\documentclass[a4paper,11pt]{article}
\usepackage{makeidx}
\usepackage[mathscr]{eucal}
\usepackage[dvips]{color}
\usepackage{amssymb, amsfonts, amsmath, amsthm, graphics} 
\usepackage{here}

\newtheorem{proposition}{Proposition}[section]
\newtheorem{theorem}{Theorem}[section]
\newtheorem{lemma}[proposition]{Lemma}
\newtheorem{remark}{Remark}[section]

\numberwithin{equation}{section}

\title{
Global dynamics above the ground state energy for the combined power-type nonlinear Schr\"odinger equations with energy-critical growth 
 at low frequencies  
}

\author{Takafumi Akahori, Slim Ibrahim, Hiroaki Kikuchi and Hayato Nawa}
\date{}
\begin{document}
\maketitle

\begin{abstract}
We consider the combined power-type nonlinear Schr\"odinger equations with 
 energy-critical growth, and study the solutions slightly above the ground state  threshold at low frequencies, so that we obtain a so-called nine-set theory developed by Nakanishi and Schlag \cite{Nakanishi-Schlag1, Nakanishi-Schlag2}.
\end{abstract}

\section{Introduction}\label{18/01/09/07:17}
We consider the following nonlinear Schr\"odinger equation: 
\begin{equation}\label{12/03/23/17:57}\tag{NLS}
i\displaystyle{\frac{\partial \psi}{\partial t}} +\Delta \psi
+|\psi|^{p-1}\psi
+|\psi|^{\frac{4}{d-2}}\psi
=0,   
\end{equation}
where $\psi=\psi(x,t)$ is a complex-valued function on $\mathbb{R}^{d}\times \mathbb{R}$ ($d\ge 3$), $\Delta$ is the Laplace operator on $\mathbb{R}^{d}$ and $p$ satisfies that 
\begin{equation}\label{09/05/13/15:03}
2+\frac{4}{d} <p+1 < 2^{*}:=2+\frac{4}{d-2} .
\end{equation}
We denote the mass and the Hamiltonian of \eqref{12/03/23/17:57} by $\mathcal{M}$ and $\mathcal{H}$, respectively, that is, 
\begin{align}
\label{13/03/30/2:01}
\mathcal{M}(u)
&:=
\frac{1}{2}\left\| u \right\|_{L^{2}}^{2},
\\[6pt]
\label{13/03/30/2:02}
\mathcal{H}(u)
&:=
\frac{1}{2}\left\|\nabla u \right\|_{L^{2}}^{2}
-
\frac{1}{p+1}\left\| u \right\|_{L^{p+1}}^{p+1}
-
\frac{1}{2^{*}}\left\|u \right\|_{L^{2^{*}}}^{2^{*}}.
\end{align}

The Cauchy problem for \eqref{12/03/23/17:57} is locally well-posed  in $H^{1}$ (see, e.g., Proposition 3.1 in \cite{Tao-Visan-Zhang}), and the mass and the Hamiltonian are conserved quantities for the flow defined by \eqref{12/03/23/17:57}. Furthermore, for any solution $\psi$ of finite variance, we have the ``virial identity'':
\begin{equation}\label{15/05/07/10:34}
\frac{d^{2}}{dt^{2}}\int_{\mathbb{R}^{d}}|x|^{2}|\psi(x,t)|^{2}\,dx
=
8\mathcal{K}(\psi(t)),
\end{equation}
where  
\begin{equation}\label{13/03/30/12:13}
\mathcal{K}(u):=
\left\|\nabla u \right\|_{L^{2}}^{2}
-
\frac{d(p-1)}{2(p+1)}\left\| u \right\|_{L^{p+1}}^{p+1}
-
\left\|u \right\|_{L^{2^{*}}}^{2^{*}}. 
\end{equation}
It is easy to see that   
\begin{equation}\label{14/09/19/16:12}
\mathcal{H}(u)> \frac{1}{2}\mathcal{K}(u).
\end{equation}

A standing wave of \eqref{12/03/23/17:57} with frequency $\omega$ is a solution  to \eqref{12/03/23/17:57} of the form $e^{it\omega}u$, so that $u$ solves 
 the following semilinear elliptic equation:

\begin{equation}\label{12/03/23/18:09} 
\omega u -\Delta u
-|u|^{p-1}u
-|u|^{\frac{4}{d-2}}u
=0,
\qquad 
u \in H^{1}(\mathbb{R}^{d})\setminus\{0\} .
\end{equation}
Moreover, a ground state of \eqref{12/03/23/18:09} is a solution to \eqref{12/03/23/18:09} of the minimal action, where by the action, we mean the functional $\mathcal{S}_{\omega}$ defined by 
\begin{equation}\label{13/03/30/2:03}
\mathcal{S}_{\omega}:= \omega \mathcal{M}+\mathcal{H}.
\end{equation}
In \cite{AIKN}, the same authors showed the existence of a ground state 
 via the variational problem   
\begin{equation}
\label{12/03/24/17:51}
m_{\omega}
:=
\inf\big\{ \mathcal{S}_{\omega}(u) \colon 
u \in H^{1}(\mathbb{R}^{d})\setminus \{0\}, \ \mathcal{K}(u)=0
\big\}.
\end{equation}
More precisely, we proved that for any $d\ge 4$, any $p$ with $2+\frac{4}{d}< p +1 < 2^{*}$ and any $\omega>0$, the variational problem associated with $m_{\omega}$ has a minimizer, and any minimizer for $m_{\omega}$ becomes a ground state of \eqref{12/03/23/18:09}. Moreover, we see from Proposition 1.2 in \cite{AIKN} that for any $d\ge 3$, any $p$ with $ 2+\frac{4}{d}< p +1 < 2^{*}$ and any $\omega>0$,   
\begin{equation}\label{12/03/23/18:17}
\begin{split}
m_{\omega}
&=
\inf\big\{ \mathcal{I}_{\omega}(u) \colon 
u \in H^{1}(\mathbb{R}^{d})\setminus \{0\}, \ \mathcal{K}(u)\le 0
\big\}
\\[6pt]
&=
\inf\big\{ \mathcal{J}_{\omega}(u) \colon 
u \in H^{1}(\mathbb{R}^{d})\setminus \{0\}, \ \mathcal{K}(u)\le 0
\big\},
\end{split}
\end{equation}
where 
\begin{align}
\label{13/03/30/12:14}
\mathcal{I}_{\omega}(u)
&:=\mathcal{S}_{\omega}(u)
-\frac{2}{d(p-1)}\mathcal{K}(u),
\\[6pt]
\label{13/03/08/10:38}
\mathcal{J}_{\omega}(u)
&:=\mathcal{S}_{\omega}(u)
-\frac{1}{2}\mathcal{K}(u)
.
\end{align}
It is worthwhile noting that $\mathcal{I}_{\omega}$  and $\mathcal{J}_{\omega}$  do not contain the $L^{p+1}$-norm and $\dot{H}^{1}$-norm, respectively: Precisely, 
\begin{align}
\label{15/12/31/11:01}
\mathcal{I}_{\omega}(u)
&=
\frac{\omega}{2}\|u \|_{L^{2}}^{2}
+
\frac{p-1-\frac{4}{d}}{2(p-1)} \|\nabla u \|_{L^{2}}^{2}
+
\frac{2^{*}-(p+1)}{2^{*}(p-1)}\|u \|_{L^{2^{*}}}^{2^{*}}
,
\\[6pt]
\label{15/12/31/11:02}
\mathcal{J}_{\omega}(u)
&=
\frac{\omega}{2}\|u \|_{L^{2}}^{2}
+
\frac{d(p-1-\frac{4}{d})}{4(p+1)}\|u \|_{L^{p+1}}^{p+1}
+
\frac{1}{d} \|u \|_{L^{2^{*}}}^{2^{*}}
.
\end{align}

\begin{remark}\label{17/12/30/14:47}
The authors have to say that Theorem 1.2 in \cite{AIKN} contains a mistake. Actually, if $d=3$ and $3<p<5$, then for any $\omega>0$, we can prove the existence of ground state of \eqref{12/03/23/18:09} (cf. Theorem 1 in \cite{Zhang-Zou}). Moreover, we find from Proposition \ref{17/12/25/17:17} in Section \ref{17/12/16/15:13} that if $d=3$ and $1<p<5$, then  there exists $\omega_{3}>0$ such that for any $0< \omega <\omega_{3}$, the equation \eqref{12/03/23/18:09} admits at least one ground state. 
\end{remark}

Our aim in this paper is to study the behavior of solutions to \eqref{12/03/23/17:57} slightly above the ``ground state threshold at low frequencies'', in the spirit of Nakanishi and Schlag \cite{Nakanishi-Schlag1, Nakanishi-Schlag2}. To this end, we need detailed information about ground states of \eqref{12/03/23/18:09} . We will see that for any  $d\ge 3$ and $2+\frac{4}{d}<p+1 <2^{*}$, there exists $\omega_{1}>0$ such that  for any $0< \omega <\omega_{1}$, a positive radial ground state $\Phi_{\omega}$ of \eqref{12/03/23/18:09} exists and unique (see Proposition \ref{13/03/16/15:33} and {\rm (i)} of Proposition \ref{13/03/29/15:30}). Furthermore, the ground state $\Phi_{\omega}$ satisfies that $\frac{d}{d\omega}\mathcal{M}(\Phi_{\omega})<0$ on $(0,\omega_{1})$ (see {\rm (iii)} of Proposition \ref{13/03/29/15:30}), so that the inverse of the mapping $\alpha \in (0,\omega_{1}) \to \mathcal{M}(\Phi_{\alpha}) \in (\mathcal{M}(\Phi_{\omega_{1}}), \infty)$ exists (we will see that $\lim_{\omega \to 0}\mathcal{M}(\Phi_{\omega})=\infty$ in \eqref{13/03/03/10:05}). 
\par 
Throughout this paper, we use the symbol $\Phi_{\omega}$ to denote the positive  radial ground state of \eqref{12/03/23/18:09}. Moreover, $\mathscr{O}(\Phi_{\omega})$ denotes the orbit     
\begin{equation}
\label{13/03/02/20:46}
\mathscr{O}(\Phi_{\omega}):=
\big\{ e^{i\theta} \Phi_{\omega} \colon \theta \in \mathbb{R} \big\}.
\end{equation}

Now,  we state our main result.
\begin{theorem}\label{15/03/22/11:28}
Assume either $d= 3$ and $3\le p < 5$, or $d\ge 4$ and $1+\frac{4}{d-1}<p< 2^{*}-1$. Then, there exists $\omega_
{*}>0$ such that for any $\omega \in (0,\omega_{*})$, there exists a 
positive function $\varepsilon_{\omega}\colon [0,\infty) \to (0,\infty)$ 
 with the following property: Set  
\begin{equation}\label{15/03/22/11:23}
\widetilde{PW}_{\omega}
:=
\big\{u \in H^{1}(\mathbb{R}^{d}) \colon 
\mathcal{S}_{\omega}(u) <m_{\omega}+\varepsilon_{\omega}(\mathcal{M}(u)) 
\big\}.
\end{equation}
Then, any radial solution $\psi$ starting from $\widetilde{PW}_{\omega}$ 
exhibits one of the following scenarios: In the statement below, $\alpha>
0$ denotes the constant such that $\mathcal{M}(\psi)=\mathcal{M}(\Phi_{\alpha})$.  
\\
{\rm (i)} Scattering both forward and backward in time;
\\
{\rm (ii)} Finite time blowup both forward and backward in time; 
\\
{\rm (iii)} Scattering forward in time, and finite time blowup backward 
in time;
\\
{\rm (iv)} Finite time blowup forward in time, and scattering backward 
in time;
\\
{\rm (v)} Trapped by $\mathscr{O}(\Phi_{\alpha})$ forward in time, and 
scattering backward in time;
\\
{\rm (vi)} Scattering forward in time, and trapped by 
$\mathscr{O}(\Phi_{\alpha})$ backward in time;
\\
{\rm (vii)} Trapped by $\mathscr{O}(\Phi_{\alpha})$ forward in time, and 
finite time blowup backward in time;
\\
{\rm (viii)} Finite time blowup forward in time, and trapped by 
 $\mathscr{O}(\Phi_{\alpha})$ backward in time;
\\ 
{\rm (ix)} Trapped by $\mathscr{O}(\Phi_{\alpha})$ both forward and 
backward in time. 
\\
Here, ``scattering forward in time'' means that the maximal lifespan of 
a solution $\psi$ is infinity and there exists $\phi \in H^{1}(\mathbb{R}
^{d})$ such that 
\begin{equation}\label{15/02/21/16:35}
\lim_{t\to \infty}\big\|\psi(t)-e^{it\Delta}\phi \big\|_{H^{1}}=0;
\end{equation}
``blowup forward in time '' means that the maximal lifespan of a 
solution is finite; and ``trapped by $\mathscr{O}(\Phi_{\alpha})$ 
forward in time'' means that 
 the maximal lifespan of a solution is infinity and the solution stays 
in some neighborhood of $\mathscr{O}(\Phi_{\alpha})$ in $H^{1}(\mathbb{R}
^{d})$ after some time. The terms corresponding to ``backward in time'' 
are used in a similar manner. 
\end{theorem}

\begin{remark}\label{17/10/29/17:28}
We continue our study in anticipation of validity of Theorem 1.1 for $d\ge 3$ and $1+\frac{4}{d} \le p < 2^{*}-1$. We describe main obstructions to the extension: 
\\
\noindent 
{\rm (1)}~The double critical case (namely, $p=1+\frac{4}{d}$) has a special difficulty that the same proof as {\rm (iii)} of Proposition \ref{13/03/29/15:30} cannot apply to derive the property \eqref{13/03/30/10:40}. In fact, the property \eqref{13/03/30/10:40} fails for the equation \eqref{15/01/29/09:41} with $p=1+\frac{4}{d}$. Moreover, in the double critical case $p=1+\frac{4}{d}$, the scaling-exponent $s_{p}$ (see \eqref{13/03/08/10:39})  vanishes ($s_{1+\frac{4}{d}}=0$), which requires some modifications in our argument, especially the proof of existence of unstable mode (Proposition \ref{13/02/20/9:41}).
\\
\noindent 
{\rm (2)}~The restriction on the exponent $p$ in Theorem \ref{15/03/22/11:28} comes from the ``one-pass theorem'' (see Theorem \ref{13/06/15/20:32} and Remark \ref{18/01/08/13:52}). 
\par 
We will undertake the extension of Theorem \ref{15/03/22/11:28} including the double critical case in a forthcoming paper as well as the high-frequency case $\omega \gg 1$.
\end{remark}

We give a proof of Theorem \ref{15/03/22/11:28} in Section \ref{15/03/22/11:50} below. In particular, we show how to construct the function $\varepsilon_{\omega}$ (see \eqref{15/05/06/09:50}). 
\par 
The equation \eqref{12/03/23/17:57} reminds us of  the following ones
\begin{align}
\label{15/05/06/11:39}%\tag{${\rm NLS}^{\dagger}$}
&i\displaystyle{\frac{\partial \psi}{\partial t}} +\Delta \psi
+|\psi|^{p-1}\psi
=0,
\\[6pt]
\label{14/01/31/11:03}%\tag{${\rm NLS}^{\ddagger}$}
&i\displaystyle{\frac{\partial \psi}{\partial t}} +\Delta \psi
+|\psi|^{\frac{4}{d-2}}\psi
=0.
\end{align}
These equations are invariant  under the following scalings, respectively: 
\begin{align}
\label{15/05/08/11:47}
&\psi(x,t) \mapsto \lambda^{-\frac{2}{p-1}} 
\psi(\frac{x}{\lambda}, \frac{t}{\lambda^{2}}),
\\[6pt]
\label{15/05/08/11:48}
&\psi(x,t) \mapsto \lambda^{-\frac{d-2}{2}} \psi(\frac{x}{\lambda}, \frac{t}{\lambda^{2}}).
\end{align}
On the other hand, there is no scaling which leaves \eqref{12/03/23/17:57} invariant. 
\par 
The stationary problems corresponding to \eqref{15/05/06/11:39} and \eqref{14/01/31/11:03} are respectively   
\begin{align}
\label{15/01/29/09:41}%\tag{$\omega$-${\rm SP}^{\dagger}$}
\omega v-\Delta v - |v|^{p-1}v&=0,
\\[6pt]
\label{12/08/22/13:34}%\tag{${\rm SP}^{\ddagger}$}
\Delta v +|v|^{\frac{4}{d-2}}v&=0.
\end{align}
Any solution $v$ to \eqref{15/01/29/09:41} gives rise to a standing wave $e^{i\omega t}v$ to \eqref{15/05/06/11:39}, and any solution to \eqref{12/08/22/13:34} solves the equation \eqref{14/01/31/11:03}. Here, we introduce the ``scaling-operator'' $T_{\omega}$ to be that for any function $f$ on $\mathbb{R}^{d}$,  
\begin{equation}\label{13/03/30/12:20}
T_{\omega} f(x):=\omega^{-\frac{1}{p-1}}f( \frac{x}{\sqrt{\omega}}).
\end{equation}
Then, putting $u:=T_{\omega}v$, we see that $u$ solves  
\begin{equation}\label{13/03/03/9:51}%\tag{${\rm SP}^{\dagger}$}
u -\Delta u -|u|^{p-1}u=0.
\end{equation}
It is well known that the equation \eqref{13/03/03/9:51} has a unique positive solution up to translations which is radially symmetric with respect to some point. We denote the positive solution symmetric about the origin by  $U$. On the other hand, the equation \eqref{12/08/22/13:34} possesses a solution of the form 
\begin{equation}\label{13/03/16/16:43}
W(x):=\bigg( \frac{\sqrt{d(d-2)}}{1+|x|^{2}}\bigg)^{\frac{d-2}{2}}.
\end{equation}
Here, a positive $C^{2}$-solution of \eqref{12/08/22/13:34}  is unique  up to 
 scalings and translations (see Corollary 8.2 in \cite{CGS}). Following Br\'ezis-Nirenberg \cite{Brezis-Nirenberg}, we introduce the variational value  
\begin{equation}\label{12/03/24/17:52}
\sigma:=\inf\big\{ 
\left\| \nabla u \right\|_{L^{2}}^{2} \colon
u \in \dot{H}^{1}(\mathbb{R}^{d}), \ 
\left\| u \right\|_{L^{2^{*}}}=1
\big\}.
\end{equation}
Then, it is known that $u=W/\|W\|_{L^{2^{*}}}$ is a minimizer of the variational problem associated with $\sigma$ and 
\begin{equation}\label{12/08/22/13:32}
\sigma^{\frac{d}{2}}
=\left\|\nabla W \right\|_{L^{2}}^{2}
=\left\|W \right\|_{L^{2^{*}}}^{2^{*}}.
\end{equation}
We also introduce the variational value
\begin{equation}\label{15/08/10/17:31}
m^{\ddagger}:=
\inf\big\{ 
\mathcal{I}^{\ddagger}(u) \colon
u \in \dot{H}^{1}(\mathbb{R}^{d})\setminus \{0\}, \ 
\mathcal{K}^{\ddagger}(u)\le 0 
\big\},
\end{equation}
where 
\begin{equation}\label{14/09/20/13:55}
\mathcal{I}^{\ddagger}(u)
:=
\mathcal{H}^{\ddagger}(u)-\frac{2}{d(p-1)} \mathcal{K}^{\ddagger}(u)
\end{equation}
with 
\begin{align}
\label{14/02/01/17:06}
\mathcal{H}^{\ddagger}(u)
&:=
\frac{1}{2}\|\nabla u\|_{L^{2}}^{2}
-
\frac{1}{2^{*}}\|u \|_{L^{2^{*}}}^{2^{*}},
\\[6pt]
\label{14/02/01/17:08}
\mathcal{K}^{\ddagger}(u)
&:=
\|\nabla u \|_{L^{2}}^{2}
-
\|u\|_{L^{2^{*}}}^{2^{*}}
.
\end{align} 
Note that 
\begin{equation}\label{18/01/30/14:39}
\mathcal{I}^{\ddagger}(u)
=
\frac{d(p-1)-4}{2d(p-1)} \|\nabla u\|_{L^{2}}^{2}
+
\frac{4-(d-2)(p-1)}{2d(p-1)}\| u\|_{L^{2^{*}}}^{2^{*}}.
\end{equation}
In particular, $\mathcal{I}^{\ddagger}(u) \ge 0$ for $1+\frac{4}{d} \le p \le 2^{*}-1$.
\par 
We can verify (see Appendix \ref{17/12/16/15:13}) that 
\begin{equation}\label{15/02/24/20:55}
m^{\ddagger}=\mathcal{I}^{\ddagger}(W)=\mathcal{H}^{\ddagger}(W)
=
\frac{1}{d}\sigma^{\frac{d}{2}}.
\end{equation}
Furthermore, we proved in \cite{AIKN} that for any $d\ge 4$, any $1+\frac{4}{d} < p < 2^{*}-1$ and any $\omega>0$,
\begin{equation}\label{12/03/23/17:48}
0<m_{\omega}
<\frac{1}{d}\sigma^{\frac{d}{2}}.
\end{equation}
We can also prove \eqref{12/03/23/17:48} for $d=3$ if $\omega>0$ is sufficiently small (see Lemma \ref{17/12/30/10:09}).
\par 
Now, we go back to  the equation \eqref{12/03/23/18:09}. If $u$ is a solution to \eqref{12/03/23/18:09}, then $v:=T_{\omega}u$ satisfies that    
\begin{equation}\label{13/03/17/18:45}
v-\Delta v -|v|^{p-1}v-\omega^{\frac{2^{*}-(p+1)}{p-1}}|v|^{\frac{4}{d-2}}v=0.
\end{equation}
We should mention that if $\omega$ is small enough, it is expected that  \eqref{12/03/23/18:09} has properties similar to \eqref{13/03/03/9:51}. Indeed, this is the heart of our analysis. 
\par 
We briefly review results for \eqref{15/05/06/11:39} and \eqref{14/01/31/11:03} which are used in this paper.  
\par 
In \cite{Kenig-Merle}, Kenig and Merle studied the equation \eqref{14/01/31/11:03} in the dimensions $d=3,4,5$, and proved the scattering of radial solutions starting from the set  
\begin{equation}\label{14/02/01/16:53}
PW_{+}^{\ddagger}
:=
\Big\{ u\in H^{1}(\mathbb{R}^{d}) \colon \mathcal{H}^{\ddagger}(u)<\mathcal{H}^{\ddagger}(W), \ 
\|\nabla u\|_{L^{2}}^{2}< \|\nabla W\|_{L^{2}}^{2} \Big\}
.
\end{equation}
Their result is extended to the higher dimensional cases by Killip and Visan \cite{Killip-Visan}. We summarize these results:
\begin{theorem}[\cite{Kenig-Merle, Killip-Visan}]\label{15/04/05/15:27}
Assume $d\ge 3$. Then, the set $PW_{+}^{\ddagger}$ is invariant under the flow defined by \eqref{14/01/31/11:03}. Furthermore, any non-trivial radial solution $\psi$ to \eqref{14/01/31/11:03} starting from $PW_{+}^{\ddagger}$ exists globally in time, and satisfies that 
\begin{align}
\label{14/02/01/16:58}
&\|\nabla \psi\|_{St(\mathbb{R})}<\infty,
\\[6pt]
\label{14/09/07/15:12}
&
\mathcal{H}^{\ddagger}(\psi)
\ge 
\frac{1}{2} \inf_{t\in \mathbb{R}}\mathcal{K}^{\ddagger}(\psi(t))
>0,
\end{align}
where $St(\mathbb{R}):=L^{\infty}(\mathbb{R},L^{2}(\mathbb{R}^{d}))\cap L^{2}(\mathbb{R},L^{2^{*}}(\mathbb{R}^{d}))$.  
\end{theorem}
The result for \eqref{15/05/06/11:39} corresponding to Theorem \ref{15/04/05/15:27} is derived by Duyckaerts, Holmer and Roudenko \cite{D-H-R, Holmer-Roudenko} for $d=p=3$. Furthermore, the first and the fourth authors extended their result to the general dimensions $d\ge 1$ and the powers $p$ satisfying \eqref{09/05/13/15:03} in \cite{Akahori-Nawa}; the result says that we have either the scattering or the blowup,  if the solutions  start from the set  
\begin{equation}\label{15/05/06/16:35}
PW^{\dagger}
:=\big\{ u \in H^{1}(\mathbb{R}^{d})\setminus\{0\} \colon 
\mathcal{H}^{\dagger}(u) 
< 
\Big( \frac{\mathcal{M}(U)}{\mathcal{M}(u)}\Big)^{\frac{1-s_{p}}{s_{p}}} 
\mathcal{H}^{\dagger}(U) 
\big\},
\end{equation}
where 
\begin{equation}\label{15/05/06/16:50}
\mathcal{H}^{\dagger}(u)
:=
\frac{1}{2} \|\nabla u \|_{L^{2}}^{2}
-
\frac{1}{p+1} \|u\|_{L^{p+1}}^{p+1},
\end{equation}
and $s_{p}$ is the ``scaling-exponent'' given by  
\begin{equation}\label{13/03/08/10:39}
s_{p}:=\frac{d}{2}-\frac{2}{p-1}.
\end{equation}
The condition \eqref{09/05/13/15:03} on $p$ implies that $0<s_{p}<1$. 
Furthermore, we can classify the behavior of the solutions by the sign of the functional $\mathcal{K}^{\dagger}$ defined by 
\begin{equation}\label{13/03/06/0:11}
\mathcal{K}^{\dagger}(u)
:=
\|\nabla u \|_{L^{2}}^{2}
-
\frac{d(p-1)}{2(p+1)}\|u\|_{L^{p+1}}^{p+1}.
\end{equation}
Note here that for any $\omega>0$, the replacement of $U$ with $T_{\omega}U$ in \eqref{15/05/06/16:35} leaves $PW^{\dagger}$ unchanged. Moreover, introducing the set $PW_{\omega}^{\dagger}$ as 
\begin{equation}\label{15/05/06/18:09}
PW_{\omega}^{\dagger}
:=
\big\{u \in H^{1}(\mathbb{R}^{d}) \setminus\{0\} \colon 
\mathcal{S}_{\omega}^{\dagger}(u) <\mathcal{S}_{\omega}^{\dagger}(U) 
\big\},
\end{equation}
we can express $PW^{\dagger}$ as the union of $PW_{\omega}^{\dagger}$ over all $\omega>0$:
\begin{equation}\label{15/05/06/18:07}
PW^{\dagger}=\bigcup_{\omega>0}PW_{\omega}^{\dagger}.
\end{equation}
Here, $\mathcal{S}_{\omega}^{\dagger}$ denotes the action for \eqref{15/01/29/09:41}, namely, 
\begin{equation}\label{13/03/05/23:59}
\mathcal{S}_{\omega}^{\dagger}(u)
:=
\omega \mathcal{M}(u)
+
\mathcal{H}^{\dagger}(u).
\end{equation}

In \cite{Nakanishi-Schlag2}, Nakanishi and Schlag considered the equation \eqref{15/05/06/11:39} in the case $d=p=3$. They developed a method to analyze the behavior of (radial) solutions starting from the set  
\begin{equation}\label{15/05/06/17:43}
PW^{\dagger,\varepsilon}:=\big\{ u \in H^{1}(\mathbb{R}^{3})\setminus \{0\} \colon 
\mathcal{H}^{\dagger}(u) 
< 
\Big( \frac{\mathcal{M}(U)}{\mathcal{M}(u)}\Big) (\mathcal{H}^{\dagger}(U)+\varepsilon ) 
\big\}
\end{equation}
for some $\varepsilon>0$. Clearly, the set $PW^{\dagger,\varepsilon}$ is an enlargement of $PW^{\dagger}$. 
In particular, they proved that there exists $\varepsilon>0$ such that all solutions starting from $PW^{\dagger,\varepsilon}$ exhibit one of the same nine scenarios as Theorem \ref{15/03/22/11:28} above. This result motivated our study. 
\par 
We note that the way to define $PW^{\dagger}$ (see \eqref{15/05/06/16:35}) is based on the scaling-invariant nature of $U$ (cf. \cite{Akahori-Nawa}). Due to lack of such a scaling property,  it is not appropriate to define the corresponding ``potential well'' for our equation \eqref{12/03/23/17:57} by simply replacing $\mathcal{H}^{\dagger}$ and $U$ with $\mathcal{H}$ and $\Phi_{\omega}$, respectively  in $PW^{\dagger}$. Instead, from the viewpoint of \eqref{15/05/06/18:07}, we consider  
 the frequency-wise ``potential well'' $PW_{\omega}$ defined by 
\begin{equation}\label{15/03/21/11:40}
PW_{\omega}
:=
\big\{u \in H^{1}(\mathbb{R}^{d}) \colon 
\mathcal{S}_{\omega}(u) <m_{\omega} 
\big\}. 
\end{equation}
This set is closely related to a variational nature of $\Phi_{\omega}$. Indeed, the definition of $m_{\omega}$ (see \eqref{12/03/24/17:51}) implies that if $u\in H^{1}(\mathbb{R}^{d})\setminus \{0\}$ and $\mathcal{K}(u)=0$, then $\mathcal{S}_{\omega}(u)\ge m_{\omega}$, so that we can 
 split $PW_{\omega}$ into three parts according to the sign of the functional $\mathcal{K}$: 
\begin{equation}\label{13/02/15/15:34}
 PW_{\omega}=PW_{\omega,+}\cup \{0\} \cup PW_{\omega,-}, 
\end{equation}
where 
\begin{align}
\label{13/02/15/15:32}
PW_{\omega,+}
&:=\big\{
u \in H^{1}(\mathbb{R}^{d}) \colon 
\mathcal{S}_{\omega}(u) <m_{\omega},\ 
\mathcal{K}(u)>0 
\big\},
\\[6pt]
\label{15/01/24/00:01}
PW_{\omega,-}
&:=
\big\{u \in H^{1}(\mathbb{R}^{d}) \colon 
\mathcal{S}_{\omega}(u) <m_{\omega},\ 
\mathcal{K}(u)<0 
\big\}.
\end{align}

The following theorem follows from the results in \cite{AIKN, AIKN2} together with the existence theorem of ground state in $\mathbb{R}^{3}$ (see  Proposition \ref{17/12/25/17:17} and Proposition \ref{17/12/30/10:09}, and \cite{Zhang-Zou}).\footnote{The papers \cite{AIKN, AIKN2} dealt with the case where $d\ge 4$ and $2+\frac{4}{d}<p+1<2^{*}$ only. However, we can easily verify that the arguments in \cite{AIKN, AIKN2} work well for the case where $d=3$ and $1+\frac{4}{3}<p<5$ as long as there exists a ground state $Q_{\omega}$ with $m_{\omega}=\mathcal{S}_{\omega}(Q_{\omega})$.}  
\begin{theorem}\label{15/03/24/16:05}
Let $\omega>0$. Assume $d\ge 3$ and $1+\frac{4}{d}<p <2^{*}-1$. Furthermore, assume $\omega \in (0,\omega_{3})$ if $d=3$, where $\omega_{3}$ is the frequency given in Proposition \ref{17/12/25/17:17}. Then, the following hold: $PW_{\omega,+}$ and $PW_{\omega,-}$ are invariant under the flow defined by \eqref{12/03/23/17:57}; and   
\\
{\rm (i)} Any radial solution $\psi$ to \eqref{12/03/23/17:57} starting from $PW_{\omega,+}$ satisfies 
\begin{align}
\label{14/02/01/16:56}
&\|\langle \nabla \rangle \psi\|_{St(\mathbb{R})}<\infty,
\\[6pt]
\label{14/09/07/15:21}
&
\mathcal{H}(\psi)
\ge 
\frac{1}{2}\inf_{t\in \mathbb{R}}\mathcal{K}(\psi(t))
>0,
\end{align}
where $St(\mathbb{R}):=L^{\infty}(\mathbb{R},L^{2}(\mathbb{R}^{d}))\cap L^{2}(\mathbb{R},L^{2^{*}}(\mathbb{R}^{d}))$. In particular, $\psi$ scatters both forward and backward in time.
\\
{\rm (ii)} Any radial solution $\psi$ to \eqref{12/03/23/17:57} starting from $PW_{\omega,-}$ satisfies 
\begin{equation}\label{17/12/02/16:59}
\sup_{t\in I_{\max}(\psi)}\mathcal{K}(\psi(t))
< 0,
\end{equation}
and blows up in a finite time both forward and backward in time, where $I_{\max}(\psi)$ denotes the maximal existence-interval of $\psi$.  
\end{theorem} 

Note that Theorem \ref{15/03/22/11:28} is an extension of Theorem \ref{15/03/24/16:05}. 
\par 
This paper is organized as follows. In Section \ref{15/05/10/16:16}, we give fundamental properties of ground state of \eqref{12/03/23/18:09}.  In Section \ref{15/03/22/11:50}, we prove Theorem \ref{15/03/22/11:28}. The main ingredient of the proof is Theorem \ref{13/03/16/12:01}. We assign preliminaries for the proof of Theorem \ref{13/03/16/12:01} to Section \ref{12/03/25/16:46} through 
 Section \ref{13/06/15/20:31}: In Section \ref{12/03/25/16:46}, we think about the decomposition of solution around the ground state in the same viewpoint of Nakanishi and Schlag \cite{Nakanishi-Schlag2}. In Section \ref{15/01/07/14:31}, we prove the ``ejection lemma'' for the equation \eqref{12/03/23/17:57}. 
 In Section \ref{15/02/08/14:28}, we introduce the distance function used in \cite{Nakanishi-Schlag2}, and derive some variational lemmas. In Section \ref{13/06/15/20:31}, we prove the ``one-pass theorem''. In Section \ref{14/01/21/14:18}, we finally prove Theorem \ref{13/03/16/12:01}. This paper also has appended sections: In Section \ref{17/12/16/15:13}, we prove the existence of ground state of \eqref{12/03/23/18:09} in three dimensions. In Section \ref{15/01/29/14:22}, we state fundamental properties of linearized operator around a ground state of \eqref{12/03/23/18:09}. In Section \ref{13/12/29/11:25}, we give two well-known inequalities for radial functions. In Section \ref{14/01/30/10:18}, we record a small-data theory, and in Section \ref{15/12/09/14:27} we give a long-time perturbation theory in a general setting. 
\\
\\
\noindent 
{\bf Notation}.~Besides the notation introduced so far, we use the following notation (see also the table of notation in Section \ref{17/12/09/16:47}): 
\\[6pt]
{\rm (i)} We use $\partial_{\omega}\Phi_{\omega}$ to denote the derivative of the positive radial grand state $\Phi_{\omega}$ of \eqref{12/03/23/18:09} with respect to $\omega$ (see Proposition \ref{13/03/29/15:30} below for the differentiability). 
\\[6pt]
\noindent 
{\rm (ii)} The daggered symbols are related to \eqref{13/03/03/9:51}: 
\begin{align}
\label{13/03/30/12:01}
&L_{+}^{\dagger}:=1-\Delta -p U^{p-1},
\\[6pt]
\label{13/04/12/12:14}
&L_{-}^{\dagger}:=1-\Delta-U^{p-1}.
\end{align}  
\noindent 
{\rm (iii)} Functionals with tilde are related to the rescaled equation \eqref{13/03/17/18:45}: 
\begin{align}
\label{13/03/17/16:24}
&\widetilde{\mathcal{K}}_{\omega}(u)
:=
\| \nabla u \|_{L^{2}}^{2}-\frac{d(p-1)}{2(p+1)}\|u\|_{L^{p+1}}^{p+1}
-\omega^{\frac{2^{*}-(p+1)}{p-1}}\|u\|_{L^{2^{*}}}^{2^{*}},
\\[6pt]
\label{13/03/17/16:25}
&\widetilde{\mathcal{I}}_{\omega}(u)
:=
\frac{1}{2}\|u\|_{L^{2}}^{2}
+
\frac{s_{p}}{d}
\| \nabla u \|_{L^{2}}^{2}
+
\omega^{\frac{2^{*}-(p+1)}{p-1}}\frac{1-s_{p}}{d}\|u\|_{L^{2^{*}}}^{2^{*}},
\\[6pt]
\label{13/03/30/11:40}
&\widetilde{L}_{\omega,+}
:= 
1-\Delta 
-p\big( T_{\omega}\Phi_{\omega}\big)^{p-1}
-
\omega^{\frac{2^{*}-(p+1)}{p-1}}(2^{*}-1)\big( T_{\omega}\Phi_{\omega} 
\big)^{\frac{4}{d-2}},
\\[6pt]
\label{13/04/12/14:40}
&\widetilde{L}_{\omega,-}
:= 
1-\Delta -\big( T_{\omega}\Phi_{\omega}\big)^{p-1}
-\omega^{\frac{2^{*}-(p+1)}{p-1}}\big( T_{\omega} \Phi_{\omega} \big)^{\frac{4}{d-2}}.
\end{align}
\noindent 
{\rm (iv)} We use $(\cdot,\cdot)_{L^{2}}$ to denote the inner product in $L^{2}(\mathbb{R}^{d})$:  
\[
(u,v)_{L^{2}}:=\int_{\mathbb{R}^{d}}u(x)\overline{v(x)}\,dx.
\]
We also use $L_{real}^{2}(\mathbb{R}^{d})$ to denote the real Hilbert space of complex-valued functions in $L^{2}(\mathbb{R}^{d})$ which is equipped with 
the inner product 
\[
( u,v )_{L_{real}^{2}}
:=
\Re \int_{\mathbb{R}^{d}}u(x)\overline{v(x)} 
\,dx .
\]
Furthermore, $H_{real}^{1}(\mathbb{R}^{d})$ denotes the real Hilbert space of functions in $H^{1}(\mathbb{R}^{d})$ equipped with the inner product 
\[
( u,v )_{H_{real}^{1}}
:=
( u,v )_{L_{real}^{2}}
+
( \nabla u, \nabla v )_{L_{real}^{2}}.
\]
{\rm (v)} We use $\langle v, u \rangle_{H^{-1},H^{1}}$ to denote the duality pair of $u\in H_{real}^{1}(\mathbb{R}^{d})$ and $v \in H_{real}^{-1}(\mathbb{R}^{d})$:
\[
\langle v, u \rangle_{H^{-1},H^{1}}:=((1-\Delta)^{-\frac{1}{2}}v, (1-\Delta)^{\frac{1}{2}} u)_{L_{real}^{2}}.
\]

\noindent 
{\rm (vi)} Let $I$ be an interval, and let $1+\frac{4}{d} \le q \le 2^{*}-1$. Then, we introduce Strichartz-type spaces on $I$ as  
\begin{align}
\label{14/01/30/12:07}
&St(I):=L_{t}^{\infty}L_{x}^{2}(I)\cap L_{t}^{2}L_{x}^{2^{*}}(I),\\[6pt]
\label{14/01/30/11:19}
&V_{q+1}(I)
:=L_{t}^{\frac{(d+2)(q-1)}{2}}L_{x}^{\frac{2d(d+2)(q-1)}{d(d+2)(q-1)-8}}(I),
\\[6pt]
\label{14/01/30/11:20}
&W_{q+1}(I):=L_{t,x}^{\frac{(d+2)(q-1)}{2}}(I)
.
\end{align}
Note here that Sobolev's embedding shows $|\nabla|^{-s_{q}}V_{q+1}(I) \hookrightarrow W_{q+1}(I)$. Moreover, by Strichartz' estimate, we mean the following estimate: for any appropriate space-time function $u$, any $t_{0}\in I$ and any pair $(q,r) \in [2,2^{*}]\times [2,\infty]$ with $\frac{2}{r}=d(\frac{1}{2}-\frac{1}{q})$, 
\begin{equation}\label{15/07/21/15:31}
\|u\|_{St(I)} 
\lesssim  
\| u(t_{0}) \|_{L^{2}}
+
\|i\displaystyle{\frac{\partial u}{\partial t}} +\Delta u \|_{L_{t}^{r'}L_{x}^{q'}(I)},
\end{equation}
where $q'$ and $r'$ denote the H\"older conjugates of $q$ and $r$ respectively.
\begin{figure}[H]
\label{15/05/10/13:27}
\input{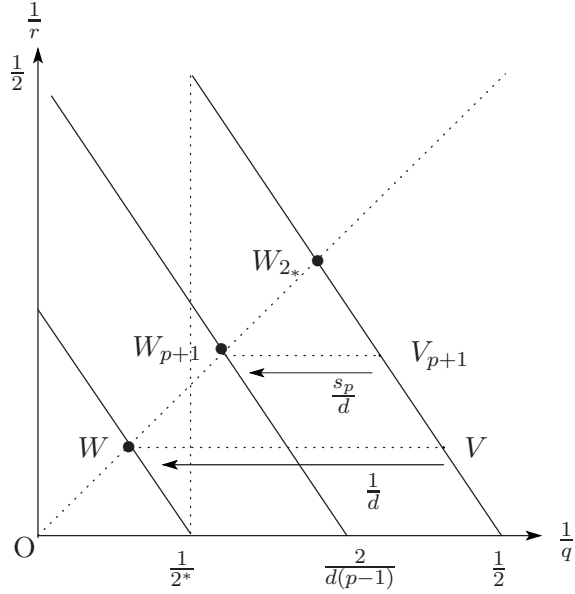}
\caption{Strichartz-type spaces}
\end{figure}

\noindent 
{\rm (vii)}  For a space-time function $u$, we define 
\begin{align}
\label{14/06/05/17:01}
F[u]&:=|u|^{p-1}u+|u|^{\frac{4}{d-2}}u,
\quad 
F^{\dagger}[u]:=|u|^{p-1}u,
\quad
F^{\ddagger}[u]:=|u|^{\frac{4}{d-2}}u,  
\\[6pt]
\label{14/05/05/17:02}
e[u]&:=i\frac{\partial u}{\partial t}+\Delta u+ F[u],
\quad 
e^{\ddagger}[u]:=i\frac{\partial u}{\partial t}+\Delta u+F^{\ddagger}[u]
.
\end{align}

%%%%%%%%%%%%%%%%%%%%%%%%%%%%%%%%%%%%%%%%%%%%%%%%%%%%%%%%%%%%%%%%%%%%%%%%%%%%%%%%

\section{Properties of Ground State}\label{15/05/10/16:16}

In this section, we give several fundamental properties of ground state of \eqref{12/03/23/18:09}. As mentioned in Section \ref{18/01/09/07:17}, for any $d\ge 4$, any $1+\frac{4}{d}<p< 2^{*}-1$ and any $\omega>0$, the equation \eqref{12/03/23/18:09} admits at least one ground state $Q_{\omega}$ satisfying $m_{\omega}=\mathcal{S}_{\omega}(Q_{\omega})$ (see Proposition 1.1 and Theorem 1.1 in \cite{AIKN}). When $d=3$, it follows from Proposition \ref{17/12/25/17:17} and Proposition \ref{17/12/30/10:09} in Appendix \ref{17/12/16/15:13} that   for any $1+\frac{4}{3}< p< 5$, there exists $\omega_{3}>0$ such that for any $0<\omega<\omega_{3}$, the equation \eqref{12/03/23/18:09} admits a ground state $Q_{\omega}$ satisfying $m_{\omega}=\mathcal{S}_{\omega}(Q_{\omega})$.

The following  proposition tells us that any ground state of \eqref{12/03/23/18:09} is essentially positive and radial: 
\begin{proposition}\label{13/03/16/15:33}
Assume $d\ge 3$ and $1+\frac{4}{d}<p<2^{*}-1$. Let $\omega>0$, and let $Q_{\omega}$ be a ground state of \eqref{12/03/23/18:09}. Then, there exist $\theta \in \mathbb{R}$, $y \in \mathbb{R}^{d}$ and a positive radial ground state $\Phi_{\omega}$ of \eqref{12/03/23/18:09} such that $\Phi_{\omega}\in C^{2}(\mathbb{R}^{d})\cap H^{1}(\mathbb{R}^{d})$, $\frac{x}{|x|}\cdot \nabla \Phi_{\omega} <0$ and $Q_{\omega}=e^{i\theta} \Phi_{\omega}(\cdot -y)$.
\end{proposition}
\begin{proof}[Proof of Proposition \ref{13/03/16/15:33}]
Let $Q_{\omega}$ be a ground state of \eqref{12/03/23/18:09}. Then, using \eqref{12/03/23/18:17}, we can verify that $|Q_{\omega}|$ is also a ground state of \eqref{12/03/23/18:09}.  We also see that $|Q_{\omega}| \in W_{loc}^{2,q}(\mathbb{R}^{d})$ for any $2\le q < \infty$ (see Theorem 2.3 in \cite{Brezis-Lieb}). Furthermore, the Schauder theory shows that $Q_{\omega} \in C^{2}(\mathbb{R}^{d})$ and $\lim_{|x|\to \infty}|Q_{\omega}(x)|= 0$. On the other hand, it follows from Theorem 9.10 in \cite{Lieb-Loss} that 
for any compact set  $K$ in $\mathbb{R}^{d}$, there exists a constant $C$ depending only on $K$ and $\omega$ such that for any $x \in K$, 
\begin{equation}\label{13/03/29/16:14}
|Q_{\omega}(x)|\ge C 
\int_{K}|Q_{\omega}(y)|\,dy .
\end{equation}
Hence, $|Q_{\omega}|$ is a positive ground state of \eqref{12/03/23/18:09}. We see from the result of Gidas, Ni and Nirenberg \cite{Gidas-Ni-Nirenberg} that there exist a positive radial function $\Phi_{\omega}$ with $\frac{x}{|x|}\cdot \nabla \Phi_{\omega}(x)<0$ for any $x\in \mathbb{R}^{d}\setminus \{0\}$, and $y \in \mathbb{R}^{d}$, such that $|Q_{\omega}|=\Phi_{\omega}(\cdot-y)$. 
 We introduce the sign of a complex-valued function $u$, denoted by ${\rm sgn}\,u$, as 
\begin{equation}\label{17/10/23/10:15}
{\rm sgn}\,u:=\frac{\overline{u}}{|u|}
.
\end{equation}
Then, we have 
\begin{equation}\label{13/04/11/12:25}
{\rm sgn}\,{\overline{Q_{\omega}}}:=\frac{Q_{\omega}}{|Q_{\omega}|}
,
\end{equation}
so that $Q_{\omega} =|Q_{\omega}| {\rm sgn}\,{\overline{Q_{\omega}}}$. We shall show that ${\rm sgn}\,{\overline{Q_{\omega}}}$ is constant, which together with  $|{\rm sgn}\,{\overline{Q_{\omega}}}|\equiv 1$  completes the proof. Since 
\begin{equation}\label{13/03/29/15:41}
\Re\big[
{\rm sgn}\,{\overline{Q_{\omega}}} \nabla {\rm sgn}\,Q_{\omega}(x)
\big]
=
\Re\bigg[
\frac{Q_{\omega}}{|Q_{\omega}|} \nabla \left( \frac{\overline{Q_{\omega}}}{|Q_{\omega}|}\right)
\bigg]
=
\frac{\Re[Q_{\omega}\overline{\nabla Q_{\omega}}]-|Q_{\omega}|\nabla |Q_{\omega}|}{|Q_{\omega}|^{2}}
=0,
\end{equation}
we have   
\begin{equation}\label{13/04/11/15:54}
\begin{split}
\|\nabla Q_{\omega} \|_{L^{2}}^{2}
&=
\big\|\nabla \bigm( |Q_{\omega}| {\rm sgn}\,{\overline{Q_{\omega}}} \bigm)  \big\|_{L^{2}}^{2}
\\[6pt]
&=
\int_{\mathbb{R}^{d}} \big| \nabla |Q_{\omega}(x)| \big|^{2}\,dx 
+
\int_{\mathbb{R}^{d}} 
\big|Q_{\omega}(x)\big|^{2}\big|\nabla {\rm sgn}\,{\overline{Q_{\omega}}}(x)\big|^{2}\,dx 
\\[6pt]
&\quad 
+
\int_{\mathbb{R}^{d}} \nabla \big| Q_{\omega}(x) \big|^{2} 
\Re\big[
{\rm sgn}\,{\overline{Q_{\omega}}} \nabla {\rm sgn}\,Q_{\omega}(x)
\big]
\,dx 
\\[6pt]
&=
\int_{\mathbb{R}^{d}} \big| \nabla |Q_{\omega}(x)| \big|^{2}\,dx 
+
\int_{\mathbb{R}^{d}} 
\big|Q_{\omega}(x)\big|^{2}\big|\nabla {\rm sgn}\,{\overline{Q_{\omega}}}(x)\big|^{2}\,dx 
.
\end{split}
\end{equation}
Hence, if $\nabla {\rm sgn}\,{\overline{Q_{\omega}}} \not \equiv 0$, then we would have 
\begin{equation}\label{13/03/29/18:17}
m_{\omega}=\mathcal{S}_{\omega}(Q_{\omega})
> \mathcal{S}_{\omega}(|Q_{\omega}|)=\mathcal{S}_{\omega}(\Phi_{\omega})=m_{\omega}.
\end{equation}
However, this is a contradiction. Thus, ${\rm sgn}\,{\overline{Q_{\omega}}}$ must be constant.
\end{proof}

Next, we give a decay property of a positive ground state.
\begin{proposition} \label{14/06/29/14:03}
Assume $d\ge 3$ and $1+\frac{4}{d}<p< 2^{*}-1$. Let $\omega>0$, and let $\Phi_{\omega}\in C^{2}(\mathbb{R}^{d})\cap H^{1}(\mathbb{R}^{d})$ be a positive radial ground state of 
\eqref{12/03/23/18:09}. Then, there exist positive constants $C(\omega)>0$ 
and $\delta(\omega) >0$ depending on $\omega$ such that  
\begin{equation}
|\Phi_{\omega}(x)|+|\nabla \Phi_{\omega}(x)|
+|\Delta \Phi_{\omega}(x)| 
\le  
C(\omega) e^{- \delta(\omega) |x|}
\end{equation}
for any $x\in \mathbb{R}^{d}$. In particular, $\Phi_{\omega} \in H^{2}(\mathbb{R}^{d})$. 
\end{proposition}
\begin{proof}[Proof of Proposition \ref{14/06/29/14:03}]
We can prove the proposition in a way similar to \cite{Berestycki-Lions}.
\end{proof}

The following convergence result justifies the intuition that the equation \eqref{12/03/23/18:09} looks like \eqref{13/03/03/9:51} for a sufficiently small $\omega$.
\begin{proposition}\label{13/03/03/9:43}
Assume $d\ge 3$ and $1+\frac{4}{d}<p < 2^{*}-1$. Let $\omega>0$, and let $\Phi_{\omega}$ be a positive radial ground state of \eqref{12/03/23/18:09}. Moreover, let $U$ be the unique positive radial ground state of \eqref{13/03/03/9:51}. Then, we have 
\begin{equation}\label{13/03/03/9:01}
\lim_{\omega \downarrow 0}
\|T_{\omega}\Phi_{\omega}-U \|_{H^{1}}
=0.
\end{equation}
\end{proposition} 
\begin{proof}[Proof of Proposition \ref{13/03/03/9:43}]
We can prove the proposition in a way similar to \cite{Fukuizumi}.
\end{proof}

\begin{remark}\label{13/03/03/10:04}
We see from Proposition \ref{13/03/03/9:43} that for any $1\le q \le 2^{*}-1$,
\begin{equation}\label{13/03/03/10:05}
\lim_{\omega \downarrow 0}\omega^{s_{p}-\frac{q-1}{p-1}}\| \Phi_{\omega}\|_{L^{q+1}}^{q+1}
=\|U\|_{L^{q+1}}^{q+1} 
\end{equation}
and 
\begin{equation}\label{18/01/09/09:35}
\lim_{\omega \downarrow 0}\omega^{s_{p}-1} \| \nabla \Phi_{\omega}\|_{L^{2}}^{2}=\|\nabla U\|_{L^{2}}^{2} 
.
\end{equation}
\end{remark}

\begin{proposition}\label{13/03/29/15:30}
Assume that $d\ge 3$ and $1+\frac{4}{d}<p <2^{*}-1$. Then, there exists $\omega_{1}>0$ with the following properties: 
\\
{\rm (i)} For any $\omega \in (0,\omega_{1})$, a positive radial ground state of \eqref{12/03/23/18:09} uniquely exists.   
\\
{\rm (ii)} The mapping $\omega \in (0, \omega_{1}) \mapsto \Phi_{\omega} \in H^{1}(\mathbb{R}^{d})$ is continuously differentiable. 
\\
{\rm (iii)} For any $\omega \in (0,\omega_{1})$, 
\begin{equation}\label{13/03/30/10:40}
\frac{d}{d\omega}\mathcal{M}(\Phi_{\omega}) 
=
\big(\Phi_{\omega}, \partial_{\omega}\Phi_{\omega} \big)_{L_{real}^{2}}<0.
\end{equation}
Here, $\Phi_{\omega}$ denotes the unique positive radial ground state of \eqref{12/03/23/18:09}.  
\end{proposition}

In order to prove the claim {\rm (i)} in Proposition \ref{13/03/29/15:30}, we need the following result (see, e.g., Theorem 0.2 in \cite{Kabeya-Tanaka}, \cite{Ni-Takagi} and Proposition 2.8 in \cite{Weinstein}): 
\begin{lemma}\label{13/03/36/15:35}
Assume either $d=1, 2$ and $1<p< \infty$, or $d\ge 3$ and $1<p <2^{*}-1$. Then, we have 
\begin{equation}\label{13/03/29/15:40}
{\rm Ker}\,L_{+}^{\dagger}={\rm span}\big\{\partial_{1}U,\ldots, \partial_{d}U \big\},
\end{equation}  
where $L_{+}^{\dagger}$ is the operator defined by \eqref{13/03/30/12:01}, and  $\partial_{1}U,\ldots, \partial_{d}U$ denote the partial derivatives of $U$.
\end{lemma}

Now, we give a proof of the claim {\rm (i)}: 
\begin{proof}[Proof of {\rm (i)} in Proposition \ref{13/03/29/15:30}]
It suffices to prove the uniqueness of the ground state. Suppose for contradiction that for any $k\in \mathbb{N}$, there exists $\omega_{k}\in (0, \frac{1}{k})$ such that the equation  \eqref{12/03/23/18:09} has two different positive radial ground states, say $U_{\omega_{k}}$ and $V_{\omega_{k}}$. Put $\widetilde{U}_{k}:=T_{\omega_{k}}U_{\omega_{k}}$, 
 $\widetilde{V}_{k}:=T_{\omega_{k}}V_{\omega_{k}}$
and  
\begin{equation}\label{13/03/29/11:24}
\widetilde{u}_{k}:=\frac{\widetilde{U}_{k}-\widetilde{V}_{k}}{\|\widetilde{U}_{k}-\widetilde{V}_{k}\|_{H^{1}}}.
\end{equation}
Then, $\widetilde{u}_{k}$ satisfies the equation  
\begin{equation}\label{13/03/29/11:26}
\widetilde{u}_{k}-\Delta \widetilde{u}_{k}
-
\frac{\widetilde{U}_{k}^{p}-\widetilde{V}_{k}^{p}}{\|\widetilde{U}_{k}
-\widetilde{V}_{k}\|_{H^{1}}}
-
\omega_{k}^{\frac{2^{*}-(p+1)}{p-1}}
\frac{\widetilde{U}_{k}^{2^{*}-1}-\widetilde{V}_{k}^{2^{*}-1}}{\|\widetilde{U}_{k}-\widetilde{V}_{k}\|_{H^{1}}}
=0
.
\end{equation}
Moreover, it follows from $\widetilde{u}_{k}\in H_{rad}^{1}(\mathbb{R}^{d})$ and $\|\widetilde{u}_{k}\|_{H^{1}}=1$ for any $k\in \mathbb{N}$ that there exists a radial function $\widetilde{u} \in H^{1}(\mathbb{R}^{d})$ such that, passing to a subsequence, we have 
\begin{equation}\label{13/03/29/11:36}
\lim_{k\to \infty}\widetilde{u}_{k}=\widetilde{u} 
\qquad 
\mbox{weakly in $H^{1}(\mathbb{R}^{d})$},
\end{equation}
and for any $2< q <2^{*}$, 
\begin{equation}\label{13/03/29/14:43}
\lim_{k\to \infty}\widetilde{u}_{k}=\widetilde{u} 
\qquad 
\mbox{strongly in $L^{q}(\mathbb{R}^{d})$}.
\end{equation}
We shall show that $\widetilde{u}\equiv 0$. First note that the fundamental theorem of calculus together with Proposition \ref{13/03/03/9:43} and \eqref{13/03/29/14:43} implies that for any $\phi \in C_{c}^{\infty}(\mathbb{R}^{d})$, 
\begin{equation}\label{13/03/29/12:11}
\begin{split}
&\lim_{k\to \infty}\bigm\langle \frac{\widetilde{U}_{k}^{p}-\widetilde{V}_{k}^{p}}{\big\|\widetilde{U}_{k}-\widetilde{V}_{k}\big\|_{H^{1}}},  \phi \bigm\rangle_{H^{-1},H^{1}}  
\\[6pt]
&= 
\lim_{k\to \infty}
p \int_{0}^{1}
\bigm\langle 
\big\{ \widetilde{V}_{k}+\theta (\widetilde{U}_{k}-\widetilde{V}_{k})\big\}^{p-1} 
\widetilde{u}_{k},  \phi \bigm\rangle_{H^{-1},H^{1}}
d\theta 
\\[6pt]
&=
p\int_{\mathbb{R}^{d}}U^{p-1} \widetilde{u} \phi \,dx
.
\end{split}
\end{equation}
We can also verify that for a given function $\phi \in C_{c}^{\infty}(\mathbb{R}^{d})$, 
\begin{equation}\label{13/03/29/14:10}
\lim_{k\to \infty}
\omega_{k}^{\frac{2^{*}-(p+1)}{p-1}}\bigm\langle \frac{\widetilde{U}_{k}^{2^{*}-1}-\widetilde{V}_{k}^{2^{*}-1}}{\|\widetilde{U}_{k}-\widetilde{V}_{k}\|_{H^{1}}},  \phi \bigm\rangle_{H^{-1},H^{1}}  
=0.
\end{equation}
Combining \eqref{13/03/29/11:26} with \eqref{13/03/29/11:36}, \eqref{13/03/29/12:11} and \eqref{13/03/29/14:10}, we find that for any $\phi \in C_{c}^{\infty}(\mathbb{R}^{d})$, 
\begin{equation}\label{13/03/29/14:15}
\bigm\langle 
L_{+}^{\dagger}\widetilde{u}, \,\phi  
\bigm\rangle_{H^{-1},H^{1}} 
=
\bigm\langle 
\widetilde{u} -\Delta \widetilde{u}
-p U^{p-1}\widetilde{u}, \, \phi  
\bigm\rangle_{H^{-1},H^{1}} 
=
0.
\end{equation}
Furthermore, since $\widetilde{u}$ is a radial $H^{1}$-solution of $L_{+}^{\dagger}\widetilde{u}=0$, we find that $\widetilde{u} \in H_{rad}^{2}(\mathbb{R}^{d})$. Hence, 
 it follows from Lemma \ref{13/03/36/15:35} (there is no radial function in the kernel of $L_{+}^{\dagger}$) that  $\widetilde{u} \equiv 0$. 
\par 
Now, multiplying the equation \eqref{13/03/29/11:26} by $\widetilde{u}_{k}$ and then integrating the resulting equation over $\mathbb{R}^{d}$, we obtain 
\begin{equation}\label{13/03/29/14:30}
\|\widetilde{u}_{k}\|_{H^{1}}^{2}
=
\int_{\mathbb{R}^{d}}
\frac{\widetilde{U}_{k}^{p}-\widetilde{V}_{k}^{p}}{\|\widetilde{U}_{k}
-\widetilde{V}_{k}\|_{H^{1}}}\widetilde{u}_{k}\,dx 
+
\omega_{k}^{\frac{2^{*}-(p+1)}{p-1}}
\int_{\mathbb{R}^{d}}
\frac{\widetilde{U}_{k}^{2^{*}-1}-\widetilde{V}_{k}^{2^{*}-1}}{\|\widetilde{U}_{k}-\widetilde{V}_{k}\|_{H^{1}}}
\widetilde{u}_{k}\,dx .
\end{equation}
Then, it follows from \eqref{13/03/29/14:43} and $\widetilde{u}\equiv 0$ that the right-hand side of \eqref{13/03/29/14:30} tends to $0$ as $k\to \infty$, whereas the left-hand side is identically  $1$ (see \eqref{13/03/29/11:24}). This is a contradiction. Thus, we have completed the proof of the claim {\rm (i)} in Proposition \ref{13/03/29/15:30}. 
\end{proof}

Next, we mention how to prove the claim {\rm (ii)} in Proposition \ref{13/03/29/15:30}. 
\begin{proof}[Proof of {\rm (ii)} in Proposition \ref{13/03/29/15:30}]
The claim {\rm (ii)} follows from the result by Shatah and Strauss \cite{Shatah-Strauss} (see also \cite{Killip-Oh-Poco-Visan}).
\end{proof}

Finally, we move on to the proof of the claim {\rm (iii)} in Proposition \ref{13/03/29/15:30}. Let us notice that the claim {\rm (iii)} is an analogy to that 
\begin{equation}\label{15/01/28/10:05}
\frac{d}{d \omega}\mathcal{M}(T_{\omega^{-1}} U) 
=
\frac{1}{2}\frac{d}{d\omega}\big(\omega^{-s_{p}} \|U \|_{L^{2}}^{2}
\big)
=-\frac{s_{p}}{2}\omega^{-s_{p}-1} \|U \|_{L^{2}}^{2}<0. 
\end{equation}
Hence, it is convenient to introduce   a function $\Lambda U$ defined by    
\begin{equation}\label{13/03/31/16:03}
\Lambda U(x):=
\partial_{\omega}\{T_{\omega^{-1}}U(x)\} \big|_{\omega=1}
=
\frac{1}{p-1}U(x)+ \frac{1}{2}x\cdot \nabla U(x).
\end{equation}
Then, we can verify that $\Lambda U$ obeys  
\begin{equation}\label{13/03/31/16:08}
L_{+}^{\dagger}\Lambda U
=
-U.
\end{equation}
Differentiation of the both sides of the equation \eqref{12/03/23/18:09} with respect to $\omega$ yields   
\begin{equation}\label{13/03/30/10:50}
\omega \partial_{\omega}\Phi_{\omega}-\Delta \partial_{\omega} \Phi_{\omega}
-p\Phi_{\omega}^{p-1} \partial_{\omega}\Phi_{\omega}
-(2^{*}-1)\Phi_{\omega}^{\frac{4}{d-2}} \partial_{\omega}\Phi_{\omega}
=
-\Phi_{\omega}.
\end{equation}
Furthermore, we see from 
 \eqref{13/03/30/10:50} that  
\begin{equation}\label{13/03/30/10:57}
\widetilde{L}_{\omega,+}(\omega T_{\omega} \partial_{\omega}\Phi_{\omega})
=
-T_{\omega}\Phi_{\omega},
\end{equation}
where $\widetilde{L}_{\omega,+}$ is the operator defined by \eqref{13/03/30/11:40}. We state a property of the operator $\widetilde{L}_{\omega,+}$. 
\begin{lemma}\label{13/03/31/16:12}
Assume $d\ge 3$ and $1+\frac{4}{d}<p< 2^{*}-1$. Then, there exist $\omega_{0}>0$ and $C>0$ 
such that for any $\omega \in (0, \omega_{0})$ and any $f \in H_{rad}^{2}(\mathbb{R}^{d})$, 
\begin{equation}\label{15/04/22/11:53}
\| \widetilde{L}_{\omega,+}f \|_{L^{2}}
\ge C \| f\|_{H^{1}}. 
\end{equation}
\end{lemma} 
\begin{proof}[Proof of Lemma \ref{13/03/31/16:12}]
Suppose for contradiction that for any $k\in \mathbb{N}$, there exist $\omega_{k}\in (0, \frac{1}{k})$ and $f_{k} \in H_{rad}^{2}(\mathbb{R}^{d})$ with $\|f_{k}\|_{H^{1}}=1$ such that 
\begin{equation}\label{13/03/31/16:24}
\| \widetilde{L}_{\omega_{k},+}f_{k} \|_{L^{2}}
< \frac{1}{k}. 
\end{equation}
Then, we can take $f_{\infty} \in H_{rad}^{1}(\mathbb{R}^{d})$ such that, passing to some subsequence,  
\begin{equation}\label{13/03/31/16:40}
\lim_{k\to \infty}f_{k}=f_{\infty}
\qquad 
\mbox{weakly in $H^{1}(\mathbb{R}^{d})$},
\end{equation}
and for any $2<q<2^{*}$, 
\begin{equation}\label{13/03/31/21:52}
\lim_{k\to \infty}f_{k}=f_{\infty}
\qquad 
\mbox{strongly in $L^{q}(\mathbb{R}^{d})$}.
\end{equation}
Furthermore, for any $\phi \in C_{c}^{\infty}(\mathbb{R}^{d})$, we have 
\begin{equation}\label{13/03/31/16:45}
\begin{split}
&\big| 
\langle L_{+}^{\dagger}f_{\infty}, \phi \rangle_{H^{-1},H^{1}}
\big| 
\\[6pt]
&\le 
\big| \bigm\langle \widetilde{L}_{\omega_{k},+}f_{k}, \phi 
\bigm\rangle_{H^{-1},H^{1}}  \big|
+
\big| 
\bigm\langle \widetilde{L}_{\omega_{k},+}\big(f_{\infty}-f_{k} \big), \phi 
\bigm\rangle_{H^{-1},H^{1}} 
\big|
\\[6pt]
&\quad 
+
p\big| 
\big\langle \big\{ (T_{\omega_{k}}\Phi_{\omega_{k}})^{p-1}-U^{p-1}\big\} f_{\infty}, \, \phi 
\big\rangle_{H^{-1},H^{1}}  \big|
\\[6pt]
&\quad +
\omega_{k}^{\frac{2^{*}-(p+1)}{p-1}} (2^{*}-1)
\big| 
\big\langle (T_{\omega_{k}} \Phi_{\omega_{k}})^{\frac{4}{d-2}}f_{\infty}, \, \phi \big\rangle_{H^{-1},H^{1}}  \big| 
.
\end{split}
\end{equation}
This together with the hypothesis \eqref{13/03/31/16:24}, \eqref{13/03/31/16:40}, \eqref{13/03/31/21:52}, Proposition \ref{13/03/03/9:43} and $\lim_{k\to \infty}\omega_{k}=0$ shows that $L_{+}^{\dagger}f_{\infty}=0$ in the distribution sense. Since $({\rm Ker}\,L_{+}^{\dagger})\big|_{H_{rad}^{1}}=\{0\}$ (see Lemma \ref{13/03/36/15:35}), we conclude that $f_{\infty}=0$. Furthermore, this fact $f_{\infty}=0$ together with $\|f_{k}\|_{H^{1}} = 1$, \eqref{13/03/31/21:52} and $\lim_{k\to \infty}\omega_{k}=0$ gives us that 
\begin{equation} \label{14/06/27/14:48}
\lim_{k \to \infty} 
( \widetilde{L}_{\omega_{k},+} f_{k}, f_{k} )_{L^{2}} 
= 1.
\end{equation}
However, it follows from the hypothesis \eqref{13/03/31/16:24} that 
\begin{equation}\label{15/01/24/16:21}
\lim_{k\to \infty}\big| ( \widetilde{L}_{\omega_{k},+} f_{k}, f_{k} )_{L^{2}}\big|
\le 
\lim_{k\to \infty}\|\widetilde{L}_{\omega_{k},+} f_{k}\|_{L^{2}}=0. 
\end{equation}
This is a contradiction. Thus, we have proved that \eqref{15/04/22/11:53} holds.\end{proof}

\begin{lemma}\label{14/06/27/15:03}
Assume $d\ge 3$ and $1+\frac{4}{d}<p < 2^{*}-1$. Then, we have 
\begin{equation} \label{15/04/22/11:50}
\lim_{\omega \downarrow 0} \| \omega T_{\omega} \partial_{\omega}\Phi_{\omega} - \Lambda U \|_{H^{1}}=0.
\end{equation}
\end{lemma}
\begin{proof}[Proof of Lemma \ref{14/06/27/15:03}]
Let $\omega_{0}>0$ and $C>0$ be constants given in Lemma \ref{13/03/31/16:12}. Furthermore, let $\{\omega_{n}\}$ be a sequence in $(0, \omega_{0})$ with $\lim_{n \to \infty}\omega_{n} = 0$, and set 
\begin{equation}\label{17/10/25/10:50}
 \Lambda \Phi_{\omega_{n}}:=\partial_{\omega}\Phi_{\omega}|_{\omega=\omega_{n}}
.
\end{equation}
Then,  we find from Proposition \ref{13/03/31/16:12}, \eqref{13/03/30/10:57} and Proposition \ref{13/03/03/9:43} that there exists a number $N$ such that for any $n\ge N$, 
\begin{equation} \label{14/06/27/15:09}
C\|\omega_{n} T_{\omega_{n}}\Lambda \Phi_{\omega_{n}} \|_{H^{1}} 
\le 
\big\|\widetilde{L}_{\omega_{n}, +} (\omega_{n} T_{\omega_{n}} 
\Lambda \Phi_{\omega_{n}}) \big\|_{L^{2}}
= 
\|T_{\omega_{n}} \Phi_{\omega_{n}}\|_{H^{1}} 
\le 
2 \|U\|_{H^{1}}.
\end{equation}
Thus, $\{\omega_{n} T_{\omega_{n}}\Lambda \Phi_{\omega_{n}} \}$ is bounded in $H^{1}(\mathbb{R}^{d})$ and therefore we can take $g_{\infty} \in H^{1}(\mathbb{R}^{d})$ such that, passing to some subsequence,  
\begin{equation}\label{14/10/07/11:23}
\lim_{n \to \infty} \omega_{n} T_{\omega_{n}} \Lambda \Phi_{\omega_{n}}
= 
g_{\infty} 
\qquad \mbox{weakly in $H^{1}(\mathbb{R}^{d})$}.
\end{equation} 
This together with Proposition \ref{13/03/03/9:43} also shows that for any 
$\phi \in C^{\infty}_{0}(\mathbb{R}^{d})$, 
\begin{equation}\label{14/06/27/15:16}
\lim_{n \to \infty} \langle \widetilde{L}_{\omega_{n},+} (\omega_{n} T_{\omega_{n}} \Lambda \Phi_{\omega_{n}}), \phi \rangle_{H^{-1},H^{1}} 
= 
\langle L^{\dagger}_{+}g_{\infty}, \phi \rangle_{H^{-1},H^{1}}
. 
\end{equation}
On the other hand, we see from \eqref{13/03/30/10:57} and Proposition \ref{13/03/03/9:43} that 
\begin{equation}\label{14/10/07/11:37}
\lim_{n \to \infty} \widetilde{L}_{\omega_{n},+} (\omega_{n} T_{\omega_{n}} 
\Lambda \Phi_{\omega_{n}}) = -U \qquad \mbox{strongly in $H^{1}(\mathbb{R}^{d})$}. 
\end{equation}
Putting \eqref{14/06/27/15:16} and \eqref{14/10/07/11:37} together, we find that  $L^{\dagger}_{+} g_{\infty} = -U$. Furthermore, since $({\rm Ker}\,L_{+}^{\dagger})\big|_{H_{rad}^{1}}=\{0\}$ (see Lemma \ref{13/03/36/15:35}), this identity together with  \eqref{13/03/31/16:08} shows that $g_{\infty} = \Lambda U$. It remains to show that 
\begin{equation}\label{15/06/16/16:02}
\lim_{n \to \infty} \| \omega_{n} T_{\omega_{n}} \Lambda \Phi_{\omega_{n}} \|_{H^{1}}
= 
\|\Lambda U \|_{H^{1}}
.
\end{equation}
We see from \eqref{13/03/30/10:57}, Proposition \ref{14/06/29/14:03}, Proposition \ref{13/03/03/9:43}, \eqref{14/10/07/11:23} with $g_{\infty}=\Lambda U$ and  \eqref{13/03/31/16:08} that 
\begin{equation}\label{15/07/11/16:15}
\begin{split}
&\lim_{n \to \infty} \| \omega_{n} T_{\omega_{n}} \Lambda \Phi_{\omega_{n}} \|_{H^{1}}^{2}
\\[6pt]
&=
-\lim_{n \to \infty}( T_{\omega_{n}} \Phi_{\omega_{n}},
\, \omega_{n} T_{\omega_{n}} \Lambda \Phi_{\omega_{n}}
)_{L_{real}^{2}}
\\[6pt]
&\quad +
\lim_{n \to \infty} p ((T_{\omega_{n}}\Phi_{\omega_{n}})^{p-1}\omega_{n} T_{\omega_{n}} \Lambda \Phi_{\omega_{n}}, \omega_{n} T_{\omega_{n}} \Lambda \Phi_{\omega_{n}} )_{L_{real}^{2}}
\\[6pt]
&\quad +
\lim_{n \to \infty} \omega_{n}^{\frac{2^{*}-(p+1)}{p-1}}(2^{*}-1)
((T_{\omega_{n}}\Phi_{\omega_{n}})^{\frac{4}{d}}\omega_{n} T_{\omega_{n}} 
\Lambda \Phi_{\omega_{n}}, 
\omega_{n} T_{\omega_{n}} \Lambda \Phi_{\omega_{n}} )_{L_{real}^{2}}
\\[6pt]
&= 
-( U, \Lambda U)_{L_{real}^{2}}+p( U^{p-1}\Lambda U, \Lambda U)_{L_{real}^{2}}
\\[6pt]
&=
( L_{+}^{\dagger}\Lambda U, \Lambda U)_{L_{real}^{2}}+p( U^{p-1}\Lambda U, \Lambda U)_{L_{real}^{2}}
=
\|\Lambda U\|_{H^{1}}^{2}.
\end{split}
\end{equation}
Thus, we have completed the proof. 
\end{proof}

Now, we are in a position to prove the claim {\rm (iii)} in Proposition \ref{13/03/29/15:30}. 
\begin{proof}[Proof of {\rm (iii)} in Proposition \ref{13/03/29/15:30}]
We see from Proposition \ref{13/03/03/9:43}, \eqref{14/06/27/15:03} and \eqref{13/03/31/16:03} that 
\begin{equation} \label{14/06/27/15:34}
\begin{split}
\lim_{\omega \to 0} 
\omega^{s_{p} +1} 
\frac{d}{d \omega} \mathcal{M}(\Phi_{\omega})  
&= 
\lim_{\omega \to 0} \omega^{s_{p} +1} 
(\Phi_{\omega}, \Phi'_{\omega})_{L^{2}_{real}} 
= 
\lim_{\omega \to 0}(T_{\omega} \Phi_{\omega}, \omega T_{\omega}\partial_{\omega}\Phi_{\omega}
)_{L^{2}_{real}} \\ 
&= 
(U, \Lambda U)_{L^{2}_{real}} 
= -\frac{s_{p}}{2} \|U\|^{2}_{L^{2}} < 0,
\end{split}
\end{equation}
which gives us the desired result. 
\end{proof}

%%%%%%%%%%%%%%%%%%%%%%%%%%%%%%%%%%%%%%%%%%%%%%%%%%%%%%%%%%%%%%%%%%%%%%%%%%%%%%%%
\section{Proof of Theorem \ref{15/03/22/11:28}}\label{15/03/22/11:50}

In order to prove Theorem \ref{15/03/22/11:28}, for a given $\varepsilon \ge  0$,  we consider the set 
\begin{equation}\label{13/03/02/20:21}
A_{\omega}^{\varepsilon}
:=
\big\{ 
u \in H^{1}(\mathbb{R}^{d}) \colon 
\mathcal{S}_{\omega}(u) < m_{\omega}+\varepsilon,
\, \mathcal{M}(u)=\mathcal{M}(\Phi_{\omega})
\big\}.
\end{equation}
Then, a key fact to prove Theorem \ref{15/03/22/11:28} is the following: 
\begin{theorem}\label{13/03/16/12:01}
Assume either $d=3$ and $3\le p < 5$, or $d\ge 4$ and $1+\frac{4}{d-1}<p< 2^{*}-1$. Then, there exists $\omega_{*}>0$ with the following property: for any $\omega \in (0,\omega_{*})$, there exists a positive constant $\varepsilon(\omega)$ such that  all radial solutions starting from $A_{\omega}^{\varepsilon(\omega)}$ exhibit one of the following scenarios: 
\\
{\rm (i)} Scattering both forward and backward in time;
\\
{\rm (ii)} Finite time blowup both forward and backward in time; 
\\
{\rm (iii)} Scattering forward in time, and finite time blowup backward 
in time;
\\
{\rm (iv)} Finite time blowup forward in time, and scattering backward 
in time;
\\
{\rm (v)} Trapped by $\mathscr{O}(\Phi_{\omega})$ forward in time, and 
scattering backward in time;
\\
{\rm (vi)} Scattering forward in time, and trapped by 
$\mathscr{O}(\Phi_{\omega})$ backward in time;
\\
{\rm (vii)} Trapped by $\mathscr{O}(\Phi_{\omega})$ forward in time, and 
finite time blowup backward in time;
\\
{\rm (viii)} Finite time blowup forward in time, and trapped by 
 $\mathscr{O}(\Phi_{\omega})$ backward in time;
\\ 
{\rm (ix)} Trapped by $\mathscr{O}(\Phi_{\omega})$ both forward and 
backward in time. 
\end{theorem}
We give a proof of this theorem in Section \ref{14/01/21/14:18}. 
\par 
In order to prove Theorem \ref{15/03/22/11:28}, we also use the following fact:
\begin{lemma}\label{15/03/22/14:59}
Assume $d\ge 3$ and $1+\frac{4}{d}<p< 2^{*}-1$. Let $\omega_{1}$ be the frequency given by Proposition \ref{13/03/29/15:30}. Then, we have the following:
\\ 
{\rm (i)} $m_{\omega}$ is differentiable on $(0,\omega_{1})$, and 
\begin{equation}\label{15/03/22/15:04}
\frac{dm_{\omega}}{d\omega}=\mathcal{M}(\Phi_{\omega}) 
\end{equation} 
for all $\omega \in (0,\omega_{1})$. In particular, $m_{\omega}$ is strictly increasing on  $(0, \omega_{1})$.
\\
{\rm (ii)} $\frac{m_{\omega}}{\omega}$ is differentiable and strictly decreasing  on $(0, \omega_{1})$. 
\\[6pt]
{\rm (iii)} Let $0<\alpha<\beta <\omega_{1}$. Then, we have that 
\begin{equation}\label{15/03/22/15:00}
\mathcal{M}(\Phi_{\beta})
<
\frac{m_{\beta}-m_{\alpha}}{\beta-\alpha}
<
\mathcal{M}(\Phi_{\alpha}). 
\end{equation}
\end{lemma}
\begin{proof}[Proof of Lemma \ref{15/03/22/14:59}]
We shall prove the claim {\rm (i)}. Since $\mathcal{S}'(\Phi_{\omega})=0$ and $\Phi_{\omega}$ is differentiable with respect to $\omega$ on $(0,\omega_{1})$ (see {\rm (ii)} of Proposition \ref{13/03/29/15:30}), we see that for any $\omega \in (0,\omega_{1})$, 
\begin{equation}\label{15/03/24/17:01}
\frac{dm_{\omega}}{d\omega}
=
\frac{d}{d\omega}\mathcal{S}_{\omega}(\Phi_{\omega})
=
\mathcal{M}(\Phi_{\omega})
+
\mathcal{S}_{\omega}'(\Phi_{\omega})\partial_{\omega} \Phi_{\omega} 
=
\mathcal{M}(\Phi_{\omega}). 
\end{equation}
Thus, we have proved the first claim. 
\par 
Next, we shall prove the second claim {\rm (ii)}. The differentiability of $\frac{m_{\omega}}{\omega}$ follows from the first claim. Furthermore, we conclude from {\rm (ii)} of Proposition \ref{13/03/29/15:30} that  for any $\omega \in (0,\omega_{1})$, 
\begin{equation}\label{15/04/16/16:20}
\frac{d}{d\omega}\Big( \frac{m_{\omega}}{\omega} \Big)
=\frac{d}{d\omega}\Big( \frac{\mathcal{S}_{\omega}(\Phi_{\omega})}{\omega} \Big)
=
\dfrac{ \mathcal{S}_{\omega}'(\Phi_{\omega}) \partial_{\omega}\Phi_{\omega}
 \omega -\mathcal{S}_{\omega}(\Phi_{\omega})}{\omega^{2}}
=
-\frac{\mathcal{S}_{\omega}(\Phi_{\omega})}{\omega^{2}}
<0.
\end{equation}
Thus, we find that $\frac{m_{\omega}}{\omega}$ is strictly decreasing on $(0,\omega_{1})$. 
\par 
Finally, we shall prove the last claim {\rm (iii)}. It follows from the mean value theorem 
 and \eqref{15/03/22/15:04} that for any $0<\alpha <\beta <\omega_{1}$, there exists $\theta \in (0,1)$ such that 
\begin{equation}\label{15/04/16/16:54}
\frac{m_{\beta}-m_{\alpha}}{\beta-\alpha}
= 
\mathcal{M}(\Phi_{\alpha+\theta(\beta -\alpha)}).
\end{equation}
On the other hand  we see from {\rm (iii)} of Proposition \ref{13/03/29/15:30} that for any $\theta \in (0,1)$, 
\begin{equation}\label{15/04/16/17:34}
\mathcal{M}(\Phi_{\beta})<\mathcal{M}(\Phi_{\alpha+\theta(\beta -\alpha)})
<\mathcal{M}(\Phi_{\alpha}). 
\end{equation}
Putting \eqref{15/04/16/16:54} and \eqref{15/04/16/17:34} together, we obtain the desired result \eqref{15/03/22/15:00}. 
\end{proof}

We see from \eqref{13/03/03/10:05} with $q=1$ and {\rm (iii)} of Proposition \ref{13/03/29/15:30} that  there is a strictly decreasing function $\alpha \colon (\mathcal{M}(\Phi_{\omega_{1}}), \infty) \to (0,\omega_{1})$ such that $M=\mathcal{M}(\Phi_{\alpha(M)})$ for any $M\in (\mathcal{M}(\Phi_{\omega_{1}}), \infty)$; $\alpha$ is the inverse of the function $\omega \mapsto \mathcal{M}(\Phi_{\omega})$. Let $\omega_{*}$ be the frequency given by Theorem \ref{13/03/16/12:01}. Then, for each $\omega \in (0,\omega_{*})$, we define a positive function $\varepsilon_{\omega} \colon [0, \infty) \to (0,\infty)$ to be that for any $M\ge 0$,
\begin{equation}\label{15/05/06/09:50}
\varepsilon_{\omega}(M)
:= 
\left\{ \begin{array}{lll} 
m_{\omega_{*}}-m_{\omega}-(\omega_{*}-\omega) \mathcal{M}(\Phi_{\omega_{*}}) &\mbox{if}& M\le \mathcal{M}(\Phi_{\omega_{*}}),
\\[9pt]
\varepsilon(\alpha(M))
+
m_{\alpha(M)}-m_{\omega}- ( \alpha(M)-\omega ) M
&\mbox{if}& \mathcal{M}(\Phi_{\omega_{*}})<M<\mathcal{M}(\Phi_{\omega}),
\\[9pt]
\varepsilon(\omega) &\mbox{if}& M=\mathcal{M}(\Phi_{\omega}), 
\\[9pt]
\varepsilon(\alpha(M))
+
(\omega -\alpha(M) ) M - (m_{\omega}-m_{\alpha(M)}) &\mbox{if}& M>\mathcal{M}(\Phi_{\omega}),
\end{array} \right.
\end{equation}
where $\varepsilon(\alpha(M))$ and $\varepsilon(\omega)$ are the positive constants given by Theorem \ref{13/03/16/12:01}.  

\begin{remark}\label{18/02/04/16:42}
It follows from {\rm (iii)} of Lemma \ref{13/03/16/12:01} that $\varepsilon_{\omega}(M)$ is positive for all $M\ge  0$. Furthermore, the continuity of $\varepsilon(\omega)$ with respect to $\omega$ implies that 
\begin{equation}\label{18/02/04/16:43}
\inf_{M\ge 0}\varepsilon_{\omega}(M)>0.
\end{equation}
\end{remark}

Now, we are in a position to prove Theorem \ref{15/03/22/11:28}. 
\begin{proof}[Proof of Theorem \ref{15/03/22/11:28}]
Let $\omega_{*}$ be the frequency found in Theorem \ref{13/03/16/12:01}, and let $\omega \in (0,\omega_{*})$. Furthermore, let $\varepsilon_{\omega}$ be the positive function define by \eqref{15/05/06/09:50}, and let $\widetilde{PW}_{\omega}$ be the set defined by \eqref{15/03/22/11:23}. 

\begin{figure}[H]\label{15/05/17/16:33}
\input{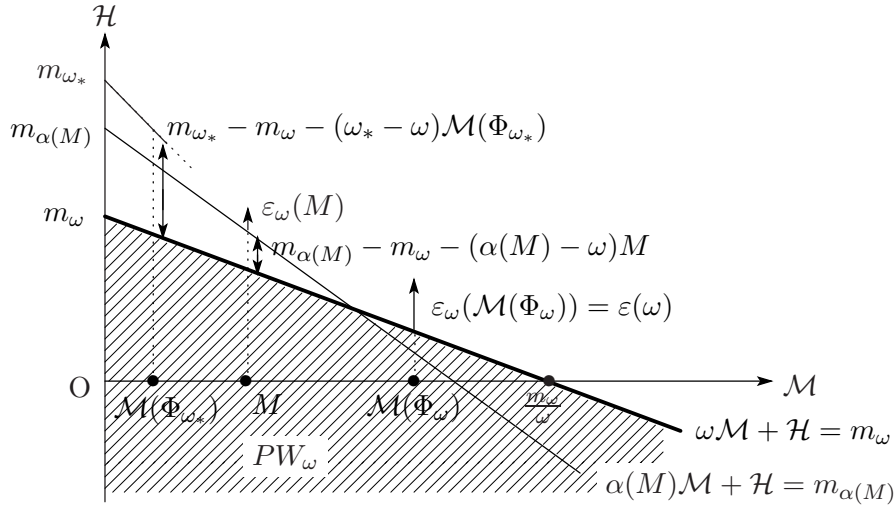}
\caption{How to extend the potential well $PW_{\omega}$}
\end{figure}

We consider a solution $\psi$ to \eqref{12/03/23/17:57} starting from $\widetilde{PW}_{\omega}$. If $\mathcal{M}(\psi) \le \mathcal{M}(\Phi_{\omega_{*}})$, then we can verify that $\psi \in PW_{\omega_{*}}$ (see Figure 2 above).  Hence, it follows from Theorem \ref{15/03/24/16:05} that $\psi$ exhibits the scenario {\rm (i)} or {\rm (ii)} in Theorem \ref{15/03/22/11:28}.  On the other hand, if $\mathcal{M}(\Phi_{\omega_{*}})< \mathcal{M}(\psi)<\mathcal{M}(\Phi_{\omega})$, then we can verify that $\psi \in A_{\alpha}^{\varepsilon(\alpha)}$, where $\alpha:=\alpha(\mathcal{M}(\psi))$. Similarly, if $\mathcal{M}(\psi)\ge \mathcal{M}(\Phi_{\omega})$, then $\psi \in A_{\alpha}^{\varepsilon(\alpha)}$. Hence, it follows from Theorem \ref{13/03/16/12:01} that $\psi$ exhibits one of the nine scenarios in Theorem \ref{15/03/22/11:28}. 
\end{proof}

%%%%%%%%%%%%%%%%%%%%%%%%%%%%%%%%%%%%%%%%%%%%%%

\section{Decomposition around ground state}
\label{12/03/25/16:46}

In this section, for a positive radial ground state $\Phi_{\omega}$ to \eqref{12/03/23/18:09} and a solution $\psi$ to \eqref{12/03/23/17:57} with the maximal existence-interval $I_{\max}$, we consider the decomposition of the form 
\begin{equation}\label{13/04/27/9:38}
\psi(x,t)=e^{i\theta(t)}\big(\Phi_{\omega}(x)+\eta(x,t)\big),
\end{equation} 
where $\theta(t)$ is some function of $t\in I_{\max}$ to be chosen later, and $\eta$ is the remainder. In this decomposition, we do not assume the orthogonality between $\Phi_{\omega}$ and $\eta(t)$ in $L_{real}^{2}$ (see \eqref{13/04/27/9:57} below).

\subsection{Linearized operator}\label{17/10/28/16:29} 
 The decomposition \eqref{13/04/27/9:38} leads us to
the linearized operator $\mathscr{L}_{\omega}$ around $\Phi_{\omega}$ which is defined by   
\begin{equation}\label{15/02/02/10:25}
\begin{split}
\mathscr{L}_{\omega}u
&:=
\omega u -\Delta u 
-
\frac{p+1}{2}\Phi_{\omega}^{p-1}u 
-\frac{p-1}{2}\Phi_{\omega}^{p-1}\overline{u}
-\frac{2^{*}}{2}\Phi_{\omega}^{\frac{4}{d-2}}u 
-\frac{2^{*}-2}{2}\Phi_{\omega}^{\frac{4}{d-2}}\overline{u}
\\[6pt]
&=
\big( \omega -\Delta  
-
p\Phi_{\omega}^{p-1}-(2^{*}-1)\Phi_{\omega}^{\frac{4}{d-2}}
\big)\Re[u] 
+
i\big( 
\omega -\Delta  
-
\Phi_{\omega}^{p-1}
-
\Phi_{\omega}^{\frac{4}{d-2}}
\big) \Im[u].
\end{split}
\end{equation}
The operator $\mathscr{L}_{\omega}$ is self-adjoint in $L_{real}^{2}(\mathbb{R}^{d})$, and for any $u,v \in H^{1}(\mathbb{R}^{d})$, 
\begin{equation}\label{13/01/04/14:46}
\begin{split}
\big[ \mathcal{S}_{\omega}''(\Phi_{\omega})u \big]v 
&= 
\langle \mathscr{L}_{\omega}u, v \rangle_{H^{-1},H^{1}}
. 
\end{split}
\end{equation} 
It is convenient to introduce the operators $L_{\omega,+}$ and $L_{\omega,-}$: 
\begin{align}\label{13/02/02/17:56}
L_{\omega, +} 
&:= 
\omega  -\Delta  
-p \Phi_{\omega}^{p-1} -(2^{*}-1)\Phi_{\omega}^{\frac{4}{d-2}}, 
\\[6pt]
\label{13/02/02/17:57}
L_{\omega, -} 
&:= 
\omega  -\Delta  - \Phi_{\omega}^{p-1} - \Phi_{\omega}^{\frac{4}{d-2}}.
\end{align}
Then, we have  
\begin{equation}\label{13/02/20/9:45}
\mathscr{L}_{\omega}u
=
L_{\omega, +}\Re[u]+ i L_{\omega,-}\Im[u] .
\end{equation}
Moreover, since $\Phi_{\omega}$ is a solution to \eqref{12/03/23/18:09}, we can verify that 
\begin{align}
\label{13/12/28/12:07}
&\mathscr{L}_{\omega}\Phi_{\omega}=L_{\omega,+}\Phi_{\omega}=-(p-1)\Phi_{\omega}^{p}-(2^{*}-2)\Phi_{\omega}^{2^{*}-1},
\\[6pt]
\label{13/12/28/15:36}
&\mathscr{L}_{\omega} (i\Phi_{\omega})=L_{\omega,-}\Phi_{\omega}=0,
\\[6pt]
\label{13/02/20/10:40}
&
\mathscr{L}_{\omega} \partial_{\omega}\Phi_{\omega}
=
L_{\omega,+}\partial_{\omega}\Phi_{\omega} =- \Phi_{\omega}.
\end{align}
Inserting the decomposition $\psi(t)=e^{i\theta(t)}(\Phi_{\omega}+\eta(t))$ into the equation \eqref{12/03/23/17:57}, we obtain the equation for $\eta$: 
\begin{equation}\label{13/05/06/13:59}
\begin{split}
\frac{\partial \eta}{\partial t}(t)
&=
-i\mathscr{L}_{\omega} \eta(t)
-
i\Big\{ \frac{d \theta}{dt}(t)- \omega  \Big\}(\Phi_{\omega}+\eta(t)) 
+
iN_{\omega}(\eta(t)),
\end{split}
\end{equation}  
where $N_{\omega}(\eta)$ denotes the higher order term of $\eta$, i.e., 
\begin{equation}\label{13/05/06/10:27}
\begin{split}
N_{\omega}(\eta)&:=
 |\Phi_{\omega}+\eta|^{p-1}\big( \Phi_{\omega}+\eta \big) 
- 
|\Phi_{\omega}|^{p-1}\Phi_{\omega}
-\frac{p+1}{2}\Phi_{\omega}^{p-1}\eta 
-\frac{p-1}{2}\Phi_{\omega}^{p-1}\overline{\eta}
\\[6pt]
&\quad 
+
 |\Phi_{\omega}+\eta|^{\frac{4}{d-2}}\big( \Phi_{\omega}+\eta \big) 
- |\Phi_{\omega}|^{\frac{4}{d-2}}\Phi_{\omega}
-\frac{2^{*}}{2}\Phi_{\omega}^{\frac{4}{d-2}}\eta 
-\frac{2^{*}-2}{2}\Phi_{\omega}^{\frac{4}{d-2}}\overline{\eta}
\\[6pt]
&=
O\big(|\eta|^{\min\{2,\, p\}}\big). 
\end{split}
\end{equation}
We see from \eqref{13/05/06/13:59} that the operator $-i\mathscr{L}_{\omega}$ relates to the behavior of the remainder $\eta(t)$. Unfortunately, $-i \mathscr{L}_{\omega}$ is not symmetric in $L_{real}^{2}(\mathbb{R}^{d})$, and therefore we do not have the orthogonality of eigenfunctions in this space. Thus, we need to work in the symplectic space $(L^{2}(\mathbb{R}^{d}), \Omega)$ instead of $L_{real}^{2}(\mathbb{R}^{2})$, where $\Omega$ is the symplectic form defined by    
\begin{equation}\label{13/02/23/10:45}
\Omega(f,g)
:=
\Im \int_{\mathbb{R}^{d}}f(x)\overline{g(x)}\,dx
=\big( f, i g\big)_{L_{real}^{2}}.
\end{equation}

\begin{proposition}\label{13/02/20/9:41}
Assume $d\ge 3$ and $1+\frac{4}{d}<p< 2^{*}-1$. Let $\omega_{1}$ be the frequency given by Proposition \ref{13/03/29/15:30}. Then, there exists $\omega_{2}\in (0,\omega_{1})$ such that for any $\omega \in (0,\omega_{2})$, $-i\mathcal{L}_{\omega}$ has a positive eigenvalue $\mu$, as an operator in $L_{real}^{2}(\mathbb{R}^{d})$. Furthermore, the eigenvalue $\mu$ satisfies that 
\begin{equation}\label{15/07/02/12:36}
-\mu^{2}= 
\inf\Bigm\{  
\dfrac{\langle L_{\omega, +} u, u \rangle_{H^{-1},H^{1}}}
{((L_{\omega, -})^{-1} u, u )_{L_{real}^{2}}}
\colon 
u \in H^{1}(\mathbb{R}^{d}),\ (u, \Phi_{\omega})_{L^{2}_{real}}=0
\Bigm\}. 
\end{equation}
\end{proposition}
\begin{proof}[Proof of Proposition \ref{13/02/20/9:41}]
We prove the proposition following the exposition in Section 2 of \cite{CGNT}. What we need to prove is that there exist a function $u \in H^{1}(\mathbb{R}^{d})$ and $\nu <0$ such that   
\begin{equation}\label{15/05/27/14:45}
L_{\omega, -} L_{\omega, +} u = \nu u. 
\end{equation}
Indeed, putting $f_{1} =-\Re[u]$ and $f_{2} = -\sqrt{- \nu} (L_{\omega, -})^{-1} \Re[u]$, 
we see from \eqref{13/02/20/9:45} and \eqref{15/05/27/14:45} that 
\begin{equation}\label{15/05/27/15:59}
\begin{split}
-i\mathcal{L}_{\omega} (f_{1}+if_{2})
&=
-iL_{\omega,+}f_{1}+L_{\omega,-}f_{2}
\\[6pt]
&=
i(L_{\omega,-})^{-1}\nu \Re[u]
-
L_{\omega,-}\sqrt{- \nu} (L_{\omega, -})^{-1} \Re[u]
=
\sqrt{- \nu}(f_{1}+if_{2}). 
\end{split}
\end{equation}
Furthermore, we can verify that the problem \eqref{15/05/27/14:45} is equivalent to that there exist $u\in H^{1}(\mathbb{R}^{d})$, $\nu <0$ and $\alpha \in \mathbb{R}$ such that 
\begin{equation} \label{14/06/27/16:35}
\left\{ \begin{array}{rl}
L_{\omega, +} u &= \nu (L_{\omega, -})^{-1}u + \alpha \Phi_{\omega}, 
\\[6pt]
(u,\Phi_{\omega})_{L_{real}^{2}}&=0.
\end{array} \right.
\end{equation}
Note here that ${\rm Ker}\, L_{\omega,-}={\rm span}\, \{\Phi_{\omega} \}$ (see Lemma \ref{13/03/18/15:50}) and $-\Phi_{\omega}=L_{\omega,+}\partial_{\omega}\Phi_{\omega}$ (see \eqref{13/02/20/10:40}).
The problem \eqref{14/06/27/16:35} leads us to the following minimizing problem
\begin{equation} \label{14/06/27/16:27}
\nu_{\omega} 
:= 
\inf\Bigm\{  
\dfrac{\langle L_{\omega, +} u, u \rangle_{H^{-1},H^{1}}}
{((L_{\omega, -})^{-1} u, u )_{L_{real}^{2}}}
\colon 
u \in H^{1}(\mathbb{R}^{d}),\ (u, \Phi_{\omega})_{L^{2}_{real}}=0
\Bigm\}. 
\end{equation}
We can verify that any minimizer of \eqref{14/06/27/16:27} satisfies the equation \eqref{14/06/27/16:35} with $\nu =\nu_{\omega}$ for some $\alpha \in \mathbb{R}$. Thus, it suffices to show that $-\infty < \nu_{\omega}<0$ and the existence of a minimizer. 
\par 
First, we shall show that $-\infty < \nu_{\omega}$. We take any $u \in  H^{1}(\mathbb{R}^{d})$ with $(u, \Phi_{\omega})_{L^{2}_{real}}=0$. Then, we see from Proposition \ref{14/06/29/14:03} that 
\begin{equation}\label{15/06/06/10:18}
\begin{split}
&\dfrac{\langle L_{\omega, +} u, u \rangle_{H^{-1},H^{1}}}
{((L_{\omega, -})^{-1} u, u )_{L_{real}^{2}}}
\\[6pt]
&=
\dfrac{\langle L_{\omega, -} u, u \rangle_{H^{-1},H^{1}}}
{\| (L_{\omega, -})^{-\frac{1}{2}} u\|_{L^{2}}^{2}}
-
\dfrac{\langle \{ (p-1)\Phi_{\omega}^{p-1}+(2^{*}-2)\Phi_{\omega}^{\frac{4}{d-2}}\} u, u \rangle_{H^{-1},H^{1}}}
{\| (L_{\omega, -})^{-\frac{1}{2}} u\|_{L^{2}}^{2}}
\\[6pt]
&\ge 
\dfrac{\|(L_{\omega, -})^{\frac{1}{2}} u \|_{H^{1}}^{2}}
{\| (L_{\omega, -})^{-\frac{1}{2}} u\|_{L^{2}}^{2}}
-
C(\omega) \dfrac{\| u\|_{L^{2}}^{2}}{\| (L_{\omega, -})^{-\frac{1}{2}} u\|_{L^{2}}^{2}},
\end{split}
\end{equation}
where $C(\omega)$ is some positive constant depending on $\omega$. Moreover, it is easy to see that   
\begin{equation}\label{15/06/06/10:55}
\begin{split}
\| u \|_{L^{2}}^{2}
&=
( (L_{\omega, -})^{\frac{1}{2}} u, (L_{\omega, -})^{-\frac{1}{2}} u )_{L_{real}^{2}}
\le  
\|(L_{\omega, -})^{\frac{1}{2}}u \|_{L^{2}} \|(L_{\omega, -})^{-\frac{1}{2}} u\|_{L^{2}}
\\[6pt]
&\le 
\dfrac{1}{2 C(\omega)}\|(L_{\omega, -})^{\frac{1}{2}}u \|_{L^{2}}^{2} 
+
\dfrac{C(\omega)}{2}\|(L_{\omega, -})^{-\frac{1}{2}} u\|_{L^{2}}^{2}.
\end{split}
\end{equation}
Putting the estimates \eqref{15/06/06/10:18} and \eqref{15/06/06/10:55} together, we find that $\nu_{\omega}>-C(\omega)^{2}>-\infty$. 
\par 
Next, we shall show that $\nu_{\omega}$ is negative for any sufficiently small $\omega > 0$. To this end, we introduce projections $\Pi_{\omega}$ and $\Pi$:  
\begin{equation}\label{15/05/26/17:50}
\Pi_{\omega}u :=  u - \dfrac{(u, \Phi_{\omega})_{L_{real}^{2}}}
{\|\Phi_{\omega}\|^{2}_{L^{2}}} \Phi_{\omega}, 
\quad 
\Pi u := u - \dfrac{(u, U)_{L^{2}_{real}}}{\|U\|^{2}_{L^{2}}} U .
\end{equation}
Then, we see from substitution of variables, Proposition \ref{13/03/03/9:43} and Lemma \ref{14/06/27/15:03} that  
\begin{equation} \label{14/07/13/14:35}
\begin{split}
&\omega^{-2} \nu_{\omega} 
=  
\omega^{-2} \inf\Big\{  
\dfrac{\langle L_{\omega, +} u, u \rangle_{H^{-1},H^{1}}}
{((L_{\omega, -})^{-1}u, u)_{L_{real}^{2}}}
\colon 
u \in H^{1}(\mathbb{R}^{d}),\  (u, \Phi_{\omega})_{L^{2}_{real}} = 0 
\Big\}
\\[6pt]
&\le 
\omega^{-2} 
\dfrac{\langle L_{\omega, +} \Pi_{\omega}\partial_{\omega}\Phi_{\omega}, \Pi_{\omega} \partial_{\omega}\Phi_{\omega} \rangle_{H^{-1},H^{1}}}
{((L_{\omega, -})^{-1}\Pi_{\omega} \partial_{\omega}\Phi_{\omega} , 
 \Pi_{\omega} \partial_{\omega}\Phi_{\omega} )_{L_{real}^{2}}}
=  
\dfrac{\langle \widetilde{L}_{\omega, +} (\omega T_{\omega}\Pi_{\omega}
\partial_{\omega}\Phi_{\omega}), \omega T_{\omega}\Pi_{\omega}\partial_{\omega}\Phi_{\omega} \rangle_{H^{-1},H^{1}}}
{((\widetilde{L}_{\omega, -})^{-1}(\omega T_{\omega}\Pi_{\omega}\partial_{\omega}\Phi_{\omega} ), \omega T_{\omega}\Pi_{\omega}\partial_{\omega}\Phi_{\omega})_{L_{real}^{2}}}
\\[6pt]
&\to 
\dfrac{\langle L^{\dagger}_{+} \Pi \Lambda U, \Pi \Lambda U \rangle_{H^{-1},H^{1}}}
{\langle (L_{-}^{\dagger})^{-1} \Pi \Lambda U, \Pi \Lambda U )_{L_{real}^{2}}}
\qquad 
\mbox{as $\omega \downarrow 0$},  
\end{split}
\end{equation}
where $\widetilde{L}_{\omega, +}$, $\widetilde{L}_{\omega, -}$, $L_{+}^{\dagger}$ and $L_{-}^{\dagger}$ are operators defined by \eqref{13/03/30/11:40}, \eqref{13/04/12/14:40}, \eqref{13/03/30/12:01} and \eqref{13/04/12/12:14}, respectively. Thus, it suffices to show that 
\begin{equation} \label{14/07/13/14:34}
\dfrac{\langle L^{\dagger}_{+} \Pi \Lambda U, \Pi \Lambda U \rangle_{H^{-1},H^{1}}}
{\langle (L^{\dagger})^{-1}_{-} \Pi \Lambda U, \Pi \Lambda U )_{L_{real}^{2}}} < 0.
\end{equation}
Note here that we see from \eqref{13/03/31/16:03} and integration by parts that \begin{equation}\label{14/06/27/17:09}
\Pi \Lambda U= \Lambda U + \frac{s_{p}}{2} U.
\end{equation}
Moreover, it follows from \eqref{13/03/31/16:03} and \eqref{13/03/31/16:08} that\begin{equation} \label{14/07/13/14:33}
L^{\dagger}_{+} \Pi \Lambda U = 
L^{\dagger}_{+} \Lambda U - \dfrac{(\Lambda U, U)_{L^{2}_{real}}}{\|U\|^{2}_{L^{2}}} L^{\dagger}_{+}U 
= 
-U -\frac{(p-1)s_{p}}{2} U^{p}.
\end{equation}
Using \eqref{14/06/27/17:09}, \eqref{14/07/13/14:33}, \eqref{13/03/31/16:03} and integration by parts,   we obtain that 
\begin{equation} \label{14/07/14/14:37}
\begin{split}
&\langle 
L^{\dagger}_{+} \Pi \Lambda U, \Pi \Lambda U  
\rangle_{H^{-1},H^{1}}
= 
\langle  
-U - \frac{(p-1)s_{p}}{2}U^{p},
\Lambda U+\frac{s_{p}}{2} U 
\rangle_{H^{-1},H^{1}} 
\\[6pt]
& = 
-
(U, \Lambda U)_{L_{real}^{2}}
-\frac{(p-1)s_{p}}{2}
(U^{p}, \Lambda U)_{L_{real}^{2}} 
-
\frac{s_{p}}{2} \|U\|^{2}_{L^{2}} 
-
\frac{(p-1)s_{p}^{2}}{4}\|U\|^{p+1}_{L^{p+1}} 
\\[6pt]
&=
\frac{d(p-1)s_{p}}{4}\Big(\frac{1}{2}-\frac{1}{p+1}\Big) \|U\|^{p}_{L^{p}}<0.  
\end{split}
\end{equation}
Since Lemma \ref{13/03/18/15:50} implies that $( \Pi \Lambda U, (L_{-}^{\dagger})^{-1} \Pi \Lambda U )_{L_{real}^{2}} > 0$, we find from \eqref{14/07/14/14:37} that \eqref{14/07/13/14:34} holds. 
\par 
Finally, we shall show the existence of a minimizer $u$ for $\nu_{\omega}$. We can  take a sequence $\{u_{n} \}$ in $H^{1}(\mathbb{R}^{d})$ such that  
\begin{align}
\label{15/05/29/14:21}
&(u_{n}, \Phi_{\omega})_{L_{real}^{2}}\equiv 0,
\\[6pt]
\label{15/05/29/14:22}
&
((L_{\omega,-})^{-1} u_{n}, u_{n} )_{L_{real}^{2}}\equiv 1,
\\[6pt] 
\label{15/05/29/14:23}
& 
\lim_{n\to \infty}\langle L_{\omega,+} u_{n}, u_{n} \rangle_{H^{-1},H^{1}}=\nu_{\omega}.
\end{align}
Note here that Lemma \ref{13/03/18/15:50} shows that the square root of $L_{\omega,-}$ is well-defined. We see from \eqref{15/05/29/14:23}, Proposition \ref{14/06/29/14:03}, the Cauchy-Schwarz inequality, \eqref{13/02/19/10:42} and \eqref{15/05/29/14:22} that for any sufficiently large number $n$,
\begin{equation}\label{15/05/29/14:28}
\begin{split}
\|u_{n}\|_{H^{1}}^{2}
&\le (1+\omega^{-1})
\langle L_{\omega,+}u_{n},u_{n} \rangle_{H^{-1},H^{1}}
\\[6pt]
&\quad +
\int_{\mathbb{R}^{d}} 
\big\{ p(\Phi_{\omega}(x))^{p-1} 
+ (2^{*}-1)(\Phi_{\omega}(x))^{\frac{4}{d-2}} 
\big\} 
|u_{n}(x)|^{2}\,dx 
\\[6pt]
&\lesssim 
(1+\omega^{-1}) \nu_{\omega} 
+ 
C_{1}(\omega) ( (L_{\omega,-})^{\frac{1}{2}}u_{n}, (L_{\omega,-})^{-\frac{1}{2}}u_{n})_{L_{real}^{2}}
\\[6pt]
&\le 
(1+\omega^{-1}) \nu_{\omega}
+
C_{1}(\omega)
\| (L_{\omega,-})^{\frac{1}{2}}u_{n}\|_{L^{2}} 
\| (L_{\omega,-})^{-\frac{1}{2}}u_{n}\|_{L^{2}} 
\\[6pt]
&=
(1+\omega^{-1}) \nu_{\omega}
+
C_{1}(\omega)
(L_{\omega,-} u_{n}, u_{n} )_{L_{real}^{2}}^{\frac{1}{2}} 
((L_{\omega,-})^{-1}u_{n}, u_{n} )_{L_{real}^{2}}^{\frac{1}{2}} 
\\[6pt]
&\le 
(1+\omega^{-1}) \nu_{\omega} +C_{2}(\omega)\|u_{n}\|_{H^{1}},
\end{split}
\end{equation}
where $C_{1}(\omega)$ and $C_{2}(\omega)$ are some constants depending only on $d$, $p$ and $\omega$. This implies that $\{u_{n}\}$ is bounded in $H^{1}(\mathbb{R}^{d})$ and therefore we can extract a subsequence of $\{u_{n}\}$ (still denoted by the same symbol) and a function $u_{\infty} \in H^{1}(\mathbb{R}^{d})$ such that 
\begin{align}
\label{15/05/30/10:47}
&\lim_{n\to \infty}u_{n}=u_{\infty} \qquad \mbox{weakly in $H^{1}(\mathbb{R}^{d})$},
\\[6pt]
\label{15/05/30/15:01}
&(u_{\infty},\Phi_{\omega})_{L_{real}^{2}}=0,
\\[6pt]
\label{15/05/30/10:54}
&\lim_{n\to \infty}\int_{\mathbb{R}^{d}} \Phi_{\omega}(x)^{q} |u_{n}(x)|^{2}\,dx=
\int_{\mathbb{R}^{d}} \Phi_{\omega}(x)^{q} |u_{\infty}(x)|^{2}\,dx,
\end{align}
where $q$ indicates $p-1$ or $\frac{4}{d-2}$: Here, \eqref{15/05/30/15:01} follows from \eqref{15/05/29/14:21} and \eqref{15/05/30/10:47}, and \eqref{15/05/30/10:54} follows from \eqref{15/05/30/10:47} and Propositionre \ref{14/06/29/14:03}. Furthermore, we can verify that 
\begin{equation}\label{15/05/30/11:48}
\langle L_{\omega,+}u_{\infty}, u_{\infty} \rangle_{H^{-1},H^{1}}
\le 
\liminf_{n\to \infty} \langle L_{\omega,+}u_{n}, u_{n} \rangle_{H^{-1},H^{1}}
=
\nu_{\omega}<0.
\end{equation} 
In particular, the limit $u_{\infty}$ is non-trivial.
We also see from Lemma \ref{13/03/18/15:50}, \eqref{13/12/28/15:36}, \eqref{15/05/29/14:21} and \eqref{15/05/29/14:22} that for any number $n$, 
\begin{equation}\label{15/05/30/11:00}
\|(L_{\omega,-})^{-1} u_{n}\|_{H^{1}}^{2} 
\lesssim  
(u_{n},  (L_{\omega,-})^{-1} u_{n})_{L_{real}^{2}}\le 1, 
\end{equation}
where the implicit constant depends on $\omega$. Hence, there exists a subsequence of $\{u_{n}\}$ (still denoted by the same symbol) and a function $v_{\infty} \in H^{1}(\mathbb{R}^{d})$ such that 
\begin{equation}\label{15/05/30/11:29}
\lim_{n\to \infty}(L_{\omega,-})^{-1}u_{n}=v_{\infty}
\qquad \mbox{weakly in $H^{1}(\mathbb{R}^{d})$}.
\end{equation}
Furthermore, we see from \eqref{13/12/28/15:36} that for any test function $\phi \in C_{c}^{\infty}(\mathbb{R}^{d})$, 
\begin{equation}\label{15/05/30/11:35}
\begin{split}
\langle L_{\omega,-}v_{\infty}, \phi \rangle_{H^{-1},H^{1}}
&=
\langle  L_{\omega,-}\phi, v_{\infty} \rangle_{H^{-1},H^{1}}
\\[6pt]
&= 
\lim_{n\to \infty} ( (L_{\omega,-})^{-1}u_{n}, L_{\omega,-}\phi)_{L_{real}^{2}}=( u_{\infty}, \phi)_{L_{real}^{2}}.
\end{split}
\end{equation}
This together with \eqref{15/05/30/15:01} shows that 
\begin{equation}\label{15/05/30/11:45}
v_{\infty}=(L_{\omega,-})^{-1}u_{\infty}. 
\end{equation}
We also see from \eqref{15/05/30/11:45} and   Lemma \ref{13/03/18/15:50}  that 
\begin{equation}\label{15/05/31/22:09}
( (L_{\omega,-})^{-1}u_{\infty}, u_{\infty} )_{L_{real}^{2}}
=
( v_{\infty}, L_{\omega,-}v_{\infty})_{L_{real}^{2}}>0. 
\end{equation}
Furthermore, it follows from \eqref{15/05/30/11:29}, \eqref{15/05/30/11:45}, \eqref{15/05/29/14:21} and \eqref{15/05/29/14:22} that 
\begin{equation}\label{15/05/30/11:54}
( (L_{\omega,-})^{-1}u_{\infty}, u_{\infty} )_{L_{real}^{2}}
\le 
\liminf_{n\to \infty}
((L_{\omega,-})^{-1}u_{n}, u_{n} )_{L_{real}^{2}}
=1,
\end{equation}
so that we can take $\rho_{1} \ge 1$ such that $((L_{\omega,-})^{-1}\rho_{1} u_{\infty}, \rho_{1} u_{\infty} )_{L_{real}^{2}}=1$. Suppose here that $\rho_{1} >1$. Then, we see from \eqref{15/05/30/15:01} and \eqref{15/05/30/11:48} that 
\begin{equation}
\label{15/05/30/12:03}
\nu_{\omega}
\le 
\langle L_{\omega,+} \rho_{1} u_{\infty}, \rho_{1} u_{\infty} \rangle_{H^{-1},H^{1}}<
\langle L_{\omega,+} u_{\infty}, u_{\infty} \rangle_{H^{-1},H^{1}}
\le \nu_{\omega}.
\end{equation}
This is a contradiction. Thus, we have shown that  
\begin{equation}\label{15/05/30/15:57}
((L_{\omega,-})^{-1}u_{\infty}, u_{\infty} )_{L_{real}^{2}}
=1.
\end{equation} Then, the same argument as \eqref{15/05/30/12:03} shows that 
\begin{equation}
\label{15/05/30/15:28}
\nu_{\omega}
=
\langle L_{\omega,+} u_{\infty}, u_{\infty} \rangle_{H^{-1},H^{1}}.
\end{equation}
This together with \eqref{15/05/30/15:01} and \eqref{15/05/30/15:57} shows that  $u_{\infty}$ is a minimizer for the problem \eqref{14/06/27/16:27}. 
\end{proof}

%%%%%%%%%%%%%%%%%%%%%%%%%%%%%%%%%%%%%%%%%%%%%%%%%%%%%%%%%%%%%%%%%%%%%%%%%%%%%%%%
\subsection{Symplectic decomposition}\label{17/10/28/16:32}

In this subsection, we assume that $d\ge 3$, $1+\frac{4}{d}<p< 2^{*}-1$ and $0<\omega <\omega_{2}$, where $\omega_{2}$ is the frequency given in Proposition \ref{13/02/20/9:41}. Note that $-i\mathscr{L}_{\omega}$ has a positive eigenvalue $\mu>0$ as an operator in $L_{real}^{2}(\mathbb{R}^{d})$. 
\par 
We will apply the ``symplectic decomposition'' corresponding to the discrete modes of $-i\mathscr{L}_{\omega}$ to the remainder $\eta$ in \eqref{13/04/27/9:38} (see \eqref{13/02/22/15:27} below). Moreover, we determine the function $\theta(t)$ in \eqref{13/04/27/9:38}.  
\par 
Let $\mathscr{U}_{+}$ be an eigenfunction corresponding to the positive eigenvalue $\mu$, and put   
\begin{equation}\label{13/04/29/11:30}
\mathscr{U}_{-}:=\overline{\mathscr{U}_{+}}.
\end{equation} 
Then, we have 
\begin{equation}\label{15/02/03/13:17}
-i\mathcal{L}_{\omega}\mathscr{U}_{-}
=
\overline{i\mathcal{L}_{\omega}\mathscr{U}_{+}}=-\overline{ \mu \mathscr{U}_{+}}=-\mu \mathscr{U}_{-}. 
\end{equation} 
Hence, $\mathscr{U}_{-}$ is an eigenfunction of $-i\mathscr{L}_{\omega}$ corresponding to $-\mu$. We assume that $\mathscr{U}_{+}$ and $\mathscr{U}_{-}$ are normalized in the following sense: 
\begin{equation}
\label{13/03/04/12:30}
\Omega\big( \mathscr{U}_{+}, \mathscr{U}_{-} \big)
=1,
\quad 
\Omega\big( \mathscr{U}_{-}, \mathscr{U}_{+} \big)
=-1.
\end{equation}
It is obvious that 
\begin{align}
\label{13/04/27/15:37}
&
\Omega\big( \mathscr{U}_{+}, \mathscr{U}_{+} \big)
=
\Omega\big( \mathscr{U}_{-}, \mathscr{U}_{-} \big)
=0.
\end{align} 
Furthermore, it follows from \eqref{13/12/28/15:36}, \eqref{13/02/20/10:40} and $\mathcal{L}_{\omega}\mathscr{U}_{\pm}=\pm i \mu \mathscr{U}_{\pm}$ that 
\begin{equation}\label{13/02/22/16:30}
\Omega\big( i\Phi_{\omega}, \mathscr{U}_{\pm} \big)
=
\Omega\big( \partial_{\omega}\Phi_{\omega}, \mathscr{U}_{\pm} \big)
=0.
\end{equation}

Now, we expand the remainder $\eta(t)$ in the decomposition \eqref{13/04/27/9:38} by the discrete modes of $-i\mathscr{L}_{\omega}$:  
\begin{equation}\label{13/02/22/15:27}
\eta(t)
=\lambda_{+}(t)\mathscr{U}_{+}+\lambda_{-}(t)\mathscr{U}_{-} 
+a(t)i\Phi_{\omega}+b(t)\partial_{\omega}\Phi_{\omega}
+\gamma(t) ,
\end{equation}
where  
\begin{equation}\label{13/02/22/15:38}
\Omega\big( \gamma(t), \mathscr{U}_{+}  \big)
=
\Omega\big( \gamma(t), \mathscr{U}_{-}  \big)
=
\Omega\big( \gamma(t), i\Phi_{\omega}  \big)
=
\Omega\big( \gamma(t), \partial_{\omega}\Phi_{\omega} \big)
=0.
\end{equation}
We see from \eqref{13/03/04/12:30}, \eqref{13/04/27/15:37} and \eqref{13/02/22/16:30} that the coefficients are as follows: 
\begin{align}\label{13/05/04/9:35}
&\lambda_{+}(t)
=
\Omega(\eta(t), \mathscr{U}_{-}), 
\qquad 
\lambda_{-}(t)
=
-\Omega(\eta(t), \mathscr{U}_{+}), 
\\[6pt]
\label{13/05/04/10:42}
&a(t)=
\frac{\Omega(\eta(t), \partial_{\omega}\Phi_{\omega})}{(\Phi_{\omega},
\partial_{\omega}\Phi_{\omega})_{L_{real}^{2}}}
,
\qquad 
b(t)=
-\frac{\Omega(\eta(t),i\Phi_{\omega})}{(\Phi_{\omega},
\partial_{\omega}\Phi_{\omega})_{L_{real}^{2}}}.
\end{align}
Note here that the denominators in \eqref{13/05/04/10:42} are non-zero (see {\rm (iii)} of Proposition \ref{13/03/29/15:30}). Moreover, it follows from H\"older's inequality that   
\begin{equation}\label{13/05/11/17:19}
\lambda_{\pm}^{2}(t)
=
\Omega(\eta(t),\mathscr{U}_{\mp})^{2}
\le
\|\mathscr{U}_{+}\|_{L^{2}}^{2} \|\eta(t)\|_{L^{2}}^{2}.
\end{equation}

We require that the mode $i \Phi_{\omega}$ does not appear in the decomposition \eqref{13/02/22/15:27}, that is, $a(t)\equiv 0$. To this end, we choose the function $\theta(t)$ in \eqref{13/04/27/9:38} so that
\begin{equation}\label{13/05/04/13:54}
\Omega(e^{-i\theta(t)}\psi(t),\partial_{\omega}\Phi_{\omega})\equiv 0.  
\end{equation}
Then, it follows from $\Omega(\Phi_{\omega},\partial_{\omega}\Phi_{\omega})=0$ that $\Omega(\eta(t),\partial_{\omega}\Phi_{\omega})\equiv 0$, and therefore $a(t)\equiv 0$. Furthermore, the choice of $\theta(t)$ has room of an integer multiple of $\pi$. Hence, in addition to \eqref{13/05/04/13:54}, we can choose $\theta(t)$ so that  
\begin{equation}\label{13/05/09/10:46}
( e^{-i\theta(t)}\psi(t), \partial_{\omega}\Phi_{\omega} )_{L_{real}^{2}}<0.
\end{equation}
This choice of $\theta(t)$ plays an important role in the argument below (see \eqref{13/05/09/11:03}). 

%%%%%%%%%%%%%%%%%%%%%%%%%%%%%%%%%%%%%%%%%%%%%%%%%%%%%%%%%%%%%%%%%%%%%%%%%%%%%%%%
\subsection{Modulation equations}\label{17/10/28/20:52}

We continue the discussion about the decompositions \eqref{13/04/27/9:38} and \eqref{13/02/22/15:27} for a solution $\psi$ to \eqref{12/03/23/17:57}. Throughout this subsection, as well as Section \ref{17/10/28/16:32}, we assume that $d\ge 3$, $1+\frac{4}{d}<p<2^{*}-1$ and $0<\omega < \omega_{2}$. Furthermore, we assume that $\psi$ satisfies \eqref{13/05/04/13:54}, \eqref{13/05/09/10:46} and 
\begin{equation}\label{13/04/21/11:42}
\mathcal{M}(\psi)= \mathcal{M}(\Phi_{\omega})
.
\end{equation}
Since 
\begin{equation}\label{13/04/27/9:43}
\mathcal{M}(\psi)
=\mathcal{M}(\Phi_{\omega})+\mathcal{M}(\eta(t))
+
\big(\Phi_{\omega}, \eta(t)\big)_{L_{real}^{2}},
\end{equation}
the condition \eqref{13/04/21/11:42} implies that for any $t\in I_{\max}$, 
\begin{equation}\label{13/04/27/9:57}
\big(\Phi_{\omega},\eta(t) \big)_{L_{real}^{2}}
= 
-\mathcal{M}(\eta(t)).
\end{equation}
In particular, there is no orthogonality between $\Phi_{\omega}$ and $\eta(t)$ in $L_{real}^{2}(\mathbb{R}^{d})$.

Our aim here is to derive ordinary differential equations for $\theta$, $\lambda_{+}$ and $\lambda_{-}$ under the condition \eqref{13/04/21/11:42}.
\par 
First, we shall derive an equation for $\theta$. We see from \eqref{13/05/04/13:54} that 
\begin{equation}\label{15/02/07/10:14}
0
=
\frac{d}{dt}\Omega( \Phi_{\omega}+\eta(t), \partial_{\omega}\Phi_{\omega})
=
(\frac{d \eta}{dt}(t), i \partial_{\omega}\Phi_{\omega})_{L_{real}^{2}}.
\end{equation}
Putting \eqref{13/05/06/13:59} and \eqref{15/02/07/10:14} together, we obtain  
\begin{equation}\label{13/04/27/16:16}
\begin{split}
0
&=
(
-\mathscr{L}_{\omega} \eta(t)
-
\Big\{ \frac{d \theta}{dt}(t)- \omega  \Big\}(\Phi_{\omega}+\eta(t)) 
+
N_{\omega}(\eta(t))
,\,
\partial_{\omega}\Phi_{\omega} 
)_{L_{real}^{2}}
\\[6pt]
&=
-
\big( \eta(t), \mathscr{L}_{\omega} \partial_{\omega}\Phi_{\omega} \big)_{L_{real}^{2}}
-\Big\{ \frac{d \theta}{dt}(t)- \omega  \Big\}
\big( \Phi_{\omega}+\eta(t), \partial_{\omega}\Phi_{\omega} \big)_{L_{real}^{2}}\\[6pt]
&\qquad 
+
\big( N_{\omega}(\eta(t)), \partial_{\omega}\Phi_{\omega}\big)_{L_{real}^{2}}.
\end{split}
\end{equation}
Furthermore, using \eqref{13/02/20/10:40} and \eqref{13/04/27/9:57}, 
 we obtain the equation for $\theta$:  
\begin{equation}\label{13/05/06/14:06}
\Big\{ \frac{d \theta}{dt}(t)- \omega  \Big\}
\big( \Phi_{\omega}+\eta(t), \partial_{\omega}\Phi_{\omega} \big)_{L_{real}^{2}}
=-\mathcal{M}(\eta(t))
+
\big( N_{\omega}(\eta(t)), \partial_{\omega}\Phi_{\omega} \big)_{L_{real}^{2}}.
\end{equation}
Next, we shall derive equations for $\lambda_{+}$ and $\lambda_{-}$. It follows from \eqref{13/05/06/13:59}, \eqref{13/02/22/16:30}, \eqref{13/05/04/9:35} and 
 $\mathscr{L}_{\omega}\mathscr{U}_{\pm}=\pm i \mu \mathscr{U}_{\pm}$ that 
\begin{equation}\label{15/02/07/10:46}
\begin{split}
\frac{d \lambda_{\pm}}{dt}(t)
&=
\pm \Omega( \frac{d \eta}{d t}(t), \mathscr{U}_{\mp})
\\[6pt]
&=
\pm \Omega( 
-i\mathscr{L}_{\omega} \eta(t)
-
i\Big\{ \frac{d \theta}{dt}(t)- \omega  \Big\}\eta(t) 
+
iN_{\omega}(\eta(t))
,\,
\mathscr{U}_{\mp}
)
\\[6pt]
&=
\pm \big( 
\eta(t)
,
i\mu \mathscr{U}_{\mp}
\big)_{L_{real}^{2}}
\mp 
\big( 
\Big\{ \frac{d \theta}{dt}(t)- \omega  \Big\}\eta(t) 
-
N_{\omega}(\eta(t))
,\,
\mathscr{U}_{\mp}
\big)_{L_{real}^{2}}
\\[6pt]
&=
\pm \mu \lambda_{\pm}(t)
\mp \big( 
\Big\{ \frac{d \theta}{dt}(t)- \omega  \Big\}\eta(t) 
-
N_{\omega}(\eta(t))
,\,
\mathscr{U}_{\mp}
\big)_{L_{real}^{2}}.
\end{split}
\end{equation}
Thus, we have obtained equations for $\lambda_{+}$ and $\lambda_{-}$: 
\begin{align}
\label{13/05/06/14:03}
&
\frac{d \lambda_{+}}{d t}(t)
=
\mu \lambda_{+}(t)
- 
\big( \Big\{ \frac{d \theta}{dt}(t)- \omega  \Big\} \eta(t) 
-N_{\omega}(\eta(t)), \, \mathscr{U}_{-} \big)_{L_{real}^{2}},
\\[6pt]
\label{13/05/06/14:04}
&\frac{d \lambda_{-}}{d t}(t)
=
-\mu \lambda_{-}(t) 
+
\big(
\Big\{ \frac{d \theta}{dt}(t)- \omega  \Big\}\eta(t) 
-N_{\omega}(\eta(t)), \, \mathscr{U}_{+} \big)_{L_{real}^{2}}. 
\end{align}
Note here that it follows from $i\mathscr{L}_{\omega}\mathscr{U}_{\pm}=\pm \mu \mathscr{U}_{\pm}$, \eqref{13/05/06/14:03}, \eqref{13/05/06/14:04} and \eqref{13/03/04/12:30} that  
\begin{equation}\label{15/06/06/11:39}
\begin{split}
&\langle \mathscr{L}_{\omega} \big(\lambda_{+}(t)\mathscr{U}_{+}+\lambda_{-}(t)\mathscr{U}_{-}\big), \, 
\frac{d\lambda_{+}}{dt}(t) \mathscr{U}_{+}
+\frac{d\lambda_{-}}{dt}(t) \mathscr{U}_{-}
\rangle_{H^{-1},H^{1}}
\\[6pt]
&=
-\mu \lambda_{-}(t) \big( \Big\{ \frac{d \theta}{dt}(t)- \omega  \Big\} \eta(t) 
-N_{\omega}(\eta(t)), \, \mathscr{U}_{-} \big)_{L_{real}^{2}}
\\[6pt]
&\quad 
+\mu \lambda_{+}(t) 
\big( \Big\{ \frac{d \theta}{dt}(t)- \omega  \Big\} \eta(t) 
-N_{\omega}(\eta(t)), \, \mathscr{U}_{+} \big)_{L_{real}^{2}}
.
\end{split}
\end{equation}

%%%%%%%%%%%%%%%%%%%%%%%%%%%%%%%%%%%%%%%%%%%%%%%%%%%%%%%%%%%%%%%%%%%%%%%%%%%%%%%%

\subsection{Linearized energy norm}\label{17/10/28/17:38}

Our aim here it to introduce the ``linearized energy norm'' for the remainder $\eta$ in the decomposition \eqref{13/04/27/9:38}. Throughout this subsection, as well as Section \ref{17/10/28/20:52}, we assume that $d\ge 3$, $1+\frac{4}{d}<p<2^{*}-1$, $0<\omega < \omega_{2}$ and $\psi$ is a solution to \eqref{12/03/23/17:57} satisfying \eqref{13/05/04/13:54} through \eqref{13/04/21/11:42}. 
\par 
Put   
\begin{equation}\label{13/05/04/14:30}
\Gamma(t):=b(t)\partial_{\omega}\Phi_{\omega}+\gamma(t).
\end{equation}
Then, since $a(t)\equiv 0$ (see \eqref{13/05/04/13:54}), the decomposition \eqref{13/02/22/15:27} is rewritten by 
\begin{equation}\label{13/05/04/14:34}
\eta(t)
=
\lambda_{+}(t)\mathscr{U}_{+}+\lambda_{-}(t)\mathscr{U}_{-}
+
\Gamma(t).
\end{equation}
We see from \eqref{13/02/22/16:30} and \eqref{13/02/22/15:38} that  
\begin{equation}\label{13/05/03/19:37}
\Omega(\Gamma(t), \mathscr{U}_{+})
=
\Omega(\Gamma(t), \mathscr{U}_{-})
=
\Omega\big( \Gamma(t), \partial_{\omega}\Phi_{\omega} \big)
=
0. 
\end{equation}
As a consequence of these orthogonalities, we have the following relationship:
\begin{lemma}\label{13/05/03/19:30}
The function $\Gamma$ in the decomposition \eqref{13/05/04/14:34} satisfies 
\begin{equation}\label{13/04/29/14:47}
\langle \mathscr{L}_{\omega} \Gamma(t), \Gamma(t)\rangle_{H^{-1},H^{1}}
\sim
\|\Gamma (t) \|_{H^{1}}^{2} 
\end{equation}
for all $t\in I_{\max}$, where the implicit constant depends on $\omega$ as well as $d$ and $p$.  
\end{lemma}
\begin{proof}[Proof of Lemma \ref{13/05/03/19:30}]
Lemma \ref{13/01/02/12:08} together with \eqref{13/05/03/19:37} gives the desired result. 
\end{proof}

Now, we see from \eqref{13/04/21/11:42}, Taylor's expansion around $\Phi_{\omega}$,  $\mathcal{S}_{\omega}'(\Phi_{\omega})=0$, \eqref{13/01/04/14:46},  $\mathscr{L}_{\omega}\mathscr{U}_{\pm}=\pm i\mu \mathscr{U}_{\pm}$ and \eqref{13/03/04/12:30} that   
\begin{equation}\label{13/04/29/12:12}
\begin{split}
&\mathcal{H}(\psi)-\mathcal{H}(\Phi_{\omega})
=
\mathcal{S}_{\omega}(\psi)-\mathcal{S}_{\omega}(\Phi_{\omega})
\\[6pt]
&=
\mathcal{S}_{\omega}(e^{-i\theta(t)}\psi)-\mathcal{S}_{\omega}(\Phi_{\omega})
=
\mathcal{S}_{\omega}(\Phi_{\omega}+\eta(t))-\mathcal{S}_{\omega}(\Phi_{\omega})
\\[6pt]
&=
\frac{1}{2}
\langle 
\mathscr{L}_{\omega}\eta(t), 
\eta(t) \rangle_{H^{-1},H^{1}}
+
O(\, \|\eta(t)\|_{H^{1}}^{\min\{3,p+1\}}\,)
\\[6pt]
&=
\frac{1}{2}\lambda_{+}(t)\lambda_{-}(t)
( 
-i\mu \mathscr{U}_{+}, \mathscr{U}_{-}
)_{L_{real}^{2}}
+
\frac{1}{2}\lambda_{+}(t)\lambda_{-}(t)
( 
i\mu \mathscr{U}_{-}, \, 
\mathscr{U}_{+}
)_{L_{real}^{2}}
\\[6pt]
&\quad 
+
\frac{1}{2}
\langle 
\mathscr{L}_{\omega}\Gamma(t), 
\Gamma(t) \rangle_{H^{-1},H^{1}}
+
O(\, \|\eta(t)\|_{H^{1}}^{\min\{3,p+1\}} \,)
\\[6pt]
&= 
-\mu \lambda_{+}(t)\lambda_{-}(t)
+
\frac{1}{2}
\langle 
\mathscr{L}_{\omega}\Gamma(t), 
\Gamma(t) \rangle_{H^{-1},H^{1}}
+ O(\, \|\eta(t)\|_{H^{1}}^{\min\{3,p+1\}} \,).
\end{split}
\end{equation}
We define the linearized energy norm $\|\eta(t)\|_{E}$ by   
\begin{equation}\label{13/04/29/15:42}
\|\eta(t) \|_{E}^{2}
:=
\frac{\mu}{2} \big( \lambda_{+}^{2}(t)+\lambda_{-}^{2}(t)\big)
+
\frac{1}{2}
\langle 
\mathscr{L}_{\omega}\Gamma(t), 
\Gamma(t) \rangle_{H^{-1},H^{1}}.
\end{equation}
Then, we see from \eqref{13/04/29/12:12} that   
\begin{equation}\label{13/05/05/12:22}
\mathcal{H}(\psi)-\mathcal{H}(\Phi_{\omega})
+
\frac{\mu}{2} \big(\lambda_{+}(t) + \lambda_{-}(t)\big)^{2}
-
\|\eta(t)\|_{E}^{2}
=O( \, \|\eta(t)\|_{H^{1}}^{\min\{3,p+1\}} \,).
\end{equation}

\begin{lemma}\label{13/05/11/11:52}
The function $\Gamma$ in the decomposition \eqref{13/05/04/14:34} satisfies
\begin{equation}\label{13/05/11/14:13}
\|\Gamma(t)\|_{H^{1}}
\lesssim 
\|\eta(t) \|_{H^{1}}
\end{equation} 
for all $t\in I_{\max}$. Moreover, we have 
\begin{equation}\label{17/12/31/15:31}
\|\eta(t) \|_{H^{1}}
\sim
\|\eta(t) \|_{E}.
\end{equation}
Here, the implicit constants in \eqref{13/05/11/14:13} and \eqref{17/12/31/15:31} depend on $\omega$ as well as $d$ and $p$.
\end{lemma}
\begin{proof}[Proof of Lemma \ref{13/05/11/11:52}]
We see from \eqref{13/05/11/17:19} that 
\begin{equation}\label{13/05/11/21:31}
\begin{split}
\|\Gamma(t)\|_{H^{1}}^{2}
&= \|\eta(t)-\lambda_{+}(t)\mathscr{U}_{+}-\lambda_{-}(t)\mathscr{U}_{-} \|_{H^{1}}^{2}
\\[6pt]
&\lesssim  
\|\eta(t)\|_{H^{1}}^{2}+ \lambda_{+}^{2}(t) \|\mathscr{U}_{+}\|_{H^{1}}^{2}
+
\lambda_{-}^{2}(t) \| \mathscr{U}_{-} \|_{H^{1}}^{2}
\\[6pt]
&\lesssim (1+2\|\mathscr{U}_{+}\|_{H^{1}}^{4}) \|\eta(t)\|_{H^{1}}^{2},
\end{split}
\end{equation}
which proves \eqref{13/05/11/14:13}. 
\par 
Next, we shall prove \eqref{17/12/31/15:31}. We see from \eqref{13/05/04/14:34}, $\mathscr{U}_{-}=\overline{\mathscr{U}_{+}}$ (see \eqref{13/04/29/11:30}), Lemma \ref{13/05/03/19:30} and the definition \eqref{13/04/29/15:42} that 
\begin{equation}\label{13/05/04/15:15}
\begin{split}
\|\eta(t) \|_{H^{1}}^{2}
&=  
\|\lambda_{+}(t)\mathscr{U}_{+}+\lambda_{-}(t)\mathscr{U}_{-}+\Gamma(t) \|_{H^{1}}^{2}
\\[6pt]
&\lesssim   
\lambda_{+}^{2}(t) \|\mathscr{U}_{+}\|_{H^{1}}^{2}
+
\lambda_{-}^{2}(t) \|\mathscr{U}_{-}\|_{H^{1}}^{2}
+
\|\Gamma(t)\|_{H^{1}}^{2}
\\[6pt]
&\lesssim
\| \mathscr{U}_{+}\|_{H^{1}}^{2} \big( \lambda_{+}^{2}(t) + \lambda_{-}^{2}(t) \big)
+
\langle \mathscr{L}_{\omega}\Gamma(t), \Gamma(t) \rangle_{H^{-1},H^{1}}
\\[6pt]
&\lesssim 
\max\{1,\mu^{-1} \}(1+\| \mathscr{U}_{+}\|_{H^{1}}^{2})
\|\eta(t) \|_{E}^{2}.
\end{split}
\end{equation}
Moreover, it follows from \eqref{13/05/11/17:19}, Lemma \ref{13/05/03/19:30} 
and \eqref{13/05/11/21:31} that 
\begin{equation}\label{13/05/11/18:12}
\begin{split}
\|\eta(t)\|_{E}^{2}
&\lesssim 
\mu \|\mathscr{U}_{+}\|_{L^{2}}^{2}\|\eta(t) \|_{L^{2}}^{2}+\|\Gamma(t) \|_{H^{1}}^{2}
\\[6pt]
&\lesssim 
\mu \|\mathscr{U}_{+}\|_{L^{2}}^{2}\|\eta(t) \|_{H^{1}}^{2}
+
\big( 1+2\| \mathscr{U}_{+}\|_{H^{1}}^{4}\big) \|\eta(t)\|_{H^{1}}^{2}.
\end{split}
\end{equation} 
Putting \eqref{13/05/04/15:15} and \eqref{13/05/11/18:12} together, we find that  \eqref{17/12/31/15:31} holds.  
\end{proof}

%%%%%%%%%%%%%%%%%%%%%%%%%%%%%%%%%%%%%%%%%%%%%%%%%%%%%%%%%%%%%%%%%%%%%%%%%%%%%%%%
\subsection{Distance function from the ground state}\label{17/10/28/22:11}
 Our aim here is to introduce a distance function from the ground state $\Phi_{\omega}$ by using the linearized energy norm \eqref{13/04/29/15:42} (cf. \cite{Nakanishi-Schlag2}). Throughout this subsection, we assume that  $d\ge 3$, $1+\frac{4}{d}<p<2^{*}-1$ and $0< \omega <\omega_{2}$, where $\omega_{2}$ is the frequency given by Proposition \ref{13/02/20/9:41}.
\par 
We see from \eqref{13/05/05/12:22}, and \eqref{17/12/31/15:31} in Lemma \ref{13/05/11/11:52} that there exists a constant $\delta_{E}(\omega) >0$ with the following property: for any solution $\psi$ to \eqref{12/03/23/17:57} satisfying 
 $\mathcal{M}(\psi)=\mathcal{M}(\Phi_{\omega})$ and any $t \in I_{\max}(\psi)$ for which $\|\eta(t)\|_{E} \le 4\delta_{E}(\omega)$,  
\begin{equation}\label{13/05/05/13:33}
\Bigm| 
\mathcal{H}(\psi(t))-\mathcal{H}(\Phi_{\omega})
+
\frac{\mu}{2} \big( \lambda_{+}(t)+\lambda_{-}(t) \big)^{2}
-
\|\eta(t)\|_{E}^{2}
\Bigm|
\le 
\frac{\|\eta(t)\|_{E}^{2}}{10}, 
\end{equation}
where $\eta(t)$, $\lambda_{+}(t)$ and $\lambda_{-}(t)$ are the functions appearing in the decomposition for $\psi$ of the form \eqref{13/04/27/9:38} with \eqref{13/05/04/14:34} (see also \eqref{13/02/22/15:27}, \eqref{13/05/04/13:54} and \eqref{13/05/09/10:46}). Regarding the initial data of a solution to \eqref{12/03/23/17:57} as a general function in $H^{1}(\mathbb{R}^{d})$, we find that the following fact holds: 
\begin{proposition}\label{18/01/27/13:12} 
There exists a constant $\delta_{E}(\omega) >0$ with the following property: let $u$ be a function in $H^{1}(\mathbb{R}^{d})$ satisfying $\mathcal{M}(u)=\mathcal{M}(\Phi_{\omega})$. Consider the decomposition of the form  
\begin{align}
\label{18/01/27/13:20}
&u=e^{i\theta[u]} (\Phi_{\omega} +\eta[u] ),
\quad 
\Omega(e^{-i\theta[u]} u, \partial_{\omega}\Phi_{\omega})\equiv 0,
\quad 
(e^{-i\theta[u]}u, \partial_{\omega}\Phi_{\omega})_{L_{real}^{2}}<0,
\\[6pt] 
\label{18/01/27/11:52} 
& 
\eta[u]
=\lambda_{+}[u] \mathscr{U}_{+}+\lambda_{-}[u] \mathscr{U}_{-} 
+\Gamma[u].
\end{align}
Define $\|\eta[u]\|_{E}$ by 
\begin{equation}\label{18/01/27/13:29}
\|\eta[u]\|_{E}^{2}
:=
\frac{\mu}{2} \big( \lambda_{+}^{2}[u]+\lambda_{-}^{2}[u]\big)
+
\frac{1}{2}
\langle 
\mathscr{L}_{\omega}\Gamma[u], 
\Gamma[u] \rangle_{H^{-1},H^{1}}
.
\end{equation}
Furthermore, assume that  
\begin{equation}\label{18/01/27/11:53}
\|\eta[u]\|_{E}
\le 
4\delta_{E}(\omega).
\end{equation}
Then,  
\begin{equation}\label{18/01/27/11:57}
\Bigm| 
\mathcal{H}(u)-\mathcal{H}(\Phi_{\omega})
+
\frac{\mu}{2} \big( \lambda_{+}[u]+\lambda_{-}[u] \big)^{2}
-
\|\eta[u]\|_{E}^{2}
\Bigm|
\le 
\frac{\|\eta[u]\|_{E}^{2}}{10}
. 
\end{equation}
\end{proposition}

\vspace{1pt}
Now, we introduce a distance function $d_{\omega}$. To this end, fix a non-increasing smooth function $\chi$ on $[0,\infty)$ such that 
\begin{equation}\label{13/05/12/10:19}
\chi(r)=\left\{ \begin{array}{ccc}
1  &\mbox{if}& 0\le r\le 1,
\\ 
0 &\mbox{if}&  r\ge 2. 
\end{array} \right.
\end{equation}
Moreover, we define 
\begin{equation}\label{18/01/26/22:15}
H_{\omega}^{1}(\mathbb{R}^{d}):=
\big\{
u \in H^{1}(\mathbb{R}^{d}) \colon \mathcal{M}(u)= \mathcal{M}(\Phi_{\omega})
\big\}
.
\end{equation}
Then, we define a function $d_{\omega}\colon H_{\omega}^{1}(\mathbb{R}^{d})\to [0,\infty)$ by    
\begin{equation}\label{13/05/05/13:27}
d_{\omega}(u)^{2}
:=
\|\eta[u]\|_{E}^{2}
+
\chi\Bigm( \frac{\|\eta[u]\|_{E}}{2\delta_{E}(\omega)}
\Bigm)
C_{\omega}(u), 
\end{equation}
where $\delta_{E}(\omega)$ is the constant given by Proposition \ref{18/01/27/13:12}, and 
\begin{equation}\label{13/05/05/13:55}
C_{\omega}(u)
:=\mathcal{H}(u)-\mathcal{H}(\Phi_{\omega})
+
\frac{\mu}{2} \big( \lambda_{+}[u]+\lambda_{-}[u] \big)^{2}
-
\|\eta[u]\|_{E}^{2}.
\end{equation}
We rephrase \eqref{18/01/27/11:57} as follows: if $\|\eta[u]\|_{E}\le 4\delta_{E}(\omega)$, then 
\begin{equation}\label{15/02/07/15:10}
|C_{\omega}(u) |\le \frac{\|\eta[u]\|_{E}^{2}}{10}. 
\end{equation}
 
Now, we consider a solution $\psi$ to \eqref{12/03/23/17:57} satisfying 
 $\mathcal{M}(\psi)=\mathcal{M}(\Phi_{\omega})$. In the decomposition \eqref{13/04/27/9:38} with \eqref{13/05/04/14:34}, it is convenient to introduce new parameters $\lambda_{1}(t)$ and $\lambda_{2}(t)$ defined by 
\begin{equation}\label{13/05/07/16:17}
\lambda_{1}(t):=\frac{\lambda_{+}(t)+\lambda_{-}(t)}{2},
\qquad 
\lambda_{2}(t):=\frac{\lambda_{+}(t)-\lambda_{-}(t)}{2}.
\end{equation}
We see from  \eqref{13/05/06/14:03} and  \eqref{13/05/06/14:04} that 
\begin{align}
\label{14/01/13/13:53}
\frac{d\lambda_{1}}{dt}(t)
&=
\mu \lambda_{2}(t)
+
\frac{1}{2}( \Big\{ \frac{d\theta}{dt}(t)-\omega \Big\}\eta(t)
-N_{\omega}(\eta(t)), \, \mathscr{U}_{+}-\mathscr{U}_{-} )_{L_{real}^{2}},
\\[6pt]
\label{15/03/30/16:37}
\frac{d\lambda_{2}}{dt}(t)
&=
\mu \lambda_{1}(t)
-
\frac{1}{2}( \Big\{ \frac{d\theta}{dt}(t)-\omega \Big\}\eta(t)
-N_{\omega}(\eta(t)), \, \mathscr{U}_{+}+\mathscr{U}_{-} )_{L_{real}^{2}}
.
\end{align}

An important property of the distance function $d_{\omega}(\psi(t))$ is the following:
\begin{lemma}\label{13/05/05/15:47}
Assume that there exists an interval $I$ on which        
\begin{equation}
\label{14/04/03/19:22}
\sup_{t\in I}d_{\omega}(\psi(t))\le \delta_{E}(\omega).
\end{equation}
Then, all of the following hold for all $t \in I$:  
\begin{align}
\label{13/05/05/17:08}
&\frac{1}{2}\|\eta(t)\|_{E}^{2} \le d_{\omega}(\psi(t))^{2} 
\le \frac{3}{2}\|\eta(t)\|_{E}^{2},
\\[6pt]
&d_{\omega}(\psi(t))^{2}
= 
\mathcal{H}(\psi)-\mathcal{H}(\Phi_{\omega})
+
2\mu \lambda_{1}^{2}(t),
\label{13/05/05/17:51}
\\[6pt]
\label{13/05/05/19:55}
&\frac{d}{dt}d_{\omega}(\psi(t))^{2}
=
4\mu^{2} \lambda_{1}(t) \lambda_{2}(t) 
+
4\mu \lambda_{1}(t) 
\Omega\big( \Big\{ \frac{d \theta}{d t}(t)-\omega \Big\}\eta(t) 
-N_{\omega}(\eta(t)), f_{2} \big).
\end{align}
Furthermore, if   
\begin{equation}
\label{13/05/05/15:49}
\mathcal{S}_{\omega}(\psi)<m_{\omega}+\frac{1}{2}d_{\omega}(\psi(t))^{2}
\end{equation} 
holds for all $t \in I$, then  
\begin{equation}
\label{13/05/05/15:53}
d_{\omega}(\psi(t)) \sim  | \lambda_{1}(t) |
\end{equation}
for all $t \in I$, where the implicit constant depends on $\omega$ as well as $d$ and $p$.  
\end{lemma}
\begin{proof}[Proof of Lemma \ref{13/05/05/15:47}] 
First, we shall show that for any $t \in I$, 
 \begin{equation}\label{15/02/07/15:16}
\|\eta(t)\|_{E}\le 4\delta_{E}(\omega).
\end{equation}
Suppose for contradiction that $\|\eta(t_{0})\|_{E}> 4\delta_{E}(\omega)$ for some $t_{0}\in I$. Then, it follows from the definition \eqref{13/05/05/13:27} that $d_{\omega}(\psi(t_{0}))=\|\eta(t_{0})\|_{E} \ge 4\delta_{E}(\omega)$. However, this contradicts the assumption \eqref{14/04/03/19:22}. 
Thus, we have proved \eqref{15/02/07/15:16}. Using \eqref{15/02/07/15:10} and \eqref{15/02/07/15:16}, we can verify \eqref{13/05/05/17:08}. Indeed,    
\begin{equation}\label{13/05/05/17:20}
\begin{split}
\frac{1}{2}\|\eta(t)\|_{E}^{2}
&\le 
\|\eta(t)\|_{E}^{2}-\big|C_{\omega}(\psi(t)) \big|
\\[6pt]
&\le 
d_{\omega}(\psi(t))^{2}
\le 
\|\eta(t)\|_{E}^{2}+\big|C_{\omega}(\psi(t)) \big|
 \le  \frac{3}{2}\|\eta(t)\|_{E}^{2}. 
\end{split}
\end{equation}
 
Next, we shall derive the equation \eqref{13/05/05/17:51}. We see from  \eqref{14/04/03/19:22} and \eqref{13/05/05/17:08}  that 
\begin{equation}\label{15/02/07/14:40}
\|\eta(t)\|_{E}^{2} 
\le 
2 d_{\omega}(\psi(t))^{2} \le 2\delta_{E}(\omega)^{2}.
\end{equation} 
Furthermore, it follows from the definition of $d_{\omega}(\psi(t))$ (see \eqref{13/05/05/13:27} and \eqref{13/05/05/13:55}) that  
\begin{equation}\label{13/05/05/17:57}
d_{\omega}(\psi(t))^{2}
= 
\|\eta(t)\|_{E}^{2}+C_{\omega}(\psi(t))
=
\mathcal{H}(\psi)-\mathcal{H}(\Phi_{\omega})
+
\frac{\mu}{2} \big( \lambda_{+}(t)+\lambda_{-}(t) \big)^{2}
,
\end{equation}
so that \eqref{13/05/05/17:51} holds. 
\par 
The equation \eqref{13/05/05/19:55} follows from \eqref{13/05/05/17:51} and \eqref{14/01/13/13:53}. 
\par 
Finally, we shall prove \eqref{13/05/05/15:53}. We see from the definition of $C_{\omega}(\psi(t))$ (see \eqref{13/05/05/13:55}), \eqref{15/02/07/15:16} and \eqref{15/02/07/15:10} that
\begin{equation}\label{13/05/05/17:40}
\begin{split}
\|\eta(t)\|_{E}^{2}
&= 
\mathcal{H}(\psi)-\mathcal{H}(\Phi_{\omega})
+
\frac{\mu}{2} \big( \lambda_{+}(t)+\lambda_{-}(t) \big)^{2} 
-C_{\omega}(\psi(t))
\\[6pt]
&\le 
\mathcal{H}(\psi)-\mathcal{H}(\Phi_{\omega})+\frac{\mu}{2} 
\big( \lambda_{+}(t)+\lambda_{-}(t) \big)^{2}
+
\frac{\|\eta(t)\|_{E}^{2}}{10},
\end{split}
\end{equation}
so that 
\begin{equation}\label{13/05/05/17:45}
\frac{9}{10}\|\eta(t)\|_{E}^{2} 
\le   
\mathcal{H}(\psi)-\mathcal{H}(\Phi_{\omega})+2 \mu \lambda_{1}(t)^{2}
.
\end{equation}
Moreover, it follows from the assumptions \eqref{13/04/21/11:42} and \eqref{13/05/05/15:49}, and $m_{\omega}=\mathcal{S}_{\omega}(\Phi_{\omega})$ that
\begin{equation}\label{13/05/05/16:09}
\mathcal{H}(\psi)-\mathcal{H}(\Phi_{\omega})
=
\mathcal{S}_{\omega}(\psi)-\mathcal{S}_{\omega}(\Phi_{\omega})
< \frac{1}{2}d_{\omega}(\psi(t))^{2}.
\end{equation}
Putting \eqref{13/05/05/17:08}, \eqref{13/05/05/17:45} and \eqref{13/05/05/16:09} together, we obtain that   
\begin{equation}\label{13/05/05/15:59}
d_{\omega}(\psi(t))^{2}
\le \frac{3}{2} 
\|\eta(t)\|_{E}^{2}
<  
\frac{5}{6}d_{\omega}(\psi(t))^{2} 
+
\frac{10}{3} \mu \lambda_{1}(t)^{2}
.
\end{equation}
Hence, we have 
\begin{equation}\label{13/05/05/16:13}
d_{\omega}(\psi(t))^{2}<200 \mu \lambda_{1}(t)^{2}.
\end{equation}
On the other hand, we see from \eqref{13/05/11/17:19}, Lemma \ref{13/05/11/11:52} and \eqref{13/05/05/17:08} that 
\begin{equation}\label{13/05/05/16:42}
\mu \lambda_{1}(t)^{2}
\le 
2 \mu \big( \lambda_{+}^{2}(t)+\lambda_{-}^{2}(t)\big)
\le 
4 \mu \|\mathscr{U}_{+}\|_{L^{2}}^{2} \|\eta(t)\|_{L^{2}}^{2}
\lesssim  d_{\omega}(\psi(t))^{2},
\end{equation}
where the implicit constant depends on $\mu$ and $\|\mathscr{U}_{+}\|_{L^{2}}^{2}$. Combining \eqref{13/05/05/16:13} and \eqref{13/05/05/16:42}, we obtain \eqref{13/05/05/15:53}.
\end{proof}

%%%%%%%%%%%%%%%%%%%%%%%%%%%%%%%%%%%%%%%%%%%%%%%%%%%%%%%%%%%%%%%%%%%%%%%%%%%%%%%%
\subsection{Fundamental properties of the eigenfunctions}\label{17/10/28/23:01}
 In connection with $\lambda_{1}$ and $\lambda_{2}$ (see \eqref{13/05/07/16:17}), we introduce the real-valued functions $f_{1}$ and $f_{2}$:  
\begin{equation}\label{15/02/07/13:35}
f_{1}
:=\frac{\mathscr{U}_{+}+\mathscr{U}_{-}}{2}
=\Re[\mathscr{U}_{+}],
\qquad 
f_{2}
:=
\frac{\mathscr{U}_{+}-\mathscr{U}_{-}}{2i}
=
\Im[\mathscr{U}_{+}].
\end{equation}
We shall observe the properties of $f_{1}$ and $f_{2}$. Throughout this subsection, we assume that  $d\ge 3$, $1+\frac{4}{d}<p<2^{*}-1$, $0< \omega <\omega_{2}$ and \eqref{13/04/21/11:42}, where $\omega_{2}$ is the constant given by Proposition \ref{13/02/20/9:41}. 
\par 
First, we note that the decomposition \eqref{13/05/04/14:34} of $\eta(t)$ is expressed as follows in terms of the functions $f_{1}$ and $f_{2}$ defined by \eqref{15/02/07/13:35}:
\begin{equation}\label{13/12/28/10:44}
\eta(t)
=
2 \lambda_{1}(t)f_{1}+ 2 i \lambda_{2}(t)f_{2}+\Gamma(t).
\end{equation}

We see from \eqref{13/02/20/9:45} and $\mathscr{L}_{\omega}\mathscr{U}_{+}= i\mu \mathscr{U}_{+}$ that  
\begin{equation}\label{13/01/02/16:25}
L_{\omega,+}f_{1}
=
-\mu f_{2},
\qquad 
L_{\omega,-}f_{2}=\mu f_{1}.
\end{equation}
Furthermore, it follows from the elliptic regularity for \eqref{13/01/02/16:25}  that 
\begin{equation}\label{15/02/11/17:02}
\| f_{1} \|_{L^{\infty}}+ \| f_{2} \|_{L^{\infty}} <\infty. 
\end{equation}

\begin{lemma}\label{13/03/04/9:00}
We have that $f_{2} \not\in {\rm span}\,\{\Phi_{\omega}\}={\rm Ker}\,L_{\omega,-}$.
\end{lemma}
\begin{proof}[Proof of Lemma \ref{13/03/04/9:00}]
Suppose for a contradiction that $f_{2}=k\Phi_{\omega}$ for some $k\neq 0$. Then, it follows from \eqref{13/12/28/15:36} and \eqref{15/02/07/13:35} that 
\begin{equation}\label{13/03/04/9:25}
0=k \mathscr{L}_{\omega}(i\Phi_{\omega})
=
\mathscr{L}_{\omega}(if_{2})
=
\frac{1}{2}\mathscr{L}_{\omega}(\mathscr{U}_{+}-\mathscr{U}_{-})
=\frac{1}{2}i \mu \mathscr{U}_{+}+ \frac{1}{2}i\mu \mathscr{U}_{-}
=i\mu f_{1}. 
\end{equation}
Furthermore, we see from \eqref{13/01/02/16:25} and \eqref{13/03/04/9:25} that  \begin{equation}\label{13/03/04/9:26}
-\mu f_{2}=L_{\omega,+}f_{1}=0. 
\end{equation}
Thus, $f_{1}=f_{2}\equiv 0$. This is a contradiction. 
\end{proof}

\begin{lemma}\label{13/03/03/14:50}
We have the following orthogonalities:
\begin{equation} \label{14/04/01/20:19}
\big(\Phi_{\omega},f_{1} \big)_{L^{2}}
=
\big( \partial_{\omega}\Phi_{\omega}, f_{2} \big)_{L^{2}}
=0.
\end{equation}
Furthermore, we have 
\begin{equation}\label{13/03/02/9:56}
\big( f_{1},f_{2} \big)_{L^{2}}>0
\end{equation}
and 
\begin{equation}\label{13/03/03/15:19}
\big(\Phi_{\omega},f_{2} \big)_{L^{2}}\neq 0.
\end{equation}
\end{lemma}
\begin{proof}[Proof of Lemma \ref{13/03/03/14:50}]
Since $L_{\omega,-}\Phi_{\omega}=0$ and $L_{\omega,-}$ is self-adjoint in $L^{2}(\mathbb{R}^{d})$, we find from \eqref{13/01/02/16:25} that 
\begin{equation}\label{13/03/03/14:43}
\big(\Phi_{\omega},f_{1} \big)_{L^{2}}
=
\mu^{-1}\langle L_{\omega,-}f_{2}, \Phi_{\omega} \rangle_{H^{-1},H^{1}}
=
\mu^{-1} \big(L_{\omega,-}\Phi_{\omega}, f_{2} \big)_{L^{2}}
=
0.
\end{equation} 
On the other hand, it follows from \eqref{13/01/02/16:25},  \eqref{13/02/20/10:40} and the self-adjointness of $L_{\omega,+}$ in $L^{2}(\mathbb{R}^{d})$ that  
\begin{equation}\label{13/03/03/15:06}
\big(\partial_{\omega}\Phi_{\omega},f_{2} \big)_{L^{2}}
=
-\mu^{-1}\langle L_{\omega,+}f_{1} , \partial_{\omega}\Phi_{\omega} \rangle_{H^{-1},H^{1}}
=
-\mu^{-1}\big( L_{\omega,+}\partial_{\omega}\Phi_{\omega} , f_{1} \big)_{L^{2}}
=
\mu^{-1}\big( \Phi_{\omega} , f_{1} \big)_{L^{2}},
\end{equation} 
which together with \eqref{13/03/03/14:43} proves \eqref{14/04/01/20:19}. 
\par 
Next, we shall prove \eqref{13/03/02/9:56}. It follows from \eqref{13/01/02/16:25} that 
\begin{equation}\label{13/03/02/9:26}
\big( f_{1},f_{2} \big)_{L^{2}}
=
\mu^{-1} \langle L_{\omega,-}f_{2},f_{2}\rangle_{H^{-1},H^{1}}. 
\end{equation}
Since $f_{2} \not\in {\rm Ker}\, L_{\omega,-}$ (see Lemma \ref{13/03/04/9:00}),  Lemma \ref{13/03/18/15:50} together with \eqref{13/03/02/9:26} shows the desired result.
\par 
Finally, we prove \eqref{13/03/03/15:19}. Suppose for contradiction that 
$(\Phi_{\omega}, f_{2})_{L^{2}}=0$. Then, we see from \eqref{13/12/28/12:07}, the self-adjointness of $L_{\omega,+}$ and \eqref{13/01/02/16:25} 
that  
\begin{equation}\label{13/03/03/14:56}
\big((p-1)\Phi_{\omega}^{p}+(2^{*}-2)\Phi_{\omega}^{2^{*}-1}, f_{1} \big)_{L^{2}}
=
-\big(L_{\omega,+}\Phi_{\omega}, f_{1} \big)_{L^{2}}
=
\mu \big(\Phi_{\omega},f_{2} \big)_{L^{2}}
=0.
\end{equation} 
Hence, it follows from Lemma \ref{13/01/01/16:58} that $\langle L_{\omega,+}f_{1},f_{1} \rangle_{H^{-1},H^{1}} \ge 0$. However, it must follow from \eqref{13/01/02/16:25} and \eqref{13/03/02/9:56} that 
\begin{equation}\label{13/03/04/9:57}
\langle L_{\omega,+}f_{1},f_{1} \rangle_{H^{-1},H^{1}}
=
-\mu \big(f_{2},f_{1}\big)_{L^{2}}<0.
\end{equation}
This is a contradiction. Thus, we find that \eqref{13/03/03/15:19} holds. 
\end{proof} 

The relation \eqref{13/03/03/15:19} in Lemma \ref{13/03/03/14:50} allows us to choose $f_{2}$ so that 
\begin{equation}\label{13/12/28/17:22}
 (\Phi_{\omega},f_{2})_{L^{2}}<0.
\end{equation}

In our analysis, the frequency $\omega$ varies. In particular, we need to take $\omega \to 0$. Hence, we have to pay attention to the dependence of $\omega$. Such a bother does not appear in the scale invariant cases such as \eqref{15/05/06/11:39} (see \cite{Nakanishi-Schlag2}). The following lemma plays an important role to prove the ejection lemma (Lemma \ref{13/05/05/22:31}):
\begin{lemma}\label{13/12/19/09:42}
For any constant $C>0$, there exists $\omega(C)>0$ such that for any $\omega \in (0,\omega(C))$, 
\begin{equation}\label{13/12/19/09:47}
\mu \big| ( \Phi_{\omega}, f_{2} )_{L^{2}}\big|
\ge C
\big| (\Phi_{\omega}^{2^{*}-1}, f_{1} )_{L^{2}}\big|.
\end{equation}
\end{lemma}
\begin{proof}[Proof of Lemma \ref{13/12/19/09:42}]
 We use the notation $\mu_{\omega}$, $f_{1,\omega}$ and $f_{2,\omega}$ instead of $\mu$, $f_{1}$ and $f_{2}$ in order to emphasise the dependence on $\omega$. 
\par 
Suppose for contradiction that there exists a constant $C_{0}>0$ with the following property: for any number $n$, there exists $\omega_{n}\in (0,\frac{1}{n})$ such that 
\begin{equation}\label{15/02/16/11:03}
\mu_{\omega_{n}} \big| ( \Phi_{\omega_{n}}, f_{2,\omega_{n}} )_{L^{2}}\big|
< C_{0}
\big| (\Phi_{\omega_{n}}^{2^{*}-1}, f_{1,\omega_{n}} )_{L^{2}}\big|,
\end{equation}
where $f_{1,\omega_{n}}$ and $f_{2,\omega_{n}}$ are functions given by \eqref{15/02/07/13:35} for $\omega_{n}$. 
We consider the functions $f_{1,n}$ and $f_{2,n}$ defined by 
\begin{equation}\label{13/12/19/09:19}
f_{1,n}:=\omega_{n}^{-s_{p}}T_{\omega_{n}}f_{1,\omega_{n}}
,\qquad 
f_{2,n}:=\omega_{n}^{-s_{p}}T_{\omega_{n}}f_{2,\omega_{n}}.
\end{equation}
It follows from \eqref{13/01/02/16:25} that 
\begin{equation}\label{13/12/19/09:29}
\widetilde{L}_{\omega_{n},+}f_{1,n}
=
- \mu_{\omega_{n}} \omega_{n}^{-1}f_{2,n},
\quad 
\widetilde{L}_{\omega_{n},-}f_{2,n}
=
\mu_{\omega_{n}} \omega_{n}^{-1}f_{1,n},
\end{equation}
where $\widetilde{L}_{\omega,+}$ and $\widetilde{L}_{\omega,-}$ are the 
 operators defined by \eqref{13/03/30/11:40} and \eqref{13/04/12/14:40}, respectively. We put 
\begin{equation}\label{15/02/16/10:21}
g_{n}:=\frac{f_{1,n}}{\|f_{1,n}\|_{H^{1}}},
\quad 
h_{n}:=\frac{f_{2,n}}{\|f_{1,n}\|_{H^{1}}}.
\end{equation}
Then, we can take a real-valued function $g$ in $H^{1}(\mathbb{R}^{d})$  such that 
\begin{equation}
\label{13/12/19/09:53}
\lim_{n\to \infty}g_{n}=g \qquad \mbox{weakly in $H^{1}(\mathbb{R}^{d})$}
.
\end{equation}
We shall show that $g$ is non-trivial. Suppose for contradiction that $g$ was trivial. Then, it follows from Lemma \ref{13/03/03/14:50}, \eqref{13/12/19/09:29}, Proposition \ref{13/03/03/9:43} and \eqref{13/12/19/09:53} that 
\begin{equation}\label{13/12/19/09:55}
\begin{split}
0&
\ge 
\lim_{n\to \infty}\frac{- \omega_{n}^{-s_{p}-1} \mu_{\omega_{n}}}{\|f_{1,n}\|_{H^{1}}^{2}}
( f_{1,\omega_{n}}, f_{2,\omega_{n}})_{L^{2}}
= 
\lim_{n\to \infty}\frac{-\omega_{n}^{-1}\mu_{\omega_{n}}}{\|f_{1,n}\|_{H^{1}}^{2}} 
( f_{1,n}, f_{2,n})_{L^{2}}
\\[6pt]
&=
\lim_{n\to \infty} \frac{1}{\|f_{1,n}\|_{H^{1}}^{2}}
\langle \widetilde{L}_{\omega_{n},+}f_{1,n}, f_{1,n}\rangle _{H^{-1},H^{1}}
\\[6pt]
&=
1- \lim_{n\to \infty}
\int_{\mathbb{R}^{d}}
\bigm\{ p(T_{\omega_{n}}\Phi_{\omega_{n}})^{p-1}
|g_{n}|^{2}+
\omega_{n}^{\frac{2^{*}-(p+1)}{p-1}}(2^{*}-1)
(T_{\omega_{n}}\Phi_{\omega_{n}})^{\frac{4}{d-2}}|g_{n}|^{2}
\big\}\,dx 
\\[6pt]
&=1.
\end{split}
\end{equation}
This is a contradiction, and therefore $g$ is non-trivial. 
\par 
Next, we shall show that there exists a constant $C_{0}>0$ such that  
 for any number $n$ and any real-valued function $f \in H^{1}(\mathbb{R}^{d})$ with 
$(f, T_{\omega_{n}}\Phi_{\omega_{n}})_{L^{2}} = 0$, 
\begin{equation} \label{15/03/15/14:03}
\langle \widetilde{L}_{\omega_{n}, -}f, f\rangle_{H^{-1},H^{1}} 
\ge 
C_{0} \|f\|^{2}_{H^{1}}. 
\end{equation}
Suppose for contradiction that we could take a subsequence of $\{\omega_{n}\}$ (still denoted by the same symbol) and a sequence $\{f_{n}\}$ in $H^{1}(\mathbb{R}^{d})$ such that 
\begin{equation} \label{15/03/15/15:13}
\lim_{n \to \infty} 
\langle \widetilde{L}_{\omega_{n}, -}f_{n}, f_{n}\rangle_{H^{-1},H^{1}}=0, 
\qquad
\|f_{n}\|_{H^{1}} = 1, \qquad   
(f_{n}, T_{\omega_{n}}\Phi_{\omega_{n}})_{L^{2}} = 0.
\end{equation}
Furthermore, we can take $f_{0} \in H^{1}(\mathbb{R}^{d})$
 such that $\lim_{n \to \infty}f_{n} = f_{0}$ weakly in $H^{1}(\mathbb{R}^{d})$. Then, we see from the weak lower semicontinuity and Proposition \ref{13/03/03/9:43} that 
\begin{equation} \label{15/03/15/15:18}
\begin{split}
0&= 
\lim_{n \to \infty} 
\langle \widetilde{L}_{\omega_{n}, -}f_{n}, f_{n}\rangle_{H^{-1},H^{1}} 
\\[6pt]
&=
\lim_{n \to \infty}
\bigm\{ 
\|f_{n}\|^{2}_{H^{1}} - 
\int(T_{\omega_{n}} \Phi_{\omega_{n}})^{p-1}|f_{n}|^{2} dx 
- \omega_{n}^{\frac{2^{*}-(p+1)}{p-1}}
\int(T_{\omega_{n}} \Phi_{\omega_{n}})^{\frac{4}{d-2}}|f_{n}|^{2} dx 
\bigm\}
\\[6pt]
&\ge 
\langle L_{-}^{\dagger}f_{0}, f_{0} \rangle_{H^{-1},H^{1}}.  
\end{split}
\end{equation}
Moreover, it follows from Proposition \ref{13/03/03/9:43} and the hypothesis \eqref{15/03/15/15:13} that 
\begin{equation}\label{15/03/30/11:42}
(f_{0}, U)_{L^{2}} = \lim_{n \to \infty}
(f_{n}, T_{\omega_{n}}\Phi_{\omega_{n}})_{L^{2}} = 0.
\end{equation}
Hence, we conclude from \eqref{15/03/15/15:18} and the positivity of $L_{-}^{\dagger}$ together with \eqref{15/03/30/11:42} that 
\begin{equation} \label{15/03/15/15:22}
0 \ge \langle L_{-}^{\dagger} f_{0}, f_{0} \rangle \gtrsim \|f_{0}\|^{2}_{H^{1}},
\end{equation}
so that $f_{0}$ is trivial. However, the same argument as \eqref{13/12/19/09:55} yields that $f_{0}$ is non-trivial. Thus, we arrive at a contradiction and therefore \eqref{15/03/15/14:03} holds.  
\par 
We shall show that $\{h_{n}\}$ is bounded in $H^{1}(\mathbb{R}^{d})$. We see from \eqref{15/03/15/14:03} and \eqref{13/12/19/09:29} that  
\begin{equation} \label{15/03/15/14:07}
\begin{split}
C_{0} \|f_{2, n}\|^{2}_{H^{1}} 
&\le 
\langle \widetilde{L}_{\omega_{n}, -}f_{2, n}, f_{2, n}\rangle_{H^{-1},H^{1}} 
= 
\mu_{\omega_{n}} \omega^{-1}_{n} (f_{1, n}, f_{2, n})_{L^{2}} 
\\[6pt]
&= 
- \langle \widetilde{L}_{\omega_{n}, +}f_{1, n}, f_{1, n} \rangle_{H^{-1},H^{1}}
\le C \|f_{1, n}\|^{2}_{H^{1}}. 
\end{split}
\end{equation}
Dividing the both sides above by $\|f_{1, n}\|_{H^{1}}^{2}$, we find that  
 $\{h_{n}\}$ is bounded in $H^{1}(\mathbb{R}^{d})$. 
\par 
We shall show that the sequence $\{\mu_{\omega_{n}} \omega^{-1}_{n}\}$ is bounded. It follows from \eqref{13/12/19/09:29} that 
\begin{equation} \label{15/03/15/14:16}
\big| \mu_{\omega_{n}} \omega^{-1}_{n} |
= 
\dfrac{|\langle \widetilde{L}_{\omega_{n}, -}f_{2, n}, f_{1, n}\rangle_{H^{-1},H^{1}}|}
{\|f_{1, n}\|^{2}_{L^{2}}} . 
\end{equation}
Furthermore, it follows from \eqref{13/12/19/09:53} and the boundedness of $\{g_{n}\}$ and $\{h_{n}\}$ in $H^{1}(\mathbb{R}^{d})$ that 
\begin{equation} \label{15/03/15/14:21}
\big| \mu_{\omega_{n}} \omega^{-1}_{n} \big|
\le \dfrac{\sup_{n \in \mathbb{N}} 
|\langle \widetilde{L}_{\omega_{n}, -}h_{n}, g_{n} \rangle_{H^{-1},H^{1}}|}
{\liminf_{n \to \infty}\|g_{n}\|^{2}_{L^{2}}} 
\lesssim \dfrac{\sup_{n \in \mathbb{N}} 
(\|g_{n}\|^{2}_{H^{1}} + \|h_{n}\|^{2}_{H^{1}})}
{\|g \|^{2}_{L^{2}}} \lesssim 1. 
\end{equation}
Thus, we find that  $\{\mu_{\omega_{n}} \omega^{-1}_{n}\}$ is bounded. 
\par 
Since $\{h_{n}\}$ is bounded in $H^{1}(\mathbb{R}^{d})$, there exists $h \in H^{1}(\mathbb{R}^{d})$ such that 
\begin{equation}\label{15/03/30/12:54}
\lim_{n \to \infty} h_{n} = h \qquad \mbox{weakly in 
$H^{1}(\mathbb{R}^{d})$}. 
\end{equation}
Moreover, we can take a subsequence of $\{\mu_{\omega_{n}} \omega^{-1}_{n}\}$ (still denoted by the same symbol) and $\nu_{*} \in [0,\infty)$ such that $\lim_{n \to \infty}\mu_{\omega_{n}} \omega^{-1}_{n} = \nu_{*}$. We shall show that $\nu_{*} \neq 0$ and $h$ is non-trivial. Recall here that $\mu_{\omega_{n}}$ is a positive eigenvalue of $-i \mathcal{L}_{\omega_{n}}$. We find from the proof of Proposition \ref{13/02/20/9:41} (see \eqref{15/05/27/15:59} and \eqref{14/06/27/16:27}) that 
\begin{equation}\label{15/06/26/14:27}
-\mu_{\omega_{n}}^{2}
=
\inf\Bigm\{ 
\frac{\langle L_{\omega_{n},+}u,u \rangle_{H^{-1},H^{1}}}{((L_{\omega_{n},-})^{-1}u, u)_{L_{real}^{2}}}
\colon u \in H^{1}(\mathbb{R}^{d}), (u,\Phi_{\omega_{n}})_{L_{real}^{2}}=0 
\Bigm\}.
\end{equation}
Furthermore, it follows from the estimates \eqref{14/07/13/14:35} and \eqref{14/07/13/14:34} in the proof of Proposition \ref{13/02/20/9:41} shows that 
\begin{equation}\label{15/05/24/15:45}
-\nu_{*}^{2}
=
\lim_{n\to \infty}\omega_{n}^{-2}(-\mu_{\omega_{n}}^{2})
=
\frac{\langle L_{+}^{\dagger} \Pi \Lambda U, \Pi \Lambda U \rangle_{H^{-1},H^{1}}}{( 
 (L_{-}^{\dagger})^{-1} \Pi \Lambda U, \Pi \Lambda U)_{L_{real}^{2}}}
<0,
\end{equation}
where $\Pi$ is the projection given by \eqref{15/05/26/17:50}. 
Thus, we find that $\nu_{*}\neq 0$. Next, suppose for contradiction that $h$ was trivial. Then, we see from \eqref{13/12/19/09:29} and \eqref{13/12/19/09:53} that for any $\varphi \in C^{\infty}_{0}(\mathbb{R}^{d})$, 
\begin{equation} \label{15/03/15/14:55}
0=
\lim_{n\to \infty} \langle \widetilde{L}_{\omega_{n}, -}h_{n}, \varphi \rangle_{H^{-1},H^{1}} 
=
\lim_{n\to \infty} ( \mu_{\omega_{n}}\omega_{n}^{-1} g_{n}, \varphi )_{L_{real}^{2}}
=
( \nu_{*}g, \varphi )_{L_{real}^{2}}.
\end{equation}
However, this contradicts that $\nu_{*}\neq 0$ and $g$ is non-trivial. Thus, we find that $h$ is non-trivial.
\par 
Note that $g$ and $h$ satisfies that 
\begin{equation} \label{15/03/15/15:06}
L_{+}^{\dagger}g = -\nu_{*} h, \qquad 
L_{-}^{\dagger}h = \nu_{*}g. 
\end{equation}
Then, applying the same proof of \eqref{13/03/03/15:19} (use Lemma 2.2 in \cite{CGNT} instead of Lemma \ref{13/01/01/16:58}), we find that 
\begin{equation}\label{15/03/30/13:13}
\langle 
L_{+}^{\dagger} U, h
\rangle_{H^{-1},H^{1}} 
=
(p-1)(U^{p},h)_{L^{2}}
\neq 0.
\end{equation}
We see from \eqref{13/01/02/16:25} and the scaling that 
\begin{equation}\label{13/12/28/17:07}
-\omega_{n}^{-1}\mu_{\omega_{n}}
( \Phi_{\omega_{n}}, f_{2,\omega_{n}} )_{L^{2}}
=
\langle \widetilde{L}_{\omega_{n},+}T_{\omega_{n}}\Phi_{\omega_{n}}, 
f_{1,n}\rangle_{H^{-1},H^{1}}.
\end{equation}
Moreover, it follows  from Proposition \ref{13/03/03/9:43} that 
\begin{equation}\label{13/12/19/13:56}
\lim_{n \to \infty} \widetilde{L}_{\omega_{n},+}T_{\omega_{n}}\Phi_{\omega_{n}}
=
L_{+}^{\dagger} 
U \qquad \mbox{strongly in $H^{-1}(\mathbb{R}^{d})$}.
\end{equation}
This together with \eqref{15/03/30/12:54} yields 
\begin{equation}\label{13/12/19/13:56}
\lim_{n\to \infty} \langle 
\widetilde{L}_{\omega_{n},+}T_{\omega_{n}}
\Phi_{\omega_{n}},\, h_{n}
\rangle_{H^{-1},H^{1}}
=
\langle 
L_{+}^{\dagger} U, h \rangle_{H^{-1},H^{1}} .
\end{equation}
We find from \eqref{13/12/28/17:07}, \eqref{13/12/19/13:56} and \eqref{15/03/30/13:13} that 
\begin{equation}\label{13/12/28/17:12}
\begin{split}
\lim_{n \to \infty}\|f_{1,n}\|_{H^{1}}^{-1}\omega_{n}^{-1}\mu_{\omega_{n}}
\big| ( \Phi_{\omega_{n}}, f_{2,\omega_{n}} )_{L^{2}}\big|
&=
\lim_{n\to \infty}  
\big|
\langle 
\widetilde{L}_{\omega_{n},+}T_{\omega_{n}}
\Phi_{\omega_{n}},\, h_{n}
\rangle_{H^{-1},H^{1}}
\big|
\\[6pt]
&=
\big| \langle 
L_{+}^{\dagger} U, h
\rangle_{H^{-1},H^{1}}
\big|  
\gtrsim 1,
\end{split}
\end{equation}
On the other hand, it follows from Proposition \ref{13/03/03/9:43} and \eqref{13/12/19/09:53} that 
\begin{equation}\label{13/12/28/16:59}
\begin{split}
\lim_{n\to 0} \|f_{1,n}\|_{H^{1}}^{-1} \omega_{n}^{-\frac{2^{*}-2}{p-1}}
\big| (\Phi_{\omega_{n}}^{2^{*}-1}, f_{1,,\omega_{n}} )_{L^{2}}\big|
&=
\lim_{n\to 0}
\big| (\big(T_{\omega_{n}} \Phi_{\omega_{n}}\big)^{2^{*}-1}, g_{n})_{L^{2}}\big|\\[6pt]
&=
\big| (U^{2^{*}-1},g)_{L^{2}} \big| \lesssim 1.
\end{split}
\end{equation} 
Since $\frac{2^{*}-2}{p-1}>1$, we conclude from \eqref{13/12/28/17:12} and \eqref{13/12/28/16:59} that for any sufficiently large number $n$, 
\begin{equation}\label{15/02/16/16:48}
\mu_{\omega_{n}}
\big| ( \Phi_{\omega_{n}}, f_{2,,\omega_{n}} )_{L^{2}}\big|
\gtrsim 
 \omega_{n}^{-\frac{2^{*}-2}{p-1}+1}
\big| (\Phi_{\omega_{n}}^{2^{*}-1}, f_{1,,\omega_{n}} )_{L^{2}}\big|
\gg C_{0}
\big| (\Phi_{\omega_{n}}^{2^{*}-1}, f_{1,\omega_{n}} )_{L^{2}}\big|,
\end{equation}
where $C_{0}$ is the constant given in the hypothesis \eqref{15/02/16/11:03}. 
However, this contradicts \eqref{15/02/16/11:03}. Thus, the desired result \eqref{13/12/19/09:47} holds.  
\end{proof}

Now, we see from \eqref{13/02/22/16:30} and \eqref{13/04/27/9:57} that 
\begin{equation}\label{13/12/29/12:12}
(\Phi_{\omega}, \Gamma(t))_{L_{real}^{2}}
=
-\mathcal{M}(\eta(t))
.
\end{equation}
Moreover, we see from \eqref{13/12/28/12:07} and \eqref{13/12/28/15:36} that  
\begin{equation}\label{15/02/15/16:08}
\begin{split}
&2\omega \Phi_{\omega}
-2\Delta \Phi_{\omega} 
-\frac{d(p-1)}{2} \Phi_{\omega}^{p}
-2^{*}\Phi_{\omega}^{2^{*}-1}
\\[6pt]
&=
(2-s_{p})L_{\omega,-} \Phi_{\omega}
+
s_{p} L_{\omega,+}\Phi_{\omega} 
-
(1-s_{p})(2^{*}-2)
\Phi_{\omega}^{2^{*}-1}
\\[6pt]
&=
s_{p}L_{\omega,+}\Phi_{\omega}
-
(1-s_{p})(2^{*}-2) \Phi_{\omega}^{2^{*}-1}
=
-s_{p}(p-1)\Phi_{\omega}^{p}
-
(2^{*}-2) \Phi_{\omega}^{2^{*}-1}.
\end{split}
\end{equation}
Furthermore, this together with \eqref{13/12/28/10:44}, \eqref{13/01/02/16:25} and \eqref{13/12/29/12:12} shows that  
\begin{equation}\label{14/02/15/15:25}
\begin{split}
&\mathcal{K}'(\Phi_{\omega}) \eta(t) 
\\[6pt]
&=
2\langle 
\omega \Phi_{\omega} -\Delta \Phi_{\omega} 
-\frac{d(p-1)}{4} \Phi_{\omega}^{p}
-\frac{2^{*}}{2}\Phi_{\omega}^{2^{*}-1}, \eta(t) 
\rangle_{H^{-1},H^{1}}
-
2\omega (\Phi_{\omega},\eta(t))_{L_{real}^{2}}
\\[6pt]
&=
2 s_{p} \lambda_{1}(t)
(
L_{\omega,+}\Phi_{\omega}, f_{1} 
)_{L^{2}}
-
2(1-s_{p})(2^{*}-2)
\lambda_{1}(t)
(\Phi_{\omega}^{2^{*}-1}, f_{1} 
)_{L^{2}}
\\[6pt]
&\quad -
s_{p}(  
(p-1)\Phi_{\omega}^{p}
+(2^{*}-2)\Phi_{\omega}^{2^{*}-1}, 
\Gamma(t) 
)_{L_{real}^{2}}
\\[6pt]
&\quad -
(1-s_{p})(2^{*}-2)
(\Phi_{\omega}^{2^{*}-1}, \Gamma(t) 
)_{L_{real}^{2}}
+ \omega \mathcal{M}(\eta(t)) 
\\[6pt]
&=
-2\mu s_{p}\lambda_{1}(t)( \Phi_{\omega}, f_{2})_{L^{2}}
-
2(1-s_{p})(2^{*}-2)\lambda_{1}(t)
(\Phi_{\omega}^{2^{*}-1}, f_{1} 
)_{L^{2}}
\\[6pt]
&\quad 
-s_{p}(p-1)(\Phi_{\omega}^{p}, \Gamma(t))_{L_{real}^{2}}
-(2^{*}-2) (\Phi_{\omega}^{2^{*}-1},\Gamma(t))_{L_{real}^{2}}
+
\omega \mathcal{M}(\eta(t)) 
.
\end{split}
\end{equation}
%%%%%%%%%%%%%%%%%%%%%%%%%%%%%%%%%%%%%%%%%%%%%%%%%%%%%%%%%%%%%%%%%%%

\section{Ejection lemma}\label{15/01/07/14:31}
Let $\omega_{2}$ denote the frequency given by Proposition \ref{13/02/20/9:41} throughout this section. Furthermore, for a given $\omega \in (0,\omega_{2})$, $\delta_{E}(\omega)$ denotes the constant given by Proposition \ref{18/01/27/13:12}. 
\par 
We see from the argument in the previous section that any solution $\psi$ satisfying $\mathcal{M}(\psi)=\mathcal{M}(\Phi_{\omega})$ has the decomposition of the form \eqref{13/04/27/9:38} with \eqref{13/02/22/15:27}, \eqref{13/05/04/13:54} and \eqref{13/05/09/10:46}. Recall that $d_{\omega} \colon H_{\omega}^{1}(\mathbb{R}^{d})\to [0,\infty)$ denotes the distance function defined by \eqref{13/05/05/13:27}. 
\par 
The following lemma is a criterion of continuation for the solutions to \eqref{12/03/23/17:57} in terms of the ground state:
\begin{lemma}\label{15/11/08/11:45}
Assume $d\ge 3$ and $1+\frac{4}{d}<p < 2^{*}-1$. Then, for any $\omega \in (0,\omega_{2})$, there exists $\delta_{0}(\omega) \in (0,\delta_{E}(\omega))$ such that if $\psi$ is a solution to \eqref{12/03/23/17:57} satisfying 
\begin{equation}
\label{15/11/08/08:17}
\mathcal{M}(\psi)=\mathcal{M}(\Phi_{\omega}),
\end{equation}
then $\psi$ extends as long as $d_{\omega}(\psi(t))\le \delta_{0}(\omega)$.
\end{lemma}
\begin{proof}[Proof of Lemma \ref{15/11/08/11:45}]
We prove the claim by contradiction. Hence, suppose to the contrary that there exists $\omega \in (0,\omega_{2})$ with the following property: for any $\delta_{0}\in (0,\delta_{E}(\omega))$, there exists a solution $\psi$ to \eqref{12/03/23/17:57} such that: $\mathcal{M}(\psi)=\mathcal{M}(\Phi_{\omega})$; $T_{\max}:=\sup{I_{\max}(\psi)}<\infty$; and $\sup_{t\in [t_{0}, T_{\max})}d_{\omega}( \psi(t) )\le \delta_{0}$ for some $t_{0}\in I_{\max}$. Then, it follows from Lemma \ref{13/05/11/11:52} and Lemma \ref{13/05/05/15:47} that 
\begin{equation}\label{15/11/08/06:00}
\sup_{t\in [t_{0},T_{\max})}\|\eta(t)\|_{H^{1}}
\lesssim
\sup_{t\in [t_{0},T_{\max})}d_{\omega}\big(\psi(t)\big)
< 
\delta_{0},
\end{equation}
where the implicit constant depends only on $d$, $p$ and $\omega$. In particular, we have \begin{equation}\label{15/07/21/14:41}
\sup_{t\in [t_{0},T_{\max})}\|\psi(t)\|_{H^{1}}
\lesssim   
\|\Phi_{\omega}\|_{H^{1}}+\delta_{0}.
\end{equation}
Let $t_{1} \in (t_{0},T_{\max})$. Then, Strichartz' estimate together with 
 \eqref{15/11/08/06:00} shows that 
\begin{equation}\label{14/09/12/01:47}
\begin{split}
\|\langle \nabla \rangle e^{i(t-t_{1})\Delta}\psi(t_{1})\|_{St([t_{1},T_{\max}])}
&\lesssim
\|\langle \nabla \rangle e^{i(t-t_{1})\Delta}\Phi_{\omega}\|_{St([t_{1},T_{\max}])}
+\| \eta(t_{1})\|_{H^{1}}
\\[6pt]
&\lesssim  
\|\langle \nabla \rangle e^{i(t-t_{1})\Delta}\Phi_{\omega}\|_{St([t_{1},T_{\max}])}
+
\delta_{0}.
\end{split}
\end{equation}
Furthermore, the small-data theory (Lemma \ref{14/01/30/10:19}) together with \eqref{15/07/21/14:41} shows that there exists $\delta(\omega)>0$  such that if
\begin{equation}\label{15/02/08/21:37}
\|\langle \nabla \rangle e^{i(t-t_{1})\Delta}\psi(t_{1})\|_{St([t_{1},T_{\max}])}
\le \delta(\omega),
\end{equation} 
then the solution $\psi$ exists on $[t_{1},T_{\max}]$. Since 
\begin{equation}\label{15/02/08/21:41}
\lim_{t_{1} \uparrow T_{\max}} \|\langle \nabla \rangle e^{i(t-t_{1})\Delta}\Phi_{\omega}\|_{St([t_{1},T_{\max}])}= 0 ,
\end{equation}  
we see from \eqref{14/09/12/01:47} that if $\delta_{0}\ll \delta(\omega)$, then $\psi$ extends beyond $T_{\max}$. However, this is a contradiction. Thus, we have proved the lemma. 
\end{proof}

Note here that if follows from Lemma \ref{13/12/19/09:42} that  there exists a frequency $\omega(2^{*}) \in (0,\omega_{2})$ such that for any $\omega \in (0,\omega(2^{*}))$,  
\begin{equation}\label{18/02/04/16:52}
\mu s_{p} \big| ( \Phi_{\omega}, f_{2} )_{L^{2}}\big|
\ge 2(1-s_{p})(2^{*}-2)
\big| (\Phi_{\omega}^{2^{*}-1}, f_{1} )_{L^{2}}\big|
.
\end{equation}
We use this fact in the proof of the ejection lemma below. 

\begin{lemma}[Ejection lemma]\label{13/05/05/22:31}
Assume $d\ge 3$ and $1+\frac{4}{d}<p< 2^{*}-1$. Then, for any $\omega \in (0,\omega(2^{*}))$, there exist constants $\delta_{X}\in (0,\delta_{0}(\omega))$ ($\delta_{0}(\omega)$ denotes the constant given by Lemma \ref{15/11/08/11:45}),  
 $A_{*}>0$, $B_{*}>0$, $C_{*}>0$ and $T_{*}>0$ with the following properties: for any $t_{0}\in \mathbb{R}$ and any solution $\psi$ to \eqref{12/03/23/17:57} defined around $t_{0}$ satisfying   
\begin{align}
\label{13/05/05/21:03}
&\mathcal{M}(\psi)=\mathcal{M}(\Phi_{\omega}),
\\[6pt]
\label{13/05/05/22:47}
&0< R_{0}:=d_{\omega}(\psi(t_{0}))< \delta_{X},
\\[6pt]
\label{13/05/05/22:48}
&\mathcal{S}_{\omega}(\psi)<m_{\omega}+\frac{R_{0}^{2}}{2},
\end{align}
we can extend $\psi$ as long as $d_{\omega}(\psi(t))\le \delta_{X}$. Furthermore, assume that there exists $T>t_{0}$ such that 
\begin{equation}
\label{13/05/05/22:35}
R_{0}\le \min_{t\in [t_{0}, T]}d_{\omega}(\psi(t)) ,
\end{equation}
and define 
\begin{equation}\label{13/06/29/10:51}
T_{X}:=
\inf\big\{t \in [t_{0}, T]  \colon d_{\omega}(\psi(t))= \delta_{X} 
\big\},
\end{equation}
where we interpret $T_{X}=T$ if $d_{\omega}(\psi)<\delta_{X}$ on $[t_{0},T]$. Then,  for any $t\in [t_{0},T_{X}]$, 
\begin{align}
\label{15/03/31/11:04}
&A_{*}e^{\mu (t-t_{0})}R_{0} \le d_{\omega}(\psi(t)) \le B_{*}e^{\mu (t-t_{0})}R_{0}, 
\\[6pt]
\label{13/05/05/22:53}
&
\|\eta(t)\|_{H^{1}}
\sim \mathfrak{s}\lambda_{1}(t) 
\sim \mathfrak{s}\lambda_{+}(t)
\sim e^{\mu (t-t_{0})}R_{0},
\\[6pt] 
\label{13/05/05/22:54}
&|\lambda_{-}(t)| + \|\Gamma(t) \|_{H^{1}} 
\lesssim 
R_{0} +\big( e^{\mu (t-t_{0})}R_{0} \big)^{\frac{\min\{3,p+1\}}{2}},
\\[6pt]
\label{13/05/06/14:30}
&\mathfrak{s}\mathcal{K}(\psi(t))
\gtrsim \big(e^{\mu (t-t_{0})}-C_{*}\big)R_{0},
\end{align}
where $\mathfrak{s}$ is either $\mathfrak{s}=1$ or $\mathfrak{s}=-1$. Moreover,  $d_{\omega}(\psi(t))$ is increasing on the region $\{ t \in [t_{0}, T_{X}] \colon t_{0}+T_{*}R_{0}^{\min\{1,p-1\}}\le t \}$; and 
\begin{equation}\label{13/05/06/14:37}
\big|d_{\omega}(\psi(t))-R_{0} \big|
\lesssim 
R_{0}^{\min\{2,p\}}
\end{equation}
on the region $\{ t \in [t_{0}, T_{X}] \colon  t_{0}\le t \le t_{0}+T_{*}R_{0}^{\min\{1,p-1 \}}\}$.
\end{lemma}
\begin{proof}[Proof of Lemma \ref{13/05/05/22:31}] Let $\delta_{X}\in (0, \delta_{0}(\omega))$ be a small constant to be chosen later, and let $\psi$ be a solution satisfying \eqref{13/05/05/21:03}, \eqref{13/05/05/22:47} and \eqref{13/05/05/22:48}. Since the equation \eqref{12/03/23/17:57} is invariant under the time translations, it suffices to consider the case where $t_{0}=0$. 
\par 
First, we find from Lemma \ref{15/11/08/11:45} that $\psi$ extends as long as $d_{\omega}(\psi(t))\le \delta_{X}$.  
\par
Next, we assume \eqref{13/05/05/22:35} as well as \eqref{13/05/05/21:03}, \eqref{13/05/05/22:47} and \eqref{13/05/05/22:48}. Then, it follows from the definition of $T_{X}$ and $t_{0}=0$ that for any $t \in [0, T_{X}]$, 
\begin{equation}\label{13/06/29/14:31}
d_{\omega}(\psi(t))\le \delta_{X}<\delta_{E}(\omega).
\end{equation}
Hence, we see from Lemma \ref{13/05/11/11:52} and Lemma \ref{13/05/05/15:47} that for any $t \in [0, T_{X}]$, 
\begin{align}
\label{13/05/06/15:02}
&\|\Gamma(t)\|_{H^{1}}
\lesssim 
\|\eta(t)\|_{H^{1}}
\sim 
\|\eta(t)\|_{E} 
\sim  
d_{\omega}(\psi(t)) 
\sim |\lambda_{1}(t)|,
\\[6pt]
\label{13/05/06/14:46}
&\frac{d}{dt} d_{\omega}\big(\psi(t)\big)^{2}
=
4\mu^{2} \lambda_{1}(t) \lambda_{2}(t) 
+
4\mu \lambda_{1}(t) 
\Omega\big( \Big\{ \frac{d \theta}{d t}(t)-\omega \Big\}\eta(t) -N_{\omega}(\eta(t)), f_{2} \big).
\end{align}
In particular, we deduce from \eqref{13/05/11/17:19}, \eqref{13/06/29/14:31} and \eqref{13/05/06/15:02} that for any $t \in [0,  T_{X}]$, 
\begin{equation}\label{14/01/13/11:51}
|\lambda_{-}(t)|+\|\Gamma(t)\|_{H^{1}}
\lesssim 
\|\eta(t)\|_{H^{1}}
\sim |\lambda_{1}(t)| \lesssim \delta_{X}\ll 1. 
\end{equation}
Furthermore, we find from \eqref{13/05/05/22:35} and \eqref{13/05/06/15:02} that for any $t\in [0,  T_{X}]$.
\begin{equation}\label{13/12/30/11:11}
0<R_{0}\lesssim |\lambda_{1}(t)|,
\end{equation}
which together with the continuity of $\lambda_{1}(t)$ shows that for any 
 $t\in [0,T_{X}]$    
\begin{equation}\label{13/12/30/11:13}
\mathfrak{s}:={\rm sgn}[\lambda_{1}(t)] \equiv  {\rm sgn}[\lambda_{1}(0)]
.
\end{equation} 

Now, we further choose $\delta_{X}$ so small that 
\begin{equation}\label{15/02/08/22:02}
\delta_{X} \ll \frac{|( \Phi_{\omega}, \partial_{\omega}\Phi_{\omega})_{L^{2}}|}{2\|\partial_{\omega}\Phi_{\omega} \|_{L^{2}}}.
\end{equation} 
Then, we find from \eqref{14/01/13/11:51} that for any $t \in [0, T_{X}]$,     
\begin{equation}\label{13/05/06/15:27}
\|\eta(t)\|_{L^{2}} \le \|\eta(t)\|_{H^{1}}
\ll 
\min\Big\{ 1, 
\frac{|( \Phi_{\omega},\partial_{\omega}\Phi_{\omega} )_{L^{2}}|}{2\|\partial_{\omega}\Phi_{\omega} \|_{L^{2}}}
\Big\}.
\end{equation}
Furthermore,  \eqref{13/05/06/15:27} implies that for any $t\in [0,T_{X}]$, 
\begin{equation}\label{13/06/12/12:37}
\big|( \Phi_{\omega}+\eta(t), \partial_{\omega}\Phi_{\omega} )_{L^{2}}\big|
\ge 
\big|( \Phi_{\omega}, \partial_{\omega}\Phi_{\omega} )_{L^{2}}\big|
-
\| \eta(t)\|_{L^{2}} \| \partial_{\omega}\Phi_{\omega} \|_{L^{2}}
\ge 
\frac{1}{2}\big|( \Phi_{\omega},\partial_{\omega}\Phi_{\omega} )_{L^{2}}\big|.
\end{equation}
We see from H\"older's inequality and $N_{\omega}(\eta)=O(|\eta|^{\min\{2,p\}})$ (see \eqref{13/05/06/10:27}) that for any function $u \in L^{1}(\mathbb{R}^{d})\cap L^{\infty}(\mathbb{R}^{d})$, 
\begin{equation}\label{15/02/09/10:26}
\begin{split}
\big| (N_{\omega}(\eta(t)),u)_{L_{real}^{2}}\big|
&\le 
\|N_{\omega}(\eta(t))\|_{L^{\frac{p+1}{\min\{2,p\}}}}
\|u \|_{L^{\frac{p+1}{\max\{1,p-1 \}}}}
\\[6pt]
&\lesssim 
\|\eta(t)\|_{L^{p+1}}^{\min\{2,p\}}\|u \|_{L^{1}\cap L^{\infty}}.
\end{split}
\end{equation}
We deduce from the equation \eqref{13/05/06/14:06} for $\theta(t)$,  \eqref{13/06/12/12:37}, \eqref{15/02/09/10:26} and \eqref{13/05/06/15:27} that for any $t \in [0, T_{X}]$, 
\begin{equation}\label{13/05/06/15:19}
\begin{split}
\Big| \frac{d\theta }{dt}(t) -\omega \Big|
&
\lesssim 
\frac{\|\eta(t)\|_{L^{2}}^{2}+\|\eta(t)\|_{L^{p+1}}^{\min\{2, p\}}
\|\partial_{\omega}\Phi_{\omega}\|_{L^{1}\cap L^{\infty}}}{|( \Phi_{\omega}, \partial_{\omega}\Phi_{\omega})_{L^{2}}|}
\lesssim  
\|\eta(t)\|_{H^{1}}^{\min\{2,p\}}, 
\end{split}
\end{equation}
where the implicit constant depends on $\omega$. Furthermore, it follows from \eqref{13/05/06/15:19}, \eqref{15/02/09/10:26} and \eqref{13/05/06/15:27} that  for any $t \in [0, T_{X}]$, 
\begin{equation}\label{13/12/07/14:46}
\begin{split}
&\bigg|\Omega\big( \Big\{ \frac{d \theta}{d t}(t)-\omega \Big\} \eta(t) 
-N_{\omega}(\eta(t)), f_{2} \big) 
\bigg|
\\[6pt]
&\le 
 \Big| \frac{d \theta}{d t}(t) -\omega \Big|  
\|\eta(t) \big\|_{L^{2}}\|f_{2}\|_{L^{2}}
+ 
\|\eta(t)\|_{L^{p+1}}^{\min\{2,p \}}
\|f_{2}\|_{L^{2}\cap L^{\infty}}
\\[6pt]
&\lesssim
\|\eta(t) \|_{H^{1}}^{\min\{2,p\}+1} + \|\eta(t)\|_{H^{1}}^{\min\{2,p\}}
\lesssim  
\|\eta(t)\|_{H^{1}}^{\min\{2,p\}},
\end{split}
\end{equation}
where the implicit constant depends on $\omega$. 
 Hence, we find form \eqref{13/12/07/14:46} and \eqref{13/05/06/15:02} that for any $t\in [0, T_{X}]$, 
\begin{equation}\label{13/06/12/13:58}
\bigg|\Omega\big( \Big\{ \frac{d \theta}{d t}(t)-\omega \Big\}
 \eta(t) -N_{\omega}(\eta(t)), f_{2} \big) 
\bigg|
\lesssim
\big|\lambda_{1}(t)\big|^{\min\{2,\, p\}},
\end{equation}
where the implicit constant depends on $\omega$. Similarly, we can verify that for any $t\in [0, T_{X}]$, 
\begin{equation}\label{15/02/29/14:23}
\bigg|\big( \Big\{ \frac{d \theta}{d t}(t)-\omega \Big\}
 \eta(t) -N_{\omega}(\eta(t)), \mathscr{U}_{\pm} \big)_{L_{real}^{2}} 
\bigg|
\lesssim
\big|\lambda_{1}(t)\big|^{\min\{2,\, p\}},
\end{equation}
where the implicit constant depends on $\omega$.
\par 
We shall show that 
\begin{equation}\label{13/12/07/13:44}
\lambda_{1}(0) \sim \lambda_{+}(0) \sim \mathfrak{s} R_{0},
\end{equation}
where the implicit constants depend on $\omega$. We see from \eqref{13/05/05/22:47} and \eqref{13/05/05/22:35} that  
\begin{equation}\label{13/05/06/14:52}
\frac{d}{dt} d_{\omega}\big(\psi(t)\big)^{2}\bigg|_{t=t_{0}=0} 
=
2R_{0} \frac{d}{dt} d_{\omega}\big(\psi(t)\big) \bigg|_{t=0} 
\ge 0, 
\end{equation}
which together with \eqref{13/05/06/14:46} yields 
\begin{equation}\label{13/06/12/20:28}
0\le 
\mu^{2} \lambda_{1}(0) \lambda_{2}(0) 
+
\mu  \lambda_{1}(0) 
\bigg|
\Omega\big( \Big\{ \frac{d \theta}{d t}(0)-\omega \Big\}\eta(0) -N_{\omega}(\eta(0)), f_{2} \big)
\bigg|.
\end{equation}
Combining \eqref{13/06/12/20:28} with \eqref{13/06/12/13:58}, we obtain that 
\begin{equation}\label{13/05/06/16:05}
0\le 
\mu^{2} \, {\rm sgn}[\lambda_{1}(0)]|\lambda_{1}(0)| \lambda_{2}(0)
+ C\mu |\lambda_{1}(0)|^{\min\{3, p+1\}} 
\end{equation} 
for some constant $C>0$ depending on $\omega$. Furthermore, it follows from \eqref{13/05/06/16:05} that  
\begin{equation}\label{13/05/07/16:48}
- C |\lambda_{1}(0)|^{\min\{2,p\}}  
\le  
\mu \mathfrak{s} \lambda_{2}(0).
\end{equation} 
Suppose here that $\mathfrak{s}={\rm sgn}[\lambda_{1}(0)]=1$. Then, we see from \eqref{13/12/30/11:11}, \eqref{13/05/07/16:48} and \eqref{14/01/13/11:51} that  
\begin{equation}\label{13/05/07/17:01}
0<\lambda_{1}(0) \lesssim \lambda_{1}(0)- \frac{C}{\mu}\big| \lambda_{1}(0)\big|^{\min\{2,p\}}
\lesssim 
\lambda_{1}(0) + \lambda_{2}(0) 
= \lambda_{+}(0),
\end{equation} 
so that ${\rm sgn}[\lambda_{+}(0)]=1$. Suppose next that ${\rm sgn}[\lambda_{1}(0)]=-1$. Then, \eqref{13/05/07/16:48} becomes 
$\mu \lambda_{2}(0) \le C |\lambda_{1}(0)|^{\min\{2,p\}}$. This together with \eqref{14/01/13/11:51} shows that 
\begin{equation}\label{13/05/07/17:06}
\lambda_{+}(0)
=
\lambda_{1}(0)+\lambda_{2}(0)
\le 
\lambda_{1}(0)+ \frac{C}{\mu}|\lambda_{1}(0)|^{\min\{2,p\}} 
\sim \lambda_{1}(0). 
\end{equation} 
Thus, we conclude that 
\begin{equation}\label{13/06/12/22:45}
{\rm sgn}[\lambda_{1}(0)]={\rm sgn}[\lambda_{+}(0)]
.
\end{equation}
Since $\lambda_{1}(0)$ and $\lambda_{+}(0)$ are independent of $t$, \eqref{13/06/12/22:45} together with \eqref{13/05/06/15:02} and $d_{\omega}(\psi(0))=R_{0}$ implies \eqref{13/12/07/13:44}. 

Next, we shall prove \eqref{13/05/05/22:53} and \eqref{15/03/31/11:04}. Since we have \eqref{13/05/06/15:02}, it suffices for \eqref{13/05/05/22:53} to show that for any $t \in [0, T_{X}]$, 
\begin{equation}\label{13/12/30/11:47}
\lambda_{1}(t) 
\sim \lambda_{+}(t) 
\sim \mathfrak{s}R_{0}e^{\mu t}.
\end{equation}
Let $\alpha >1$ be a constant satisfying $|\lambda_{1}(0)|\le \frac{1}{2}\alpha R_{0}$, and consider 
\begin{equation}\label{15/02/09/15:52}
T_{\alpha}
:=
\sup\bigm\{ 
T \in [0,T_{X}] \colon |\lambda_{1}(t)| \le \alpha R_{0}e^{\mu t} \mbox{ for all $t \in [0,T]$} \bigm\}.
\end{equation}
Then, we have $T_{\alpha}>0$. We shall show that $T_{\alpha}=T_{X}$. Suppose for contradiction that $T_{\alpha} <T_{X}$. Then,  
\begin{equation}\label{15/02/09/20:29}
\alpha R_{0}e^{\mu T_{\alpha}}=|\lambda_{1}(T_{\alpha})|  \lesssim \delta_{X},
\end{equation}
where we have used \eqref{14/01/13/11:51} to obtain the inequality. 
Using the equation \eqref{13/05/06/14:03} for $\lambda_{+}$, we have that  
\begin{equation}\label{15/02/09/14:04}
\begin{split}
\frac{d}{dt}\big( e^{-\mu t}\lambda_{+}(t) \big)
&=
e^{-\mu t} 
\Bigm( \frac{d\lambda_{+}}{dt}(t)
- \mu \lambda_{+}(t) 
\Bigm)
\\[6pt]
&=
-e^{-\mu t}
\big( \Big\{ \frac{d \theta}{d t}(t) -
\omega \Big\} \eta(t) -N_{\omega}\big(\eta(t)\big), \, \mathscr{U}_{-} \big)_{L_{real}^{2}}.
\end{split}
\end{equation}
Furthermore, integrating the equation \eqref{15/02/09/14:04}, and then using \eqref{15/02/29/14:23} and \eqref{14/01/13/11:51}, we find that for any $t\in [0,T_{X}]$,  
\begin{equation}\label{13/05/06/16:28}
\begin{split}
\big| \lambda_{+}(t) -e^{\mu t}\lambda_{+}(0) \big|
&=
\bigg|e^{\mu t}
\int_{0}^{t} e^{-\mu s} 
\big(\Big\{ \frac{d \theta}{d t}(s) -\omega \Big\} \eta(s) -N_{\omega}(\eta(s)),\,  \mathscr{U}_{-} \big)_{L_{real}^{2}}\,ds
\bigg|
\\[6pt]
&\lesssim
\int_{0}^{t} e^{\mu (t-s)} 
\big| \lambda_{1}(s)\big|^{\min\{2,p\}}\,ds, 
\end{split}
\end{equation}
where the implicit constant depends on $\omega$. Similarly, we have 
\begin{equation}\label{15/02/09/14:45}
\big| \lambda_{-}(t) -e^{-\mu t}\lambda_{-}(0) \big|
\lesssim
\int_{0}^{t} e^{-\mu(t-s)} 
\big| \lambda_{1}(s)\big|^{\min\{2,p\}}\,ds.
\end{equation}
Then, we find from \eqref{13/05/06/16:28}, \eqref{15/02/09/14:45}, \eqref{13/12/07/13:44} and \eqref{14/01/13/11:51} that for any $t \in [0,T_{X}]$, 
\begin{equation}\label{15/02/09/16:31}
\begin{split}
|\lambda_{1}(t)| 
& \le 
|\lambda_{+}(t)| + |\lambda_{-}(t)|
\\[6pt]
&\lesssim   
e^{\mu t}|\lambda_{+}(0)| + e^{-\mu t}|\lambda_{-}(0)|
+ 
\int_{0}^{t} e^{\mu (t-s)} 
\big| \lambda_{1}(s)\big|^{\min\{2,p\}}\,ds
\\[6pt]
&\lesssim
e^{\mu t}R_{0} +R_{0}
+
e^{\mu t}
\int_{0}^{t} e^{-\mu s} 
\big| \lambda_{1}(s)\big|^{\min\{2,p\}}\,ds.
\end{split}
\end{equation}
This together with the definition of $T_{\alpha}$ (see \eqref{15/02/09/15:52}) and \eqref{15/02/09/20:29} shows that 
\begin{equation}\label{15/02/13/14:35}
\begin{split}
|\lambda_{1}(T_{\alpha})| 
&\lesssim 
e^{\mu T_{\alpha}}R_{0}
+
\frac{1}{ \mu \min\{1,p-1\}} (\alpha R_{0}e^{\mu T_{\alpha}})^{\min\{2,p\}} 
\\[6pt]
&\le 
e^{\mu T_{\alpha}}R_{0}
+
\frac{1}{ \mu \min\{1,p-1\}} \delta_{X}^{\min\{1,p-1\}}\alpha R_{0}e^{\mu T_{\alpha}} .
\end{split}
\end{equation}
Hence, if $\frac{1}{ \mu \min\{1,p-1\}} \delta_{X}^{\min\{1,p-1\}}\ll 1$, then \eqref{15/02/13/14:35} implies that $|\lambda_{1}(T_{\alpha})|\le \frac{1}{2}\alpha R_{0}e^{\mu T_{\alpha}}$. However, this contradicts \eqref{15/02/09/20:29}. Thus, we have proved that $T_{\alpha}=T_{X}$, and therefore for any $t \in [0,T_{X}]$, 
\begin{equation}\label{15/02/13/15:07}
|\lambda_{1}(t)|\lesssim R_{0}e^{\mu t}. 
\end{equation}
We also see from \eqref{13/05/06/16:28}, \eqref{15/02/09/14:45} and \eqref{15/02/13/15:07} that for any $t\in [0,T_{X}]$, 
\begin{align}
\label{13/12/30/11:48}
&\big| \lambda_{+}(t) -e^{\mu t}\lambda_{+}(0)\big| 
\lesssim (R_{0}e^{\mu t})^{\min\{2,p\}},
\\[6pt]
\label{15/02/09/14:32}
&\big| \lambda_{-}(t) -e^{- \mu t}\lambda_{-}(0)\big| 
\lesssim (R_{0}e^{\mu t})^{\min\{2,p\}}.
\end{align}
Furthermore, we find  from \eqref{13/12/07/13:44}, \eqref{13/12/30/11:48} and \eqref{15/02/09/14:32} that \eqref{13/12/30/11:47} holds. Combining \eqref{13/05/06/15:02} and \eqref{13/12/30/11:47}, we also obtain \eqref{15/03/31/11:04}. 
\par 
We shall prove \eqref{13/05/05/22:54}. It follows  from \eqref{13/12/30/11:47} that for any $t\in [0,T_{X}]$, 
\begin{equation}\label{14/01/13/11:59}
|\lambda_{+}(t)|+|\lambda_{-}(t)|
\lesssim 
|\lambda_{+}(t)|+|\lambda_{1}(t)|
\sim R_{0}e^{\mu t}.
\end{equation}
Furthermore, \eqref{15/02/09/14:32} together with  \eqref{14/01/13/11:59} shows  that 
\begin{equation}\label{13/12/31/17:49}
\big|\lambda_{-}(t)\big|
\lesssim 
e^{-\mu t}\big|\lambda_{-}(0)\big|
+(R_{0}e^{\mu t})^{\min\{2,p\}}
\lesssim 
R_{0}+(R_{0}e^{\mu t})^{\min\{2,p\}}.
\end{equation}
In order to complete the proof of \eqref{13/05/05/22:54}, we employ the ``nonlinear energy projected onto $\mathscr{U}_{\pm}$ plane'':
\begin{equation}\label{13/12/30/15:40}  
 E_{\{\mathscr{U}_{+},\mathscr{U}_{-}\}}(t):=
\mathcal{S}_{\omega}\big(\Phi_{\omega}+\lambda_{+}(t)\mathscr{U}_{+}+\lambda_{-}(t)\mathscr{U}_{-}\big)
-
\mathcal{S}_{\omega}\big(\Phi_{\omega}\big).
\end{equation}  
The second order Taylor's expansion around $\Phi_{\omega}$ together with $\mathcal{S}_{\omega}'(\Phi_{\omega})=0$, \eqref{13/01/04/14:46}, $\mathscr{L}_{\omega}\mathscr{U}_{\pm}=\pm i\mu \mathscr{U}_{\pm}$ and \eqref{13/03/04/12:30} shows that  
\begin{equation}\label{13/06/16/17:15}
\begin{split}
E_{\{\mathscr{U}_{+},\mathscr{U}_{-}\}}(t)
&=
\frac{1}{2}\big[ 
\mathcal{S}_{\omega}''(\Phi_{\omega}) \big\{ \lambda_{+}(t)\mathscr{U}_{+}+\lambda_{-}(t)\mathscr{U}_{-}\big\}\big]
\big\{ \lambda_{+}(t)\mathscr{U}_{+}+\lambda_{-}(t)\mathscr{U}_{-}\big\}
\\[6pt]
&\quad 
+
O(\big\| \lambda_{+}(t)\mathscr{U}_{+}+\lambda_{-}(t)\mathscr{U}_{-} 
\big\|_{H^{1}}^{\min\{3,p+1\}})
\\[6pt]
&=
\frac{1}{2}\langle 
\mathscr{L}_{\omega} \big( \lambda_{+}(t)\mathscr{U}_{+}+\lambda_{-}(t)\mathscr{U}_{-} \big), \, 
 \lambda_{+}(t)\mathscr{U}_{+}+\lambda_{-}(t)\mathscr{U}_{-}
\rangle_{H^{-1}, H^{1}}
\\[6pt]
&\quad 
+
O(\big\| \lambda_{+}(t)\mathscr{U}_{+}+\lambda_{-}(t)\mathscr{U}_{-} 
\big\|_{H^{1}}^{\min\{3,p+1\}})
\\[6pt]
&=
\frac{1}{2}\big( 
i\mu \lambda_{+}(t)\mathscr{U}_{+}-i\mu \lambda_{-}(t)\mathscr{U}_{-}, \,\lambda_{+}(t)\mathscr{U}_{+}+\lambda_{-}(t)\mathscr{U}_{-}
\big)_{L_{real}^{2}}
\\[6pt]
&\quad 
+
O(\big\| \lambda_{+}(t)\mathscr{U}_{+}+\lambda_{-}(t)\mathscr{U}_{-} 
\big\|_{H^{1}}^{\min\{3,p+1\}})
\\[6pt]
&=
-\mu \lambda_{+}(t) \lambda_{-}(t) 
+
O(\big\| \lambda_{+}(t)\mathscr{U}_{+}+\lambda_{-}(t)\mathscr{U}_{-} 
\big\|_{H^{1}}^{\min\{3,p+1\}}).
\end{split}
\end{equation}
We find from  \eqref{13/04/29/12:12}, \eqref{13/05/05/21:03}, \eqref{13/06/16/17:15} and Lemma \ref{13/05/03/19:30} that 
\begin{equation}\label{13/12/10/10:28}
\begin{split}
&\mathcal{S}_{\omega}(\psi)-\mathcal{S}_{\omega}(\Phi_{\omega})
-E_{\{ \mathscr{U}_{+},\mathscr{U}_{-}\}}(t)
\\[6pt]
&=
\frac{1}{2}
\langle \mathscr{L}_{\omega}
\Gamma(t), \Gamma(t)\rangle_{H^{-1},H^{1}}
\\[6pt]
&\quad +
O(\| \eta(t) \|_{H^{1}}^{\min\{3,p+1\}})
+
O(\big\| \lambda_{+}(t)\mathscr{U}_{+}+\lambda_{-}(t)\mathscr{U}_{-} 
\big\|_{H^{1}}^{\min\{3,p+1\}})
\\[6pt]
&\sim  
\|\Gamma(t)\|_{H^{1}}^{2}
+
O(\| \eta(t) \|_{H^{1}}^{\min\{3,p+1\}})
+
O(\big\| \lambda_{+}(t)\mathscr{U}_{+}+\lambda_{-}(t)\mathscr{U}_{-} 
\big\|_{H^{1}}^{\min\{3,p+1\}}).
\end{split}
\end{equation}
Moreover, it follows from $\mathcal{S}_{\omega}'(\Phi_{\omega})=0$, \eqref{13/01/04/14:46}, \eqref{15/06/06/11:39} and \eqref{15/02/29/14:23} that 
\begin{equation}\label{13/05/08/10:15}
\begin{split}
&
\Big| 
\frac{d}{dt}E_{\{\mathscr{U}_{+},\mathscr{U}_{-}\}}(t)  
\Big|
\\[6pt]
&=
\Big|
\mathcal{S}_{\omega}'( \Phi_{\omega}+\lambda_{+}(t)\mathscr{U}_{+}+\lambda_{-}(t)\mathscr{U}_{-})
\Big\{ 
\frac{d\lambda_{+}}{dt}(t) \mathscr{U}_{+}
+\frac{d\lambda_{-}}{dt}(t) \mathscr{U}_{-}
\Big\}
\Big|
\\[6pt]
&\lesssim  
\bigg|
\Big[
\mathcal{S}_{\omega}''\big( \Phi_{\omega} \big) 
\big(\lambda_{+}(t)\mathscr{U}_{+}+\lambda_{-}(t)\mathscr{U}_{-}\big)
\Big]
\Big\{ 
\frac{d\lambda_{+}}{dt}(t) \mathscr{U}_{+}
+\frac{d\lambda_{-}}{dt}(t) \mathscr{U}_{-}
\Big\}
\bigg|
\\[6pt]
&\quad 
+
\big\| \lambda_{+}(t)\mathscr{U}_{+}+\lambda_{-}(t)\mathscr{U}_{-} \big\|_{H^{1}}^{\min\{2,p\}}
\Big\| 
\frac{d\lambda_{+}}{dt}(t) \mathscr{U}_{+}
+\frac{d\lambda_{-}}{dt}(t) \mathscr{U}_{-}
\Big\|_{H^{1}}
\\[6pt]
&= 
\bigg|
\langle \mathscr{L}_{\omega} \big(\lambda_{+}(t)\mathscr{U}_{+}+\lambda_{-}(t)\mathscr{U}_{-}\big), \, 
\frac{d\lambda_{+}}{dt}(t) \mathscr{U}_{+}
+\frac{d\lambda_{-}}{dt}(t) \mathscr{U}_{-}
\rangle_{H^{-1},H^{1}}
\bigg|
\\[6pt]
&\quad 
+
\big\| \lambda_{+}(t)\mathscr{U}_{+}+\lambda_{-}(t)\mathscr{U}_{-} \big\|_{H^{1}}^{\min\{2,p\}}
\Big\| 
\frac{d\lambda_{+}}{dt}(t) \mathscr{U}_{+}
+\frac{d\lambda_{-}}{dt}(t) \mathscr{U}_{-}
\Big\|_{H^{1}}
\\[6pt]
&\lesssim 
\mu ( | \lambda_{-}(t)|+| \lambda_{+}(t)|)^{\min\{3,p+1\}}
\\[6pt]
&\quad + 
\big( |\lambda_{+}(t)| + |\lambda_{-}(t)| )^{\min\{2,p\}}
\Big( 
\Big|
\frac{d\lambda_{+}}{dt}(t)\Big|
+
\Big|
\frac{d\lambda_{-}}{dt}(t)\Big|
\Big).
\end{split}
\end{equation}
Here, the equations \eqref{13/05/06/14:03} and \eqref{13/05/06/14:04} together with \eqref{13/05/06/15:19} and \eqref{15/02/09/10:26} show that for any $t\in [0,T_{X}]$,
\begin{equation}\label{15/02/11/14:01}
\begin{split}
\Big| \frac{d\lambda_{\pm}}{dt}(t)\Big|
&\le 
\mu |\lambda_{\pm}(t)|
+
\Big| \frac{d \theta}{d t}(t) -
\omega \Big| 
\|\eta(t)\|_{L^{2}}\|\mathscr{U}_{\pm}\|_{L^{2}} 
+ 
\big|
\big(N_{\omega}\big(\eta(t) \big), \, 
\mathscr{U}_{\mp} \big)_{L_{real}^{2}}
\big|
\\[6pt]
&\lesssim 
|\lambda_{+}(t)|+|\lambda_{-}(t)|
+
\| \eta(t) \|_{H^{1}}^{\min\{3,p+1\}} 
+
\| \eta(t) \|_{H^{1}}^{\min\{2,p\}} 
\\[6pt]
&\lesssim 
|\lambda_{+}(t)|+|\lambda_{-}(t)|
+
\big\|
\lambda_{+}(t)\mathscr{U}_{+}+\lambda_{-}(t)\mathscr{U}_{-}+\Gamma(t)
\big\|_{H^{1}}^{\min\{2,p\}} 
\\[6pt]
&\lesssim 
|\lambda_{+}(t)|+|\lambda_{-}(t)|
+
\big\|\Gamma(t)\|_{H^{1}}^{\min\{2,p\}}
,
\end{split}
\end{equation}
where the implicit constant depends on $\omega$.
 Putting \eqref{13/05/08/10:15} and \eqref{15/02/11/14:01} together, we find that for any $t\in [0,T_{X}]$, 
\begin{equation}\label{13/06/16/16:53}
\begin{split}
&\Big| \frac{d}{dt}E_{\{\mathscr{U}_{+},\mathscr{U}_{-}\}}(t) \Big|
\\[6pt]
&\lesssim 
\big\|\Gamma(t)\|_{H^{1}}^{\min\{2,p\}}
\big( |\lambda_{+}(t)|+|\lambda_{-}(t)\big| \big)
 + \big( |\lambda_{+}(t)|+|\lambda_{-}(t)\big| \big)^{\min\{3,p+1\}},
\end{split}
\end{equation}
where the implicit constant depends on $\omega$. Furthermore, we find from 
 \eqref{13/12/10/10:28}, the decomposition \eqref{13/05/04/14:34} and \eqref{13/06/16/16:53} that for any $t \in [0, T_{X}]$, 
\begin{equation}\label{13/12/08/16:50}
\begin{split}
&\|\Gamma(t)\|_{H^{1}}^{2}
\\[6pt]
&\lesssim
\mathcal{S}_{\omega}(\psi)-\mathcal{S}_{\omega}(\Phi_{\omega})-E_{\{\mathscr{U}_{+},\mathscr{U}_{-}\}}(0)
-
\big( 
E_{\{\mathscr{U}_{+},\mathscr{U}_{-}\}}(t)
-
E_{\{\mathscr{U}_{+},\mathscr{U}_{-}\}}(0)
\big)
\\[6pt]
&\quad +
O(\big\| \eta(t) 
\big\|_{H^{1}}^{\min\{3,p+1\}})
+
O(\big\| \lambda_{+}(t)\mathscr{U}_{+}+\lambda_{-}(t)\mathscr{U}_{-} 
\big\|_{H^{1}}^{\min\{3,p\}})
\\[6pt]
&\lesssim  
\|\Gamma(0)\|_{H^{1}}^{2}
+
\int_{0}^{t}
\bigg| \frac{d}{dt}E_{\{\mathscr{U}_{+},\mathscr{U}_{-}\}}(t') \bigg|\,dt' 
\\[6pt]
&\quad +
\sup_{0\le t'\le t}
\big\| \Gamma(t') +\lambda_{+}(t')\mathscr{U}_{+}
+\lambda_{-}(t') \mathscr{U}_{-} \big\|_{H^{1}}^{\min\{3,p+1\}}
\\[6pt]
&\quad 
+
\sup_{0\le t'\le t}
\big(
|\lambda_{+}(t')|+|\lambda_{-}(t')|
\big)^{\min\{3,p+1\}}
\\[6pt]
&\lesssim 
\|\Gamma(0)\|_{H^{1}}^{2}
+
\int_{0}^{t} \|\Gamma(t')\|_{H^{1}}^{\min\{2,p\}}
\big( |\lambda_{+}(t')|+|\lambda_{-}(t')| \big)\,dt'
\\[6pt]
&\quad +
\int_{0}^{t} \big( |\lambda_{+}(t')|+|\lambda_{-}(t')| \big)^{\min\{3,p+1\}}\,dt' \\[6pt]
&\quad +
\sup_{0\le t'\le t}\| \Gamma(t') \|_{H^{1}}^{\min\{3,p+1\}}
+
\sup_{0\le t'\le t}
\big( |\lambda_{+}(t')| + |\lambda_{-}(t')| \big)^{\min\{3,p+1\}}
.
\end{split}
\end{equation}
Suppose here that $p<2$, so that $\min\{3,p+1\}=p+1$. Note that $p+1<\frac{2}{2-p}$. Then, it follows from \eqref{13/12/08/16:50}, Young's inequality, \eqref{13/05/06/15:02},  \eqref{14/01/13/11:59} and \eqref{13/12/30/11:47} that for any $t \in [0, T_{X}]$,  
\begin{equation}\label{14/01/13/11:25}
\begin{split}
\sup_{0\le t'\le t}\|\Gamma(t')\|_{H^{1}}^{2}
&\lesssim 
\|\Gamma(0)\|_{H^{1}}^{2}
+
\int_{0}^{t} \big( |\lambda_{+}(t')|+|\lambda_{-}(t')| \big)^{\frac{2}{2-p}}
\,dt'
\\[6pt]
&\quad +
\int_{0}^{t} \big( |\lambda_{+}(t')|+|\lambda_{-}(t')| \big)^{p+1}\,dt' 
\\[6pt]
&\quad +
\sup_{0\le t' \le t}\big\{
|\lambda_{+}(t')|+|\lambda_{-}(t')|\big\}^{p+1}
\\[6pt]
&\lesssim 
|\lambda_{1}(0)|^{2}
+
\int_{0}^{t} |\lambda_{1}(t')|^{p+1}\,dt' 
+
\sup_{0\le t' \le t}|\lambda_{1}(t')|^{p+1}
\\[6pt]
&\lesssim 
R_{0}^{2}
+
\int_{0}^{t} \big( R_{0}e^{\mu t'} \big)^{p+1}\,dt
+\sup_{0\le t' \le t}\big( R_{0}e^{\mu t'} \big)^{p+1}
\\[6pt]
&\lesssim 
R_{0}^{2}+ (R_{0}e^{\mu t})^{p+1}.
\end{split}
\end{equation}
This together with \eqref{13/12/31/17:49} gives us the desired result \eqref{13/05/05/22:54} for $p<2$. 
Similarly, we can prove \eqref{13/05/05/22:54} for $p\ge 2$.  
\par 
We shall prove that there exists $T_{*}>0$ such that $d_{\omega}(\psi(t))$ is strictly increasing on $[T_{*}R_{0}^{\min\{1,p-1 \}}, T_{X}]$. It follows from \eqref{13/05/06/15:02}, \eqref{13/05/06/14:46} and \eqref{13/06/12/13:58}  that for any $t \in [0,T_{X}]$, 
\begin{equation}\label{15/03/17/11:41}
\begin{split}
|\lambda_{1}(t)|\frac{d}{dt}d_{\omega}(\psi(t))
&\sim 
d_{\omega}(\psi(t))\frac{d}{dt}d_{\omega}(\psi(t))
=
\frac{1}{2}\frac{d}{dt}d_{\omega}(\psi(t))^{2}
\\[6pt]
&\ge 
2\mu^{2} |\lambda_{1}(t)| \mathfrak{s} \lambda_{2}(t) 
-C_{1}\mu |\lambda_{1}(t)|^{\min\{3,p+1\}},
\end{split}
\end{equation}
so that  
\begin{equation}\label{14/01/13/20:02}
\begin{split}
\frac{d}{dt}d_{\omega}(\psi(t))
&\gtrsim 
\mu^{2} \mathfrak{s} \lambda_{2}(t)-C_{1}\mu |\lambda_{1}(t)|^{\min\{2,p\}},
\end{split}
\end{equation}
where $C_{1}>0$ is constant depending only on $d$, $p$ and $\omega$. Here,  multiplying the equation \eqref{15/03/30/16:37} by $\mathfrak{s}$, integrating the resulting equation, and using \eqref{15/02/29/14:23} and \eqref{13/12/30/11:47}, we obtain that 
\begin{equation}\label{14/01/13/14:12}
\begin{split}
\mathfrak{s}\lambda_{2}(t)
&=
\mathfrak{s}\lambda_{2}(0)
+
\mu \int_{0}^{t}\mathfrak{s} \lambda_{1}(t')\,dt' 
\\[6pt]
&\quad 
-\mathfrak{s}
\frac{1}{2}\int_{0}^{t}( \Big\{ \frac{d\theta}{dt}(t')-\omega \Big\}\eta(t')
-N_{\omega}(\eta(t')), \, \mathscr{U}_{+}+\mathscr{U}_{-} )_{L_{real}^{2}}\,dt'
\\[6pt]
&\ge 
\mathfrak{s}\lambda_{2}(0)
+C_{2}R_{0} (e^{\mu t}-1)
-C_{3}
\int_{0}^{t} (R_{0}e^{\mu t'})^{\min\{2,p\}} \,dt' ,
\end{split}
\end{equation}
where $C_{2}$ and $C_{3}$ are some positive constants depending only on $d$, $p$ and $\omega$. Here, we see from \eqref{13/05/07/16:48} and \eqref{13/12/07/13:44} that 
\begin{equation}\label{14/01/13/16:22}
\mu \mathfrak{s}\lambda_{2}(0)
\gtrsim  
- |\lambda_{1}(0)|^{\min\{2,p\}}
\sim 
-R_{0}^{\min\{2,p\}},
\end{equation}
where the implicit constants depend on $\omega$. 
Moreover, it follows from \eqref{13/12/30/11:47} and \eqref{14/01/13/11:51} that for any $t \in [0, T_{X}]$, 
\begin{equation}\label{14/01/13/17:30}
|\lambda_{1}(t)| \sim R_{0}e^{\mu t} \lesssim \delta_{X},
\end{equation}
so that, taking $\delta_{X}$ sufficiently small dependently on $\omega$, we have 
\begin{equation}\label{14/01/13/17:33}
C_{3}\int_{0}^{t} (R_{0}e^{\mu t'})^{\min\{2,p\}}\,dt'
\le 
\frac{C_{2}}{2} \mu \int_{0}^{t} R_{0}e^{\mu t'}\,dt' 
= 
\frac{C_{2}}{2}R_{0}(e^{\mu t}-1).
\end{equation}
Putting the estimates \eqref{14/01/13/14:12}, \eqref{14/01/13/16:22} and \eqref{14/01/13/17:33} together, we find that for any $t \in [0, T_{X}]$, 
\begin{equation}\label{14/01/13/17:05}
\mathfrak{s}\lambda_{2}(t) 
\ge \frac{C_{2}}{2}R_{0} (e^{\mu t}-1)-C_{4}R_{0}^{\min\{2,p\}}
\ge \frac{C_{2}}{2}R_{0} \mu t -C_{4}R_{0}^{\min\{2,p\}}
\end{equation}
for some constants $C_{4}>0$ depending only on $d$, $p$ and $\omega$. Hence, if we choose $T_{*}\ge  \frac{4 C_{4}}{\mu C_{2}}$, then for any $t\in [T_{*}R_{0}^{\min\{1,p-1\}},T_{X}]$, 
\begin{equation}\label{14/01/13/15:57}
\mathfrak{s}\lambda_{2}(t) 
\ge 
T_{*}C_{2}\mu R_{0}^{\min\{2,p\}}. 
\end{equation}

Furthermore, it follows from \eqref{15/02/09/14:32}, \eqref{13/12/30/11:47} and  \eqref{14/01/13/11:59} that 
\begin{equation}\label{15/05/02/15:25}
\begin{split}
|\lambda_{2}(t)| 
&\ge 
|\lambda_{1}(t)|-|\lambda_{-}(t)| 
\\[6pt]
&\ge 
|\lambda_{1}(t)|-e^{-\mu t}\lambda_{-}(0)- C_{5}(e^{\mu t}R_{0})^{\min\{2,p\}}
\\[6pt]
&\ge 
c_{1}R_{0}e^{\mu t}-C_{6}R_{0}e^{-\mu t}
-
C_{5}(R_{0}e^{\mu t})^{\min\{2,p\}}
\end{split}
\end{equation}
for some constants $c_{1}\in (0,1)$, $C_{5}>0$ and $C_{6}>0$ depending only on $d$, $p$ and $\omega$. Hence, if $t\ge \frac{1}{2\mu}\log{(\frac{4C_{6}}{c_{1}})}$, then we find from \eqref{15/05/02/15:25} and \eqref{13/12/30/11:47} that  
\begin{equation}\label{15/05/03/11:23}
|\lambda_{2}(t)| \ge \frac{c_{1}}{10}R_{0}e^{\mu t} \sim |\lambda_{1}(t)|.
\end{equation}
Furthermore, we see from  \eqref{14/01/13/20:02}, \eqref{15/05/03/11:23} and \eqref{14/01/13/17:30} that for any  $t\ge \frac{1}{2\mu}\log{(\frac{4C_{6}}{c_{1}})}$, 
\begin{equation}\label{15/05/03/11:32}
\frac{d}{dt}d_{\omega}(\psi(t)) >0.
\end{equation}
On the other hand, if $t\le \frac{1}{2\mu}\log{(\frac{4C_{6}}{c_{1}})}$, then we have 
\begin{equation}\label{15/05/03/11:28}
|\lambda_{1}(t)| \sim R_{0}e^{\mu t} \lesssim R_{0}. 
\end{equation}
Hence, choosing $T_{*}$ suitably, and using \eqref{14/01/13/20:02}, \eqref{14/01/13/15:57} and \eqref{15/05/03/11:28}, we conclude that for any  $t \in [T_{*}R_{0}^{\min\{1,p-1\}}, T_{X}]$ with $t\le \frac{1}{2\mu}\log{(\frac{4C_{6}}{c_{1}})}$, 
\begin{equation}\label{14/01/15/10:26}
\frac{d}{dt}d_{\omega}(\psi(t))
\ge 
\frac{\mu^{3}C_{2}T_{*}}{2} R_{0}^{\min\{2,p\}}
-C_{1}|\lambda_{1}(t)|^{\min\{2,p\}}
>0.
\end{equation}
Putting \eqref{15/05/03/11:32} and \eqref{14/01/15/10:26} together, we obtain the desired result. 
\par  
We shall prove \eqref{13/05/06/14:37}. It follows from $d_{\omega}(\psi(0))=R_{0}$ and the fundamental theorem of calculus that 
\begin{equation}\label{15/02/11/20:42}
\big|d_{\omega}(\psi(t))-R_{0}\big|
\le 
\int_{0}^{t}
\Big|
\frac{d}{ds}d_{\omega}(\psi(s))\Big|
\,ds.
\end{equation}
Here, it follows from \eqref{13/05/06/15:02}, \eqref{13/05/06/14:46},  \eqref{13/06/12/13:58} and \eqref{13/12/30/11:47} that  
\begin{equation}\label{14/01/15/10:20}
\begin{split}
|\lambda_{1}(t)| \Big| \frac{d}{dt}d_{\omega}(\psi(t))\Big|
&\sim 
\Big|
2d_{\omega}(\psi(t))\frac{d}{dt}d_{\omega}(\psi(t))\Big|
=
\Big|\frac{d}{dt} d_{\omega}(\psi(t))^{2}\Big|
\\[6pt]
&\lesssim  |\lambda_{1}(t)| |\lambda_{2}(t)|+ |\lambda_{1}(t)|^{\min\{3,p+1\}}
\\[6pt]
&=
|\lambda_{1}(t)| |\lambda_{+}(t)-\lambda_{1}(t)|+ |\lambda_{1}(t)|^{\min\{3,p+1\}}
\lesssim |\lambda_{1}(t)|^{2},
\end{split}
\end{equation}
where the implicit constants depend on $\omega$. Furthermore, this together with \eqref{13/12/30/11:47} gives us that for any $t\in [0, T_{*}R_{0}^{\min\{1,p-1\}}]$ 
\begin{equation}\label{14/01/15/12:11}
\Big| \frac{d}{dt}d_{\omega}(\psi(t)) \Big|
\lesssim 
R_{0}e^{\mu t}.
\end{equation}
Putting \eqref{15/02/11/20:42} and \eqref{14/01/15/12:11} together, we find that for any $t \in [0,  T_{*}R_{0}^{\min\{1,p-1\}}]$, 
\begin{equation}\label{14/01/15/12:07}
\big| d_{\omega}(\psi(t))-R_{0} \big|
\lesssim 
\int_{0}^{t} R_{0}e^{\mu s}\,ds  
\lesssim 
T_{*}R_{0}^{\min\{2,p\}}
e^{\mu T_{*}R_{0}^{\min\{1,p-1\}}}
\lesssim 
R_{0}^{\min\{2,p\}},
\end{equation}
where the implicit constants depend on $\omega$. Thus, we have proved \eqref{13/05/06/14:37}.
\par 
Finally, we shall prove \eqref{13/05/06/14:30}.  Taylor's expansion of $\mathcal{K}$ around $\Phi_{\omega}$ together with $\mathcal{K}(\Phi_{\omega})=0$ and \eqref{14/02/15/15:25} shows that 
\begin{equation}\label{14/01/15/20:42}
\begin{split}
\mathcal{K}(\psi(t))
&=
\mathcal{K}(\Phi_{\omega}+\eta(t))
=
\mathcal{K}'(\Phi_{\omega})\eta(t)+O(\|\eta(t)\|_{H^{1}}^{2})
\\[6pt]
&=
-2\mu s_{p}\lambda_{1}(t)( \Phi_{\omega}, f_{2})_{L^{2}}
-
2(1-s_{p})(2^{*}-2)
\lambda_{1}(t)
(\Phi_{\omega}^{2^{*}-1}
, 
f_{1} )_{L^{2}}
\\[6pt]
&\quad 
-s_{p}(p-1)(\Phi_{\omega}^{p}, \Gamma(t))_{L_{real}^{2}}
-(2^{*}-2) (\Phi_{\omega}^{2^{*}-1},\Gamma(t))_{L_{real}^{2}}
\\[6pt]
&\quad +
\omega \mathcal{M}(\eta(t)) 
+
O(\| \eta(t) \|_{H^{1}}^{2})
.
\end{split}
\end{equation}
Furthermore, multiplying the both sides above by $\mathfrak{s}={\rm sgn}[ \lambda_{1}]$, and using $(\Phi_{\omega},f_{2})_{L^{2}}<0$ (see \eqref{13/12/28/17:22}), \eqref{18/02/04/16:52} and \eqref{13/05/06/15:02}, we find that 
 for any $t \in [0, T_{X}]$, 
\begin{equation}\label{14/01/15/21:20}
\begin{split}
\mathfrak{s} \mathcal{K}(\psi(t))
&\ge 2s_{p} \mu 
\big|(\Phi_{\omega},f_{2})_{L^{2}}\big| |\lambda_{1}(t)|
-
2(1-s_{p})(2^{*}-2)\big|(\Phi_{\omega}^{2^{*}-1},f_{1})_{L^{2}}\big| |\lambda_{1}(t)|
\\[6pt]
&\quad -
(p-1) \| \Phi_{\omega}\|_{L^{p+1}}^{p}\|\Gamma(t)\|_{L^{p+1}}
-
(2^{*}-2)\| \Phi_{\omega}\|_{L^{2^{*}}}^{2^{*}-1}\|\Gamma(t)\|_{L^{2^{*}}}
\\[6pt]
&\quad -
O(| \lambda_{1}(t) |^{2})
\\[6pt]
&\ge  s_{p} \mu 
\big|(\Phi_{\omega},f_{2})_{L^{2}}\big| |\lambda_{1}(t)|
\\[6pt]
&\quad -
(p-1) \| \Phi_{\omega}\|_{L^{p+1}}^{p}\|\Gamma(t)\|_{L^{p+1}}
-
(2^{*}-2)\| \Phi_{\omega}\|_{L^{2^{*}}}^{2^{*}-1}\|\Gamma(t)\|_{L^{2^{*}}}
\\[6pt]
&\quad -
O(| \lambda_{1}(t) |^{2}).
\end{split}
\end{equation}
Recall here that $|\lambda_{1}(t)|\lesssim \delta_{X} \ll 1$ for $t\in [0, T_{X}]$ (see \eqref{13/05/06/15:02}).  Then, we conclude from \eqref{14/01/15/21:20}, \eqref{13/12/30/11:47} and the proved result \eqref{13/05/05/22:54} that there exists a constant $C_{*}>0$ depending only on $d$, $p$ and $\omega$ such that 
\begin{equation}\label{14/01/15/21:16}
\begin{split}
\mathfrak{s}\mathcal{K}(\psi(t))
&\gtrsim  
R_{0}e^{\mu t} -C_{*} R_{0},
\end{split}
\end{equation}
where the implicit constant depends on $\omega$. 
Thus, we have completed the proof. 
\end{proof}

%%%%%%%%%%%%%%%%%%%%%%%%%%%%%%%%%%%%%%%%%%%%%%%%%%%%%%%%%%%%%%%%%%%%%%%%%%

\section{Modified distance function}\label{15/02/08/14:28}
In this section, we continue to study the decomposition of the form \eqref{13/04/27/9:38} with \eqref{13/02/22/15:27}, \eqref{13/05/04/13:54} and \eqref{13/05/09/10:46}. Our aim here is to introduce the ``modified distance function'' $\widetilde{d}_{\omega}$ (see Proposition \ref{18/01/26/21:25}). To this end, we introduce two distances between a function $u \in H^{1}(\mathbb{R}^{d})$ and the orbit of $\Phi_{\omega}$: 
\begin{align}
\label{13/05/05/14:05}
{\rm dist}_{H^{1}}\big(u, \mathscr{O}(\Phi_{\omega})\big)
&:=
\inf_{\theta \in \mathbb{R}}\|u -e^{i\theta}\Phi_{\omega}\|_{H^{1}},
\\[6pt]
\label{15/06/25/11:33}
{\rm dist}_{L^{2}}\big(u, \mathscr{O}(\Phi_{\omega})\big)
&:=
\inf_{\theta \in \mathbb{R}}\|u-e^{i\theta}\Phi_{\omega}\|_{L^{2}}.
\end{align}

\begin{lemma}\label{15/06/25/11:15}
Assume $d\ge 3$, $1+\frac{4}{d}<p<2^{*}-1$, and let $\omega_{2}$ be the frequency given by Proposition \ref{13/02/20/9:41}. Then, for any $\omega \in (0, \omega_{2})$, there exists $\gamma_{1}(\omega)>0$ with the following property: let $\psi$ be a solution to \eqref{12/03/23/17:57} satisfying $\mathcal{M}(\psi)=\mathcal{M}(\Phi_{\omega})$. If ${\rm dist}_{H^{1}}\big(\psi(t), \mathscr{O}(\Phi_{\omega})\big)\le \gamma_{1}(\omega)$, then 
\begin{equation}\label{13/05/14/16:47}
\|\eta(t)\|_{E} \lesssim 
{\rm dist}_{H^{1}}\big(\psi(t), \mathscr{O}(\Phi_{\omega})\big)
\le \|\eta(t)\|_{E},
\end{equation} 
and if ${\rm dist}_{L^{2}}\big(\psi(t), \mathscr{O}(\Phi_{\omega})\big)\le \gamma_{1}(\omega)$, then 
\begin{equation}\label{14/01/23/20:23}
\|\eta(t)\|_{L^{2}} \sim 
{\rm dist}_{L^{2}}\big(\psi(t), \mathscr{O}(\Phi_{\omega})\big).
\end{equation} 
Here, $\eta$ is the function appearing in the decomposition for $\psi$ of the form \eqref{13/04/27/9:38} with \eqref{13/02/22/15:27}, \eqref{13/05/04/13:54} and \eqref{13/05/09/10:46}. Moreover, the implicit constants in \eqref{13/05/14/16:47} and \eqref{14/01/23/20:23} depend on $\omega$ as well as $d$ and $p$. 
\end{lemma}
\begin{proof}[Proof of Lemma \ref{15/06/25/11:15}]
Note first that Lemma \ref{13/05/11/11:52} shows   
\begin{equation}\label{13/05/14/16:45}
{\rm dist}_{H^{1}}\big(\psi(t), \mathscr{O}(\Phi_{\omega})\big)
\le \|\psi(t)-e^{i\theta(t)}\Phi_{\omega}\|_{H^{1}} = \|\eta(t)\|_{H^{1}}
\lesssim 
\|\eta(t)\|_{E}.
\end{equation}
Moreover, we can take a continuous function $\theta_{0}(t)$ of $t$ such that  
\begin{equation}\label{13/05/05/15:02}
\|\psi(t)-e^{i\theta_{0}(t)}\Phi_{\omega}\|_{H^{1}}
=
{\rm dist}_{H^{1}}\big(\psi(t), \mathscr{O}(\Phi_{\omega})\big)
.
\end{equation}
Put $\eta_{0}(t):=\psi(t)-e^{i\theta_{0}(t)}\Phi_{\omega}$, so that $\|\eta_{0}(t)\|_{H^{1}}={\rm dist}_{H^{1}}\big(\psi(t), \mathscr{O}(\Phi_{\omega})\big)$. Then, it follows from $\Omega(\Phi_{\omega},\partial_{\omega}\Phi_{\omega})=(\Phi_{\omega},i\partial_{\omega}\Phi_{\omega})_{L_{real}^{2}}=0$ and \eqref{13/05/04/13:54} that  
\begin{equation}\label{13/05/05/15:08}
\begin{split}
\Omega(e^{-i\theta(t)}\eta_{0}(t),\partial_{\omega}\Phi_{\omega})
&=
\Omega(e^{-i\theta(t)}\psi(t), \partial_{\omega}\Phi_{\omega})
-
\Omega(e^{i(\theta_{0}(t)-\theta(t))}\Phi_{\omega}, \partial_{\omega}\Phi_{\omega})
\\[6pt]
&=
-
\big( ie^{i(\theta_{0}(t)-\theta(t))} \Phi_{\omega},  \partial_{\omega}\Phi_{\omega}\big)_{L_{real}^{2}}
\\[6pt]
&=
\sin{(\theta_{0}(t)-\theta(t))}(\Phi_{\omega},\partial_{\omega}\Phi_{\omega})_{L_{real}^{2}}.
\end{split}
\end{equation}
Hence, we have 
\begin{equation}\label{13/05/09/10:08}
\begin{split}
\big| \sin{(\theta_{0}(t)-\theta(t))} \big| 
\big| (\Phi_{\omega},\partial_{\omega}\Phi_{\omega})_{L_{real}^{2}}\big|
&\le \|\eta_{0}(t) \|_{L^{2}}\|\partial_{\omega}\Phi_{\omega} \|_{L^{2}}
\\[6pt]
&\le 
{\rm dist}_{H^{1}}\big(\psi(t), \mathscr{O}(\Phi_{\omega})\big)\| \partial_{\omega}\Phi_{\omega}\|_{L^{2}} .
\end{split}
\end{equation}
We see from \eqref{13/05/09/10:08} and ${\rm dist}_{H^{1}}\big(\psi(t), \mathscr{O}(\Phi_{\omega})\big)\ll 1$ that 
\begin{equation}\label{13/05/09/10:21}
\inf_{k\in \mathbb{Z}}\big| \theta_{0}(t)-\theta(t) +k\pi \big|
\lesssim 
{\rm dist}_{H^{1}}\big(\psi(t), \mathscr{O}(\Phi_{\omega})\big),
\end{equation}
where the implicit constant depends on $\omega$. 
Here, we can eliminate the case where $k$ is odd in the above infimum. Indeed, it follows from the choice of $\theta(t)$ (see \eqref{13/05/09/10:46}) that 
\begin{equation}\label{13/05/09/11:03}
\begin{split}
0&>(e^{-i\theta(t)}\psi(t),\partial_{\omega}\Phi_{\omega})_{L_{real}^{2}}
\\[6pt]
&=
\cos{(\theta_{0}(t)-\theta(t))}(\Phi_{\omega},\partial_{\omega}\Phi_{\omega})_{L_{real}^{2}}
+
(e^{-i\theta(t)}\eta_{0}(t), \partial_{\omega}\Phi_{\omega})_{L_{real}^{2}},
\end{split}
\end{equation}
so that 
\begin{equation}\label{13/05/09/11:13}
-\cos{(\theta_{0}(t)-\theta(t))}(\Phi_{\omega},\partial_{\omega}\Phi_{\omega})_{L_{real}^{2}}
>
(e^{-i\theta(t)}\eta_{0}(t), \partial_{\omega}\Phi_{\omega})_{L_{real}^{2}}.
\end{equation}
Thus, supposing for contradiction that $\theta_{0}(t)-\theta(t)$ lay near an odd multiple of $\pi$, we have 
\begin{equation}\label{13/05/09/11:41}
\frac{1}{2}(\Phi_{\omega},\partial_{\omega}\Phi_{\omega})_{L_{real}^{2}}
>
(e^{-i\theta(t)}\eta_{0}(t), \partial_{\omega}\Phi_{\omega})_{L_{real}^{2}}.
\end{equation}
However, this together with {\rm (iii)} of Proposition \ref{13/03/29/15:30} 
 shows   
\begin{equation}\label{13/05/09/11:45}
1\lesssim 
\frac{1}{2}\big|(\Phi_{\omega},\partial_{\omega}\Phi_{\omega})_{L_{real}^{2}}\big|
<
\|\eta_{0}(t)\|_{L^{2}}\| \partial_{\omega}\Phi_{\omega} \|_{L^{2}}
\le 
{\rm dist}_{H^{1}}\big(\psi(t), \mathscr{O}(\Phi_{\omega})\big)\| \partial_{\omega}\Phi_{\omega} \|_{L^{2}}
,
\end{equation}
which contradicts that ${\rm dist}_{H^{1}}\big(\psi(t), \mathscr{O}(\Phi_{\omega})\big)\ll 1$.  Since ${\rm dist}_{H^{1}}\big(\psi(t), \mathscr{O}(\Phi_{\omega})\big)\ll 1$ implies that $\theta_{0}(t)-\theta(t)$ lies near an even multiple of $\pi$, we see from Lemma \ref{13/05/11/11:52} and \eqref{13/05/09/10:21} that  
\begin{equation}\label{13/05/09/11:16}
\begin{split}
\|\eta(t)\|_{E}
&\lesssim
\|\eta(t)\|_{H^{1}}
=
\| e^{-i\theta(t)}\psi(t)-\Phi_{\omega}\|_{H^{1}}
\\[6pt]
&\le 
\big| e^{i(\theta_{0}(t)-\theta(t))} -1 \big| 
\| \Phi_{\omega} \|_{H^{1}} + \|\eta_{0}(t)\|_{H^{1}}
\\[6pt]
&\lesssim  
\big| \cos{(\theta_{0}(t)-\theta(t))}-1\big| + \big| \sin{(\theta_{0}(t)-\theta(t))}\big| + \|\eta_{0}(t)\|_{H^{1}}
\\[6pt]
&\lesssim \inf_{k \in \mathbb{Z}}|\theta_{0}(t)-\theta(t)- 2k\pi |+ \|\eta_{0}(t)\|_{H^{1}}
\\[6pt]
&\lesssim 
{\rm dist}_{H^{1}}\big(\psi(t), \mathscr{O}(\Phi_{\omega})\big)
.
\end{split}
\end{equation}
Putting \eqref{13/05/14/16:45} and \eqref{13/05/09/11:16} together, we obtain the desired result \eqref{13/05/14/16:47}. The same argument as the above shows \eqref{14/01/23/20:23}.
\end{proof}

Now, we recall that $H_{\omega}^{1}(\mathbb{R}^{d})$ denotes the set $\big\{
u \in H^{1}(\mathbb{R}^{d}) \colon \mathcal{M}(u)= \mathcal{M}(\Phi_{\omega})
\big\}$ (see \eqref{18/01/26/22:15}). The following proposition is an immediate  consequence of Lemma \ref{15/06/25/11:15}:
\begin{proposition}\label{18/01/26/21:25}
Assume $d\ge 3$, $1+\frac{4}{d}<p<2^{*}-1$, and let $\omega_{2}$ be the frequency given by Proposition \ref{13/02/20/9:41}. Then, for any $\omega \in (0, \omega_{2})$, there exist a constant $\widetilde{\gamma}(\omega)\in (0, \delta_{E}(\omega))$ ($\delta_{E}(\omega)$ denotes the constant given by Proposition \ref{18/01/27/13:12}) and a continuous function $\widetilde{d}_{\omega} \colon H_{\omega}^{1}(\mathbb{R}^{d}) \to [0,\infty)$ such that: 
\begin{equation}\label{13/05/14/21:30}
\widetilde{d}_{\omega}(u)
\sim 
{\rm dist}_{H^{1}}\big(u, \mathscr{O}(\Phi_{\omega})\big),
\end{equation}
where the implicit constant depends only on $d$, $p$ and $\omega$; and if $\widetilde{d}_{\omega}(u)\le \widetilde{\gamma}(\omega)$, then 
\begin{equation}\label{13/05/14/21:37}
\widetilde{d}_{\omega}(u)= d_{\omega}(u), 
\end{equation}
where $d_{\omega}$ is the distance function defined by \eqref{13/05/05/13:27}.
\end{proposition} 
\begin{remark}\label{18/01/27/14:03}
 We can take the constant $\delta_{X}$ given by the ejection lemma (Lemma \ref{13/05/05/22:31}) so small that $\delta_{X} \le \widetilde{\gamma}(\omega)$. Hence, we may assume that $\delta_{X} < \widetilde{\gamma}(\omega)$ in what follows.
\end{remark}

\begin{lemma}\label{13/05/15/9:32}
Assume $d\ge 3$ and $1+\frac{4}{d}<p< 2^{*}-1$, and let $\omega_{2}$ be the frequency given by Proposition \ref{13/02/20/9:41}. Then, for any $\omega \in (0, \omega_{2})$ and any $\delta>0$, there exist $\varepsilon_{0}(\delta)>0$ and $\kappa_{1}(\delta)>0$ with the following properties: $\varepsilon_{0}(\delta)$ and $\kappa_{1}(\delta)$ are non-decreasing with respect to $\delta$; and for any  radial function $u \in H^{1}(\mathbb{R}^{d})$ satisfying  
\begin{align}
\label{13/05/14/17:16}
&\mathcal{M}(u)=\mathcal{M}(\Phi_{\omega}),
\\[6pt]
\label{13/05/14/17:17}
&\mathcal{S}_{\omega}(u)<m_{\omega}+\varepsilon_{0}(\delta),
\\[6pt]
\label{13/05/14/17:18}
& \widetilde{d}_{\omega}(u)
\ge \delta,
\end{align}
we have either 
\begin{equation}\label{13/05/14/17:21}
\mathcal{K}(u) \le -\kappa_{1}(\delta)
\end{equation}
or 
\begin{equation}\label{13/05/14/17:22}
\mathcal{K}(u) 
\ge 
\min{\big\{ \kappa_{1}(\delta), \, \frac{1}{2}\|\nabla u\|_{L^{2}}^{2} \big\}}.
\end{equation} 
Here, $\widetilde{d}_{\omega}$ denotes the function given by Proposition \ref{18/01/26/21:25}. 
\end{lemma}
\begin{proof}[Proof of Lemma \ref{13/05/15/9:32}]
We see from H\"older's inequality that if $u \in H^{1}(\mathbb{R}^{d})$ satisfies $\mathcal{M}(u)=\mathcal{M}(\Phi_{\omega})$, then 
\begin{equation}\label{13/05/16/09:01}
\|u\|_{L^{p+1}}^{p+1}
\le 
\|u\|_{L^{2}}^{p+1-\frac{d(p-1)}{2}} \|u\|_{L^{2^{*}}}^{\frac{d(p-1)}{2}}
=
\mathcal{M}(\Phi_{\omega})^{\frac{p+1}{2}-\frac{d(p-1)}{4}}  
\|u\|_{L^{2^{*}}}^{\frac{d(p-1)}{2}}.
\end{equation} 
Since $2< \frac{d(p-1)}{2}< 2^{*}$, we see from \eqref{13/05/16/09:01} and Sobolev's embedding that there exists $c_{0}>0$ depending only on $d$, $p$ and $\omega$ such that if $u\in H^{1}(\mathbb{R}^{d})$ satisfies $\mathcal{M}(u)=\mathcal{M}(\Phi_{\omega})$ and $\|\nabla u\|_{L^{2}}^{2}\le c_{0}$, then $\mathcal{K}(u)\ge\frac{1}{2}\|\nabla u \|_{L^{2}}^{2}$. Thus, for the desired result, it is sufficient to prove that for any $\delta>0$, there exist $\varepsilon_{0}(\delta)>0$ and $\kappa_{0}(\delta)>0$ with the following property: for any radial function $u \in H^{1}(\mathbb{R}^{d})$ satisfying \eqref{13/05/14/17:16} through \eqref{13/05/14/17:18} and  $\|\nabla u\|_{L^{2}}^{2}\ge c_{0}$, we have $|\mathcal{K}(u)| > \kappa_{1}(\delta)$. We prove this by contradiction argument. Hence, we suppose to the contrary that there exists $\delta_{0}>0$ with the following property: for any number $n\ge 1$, there exists a radial function $u_{n} \in H^{1}(\mathbb{R}^{d})$ satisfying 
\begin{align}
\label{13/05/17/15:06}
&\mathcal{M}(u_{n})=\mathcal{M}(\Phi_{\omega}),
\\[6pt]
\label{13/05/15/20:39}
&\mathcal{S}_{\omega}(u_{n})<m_{\omega}+\frac{1}{n},
\\[6pt]
\label{13/05/15/20:40}
& \widetilde{d}_{\omega}(u_{n})
\ge \delta_{0},
\\[6pt]
\label{17/12/08/15:05}
&
\inf_{n\ge 1}\|\nabla u_{n}\|_{L^{2}}\ge c_{0},
\\[6pt]
\label{13/05/15/20:41}
& |\mathcal{K}(u_{n})| \le \frac{1}{n}
.
\end{align}
Then, we see from \eqref{13/05/15/20:39} and \eqref{13/05/15/20:41}  that  
\begin{equation}\label{13/05/15/21:17}
\frac{s_{p}}{d}\|\nabla u_{n}\|_{L^{2}}^{2}
\le 
\mathcal{I}_{\omega}(u_{n})
=\mathcal{S}_{\omega}(u_{n})-\frac{2}{d(p-1)}\mathcal{K}(u_{n})
\le m_{\omega}+2,
\end{equation}
so that  
\begin{equation}\label{13/05/15/21:27}
\|u_{n}\|_{H^{1}}^{2} 
\le 
2 \mathcal{M}(\Phi_{\omega}) 
+
\frac{d}{s_{p}}(m_{\omega}+2)
.
\end{equation}
Hence, we can take a radial function $u_{\infty}\in H^{1}(\mathbb{R}^{d})$ and a subsequence of $\{u_{n}\}$ (still denoted by the same symbol $\{u_{n}\}$) such that 
\begin{equation}\label{14/12/26/16:27}
\lim_{n\to \infty}u_{n}=u_{\infty} \qquad \mbox{weakly in $H^{1}(\mathbb{R}^{d})$}
\end{equation}
and for any $2<q<2^{*}$, 
\begin{equation}\label{14/12/26/16:30}
\lim_{n\to \infty}u_{n}=u_{\infty} \qquad \mbox{strongly in $L^{q}(\mathbb{R}^{d})$}.
\end{equation}
We shall show that $u_{\infty}$ is non-trivial by using the contradiction argument. Hence, suppose to the contrary that $u_{\infty}\equiv 0$. Then, we see from \eqref{14/12/26/16:30} that 
\begin{equation}\label{13/05/15/21:29}
\lim_{n\to \infty}\|u_{n}\|_{L^{p+1}}= 0 .
\end{equation}
Moreover, \eqref{13/05/15/20:41} together with \eqref{13/05/15/21:29} shows that, passing to some subsequence, 
\begin{equation}\label{13/05/17/15:13}
\lim_{n\to \infty}\|\nabla u_{n}\|_{L^{2}}^{2}
=\lim_{n\to \infty}\|u_{n}\|_{L^{2^{*}}}^{2^{*}}.
\end{equation} 
We see from the definition of $\sigma$ (see \eqref{12/03/24/17:52}) and \eqref{13/05/17/15:13} that 
\begin{equation}\label{13/05/17/15:18}
\lim_{n\to \infty}\|\nabla u_{n}\|_{L^{2}}^{2}
\ge \sigma \lim_{n\to \infty}\|u_{n}\|_{L^{2^{*}}}^{2}
=\sigma \lim_{n\to \infty}\|\nabla u_{n}\|_{L^{2}}^{\frac{4}{2^{*}}},
\end{equation} 
which together with \eqref{13/05/17/15:13} and \eqref{17/12/08/15:05} yields   
\begin{equation}\label{13/05/17/15:24}
\lim_{n\to \infty}\|u_{n}\|_{L^{2^{*}}}^{2^{*}}
\ge \sigma^{\frac{d}{2}}.
\end{equation} 
Furthermore, we see from the definition of $\mathcal{J}_{\omega}$ (see \eqref{13/03/08/10:38} and \eqref{15/12/31/11:02}),  \eqref{13/05/15/20:39}, \eqref{13/05/15/20:41} and \eqref{13/05/17/15:24}, we find that  
\begin{equation}\label{13/05/17/15:28}
\frac{1}{d}\sigma^{\frac{d}{2}}
\le 
\lim_{n\to \infty}\frac{1}{d}\|u_{n}\|_{L^{2^{*}}}^{2^{*}}
\le 
\lim_{n\to \infty}\mathcal{J}_{\omega}(u_{n})
\le
\lim_{n\to \infty}\big\{ \mathcal{S}_{\omega}(u_{n})+\frac{1}{2}|\mathcal{K}(u_{n})| \big\}
 \le m_{\omega}.
\end{equation} 
However, this contradicts \eqref{12/03/23/17:48} and therefore $u_{\infty}$ is non-trivial. 
\par 
Next, we shall prove $u_{\infty}=e^{i\theta_{0}}\Phi_{\omega}$ for some $\theta_{0}$. We see from \eqref{13/05/15/20:39} and \eqref{13/05/15/20:41} that  
\begin{equation}\label{13/05/17/15:43}
\mathcal{I}_{\omega}(u_{\infty})
\le \lim_{n\to \infty} \mathcal{I}_{\omega}(u_{n})
\le \lim_{n\to \infty} \big\{ \mathcal{S}_{\omega}(u_{n})+\frac{2}{d(p-1)}|\mathcal{K}(u_{n})| \big\} \le m_{\omega}. 
\end{equation}
Hence, from the point of view of \eqref{12/03/23/18:17}, what we need to prove   is that $\mathcal{K}(u_{\infty})\le 0$. We prove this by contradiction. Hence, suppose to the contrary that $\mathcal{K}(u_{\infty})>0$. 
Then, the Brezis-Lieb lemma together with \eqref{13/05/15/20:41} shows  
\begin{equation}\label{13/05/17/16:09}
\lim_{n\to \infty} \mathcal{K}(u_{n}-u_{\infty})
=
-\lim_{n\to \infty}
\big\{ \mathcal{K}(u_{n})-
\mathcal{K}(u_{n}-u_{\infty})
\big\}
=-\mathcal{K}(u_{\infty})<0,
\end{equation}
so that for any sufficiently large $n\in \mathbb{N}$, 
\begin{equation}\label{13/05/17/16:09}
\mathcal{K}(u_{n}-u_{\infty})<0. 
\end{equation}
Hence, we can take $\lambda_{n} < 1$ such that $\mathcal{K}(\lambda_{n}^{\frac{d}{2}}\big\{u_{n}(\lambda_{n}\cdot)-u_{\infty}(\lambda_{n}\cdot)\big\})=0$ (cf. Lemma 2.1 in \cite{AIKN2}). Moreover, this implies that 
\begin{equation}\label{13/05/17/16:17}
\begin{split}
m_{\omega}
&\le \mathcal{I}_{\omega}\big(\lambda_{n}^{\frac{d}{2}}\big\{ 
u_{n}(\lambda_{n}\cdot)-u_{\infty}(\lambda_{n}\cdot)\big\} \big)
\le \mathcal{I}_{\omega}(u_{n}-u_{\infty})
\\[6pt]
&= \mathcal{I}_{\omega}(u_{n})
-\mathcal{I}_{\omega} (u_{\infty})+o_{n}(1) 
=\mathcal{S}_{\omega}(u_{n})-\frac{2}{d(p-1)}\mathcal{K}(u_{n})
-\mathcal{I}_{\omega} (u_{\infty})+o_{n}(1) ,
\end{split}
\end{equation}
which together with \eqref{13/05/15/20:39} and \eqref{13/05/15/20:41} shows 
\begin{equation}\label{13/05/17/16:17}
m_{\omega}\le 
m_{\omega}-\mathcal{I}_{\omega}(u_{\infty}).
\end{equation}
Since $\mathcal{I}_{\omega}(u_{\infty})>0$, this is a contradiction. Hence, $\mathcal{K}(u_{\infty})\le 0$. In particular, we find from \eqref{13/05/17/15:43} 
 that $u_{\infty}$ is a minimizer of the variational problem \eqref{12/03/23/18:17}. Since any minimizer of $\mathcal{I}_{\omega}$ is also a ground state (see Proposition 1.2 in \cite{AIKN}), we conclude from the radial symmetry and non-triviality of $u_{\infty}$, and Proposition \ref{13/03/16/15:33} that $u_{\infty}=e^{i\theta_{0}}\Phi_{\omega}$ for some $\theta_{0} \in \mathbb{R}$. In particular, $\mathcal{K}(u_{\infty})=0$ and $\mathcal{I}(u_{\infty})=m_{\omega}$. However, the condition \eqref{13/05/15/20:40} prevents $u_{\infty}$ from existing. Indeed, it follows from the weak convergence \eqref{14/12/26/16:27} and  $\lim_{n\to \infty}\mathcal{I}_{\omega}(u_{n})\le m_{\omega}= \mathcal{I}(u_{\infty})$ that 
\begin{equation}\label{13/05/17/17:02}
\lim_{n\to \infty}\|u_{n}\|_{L^{2}}=\|u_{\infty}\|_{L^{2}},
\qquad 
\lim_{n\to \infty}\|\nabla u_{n}\|_{L^{2}}=\|\nabla u_{\infty}\|_{L^{2}}.
\end{equation}
Hence, we find that 
\begin{equation}\label{13/05/17/16:57}
\lim_{n\to \infty}u_{n} =u_{\infty}
\qquad \mbox{strongly in $H^{1}(\mathbb{R}^{d})$}.
\end{equation}
Using the strong convergence \eqref{13/05/17/16:57}, we have 
\begin{equation}\label{13/05/17/16:50}
\delta_{0}<\inf_{\theta}\|u_{n}-e^{i\theta}\Phi_{\omega}\|_{H^{1}}
\le 
\|u_{n}-e^{i\theta_{0}}\Phi_{\omega}\|_{H^{1}}
=
\|u_{n}-u_{\infty}\|_{H^{1}}  
=o_{n}(1). 
\end{equation}
However, this is a contradiction. Thus, we cannot take a sequence $\{u_{n}\}$ satisfying \eqref{13/05/17/15:06} through \eqref{13/05/15/20:41} and therefore the lemma holds. 
\end{proof}

\begin{lemma}\label{13/05/14/17:12}
Assume $d\ge 3$ and $1+\frac{4}{d}<p < 2^{*}-1$, and let $\omega_{2}$ be the frequency given by Proposition \ref{13/02/20/9:41}. Then, for any $\omega \in (0, \omega_{2})$ and any $\delta>0$, there exists $\kappa_{2}(\delta)>0$ such that for any radial function $u\in H^{1}(\mathbb{R}^{d})$ satisfying $\mathcal{J}_{\omega}(u) \le m_{\omega}-\delta$, the following holds:  
\begin{equation}\label{13/05/14/17:31}
\mathcal{K}(u) \ge \min\big\{ \kappa_{2}(\delta), \,  \frac{1}{2}\|\nabla u \|_{L^{2}}^{2} \big\}.
\end{equation}
\end{lemma}
\begin{proof}[Proof of Lemma \ref{13/05/14/17:12}]
Let $u$ be a function satisfying $\mathcal{J}_{\omega}(u) \le m_{\omega}-\delta$. Then, we see from \eqref{15/12/31/11:02} and \eqref{12/03/23/18:17} that 
\begin{align}\label{13/05/14/20:48}
&\|u\|_{L^{2}}^{2} \le \frac{2}{\omega}\mathcal{J}_{\omega}(u) \le \frac{2}{\omega}m_{\omega}, 
\\[6pt]
\label{13/05/15/9:38}
&\mathcal{K}(u)>0. 
\end{align}
Furthermore, it follows from H\"older's inequality and \eqref{13/05/14/20:48}
 that 
\begin{equation}\label{13/05/17/09:01}
\|u\|_{L^{p+1}}^{p+1}
\le 
\|u\|_{L^{2}}^{p+1-\frac{d(p-1)}{2}} \|u\|_{L^{2^{*}}}^{\frac{d(p-1)}{2}}
\le
\Bigm( \frac{2}{\omega^{2}}m_{\omega}
\Bigm)^{\frac{p+1}{2}-\frac{d(p-1)}{4}}\|u\|_{L^{2^{*}}}^{\frac{d(p-1)}{2}}
.
\end{equation} 
Hence, we find from \eqref{13/05/17/09:01}, $2<\frac{d(p-1)}{2}< 2^{*}$ and Sobolev's embedding that there exists $c_{0}>0$ depending only on $d$, $p$ and $\omega$ such that if $u\in H^{1}(\mathbb{R}^{d})$ satisfies $\mathcal{J}_{\omega}(u) \le m_{\omega}-\delta$ and $\|\nabla u\|_{L^{2}}^{2}\le c_{0}$,  then $\mathcal{K}(u)\ge\frac{1}{2}\|\nabla u \|_{L^{2}}^{2}$. Thus, for the desired result \eqref{13/05/14/17:31}, it is sufficient to prove that for any $\delta>0$, there exists $\kappa_{2}(\delta)>0$ with the following property: if $u \in H^{1}(\mathbb{R}^{d})$ is radial and satisfies that $\mathcal{J}_{\omega}(u) \le m_{\omega}-\delta$ and $\|\nabla u\|_{L^{2}}^{2}\ge c_{0}$, then $\mathcal{K}(u) > \kappa_{2}(\delta)$. 
We prove this by contradiction. Hence, suppose to the contrary that there exists  $\delta_{0}>0$ with the following property: for any $n\in \mathbb{N}$, there exists a radial function $u_{n} \in H^{1}(\mathbb{R}^{d})$ such that 
\begin{align}
\label{17/12/08/17:14}
& \mathcal{J}_{\omega}(u_{n})\le m_{\omega}-\delta_{0}, 
\\[6pt]
\label{17/12/08/17:15}
&\inf_{n\ge 1}\|\nabla u_{n} \|_{L^{2}}\ge c_{0},
\\[6pt]
\label{17/12/08/17:16}
&0<\mathcal{K}(u_{n}) \le \frac{1}{n}.
\end{align}
Furthermore, it follows from \eqref{13/03/08/10:38}, \eqref{17/12/08/17:14} and \eqref{17/12/08/17:16} that  
\begin{equation}\label{17/12/08/17:21}
\mathcal{S}_{\omega}(u_{n})
\le
\mathcal{J}_{\omega}(u_{n}) +\frac{2}{n} \le m_{\omega} +\frac{2}{n}-\delta_{0}.\end{equation}

Now, using \eqref{15/12/31/11:02}, \eqref{17/12/08/17:14} and \eqref{17/12/08/17:16}, we find that for any $n \ge 1$, 
\begin{equation}\label{17/12/08/15:42}
\begin{split}
\|u\|_{L^{2}}^{2}+\|\nabla u_{n}\|_{L^{2}}^{2}  
&\le
\mathcal{K}(u_{n})+\|u\|_{L^{2}}^{2}+ \|u\|_{L^{p+1}}^{p+1}+\|u\|_{L^{2^{*}}}^{2^{*}}
\\[6pt]
&\lesssim 1+ \mathcal{J}_{\omega}(u_{n})
\le 1+m_{\omega}. 
\end{split} 
\end{equation}
Hence, we can extract a subsequence of $\{u_{n}\}$ (still denoted by the same symbol $\{u_{n}\}$) and a radial function $u_{\infty}\in H^{1}(\mathbb{R}^{d})$ such that 
\begin{equation}\label{17/12/08/15:58}
\lim_{n\to \infty}u_{n}=u_{\infty} \qquad \mbox{weakly in $H^{1}(\mathbb{R}^{d})$}
\end{equation}
and for any $2<q<2^{*}$, 
\begin{equation}\label{17/12/08/15:59}
\lim_{n\to \infty}u_{n}=u_{\infty} \qquad \mbox{strongly in $L^{q}(\mathbb{R}^{d})$}.
\end{equation}
Then, we see from the same argument as the proof of Lemma \ref{13/05/15/9:32} that $u_{\infty}$ is non-trivial. 
Furthermore, the lower semi-continuity of the weak limit together with \eqref{17/12/08/15:58}, \eqref{17/12/08/15:59} and \eqref{17/12/08/17:14} implies that 
\begin{equation}\label{17/12/08/17:00}
\mathcal{J}_{\omega}(u_{\infty}) \le \liminf_{n\to \infty}\mathcal{J}_{\omega}(u_{n}) < m_{\omega}. 
\end{equation}
Hence, we find from \eqref{12/03/23/18:17} and $u_{\infty}\not \equiv 0$ that $\mathcal{K}(u_{\infty})>0$. However, the same argument as the proof of  Lemma \ref{13/05/15/9:32} shows that $\mathcal{K}(u_{\infty})\le 0$. This is a contradiction. Thus, we have completed the proof. 
\end{proof}

%%%%%%%%%%%%%%%%%%%%%%%%%%%%%%%%%%%%%%%%%%%%%%%%%%%%%%%%%%%%%%%%%%%%%%%%%%%%%%%%
\section{One-pass theorem}\label{13/06/15/20:31}

In this section, we derive the one-pass theorem for our equation \eqref{12/03/23/17:57} (cf. Theorem 4.1 in \cite{Nakanishi-Schlag2}). To this end, we need the following inequalities for radial functions: 
\begin{lemma}\label{18/01/03/11:27}
The following hold for all $M>0$ and all radial functions $g$ in $H^{1}(\mathbb{R}^{d})$: if $d=3$ and $3\le q \le 5=2^{*}-1$, then 
\begin{equation}\label{18/01/08/14:06}
\int_{\mathbb{R}^{3}}
\frac{M|x|}{(M+|x|)^{2}}|g(x)|^{q+1}\,dx 
\lesssim
\frac{1}{M}
\|g\|_{H^{1}}^{q-1}
\Big\|
\nabla\Big( \frac{M}{M+|x|} g \Big)
\Big\|_{L^{2}}^{2},
\end{equation}
and if $d\ge 4$ and $1+\frac{4}{d-1}\le q \le 2^{*}-1$, then 
\begin{equation}\label{18/01/03/23:08}
\begin{split}
\int_{\mathbb{R}^{d}}
\frac{M|x|}{(M+|x|)^{2}}|g(x)|^{q+1}\,dx 
&\lesssim
M^{-\frac{(d-2)q-d}{4}}
\|\nabla g\|_{L^{2}}^{q-1}
\Big\|
\nabla\Big( \frac{M}{M+|x|} g \Big)
\Big\|_{L^{2}}^{2}
\\[6pt]
&\quad +
M^{-\frac{(d-1)(q-1)-4}{4}}
\|g\|_{H^{1}}^{q-1}
\int_{\mathbb{R}^{d}}
\frac{M}{|x|(M+|x|)^{2}}|g(x)|^{2} \,dx
.
\end{split}
\end{equation}
Here, the implicit constants depend only on $d$ and $q$. 
\end{lemma}
\begin{remark}\label{18/01/08/13:52}
We rely on the radial Sobolev inequalities to prove Lemma \ref{18/01/03/11:27}, which causes the restriction of $p$ in the one-pass theorem (Theorem \ref{13/06/15/20:32}) and Theorem \ref{15/03/22/11:28}.
\end{remark}
\begin{proof}[Proof of Lemma \ref{18/01/03/11:27}]
First, we prove the inequality \eqref{18/01/08/14:06}. We see from \eqref{13/12/26/21:32} and Sobolev's embedding that for any $3\le q \le 5$, 
\begin{equation}\label{18/01/08/15:58}
\begin{split}
\int_{\mathbb{R}^{3}}\frac{M|x|}{(M+|x|)^{2}}|g(x)|^{q+1}\,dx
&\le 
\frac{1}{M}
\sup_{x\in \mathbb{R}^{3}} 
|x| \Big| \frac{M}{M+|x|} g(x) \Big|^{2} 
\int_{\mathbb{R}^{3}}
|g(x)|^{q-1}
\,dx 
\\[6pt]
&\lesssim 
\frac{1}{M}
\|g\|_{H^{1}}^{q-1} \Big\|
\nabla\Big( \frac{M}{M+|x|} g \Big)
\Big\|_{L^{2}}^{2}.
\end{split}
\end{equation}

Next, we prove \eqref{18/01/03/23:08}. We see from \eqref{13/12/26/21:32}, H\"older's inequality and Hardy's inequality that for any $d\ge 4$ and $1+\frac{4}{d}\le q \le 2^{*}-1$ (hence $\frac{d}{d-2} \le q \le 3$), 
\begin{equation}\label{13/12/23/19:03}
\begin{split}
&\int_{|x|\le \sqrt{M}}\frac{M|x|}{(M+|x|)^{2}}|g(x)|^{q+1}\,dx
\\[6pt]
&\le 
\frac{1}{M}
\Big\|
|x|^{\frac{d-2}{2}} \frac{M}{M+|x|} g  
\Big\|_{L^{\infty}}^{2}
\int_{|x|\le \sqrt{M}}
\frac{|g(x)|^{q-1}}{|x|^{d-3}}
\,dx 
\\[6pt]
&\lesssim  
\frac{1}{M}
\Big\|
\nabla \Big( \frac{M}{M+|x|} g \Big) 
\Big\|_{L^{2}}^{2}
\bigg(
\int_{|x|\le \sqrt{M}}
|x|^{-\frac{2(d-2-q)}{3-q}} \,dx 
\bigg)^{\frac{3-q}{2}}
\bigg( \int_{\mathbb{R}^{d}}
\frac{|g|^{2}}{|x|^{2}}\,dx \bigg)^{\frac{q-1}{2}}
\\[6pt]
&\lesssim 
M^{-\frac{(d-2)q-d}{4}}
\|\nabla g\|_{L^{2}}^{q-1}
\Big\|
\nabla\Big( \frac{M}{M+|x|} g \Big)
\Big\|_{L^{2}}^{2}.
\end{split}
\end{equation}
On the other hand, we see from \eqref{14/01/18/14:30} that for any $1+\frac{4}{d-1}\le q \le 2^{*}-1$, 
\begin{equation}\label{18/01/07/20:12}
\begin{split}
&
\int_{\sqrt{M} \le |x|}
\frac{M|x|}{(M+|x|)^{2}} 
|g(x)|^{q+1}\,dx 
\\[6pt]
&\le 
\big\| |x|^{d-1} |g|^{2} \big\|_{L^{\infty}}^{\frac{q-1}{2}} 
\int_{\sqrt{M}\le |x|} 
\frac{|x|^{-\frac{(d-1)(q-1)}{2}+2}M}{|x|(M+|x|)^{2}}|g(x)|^{2} \,dx
\\[6pt]
&\lesssim  
M^{-\frac{(d-1)(q-1)-4}{4}}
\|g\|_{H^{1}}^{q-1}
\int_{\mathbb{R}^{d}}
\frac{M}{|x|(M+|x|)^{2}}|g(x)|^{2} \,dx
.
\end{split}
\end{equation}
Putting \eqref{13/12/23/19:03} and \eqref{18/01/07/20:12} together, we obtain 
 the desired estimate \eqref{18/01/03/23:08}. 
\end{proof}

Next, we recall that: $\omega_{2}$ denotes the frequency given by Proposition \ref{13/02/20/9:41}; $\omega(2^{*}) \in (0,\omega_{2})$ is the frequency satisfying \eqref{18/02/04/16:52}; and for a given $\omega \in (0, \omega(2^{*}))$, $\delta_{X}$, $A_{*}$, $B_{*}$ and $C_{*}$ denote the constants given by the ejection lemma (Lemma \ref{13/05/05/22:31}).   Moreover, for given $\omega \in (0, \omega_{2})$ and $\delta>0$, $\varepsilon_{0}(\delta)$ denotes the constant determined by Lemma \ref{13/05/15/9:32}. Here, we may assume that 
\begin{equation}\label{18/01/17/10:45}
\delta_{X} \ll \mu | (\Phi_{\omega} , f_{2})_{L^{2}}|
.
\end{equation}
Furthermore, we define $\delta_{S}>0$ as a constant satisfying  
\begin{equation}\label{15/07/11/11:18}
\delta_{S}\le \frac{A_{*}\delta_{X}}{2B_{*}C_{*}} \quad \mbox{and} \quad 
\delta_{S} <\delta_{X}. 
\end{equation} 
Note that $\delta_{S}<\delta_{X} <\delta_{0}(\omega) <\delta_{E}(\omega)$ and $\delta_{X}<\widetilde{\gamma}(\omega)$, where $\delta_{E}(\omega)$, $\delta_{0}(\omega)$ and $\widetilde{\gamma}(\omega)$ are the constants given by Proposition \ref{18/01/27/13:12}, Lemma \ref{15/11/08/11:45} and Proposition \ref{18/01/26/21:25}, respectively.  
\par 
In the proof of the one-pass theorem (Theorem \ref{13/06/15/20:32}) below, we use the following sign function:  
\begin{equation}\label{18/01/01/01:01}
{\rm sign}[\alpha] 
=
\left\{ \begin{array}{ccc}
1 &\mbox{if}& \alpha\ge 0,
\\[6pt]
-1 &\mbox{if}& \alpha< 0.
\end{array} \right. 
\end{equation}
In particular, ${\rm sign}[0]=1$. 

\begin{theorem}[One-pass theorem]\label{13/06/15/20:32}
Assume that either $d=3$ and $3\le p <5$, or $d\ge 4$ and $1+\frac{4}{d-1}< p <2^{*}-1$. Then, there exists $\omega_{*}>0$ such that for any $\omega \in (0,\omega_{*})$, there exist positive constants $\varepsilon_{*}$, $\delta_{*}$ and $R_{*}$ with the following properties: 
\begin{equation}\label{18/01/31/18:02}
\sqrt{\varepsilon_{*}} \ll R_{*} \ll \delta_{*} \ll \delta_{S}, 
\quad 
\varepsilon_{*} \le \varepsilon_{0}(\delta_{*}), 
\end{equation} 
and for any $\varepsilon \in (0, \varepsilon_{*}]$, any $R \in (\sqrt{2\varepsilon}, R_{*})$ and any radial solution $\psi$ to \eqref{12/03/23/17:57} satisfying\begin{align}
\label{13/06/15/20:40}
&\mathcal{M}(\psi)=\mathcal{M}(\Phi_{\omega}),
\\[6pt]
\label{13/06/15/20:41}
&\mathcal{S}_{\omega}(\psi)< m_{\omega}+\varepsilon,
\\[6pt]
\label{13/06/15/20:42}
&\widetilde{d}_{\omega}(\psi(0))<R ,
\end{align}
we have either 
\\[6pt]
{\rm (i)} $\widetilde{d}_{\omega}(\psi(t))< R+R^{\frac{\min\{3,p+1\}}{2}}$ 
 for all $t\in [0, T_{\max}]$; or 
\\[6pt]
{\rm (ii)} there exists $t_{*}>0$ such that $\widetilde{d}_{\omega}(\psi(t))\ge R+R^{\frac{\min\{3,p+1\}}{2}}$ for all $t\in [t_{*}, T_{\max}]$. 
\\
Here, $T_{\max}$ denotes the maximal lifespan of $\psi$. 
\end{theorem}
\begin{proof}[Proof of Theorem \ref{13/06/15/20:32}]
We prove the claim by contradiction. Hence, we suppose to the contrary 
 that for any $0< \omega_{*}\ll 1$, there exists $0< \omega <\omega_{*}$ such that for any $\varepsilon_{*}, \delta_{*}, R_{*}>0$ with $\sqrt{\varepsilon_{*}} \ll R_{*}\ll \delta_{*} \ll \delta_{S}$ and $\varepsilon_{*}<\varepsilon_{0}(\delta_{*})$, there exist $\varepsilon\in (0, \varepsilon_{*}]$, $R\in (\sqrt{2 \varepsilon}, R_{*})$ and a radial solution $\psi$ such that for these $\omega$, $\varepsilon$, $R$ and $\psi$, the conditions \eqref{13/06/15/20:40}, \eqref{13/06/15/20:41} and \eqref{13/06/15/20:42} hold, while both {\rm (i)} and {\rm (ii)} fail. Here, we may assume that $\delta_{*}$ is so small that 
\begin{equation}\label{18/01/15/10:11}
\log{\bigg( \frac{\delta_{X}}{\delta_{*}}\bigg)} 
\le 
\frac{1}{M_{0}}\frac{\delta_{X}}{\delta_{*}}
\end{equation}
for some positive constant $M_{0}$ to be specified later dependently only on $d$, $p$ and $\omega$.
\par 
The failure of {\rm (i)} together with \eqref{18/01/31/18:02} and \eqref{13/06/15/20:42} shows that there exist times $t_{2}>t_{1}>0$ such that 
\begin{align}
\label{13/06/15/22:27}
&\widetilde{d}_{\omega}(\psi(t_{1}))=R<\delta_{*},
\\[6pt]
\label{13/06/15/22:29}
&\widetilde{d}_{\omega}(\psi(t_{2}))
= 
R+R^{\frac{\min\{3,p+1\}}{2}}
<\delta_{*}\ll \delta_{S}<\delta_{X},
\\[6pt]
\label{13/12/11/12:26}
R< 
&\widetilde{d}_{\omega}(\psi(t))
<R+R^{\frac{\min\{3,p+1\}}{2}}
<
\delta_{*} \ll \delta_{S} <\delta_{X}
\quad 
\mbox{for all $t\in (t_{1},t_{2})$}.
\end{align}
Moreover, the failure of {\rm (ii)} shows that there exists $t_{3}>t_{2}$ such that 
\begin{align}
\label{13/12/11/16:11}
R< &\widetilde{d}_{\omega}(\psi(t_{3}))
< R+R^{\frac{\min\{3,p+1\}}{2}},
\\[6pt]
\label{13/12/11/16:40}
&\widetilde{d}_{\omega}(\psi(t_{3})) < \widetilde{d}_{\omega}(\psi(t))
\quad \mbox{for all $t \in (t_{2},t_{3})$}.
\end{align}
We see from \eqref{13/06/15/22:27} through \eqref{13/12/11/16:40} that 
\begin{equation}\label{15/04/02/16:47}
R
=
\widetilde{d}_{\omega}(\psi(t_{1})) 
= 
\min_{t\in [t_{1},t_{3}]}\widetilde{d}_{\omega}(\psi(t))
.
\end{equation}

Note here that the ejection lemma (Lemma \ref{13/05/05/22:31}) together with \eqref{13/06/15/22:29} shows that 
\begin{equation}\label{17/12/31/17:41}
t_{1}+T_{*}R^{\min\{1,p-1\}}<t_{2},
\end{equation} 
where $T_{*}$ is the constant given by the ejection lemma. Put 
\begin{equation}\label{15/04/02/17:07}
t_{2}':=\inf\{t\in [t_{1},t_{3}] \colon \widetilde{d}_{\omega}(\psi(t))= \delta_{X}\}.
\end{equation}
Then, \eqref{13/12/11/12:26} shows $t_{2}<t_{2}'$. Moreover, we see from \eqref{13/05/14/21:37} and the ejection lemma that  
\begin{align}
\label{15/03/31/17:13}
& A_{*}e^{\mu (t-t_{1})}R \le d_{\omega}(\psi(t))\le B_{*}e^{\mu (t-t_{1})}R, 
\\[6pt]
\label{18/01/01/11:55}
&
\|\eta(t)\|_{H^{1}}
\sim |\lambda_{1}(t)| 
\sim e^{\mu(t-t_{1})}R 
\sim d_{\omega}(\psi(t)),
\\[6pt]
\label{18/01/01/11:56}
&\|\Gamma(t)\|_{H^{1}} 
\lesssim 
R+ \{ e^{\mu(t-t_{1})}R\}^{\frac{\min\{3,p+1\}}{2}}
\sim \widetilde{d}_{\omega}(\psi(t_{1}))
+|\lambda_{1}(t)|^{\frac{\min\{3,p+1\}}{2}},
\\[6pt]
\label{15/03/31/14:52}
&{\rm sign}[\lambda_{1}(t)] \mathcal{K}(\psi(t))
\gtrsim (e^{\mu (t-t_{1})}-C_{*})R 
\end{align} 
for all $t\in [t_{1},t_{2}']$,  and $\widetilde{d}_{\omega}(\psi(t))$ increases on  $[t_{1}+T_{*}R^{\min\{1,p-1\}}, t_{2}']$. In particular, $\widetilde{d}_{\omega}(\psi(t))$ must reach $\delta_{X}$ before $t_{3}$ and therefore
\begin{equation}\label{17/12/31/18:03}
 t_{2}<t_{2}'<t_{3}.
\end{equation} 
Furthermore, we see from \eqref{13/12/11/12:26} that there exists $t_{S} \in (t_{2},t_{2}')$ such that 
 $\widetilde{d}_{\omega}(\psi(t_{S}))=\delta_{S}$.

Now, note that \eqref{13/06/15/20:41} together with $\sqrt{2\varepsilon} <R$ and \eqref{15/04/02/16:47} shows that 
\begin{equation}\label{15/04/02/20:50}
\mathcal{S}_{\omega}(\psi)<m_{\omega}+\frac{1}{2}\widetilde{d}_{\omega}(\psi(t))^{2}
\quad 
\mbox{for all $t\in [t_{1},t_{3}]$}.
\end{equation}
Hence, we find from Lemma \ref{13/05/05/15:47}, \eqref{13/05/14/21:37} and \eqref{15/04/02/16:47} that  
\begin{equation}\label{15/03/30/18:03}
R\le d_{\omega}(\psi(t)) \sim |\lambda_{1}(t)|
\quad 
\mbox{for all $t \in [t_{1},t_{3}]$ with $\widetilde{d}_{\omega}(\psi(t))\le \delta_{E}(\omega)$}
.
\end{equation}
In particular, ${\rm sign}[\lambda_{1}(t)]$ is constant on the set $\{t \in [t_{1},t_{3}]\colon \widetilde{d}_{\omega}(\psi(t))\le \delta_{E}(\omega)\}$. Moreover, Lemma \ref{13/05/15/9:32} together with \eqref{13/06/15/20:41} and $\varepsilon <\varepsilon_{*}< \varepsilon_{0}(\delta_{*})$ shows that ${\rm sign}[\mathcal{K}(\psi(t))]$ is constant on the set $\{ t\in[t_{1},t_{3}]\colon \widetilde{d}_{\omega}(\psi(t))\ge \delta_{*}\}$. 
\par 
We shall show that 
 for any $t\in [t_{1},t_{3}]$ with $\delta_{*}\le \widetilde{d}_{\omega}(\psi(t)) \le \delta_{E}(\omega)$, 
\begin{equation}\label{15/04/02/21:17}
{\rm sign}[\lambda_{1}(t)]={\rm sign}[\mathcal{K}(\psi(t))]
.
\end{equation} 
Since both ${\rm sign}[\lambda_{1}(t)]$ and ${\rm sign}[\mathcal{K}(\psi(t))]$ are constant on the set $\{ t\in [t_{1},t_{3}]\colon \delta_{*}\le \widetilde{d}_{\omega}(\psi(t)) \le \delta_{E}(\omega)\}$, it suffices to show that ${\rm sign}[\lambda_{1}(t_{2}')]={\rm sign}[\mathcal{K}(\psi(t_{2}'))]$. 
We see from \eqref{15/04/02/17:07}, \eqref{15/03/31/17:13},   \eqref{13/05/14/21:37} and \eqref{15/07/11/11:18} that 
\begin{equation}\label{15/04/02/21:58}
\begin{split}
\delta_{X}
\le 
B_{*} e^{\mu (t_{2}'-t_{1})}R
&\le 
B_{*} e^{\mu (t_{2}'-t_{1})} \frac{e^{-\mu (t_{S}-t_{1})}}{A_{*}}\widetilde{d}_{\omega}(\psi(t_{S}))
\\[6pt]
&=
\frac{B_{*}}{A_{*}} e^{\mu (t_{2}'-t_{S})}
\delta_{S}
\le 
\frac{\delta_{X}}{2C_{*}}e^{\mu (t_{2}'-t_{S})}.
\end{split} 
\end{equation}
Dividing the both sides of \eqref{15/04/02/21:58} by $\delta_{X}$, and taking the logarithm, we obtain  
\begin{equation}\label{15/06/25/09:55}
\frac{1}{\mu}\log{C_{*}} < t_{2}'-t_{S} .
\end{equation} 
Since $t_{1}< t_{S}$, we find from  \eqref{15/03/31/14:52} and \eqref{15/06/25/09:55} that ${\rm sign}[\lambda_{1}(t_{2}')]={\rm sign}[\mathcal{K}(\psi(t_{2}'))]$. Hence, \eqref{15/04/02/21:17} holds. 
\par 
We find from \eqref{15/04/02/21:17} that the following function $\mathfrak{S}\colon [t_{1},t_{3}] \to \{1,-1\}$ is well-defined:   
\begin{equation}\label{15/04/02/20:40}
\mathfrak{S}(t)
:=\left\{ \begin{array}{ccc}
{\rm sign}[\lambda_{1}(t)]
&\mbox{if}& \widetilde{d}_{\omega}(\psi(t))\le \delta_{E}(\omega), 
\\[6pt]
{\rm sign}[\mathcal{K}(\psi(t))] 
&\mbox{if}& \widetilde{d}_{\omega}(\psi(t))\ge \delta_{*}.
\end{array} \right.
\end{equation} 
Note that the function $\mathfrak{S}(t)$ is constant on $[t_{1},t_{3}]$.  Put $\mathfrak{s}:=\mathfrak{S}(t) \in \{1,-1\}$. When $\mathfrak{s}=1$, we see that there exists $C(\omega)>0$ such that for any $t \in [t_{1},t_{3}]$, 
\begin{equation}\label{17/11/26/21:26}
\sup_{t\in [t_{1},t_{3}]}\|\psi(t)\|_{H^{1}}\le C(\omega)
.
\end{equation}
Indeed, if $\widetilde{d}_{\omega}(\psi(\tau))\ge \delta_{S}$ for some $\tau \in [t_{1},t_{3}]$, then \eqref{15/04/02/20:40} shows $\inf_{t \in [t_{1},t_{3}]}\mathcal{K}(\psi(t))\ge 0$. Hence, we see from \eqref{13/03/30/12:14}, \eqref{15/12/31/11:01} and the assumption \eqref{13/06/15/20:41} that for any $t \in [t_{1},t_{3}]$, 
\begin{equation}\label{13/12/10/20:25}
\frac{\omega}{2}\|\psi(t) \|_{L^{2}}^{2}
+
\frac{s_{p}}{d}\|\nabla \psi(t) \|_{L^{2}}^{2}
\le 
\mathcal{I}_{\omega}(\psi(t))
\le 
\mathcal{S}_{\omega}(\psi) 
\le m_{\omega}+1.
\end{equation}
On the other hand, if $\widetilde{d}_{\omega}(\psi(t))\le \delta_{S}$ for all $t \in [t_{1},t_{3}]$, then we see from Lemma \ref{13/05/11/11:52}, Lemma \ref{13/05/05/15:47} and \eqref{13/05/14/21:37} that for any $t \in [t_{1},t_{3}]$,
\begin{equation}\label{13/12/10/11:47}
\| \psi(t) \|_{H^{1}}^{2}
\lesssim 
\|\Phi_{\omega}\|_{H^{1}}^{2}
+
\| \eta(t) \|_{H^{1}}^{2} 
\lesssim 
\|\Phi_{\omega}\|_{H^{1}}^{2}
+
d_{\omega}(\psi(t))^{2}
\lesssim 
\|\Phi_{\omega}\|_{H^{1}}^{2}+\delta_{S}.
\end{equation}
Thus, we have verified that \eqref{17/11/26/21:26} holds. 
\par 
Let us summarize information obtained so far: there exist $t_{1}'(=t_{1}+T_{*}R^{\min\{1,p-1\}}) \in (t_{1},t_{2})$ and $t_{2}' \in (t_{2},t_{3})$ such that: $\widetilde{d}_{\omega}(\psi(t))$ increases on $[t_{1}',t_{2}']$;  $\widetilde{d}_{\omega}(\psi(t_{2}'))=\delta_{X}$; and for any $t\in (t_{1}, t_{2}')$, 
\begin{align}
\label{13/12/11/17:06}
&
\delta_{X}\ge 
\widetilde{d}_{\omega}(\psi(t)) 
\sim \|\eta(t)\|_{H^{1}}
\sim |\lambda_{1}(t)| 
\sim e^{\mu(t-t_{1})}R,
\\[6pt]
\label{14/01/18/16:34}
&\|\Gamma(t)\|_{H^{1}} 
\lesssim 
\widetilde{d}_{\omega}(\psi(t_{1}))
+|\lambda_{1}(t)|^{\frac{\min\{3,p+1\}}{2}},
\\[6pt]
\label{13/12/11/20:08}
&\mathfrak{s}\mathcal{K}(\psi(t))\gtrsim (e^{\mu(t-t_{1})}-C_{*})R.
\end{align}
Arguing from $t_{3}$ backward in time, we are also able to obtain a time interval $(t_{2}'',t_{3})\subset (t_{2}',t_{3})$ such that: $\widetilde{d}_{\omega}(\psi(t_{2}''))=\delta_{X}$; for any $t\in (t_{2}'', t_{3})$, 
\begin{align}
\label{13/12/11/20:12}
&
\delta_{X}\ge 
\widetilde{d}_{\omega}(\psi(t)) 
\sim \|\eta(t)\|_{H^{1}}
\sim |\lambda_{1}(t)|  
\sim e^{\mu(t_{3}-t)}\widetilde{d}_{\omega}(\psi(t_{3})),
\\[6pt]
\label{14/01/18/16:38}
&\|\Gamma(t)\|_{H^{1}} 
\lesssim 
\widetilde{d}_{\omega}(\psi(t_{3}))
+|\lambda_{1}(t)|^{\frac{\min\{3,p+1\}}{2}},
\\[6pt]
\label{13/12/11/20:13}
&\mathfrak{s}\mathcal{K}(\psi(t))
\gtrsim 
(e^{\mu(t_{3}-t)}-C_{*})\widetilde{d}_{\omega}(\psi(t_{3})),
\end{align}
and $\widetilde{d}_{\omega}(\psi(t))$ decreases at least in the region $\{t\in [t_{2}'',t_{3}] \colon \widetilde{d}_{\omega}(\psi(t))\ge 2\widetilde{d}_{\omega}(\psi(t_{3}))\}$ (cf. \eqref{13/05/06/14:37} in the ejection lemma (Lemma \ref{13/05/05/22:31})).
\par 
Suppose here that there exists a time $\tau \in (t_{2}',t_{2}'')$ such that $\widetilde{d}_{\omega}(\psi(\tau))$ is a local minimum and $\widetilde{d}_{\omega}(\psi(\tau))<\delta_{*}$. Then, we can apply the ejection lemma (Lemma \ref{13/05/05/22:31}) from $\tau$ both forward and backward in time to obtain an open interval $I_{\tau} \subset (t_{2}',t_{2}'')$ such that 
\begin{equation}\label{14/01/18/11:50}
\widetilde{d}_{\omega}(\psi(\inf{I_{\tau}}))
=
\widetilde{d}_{\omega}(\psi(\sup{I_{\tau}}))=\delta_{X},
\end{equation}
and for any $t\in I_{\tau}$, 
\begin{align}
\label{13/12/11/20:32}
&
\delta_{X}\ge 
\widetilde{d}_{\omega}(\psi(t)) 
\sim \|\eta(t)\|_{H^{1}}
\sim |\lambda_{1}(t)| 
\sim e^{\mu|t-\tau|}\widetilde{d}_{\omega}(\psi(\tau)),
\\[6pt]
\label{14/01/18/16:41}
&
\|\Gamma(t)\|_{H^{1}} 
\lesssim 
\widetilde{d}_{\omega}(\psi(\tau))
+|\lambda_{1}(t)|^{\frac{\min\{3,p+1\}}{2}},
\\[6pt]
\label{13/12/11/20:33}
&\mathfrak{s}\mathcal{K}(\psi(t))
\gtrsim 
(e^{\mu|t-\tau|}-C_{*})\widetilde{d}_{\omega}(\psi(\tau)),
\end{align}
and $\widetilde{d}_{\omega}(\psi(t))$ is monotone in the region 
$\widetilde{d}_{\omega}(\psi(t))\ge 2\widetilde{d}_{\omega}(\psi(\tau))$. Note that for any distinct local minimum points $\tau_{1}$ and $\tau_{2}$ of $\widetilde{d}_{\omega}(\psi(t))$ in $(t_{2}',t_{2}'')$, the monotonicity away from them implies that the intervals $I_{\tau_{1}}$ and $I_{\tau_{2}}$ are either disjoint or identical. 

\begin{figure}[H]
\hspace{-4cm}
\input{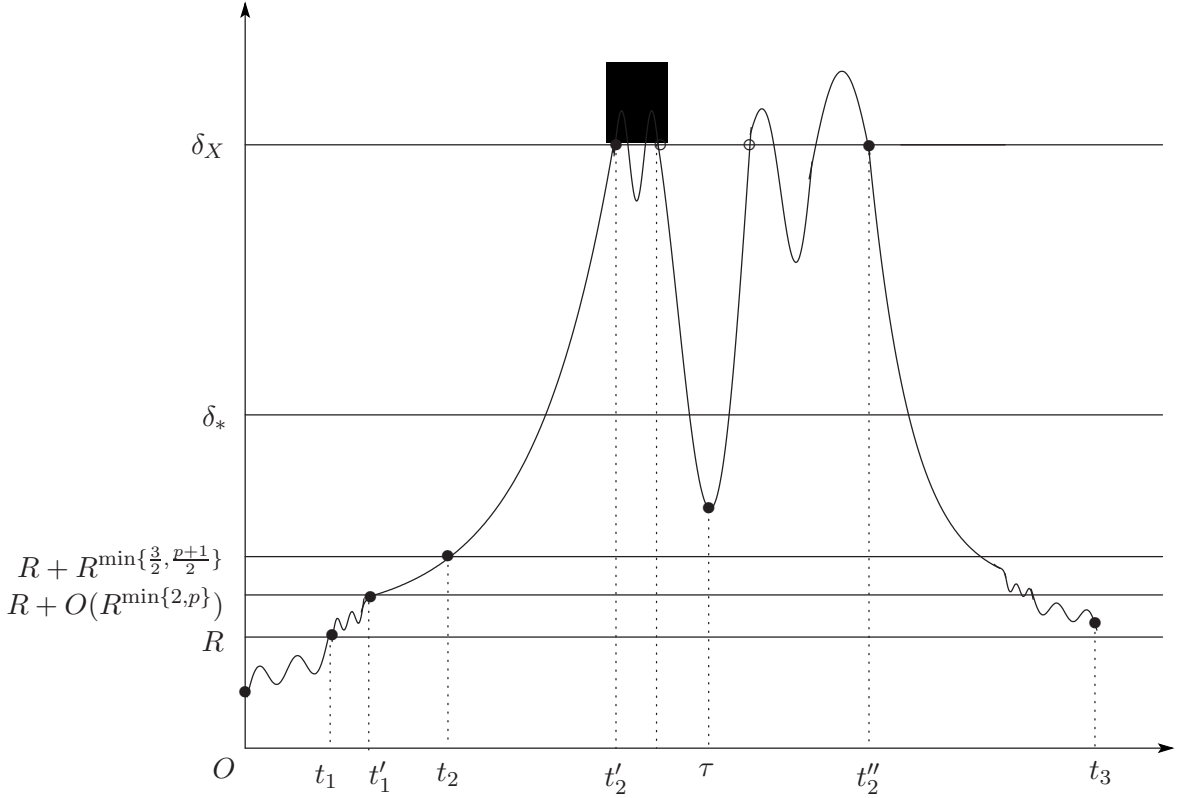}
\caption{Behavior of $\widetilde{d}_{\omega}(\psi(t))$ in the case where there exists a local minimum point $\tau \in (t_{2}',t_{2}'')$ such that $\widetilde{d}_{\omega}(\psi(\tau))<\delta_{*}$.}
\end{figure}
\par 
We find from the above observation that there exist a number $n\ge 2$ and disjoint open subintervals $I_{1},\ldots, I_{n}$ of $[t_{1},t_{3}]$ with the following properties\footnote{Note that $I_{1}=(t_{1},t_{2}')$. Moreover, if $\widetilde{d}_{\omega}(\psi(t))\ge \delta_{*}$ on $[t_{2}',t_{2}'']$, then $n=2$ and $I_{2}=(t_{2}'',t_{3})$.}: 
\\
\noindent 
\textbullet~$\inf{I_{1}}=t_{1}$, $\sup{I_{n}}=t_{3}$, $\widetilde{d}_{\omega}(\psi(\sup{I_{1}}))=\widetilde{d}_{\omega}(\psi(\inf{I_{n}}))=\delta_{X}$, and 
\begin{equation}\label{14/01/18/12:28}
\widetilde{d}_{\omega}(\psi(\inf{I_{j}}))
=\widetilde{d}_{\omega}(\psi(\sup{I_{j}}))=\delta_{X}
\quad \mbox{except for $j=1,n$}.
\end{equation} 
\noindent 
\textbullet~For each $1\le j \le n$, there exists $\tau_{j}\in I_{j}$ such that \begin{equation}\label{18/01/15/09:26}
\widetilde{d}_{\omega}(\psi(\tau_{j}))<\delta_{*}, 
\end{equation}  
and for any $t\in I_{j}$,
\begin{align}
\label{14/01/18/11:35}
&\delta_{X} \ge \widetilde{d}_{\omega}(\psi(t)) 
\sim \|\eta(t)\|_{H^{1}}
\sim |\lambda_{1}(t)| 
\sim e^{\mu|t-\tau_{j}|}\widetilde{d}_{\omega}(\psi(\tau_{j})),
\\[6pt]
\label{14/01/18/16:43}
&\|\Gamma(t)\|_{H^{1}} 
\lesssim 
\widetilde{d}_{\omega}(\psi(\tau_{j}))
+
|\lambda_{1}(t)|^{\frac{\min\{3,p+1\}}{2}},
\\[6pt]
\label{14/01/18/11:36}
& \mathfrak{s}\mathcal{K}(\psi(t))\gtrsim (e^{\mu|t-\tau_{j}|}-C_{*})\widetilde{d}_{\omega}(\psi(\tau_{j})).
\end{align}
\noindent
\textbullet~If $t$ satisfies $\sup{I_{j-1}}<t<\inf{I_{j}}$ for some $2\le j \le n$, then 
\begin{equation}\label{18/01/22/10:37}
\widetilde{d}_{\omega}(\psi(t)) >\delta_{*}.
\end{equation}

Put  $I':=[t_{1},t_{3}]\setminus \bigcup_{j=1}^{n}I_{j}$. Then, for any $t\in I'$, we have $\widetilde{d}_{\omega}(\psi(t))>\delta_{*}$ (see \eqref{18/01/22/10:37}), which together with Lemma \ref{13/05/15/9:32} gives us that for any $t\in I'$, 
\begin{equation}\label{13/12/11/20:51}
\left\{ \begin{array}{lcc}
\displaystyle{
\mathcal{K}(\psi(t))\ge 
\min\big\{ \kappa_{1}(\delta_{*}),\, \frac{1}{2}\|\nabla \psi(t)\|_{L^{2}}^{2} \big\} }
&\mbox{if}& \mathfrak{s}=1,
\\[6pt]
\mathcal{K}(\psi(t))\le -\kappa_{1}(\delta_{*}) &\mbox{if}& \mathfrak{s}=-1. 
\end{array} \right.
\end{equation}
We can find from elementary calculations, \eqref{14/01/18/11:35}, $\widetilde{d}_{\omega}(\psi(\tau_{j})) \le \delta_{*}$ and \eqref{18/01/15/10:11} that for any $1\le j \le n$, 
\begin{equation}\label{13/12/18/05:36}
\begin{split}
\mu |I_{j}| \widetilde{d}_{\omega}(\psi(\tau_{j})) 
&=
\big\{ \log{e^{\mu|\sup{I_{j}-\tau_{j}|}}}+ \log{e^{\mu|\inf{I_{j}-\tau_{j}|}}}
\big\}
\widetilde{d}_{\omega}(\psi(\tau_{j}))
\\[6pt]
&\lesssim 
\log{\Bigm( \frac{\delta_{X}}{\widetilde{d}_{\omega}(\psi(\tau_{j}))}\Bigm)}
\widetilde{d}_{\omega}(\psi(\tau_{j})) 
\le \frac{1}{M_{0}}\delta_{X}.
\end{split}
\end{equation}

We shall derive a contradiction in the case where $\mathfrak{s}=1$. To this end, we use the following identity: for any $t\in I_{\max}(\psi)$ and any $M>0$, 
\begin{equation}\label{13/12/16/15:45}
\begin{split}
&\frac{d}{dt}\Im \int_{\mathbb{R}^{d}} 
\frac{M|x|}{M+|x|}
\frac{x}{|x|} \cdot \nabla \psi(x,t)\overline{\psi(x,t)}\,dx
\\[6pt]
&=2 \mathcal{K}\bigg(\frac{M}{M+|x|} \psi(t) \bigg)
\\[6pt]
&\quad 
-\frac{p-1}{p+1}
\int_{\mathbb{R}^{d}}
\bigg\{ \frac{M^{2}}{(M+|x|)^{2}}
+\frac{(d-1)M}{M+|x|}
-\frac{dM^{p+1}}{(M+|x|)^{p+1}} 
\bigg\}
|\psi(t)|^{p+1}
\,dx 
\\[6pt]
&\quad -
\frac{2}{d}
\int_{\mathbb{R}^{d}} 
\bigg\{ \frac{M^{2}}{(M+|x|)^{2}}
+\frac{(d-1)M}{M+|x|}
-\frac{dM^{2^{*}}}{(M+|x|)^{2^{*}}}
\bigg\}
 |\psi(t)|^{2^{*}}
\,dx
\\[6pt]
&\quad 
+
\frac{1}{2}
\int_{\mathbb{R}^{d}} 
\bigg\{ 
\frac{2M^{2}}{(M+|x|)^{4}}
+
\frac{(d-3)(d-1)M}{|x| (M+|x|)^{2}}
\bigg\}  |\psi(t)|^{2}
\,dx 
.
\end{split}
\end{equation}
Here, note that an elementary computation shows that for any $d\ge 3$ and 
 any $1 \le q \le 2^{*}-1$,  
\begin{equation}\label{18/01/02/08:30}
\begin{split}
&
\int_{\mathbb{R}^{d}}
\bigg\{ \frac{M^{2}}{(M+|x|)^{2}}
+
\frac{(d-1)M}{M+|x|}
-
\frac{dM^{q+1}}{(M+|x|)^{q+1}} 
\bigg\}
|\psi(t)|^{q+1}
\,dx
\\[6pt]
&=  
\int_{\mathbb{R}^{d}}
\bigg\{ 
\frac{d M }{M+|x|}
\bigg( \frac{M}{M+|x|}-\frac{M^{q}}{(M+|x|)^{q}} \bigg)
+
\frac{(d-1) M|x| }{(M+|x|)^{2}}
\bigg\}
|\psi(t)|^{q+1}
\,dx
\\[6pt]
&\lesssim  
\int_{\mathbb{R}^{d}}
\frac{ M|x| }{(M+|x|)^{2}} |\psi(t)|^{q+1}\,dx 
.
\end{split}
\end{equation}

We shall derive a contradiction in three steps: 
\\
\noindent 
{\bf Step 1.}~We shall derive an estimate for \eqref{13/12/16/15:45} on the region $\bigcup_{j=1}^{n}I_{j}$. 
\par 
We see from Sobolev's embedding that for any $d\ge 3$, any $1+\frac{4}{d} \le q \le 2^{*}-1$, any $M>0$ and any $t\in I_{\max}(\psi)$, 
\begin{equation}\label{18/01/19/17:27}
\begin{split}
\int_{\mathbb{R}^{d}}
\frac{M|x|}{(M+|x|)^{2}} 
|\psi(t)|^{q+1}\,dx 
&\lesssim  
\frac{1}{M} 
\int_{\mathbb{R}^{d}} 
|x| |\Phi_{\omega}(x)|^{q+1} \,dx
+
\|\eta(t)\|_{L^{q+1}}^{q+1} 
\\[6pt]
&\lesssim  
\frac{C(\omega)}{M}
+
\|\eta(t)\|_{H^{1}}^{q+1}
,
\end{split}
\end{equation}
where $C(\omega)$ is some positive constant depending only on $d$, $p$ and $\omega$. Furthermore, it follows from \eqref{13/12/16/15:45}, \eqref{18/01/02/08:30} and \eqref{18/01/19/17:27} that for any $d\ge 3$, any $1+\frac{4}{d}\le p<2^{*}-1$, any $M>0$ and any $t\in I_{\max}(\psi)$, then 
\begin{equation}\label{18/01/19/15:48}
\begin{split}
&\frac{d}{dt}\Im \int_{\mathbb{R}^{d}}
\frac{M|x|}{M+|x|}
\frac{x}{|x|} \cdot \nabla \psi(x,t)\overline{\psi(x,t)}\,dx
\\[6pt]
&\ge
2 \mathcal{K}\Big(\frac{M}{M+|x|} \psi(t) \Big)
-
\frac{C(\omega)}{M}
-
C \big\{ \|\eta(t)\|_{H^{1}}^{p+1}
+
\|\eta(t)\|_{H^{1}}^{2^{*}}
\big\}
\end{split}
\end{equation}
where $C(\omega)$ is some positive constant depending only on $d$, $p$ and $\omega$, and $C$ is some positive constant depending only on $d$ and $p$.
\par 
We consider the first two terms on the right-hand side of \eqref{18/01/19/15:48}. 
 Note first that it follows from \eqref{15/02/15/16:08} that for any $u \in H^{1}(\mathbb{R}^{d})$,
\begin{equation}\label{18/01/08/20:47}
\begin{split}
\mathcal{K}'(\Phi_{\omega}) u 
&=
-s_{p}
\langle 
(p-1)\Phi_{\omega}^{p} , u
\rangle_{H^{-1},H^{1}}
-
(2^{*}-2) \langle 
\Phi_{\omega}^{2^{*}-1} , u
\rangle_{H^{-1},H^{1}}
\\[6pt]
&\quad -2\omega 
\langle 
\Phi_{\omega} , u
\rangle_{H^{-1},H^{1}}
.
\end{split} 
\end{equation}
Furthermore, the first order Taylor's expansion of $\mathcal{K}$ around $\Phi_{\omega}$ together with \eqref{18/01/08/20:47}  shows that 
\begin{equation}\label{18/01/08/20:08}
\begin{split}
&\mathcal{K}\Big(\frac{M}{M+|x|} \psi(t) \Big)
=
\mathcal{K}\Big(\Phi_{\omega}+ \eta(t)-\frac{|x|}{M+|x|}
\{ \Phi_{\omega}+\eta(t)\} \Big)
\\[6pt]
&=
-s_{p}
\langle 
(p-1)\Phi_{\omega}^{p} , \eta(t)
\rangle_{H^{-1},H^{1}}
+s_{p}
\langle (p-1)
\Phi_{\omega}^{p}, \frac{|x|}{M+|x|}\big\{ 
\Phi_{\omega}+\eta(t) \big\}
\rangle_{H^{-1},H^{1}}
\\[6pt]
&\quad 
-
(2^{*}-2)
\langle 
 \Phi_{\omega}^{2^{*}-1}, \eta(t) 
\rangle_{H^{-1},H^{1}}
+
(2^{*}-2)
\langle 
 \Phi_{\omega}^{2^{*}-1}, 
\frac{|x|}{M+|x|}
\big\{ \Phi_{\omega}+\eta(t) \big\}
\rangle_{H^{-1},H^{1}}
\\[6pt]
&\quad 
-2\omega 
\langle 
\Phi_{\omega}, \eta(t) 
\rangle_{H^{-1},H^{1}}
+
2\omega 
\langle 
\Phi_{\omega}, 
\frac{|x|}{M+|x|}\big\{ \Phi_{\omega}+\eta(t)\big\}
\rangle_{H^{-1},H^{1}}
\\[6pt]
&\quad +
O\bigm(\Big\| \frac{M}{M+|x|}\eta(t)-\frac{|x|}{M+|x|}\Phi_{\omega} 
\Big\|_{H^{1}}^{2}\bigm)
.
\end{split}
\end{equation}

Here, we see from the decomposition \eqref{13/12/28/10:44}, \eqref{13/12/28/12:07}, \eqref{13/01/02/16:25}, \eqref{13/12/28/17:22} and $\mathfrak{s}=1$ (hence $\lambda_{1}(t)\ge 0$ by \eqref{15/04/02/20:40}) that the first and third terms on the right-hand side of \eqref{18/01/08/20:08} are rewritten as follows:
\begin{equation}\label{18/01/08/20:51}
\begin{split}
&
-s_{p}
\langle 
(p-1)\Phi_{\omega}^{p} , \eta(t)
\rangle_{H^{-1},H^{1}}
-
(2^{*}-2) \langle 
\Phi_{\omega}^{2^{*}-1} , \eta(t)
\rangle_{H^{-1},H^{1}}
\\[6pt]
&=
-s_{p}
\langle 
(p-1)\Phi_{\omega}^{p} , 2 \lambda_{1}(t)f_{1} 
\rangle_{H^{-1},H^{1}}
-s_{p}(p-1)
\langle 
\Phi_{\omega}^{p} , \Gamma(t) 
\rangle_{H^{-1},H^{1}}
\\[6pt]
&\quad -
(2^{*}-2) \langle 
\Phi_{\omega}^{2^{*}-1} , 2\lambda_{1}(t)f_{1}
\rangle_{H^{-1},H^{1}}
-
(2^{*}-2) \langle 
\Phi_{\omega}^{2^{*}-1} , \Gamma(t)
\rangle_{H^{-1},H^{1}}
\\[6pt]
&=
2 s_{p} \lambda_{1}(t)
\langle 
L_{\omega,+}\Phi_{\omega} , f_{1}
\rangle_{H^{-1},H^{1}}
-s_{p}(p-1)
\langle 
\Phi_{\omega}^{p} , \Gamma(t) 
\rangle_{H^{-1},H^{1}}
\\[6pt]
&\quad -
2(1-s_{p})(2^{*}-2) \lambda_{1}(t)
\langle 
\Phi_{\omega}^{2^{*}-1}, f_{1}
\rangle_{H^{-1},H^{1}}
-
(2^{*}-2) \langle 
\Phi_{\omega}^{2^{*}-1} , \Gamma(t)
\rangle_{H^{-1},H^{1}}
\\[6pt]
&=
2 s_{p} \mu | \lambda_{1}(t) |
| ( \Phi_{\omega} , f_{2})_{L^{2}} |
-s_{p}(p-1)
\langle 
\Phi_{\omega}^{p} , \Gamma(t) 
\rangle_{H^{-1},H^{1}}
\\[6pt]
&\quad -
2(1-s_{p})(2^{*}-2) | \lambda_{1}(t)|
\langle 
\Phi_{\omega}^{2^{*}-1}, f_{1}
\rangle_{H^{-1},H^{1}}
-
(2^{*}-2) \langle 
\Phi_{\omega}^{2^{*}-1} , \Gamma(t)
\rangle_{H^{-1},H^{1}}
.
\end{split} 
\end{equation}
We also see from the decomposition \eqref{13/12/28/10:44}, Lemma \ref{13/03/03/14:50} and \eqref{13/12/29/12:12} that 
\begin{equation}\label{18/01/08/21:32}
-2\omega 
\langle 
\Phi_{\omega}, \eta(t) 
\rangle_{H^{-1},H^{1}}
=
-2\omega 
\langle 
\Phi_{\omega}, \Gamma(t)  
\rangle_{H^{-1},H^{1}}
=
\omega \|\eta (t)\|_{L^{2}}^{2}
.
\end{equation}
Moreover, it follows from H\"older's inequality, Sobolev's embedding and \eqref{13/03/03/10:05} that for  any $p \le q \le 2^{*}-1$ and any sufficiently small $\omega >0$, 
\begin{equation}\label{18/01/09/08:58}
\big|
(\Phi_{\omega}^{q} , \Gamma(t) 
)_{L^{2}}
\big|
\le 
\|\Phi_{\omega}\|_{L^{q+1}}^{q}\|\Gamma(t)\|_{L^{q+1}}
\lesssim  
\omega^{\frac{q}{q+1}(\frac{q-1}{p-1}-s_{p})}
\|\Gamma(t)\|_{H^{1}}
\le 
\omega^{\frac{1-s_{p}}{2}}
\|\Gamma(t)\|_{H^{1}}, 
\end{equation}
where the implicit constant depends only on $d$ and $q$; it follows from H\"older's inequality and Sobolev's embedding that for any $1\le q \le 2^{*}-1$ and any sufficiently small $\omega >0$, 
\begin{equation}\label{18/01/10/11:11}
\begin{split}
&\langle 
\Phi_{\omega}^{q}, \frac{|x|}{M+|x|}\big\{ 
\Phi_{\omega}+\eta(t) \big\}
\rangle_{H^{-1},H^{1}}
\ge 
-
\int_{\mathbb{R}^{d}}
\frac{|x|}{M+|x|} \Phi_{\omega}^{q}  |\eta(t)| 
\,dx 
\\[6pt]
&\ge 
-
\frac{1}{M}
\bigg(
\int_{\mathbb{R}^{d}}
|x|^{\frac{q+1}{q}} \Phi_{\omega}^{q+1}
\,dx 
\bigg)^{\frac{q}{q+1}}
\|\eta(t)\|_{L^{q+1}} 
\ge 
-
\frac{1}{M} G(\omega) \|\eta(t)\|_{H^{1}},
\end{split}
\end{equation}
where $G(\omega)$ is some positive constant depending only on $d$, $q$ and $\omega$; and it follows from Sobolev's embedding that 
\begin{equation}\label{18/01/11/13:21}
\Big\| \frac{M}{M+|x|}\eta(t)-\frac{|x|}{M+|x|}\Phi_{\omega}
\Big\|_{H^{1}}^{2}
\lesssim 
\|\eta(t)\|_{H^{1}}^{2}
+
\frac{1}{M}G(\omega), 
\end{equation} 
where $G(\omega)$ is some positive constant depending only on $d$, $q$ and $\omega$.

Plugging \eqref{18/01/08/20:51} and \eqref{18/01/08/21:32} into \eqref{18/01/08/20:08}, and then using Lemma \ref{13/12/19/09:42} with $C=s_{p}-2(1-s_{p})(2^{*}-2)$, \eqref{18/01/09/08:58}, \eqref{18/01/10/11:11} and \eqref{18/01/11/13:21}, we find that if $\omega_{*}$ is sufficiently small (hence so is $\omega$), then 
\begin{equation}\label{18/01/08/21:45}
\begin{split}
&\mathcal{K}\Big(\frac{M}{M+|x|} \psi(t) \Big)
-
\frac{C(\omega)}{M}
\\[6pt]
&=
2 s_{p} \mu |\lambda_{1}(t)|
| (\Phi_{\omega} , f_{2})_{L^{2}}|
-s_{p}(p-1)
\langle 
\Phi_{\omega}^{p} , \Gamma(t) 
\rangle_{H^{-1},H^{1}}
\\[6pt]
&\quad -
2(1-s_{p})(2^{*}-2) | \lambda_{1}(t)|
( 
\Phi_{\omega}^{2^{*}-1}, f_{1}
)_{L^{2}}
-
(2^{*}-2) \langle 
\Phi_{\omega}^{2^{*}-1} , \Gamma(t)
\rangle_{H^{-1},H^{1}}
\\[6pt]
&\quad 
+s_{p}
\langle (p-1)
\Phi_{\omega}^{p}, \frac{|x|}{M+|x|}\big\{ 
\Phi_{\omega}+\eta(t) \big\}
\rangle_{H^{-1},H^{1}}
\\[6pt]
&\quad 
+
(2^{*}-2)
\langle 
 \Phi_{\omega}^{2^{*}-1}, 
\frac{|x|}{M+|x|}
\big\{ \Phi_{\omega}+\eta(t) \big\}
\rangle_{H^{-1},H^{1}}
\\[6pt]
&\quad 
+
\omega 
\|\eta(t)\|_{L^{2}}^{2} 
+
2\omega 
\langle 
\Phi_{\omega}, 
\frac{|x|}{M+|x|}\big\{ \Phi_{\omega}+\eta(t)\big\}
\rangle_{H^{-1},H^{1}}
\\[6pt]
&\quad +
O\bigm(\Big\| \frac{M}{M+|x|}\eta(t)-\frac{|x|}{M+|x|}\Phi_{\omega}
\Big\|_{H^{1}}^{2}\bigm)
-
\frac{C(\omega)}{M}
\\[6pt]
&\ge 
s_{p} \mu |\lambda_{1}(t)|
| (\Phi_{\omega} , f_{2})_{L^{2}}|
-
C_{1} \omega^{\frac{1-s_{p}}{2}} 
\| \Gamma(t) \|_{H^{1}}
- 
\frac{1}{M} G_{1}(\omega)
\|\eta(t)\|_{H^{1}} 
\\[6pt]
&\quad 
-C_{2}
\| \eta(t) \|_{H^{1}}^{2}
-\frac{1}{M}G_{2}(\omega)
,
\end{split}
\end{equation}
where $C_{1}$ and $C_{2}$ are some positive constants depending only on $d$ and $p$, and $G_{1}(\omega)$ and $G_{2}(\omega)$ are some positive constants depending only on $d$, $q$ and $\omega$. Recall here that \eqref{14/01/18/11:35}, \eqref{14/01/18/16:43} and \eqref{18/01/17/10:45} give us that for any $t \in I_{j}$,\begin{align}
\label{18/01/11/09:25}
&\| \eta(t) \|_{H^{1}} \sim |\lambda_{1}(t)| \sim e^{\mu|t-\tau_{j}|}\widetilde{d}_{\omega}(\psi(\tau_{j})),
\\[6pt]
\label{18/01/11/09:35}
&
\| \eta(t) \|_{H^{1}}^{2} 
\lesssim 
\delta_{X} |\lambda_{1}(t)| 
\ll 
\mu | (\Phi_{\omega} , f_{2})_{L^{2}}|
|\lambda_{1}(t)|
,
\\[6pt]
\label{18/01/11/09:29}
&\|\Gamma(t)\|_{H^{1}} 
\lesssim 
\widetilde{d}_{\omega}(\psi(\tau_{j}))
.
\end{align}
Moreover, it follows from \eqref{15/04/02/16:47} that if $M\ge G_{2}(\omega)/R$, then 
\begin{equation}\label{18/01/17/09:07}
\frac{G_{2}(\omega)}{M} 
\le  
R
\le
\widetilde{d}_{\omega}(\psi(\tau_{j}))
.
\end{equation}
Hence, \eqref{18/01/08/21:45} together with \eqref{18/01/11/09:25} through 
 \eqref{18/01/17/09:07} shows that for any sufficiently small $\omega>0$, 
 any $M\ge G_{2}(\omega)/R$ and any $t \in I_{j}$, 
\begin{equation}\label{14/01/18/17:13}
\mathcal{K}\Big(\frac{M}{M+|x|} \psi(t) \Big) 
-
\frac{C(\omega)}{M}
\ge 
2 C_{0}^{*}(\omega)   
\big\{ e^{\mu |t-\tau_{j}|}-C_{1}^{*}(\omega)\big\} \widetilde{d}_{\omega}(\psi(\tau_{j})),
\end{equation}
where $C_{0}^{*}(\omega)$ and $C_{1}^{*}(\omega)$ are some positive constant depending only on $d$, $p$ and $\omega$. 
\par  
We find from \eqref{18/01/19/15:48}, \eqref{14/01/18/17:13}, $\|\eta(t)\|_{H^{1}}\lesssim \delta_{X}$ (see \eqref{14/01/18/11:35}) and \eqref{18/01/11/09:25} that if we take $\delta_{X}$ sufficiently small dependently only on $d$, $p$ and $\omega$ in advance, then for any $M\ge G_{2}(\omega)/R$ and any $t \in I_{j}$, 
\begin{equation}\label{18/01/20/15:26}
\begin{split}
&\frac{d}{dt}\Im \int_{\mathbb{R}^{d}}
\frac{M|x|}{M+|x|}
\frac{x}{|x|} \cdot \nabla \psi(x,t)\overline{\psi(x,t)}\,dx
\\[6pt]
&\ge
2 C_{0}^{*}(\omega)   
\big\{ e^{\mu |t-\tau_{j}|}-C_{1}^{*}(\omega)\big\} \widetilde{d}_{\omega}(\psi(\tau_{j}))
-
C \delta_{X}^{p-1} e^{\mu |t-\tau_{j}|}
\widetilde{d}_{\omega}(\psi(\tau_{j}))
\\[6pt]
&\ge
C_{0}^{*}(\omega)   
\big\{ e^{\mu |t-\tau_{j}|}-C_{1}^{*}(\omega)\big\} \widetilde{d}_{\omega}(\psi(\tau_{j}))
.
\end{split}
\end{equation}

\noindent 
{\bf Step 2}.~Next, we consider the ``variational region'' $I'$. We have $\widetilde{d}_{\omega}(\psi)\ge \delta_{*}$ on $I'$. Moreover, it follows from \eqref{13/12/16/15:45}, \eqref{18/01/02/08:30}, Lemma \ref{18/01/03/11:27} and  \eqref{17/11/26/21:26} that if $d\ge 3$ and $3\le p <5$, or $d\ge 4$ and $1+\frac{4}{d-1}< p<2^{*}-1$, then for any sufficiently large $M>0$ depending only on $d$, $p$ and $\omega$, we have  
\begin{equation}\label{13/12/27/17:05}
\begin{split}
&\frac{d}{dt}\Im \int_{\mathbb{R}^{d}}
\frac{M|x|}{M+|x|}
\frac{x}{|x|} \cdot \nabla \psi(x,t)\overline{\psi(x,t)}\,dx
\\[6pt]
&\ge
2 \mathcal{K}\Big(\frac{M}{M+|x|} \psi(t) \Big)
-o_{M}(1) \Big\|\nabla \Big( \frac{M}{M+|x|}\psi(t) \Big) \Big\|_{L^{2}}^{2}.
\end{split}
\end{equation}
Suppose here that $\|\nabla \psi(t)\|_{L^{2}}\le \alpha$ for some $\alpha>0$. Then, it follows from H\"older's inequality, $\frac{d(p-1)}{2}>2$, Sobolev's embedding and the assumption \eqref{13/06/15/20:40} that 
\begin{equation}\label{14/01/18/17:32}
\begin{split}
\Big\|  \frac{M}{M+|x|} \psi(t) \Big\|^{p+1}_{L^{p+1}}
&\le 
\Big\|  \frac{M}{M+|x|} \psi(t) \Big\|_{L^{2}}^{p+1-\frac{d(p-1)}{2}}
\Big\|  \frac{M}{M+|x|} \psi(t) \Big\|_{L^{2^{*}}}^{\frac{d(p-1)}{2}}
\\[6pt]
&\le 
\|\psi(t)\|_{L^{2}}^{p+1-\frac{d(p-1)}{2}}
\|\psi(t)\|_{L^{2^{*}}}^{\frac{d(p-1)}{2}-2}
\Big\|  \frac{M}{M+|x|} \psi(t) \Big\|_{L^{2^{*}}}^{2}
\\[6pt]
&\lesssim 
\mathcal{M}(\Phi_{\omega})^{\frac{p+1}{2}-\frac{d(p-1)}{4}} 
\alpha^{\frac{d(p-1)}{2}-2}
\Big\| \nabla \Big( \frac{M}{M+|x|} \psi(t) \Big) \Big\|_{L^{2}}^{2}
.
\end{split}
\end{equation} 
Similarly, we have 
\begin{equation}\label{14/01/18/17:43}
\Big\|  \frac{M}{M+|x|} \psi(t) \Big\|_{L^{2^{*}}}^{2^{*}}
\lesssim 
\alpha^{2^{*}-2} 
\Big\|  \nabla \Big( \frac{M}{M+|x|} \psi(t) \Big)\Big\|_{L^{2}}^{2}.
\end{equation}
Furthermore, we find from \eqref{14/01/18/17:32} and \eqref{14/01/18/17:43} that for a given $\omega>0$ there exist $\alpha_{0}>0$ and $\kappa_{0}>0$ such that for any $t\in I_{\max}(\psi)$ with $\|\nabla \psi(t)\|_{L^{2}}\le \alpha_{0}$, 
\begin{equation}\label{14/01/18/17:28}
\mathcal{K}\Big(\frac{M}{M+|x|} \psi(t) \Big)
\ge \kappa_{0}
\Big\| \nabla \Big( \frac{M}{M+|x|} \psi(t) \Big) \Big\|_{L^{2}}^{2}.
\end{equation}
On the other hand, we see from \eqref{13/12/11/20:51} that for any $t \in I'$ with $\|\nabla \psi(t)\|_{L^{2}}> \alpha_{0}$,  
\begin{equation}\label{14/01/18/17:50}
\mathcal{K}(\psi(t))\ge  
\min\big\{\kappa_{1}(\delta_{*}),\, \frac{1}{2} \alpha_{0}^{2}\big\}.
\end{equation}
Choosing $\varepsilon_{*}>0$ so small that $\varepsilon_{*}^{2}<\frac{1}{4}\min\{\kappa_{1}(\delta_{*}), \frac{1}{2} \alpha_{0}^{2}\}$, and using \eqref{13/06/15/20:41} and \eqref{14/01/18/17:50}, we obtain that for any $t \in I'$ with $\|\nabla \psi(t)\|_{L^{2}}> \alpha_{0}$,
\begin{equation}\label{14/01/19/10:40}
\mathcal{J}_{\omega}\Big(\frac{M}{M+|x|} \psi(t) \Big)
\le 
\mathcal{J}_{\omega}(\psi(t))
=\mathcal{S}_{\omega}(\psi)-\frac{1}{2}\mathcal{K}(\psi(t))
<m_{\omega}-\varepsilon_{*}^{2},
\end{equation}
which together with Lemma \ref{13/05/14/17:12} gives us that 
\begin{equation}\label{14/01/19/10:53}
\mathcal{K}\Big(\frac{M}{M+|x|} \psi(t) \Big)
\ge \min\{\kappa_{2}(\varepsilon_{*}^{2}),\  \frac{1}{2}\Big\| \nabla \Big(\frac{M}{M+|x|} \psi(t) \Big) \Big\|_{L^{2}}^{2} \}.
\end{equation}
Putting \eqref{14/01/18/17:28} and \eqref{14/01/19/10:53} together, we find that for any $t \in I'$, 
\begin{equation}\label{14/01/19/11:01}
\mathcal{K}\Big(\frac{M}{M+|x|} \psi(t) \Big)
\ge \min\{\kappa_{2}(\varepsilon_{*}^{2}),\  \kappa_{0}\Big\| \nabla \Big(\frac{M}{M+|x|} \psi(t) \Big) \Big\|_{L^{2}}^{2} \}
\end{equation}
for some constant $\kappa_{0}>0$ depending only on $d$, $p$ and $\omega$. Note here that \eqref{17/11/26/21:26} shows  
\begin{equation}\label{14/01/20/15:29}
\Big\| \nabla \Big(\frac{M}{M+|x|} \psi(t) \Big) \Big\|_{L^{2}}
\le 
\Big\| \frac{M}{(M+|x|)^{2}} \psi(t)  \Big\|_{L^{2}}
+
\Big\| \frac{M}{M+|x|} \nabla \psi(t)  \Big\|_{L^{2}}
\lesssim C(\omega)
\end{equation}
for all $t\in [t_{1},t_{3}]$. Furthermore, we find from \eqref{13/12/27/17:05}, \eqref{14/01/19/11:01} and \eqref{14/01/20/15:29} that if $M$ satisfies $o_{M}(1) \ll \min\{ \kappa_{2}(\varepsilon_{*}^{2})/ C(\omega)^{2}, \kappa_{0}\}$ and $t \in I'$, then 
\begin{equation}\label{14/01/20/15:35}
\frac{d}{dt}\Im \int_{\mathbb{R}^{d}}
\frac{M|x|}{M+|x|}
\frac{x}{|x|} \cdot \nabla \psi(x,t)\overline{\psi(x,t)}\,dx
\ge 0
.
\end{equation}

\vspace{1pt}
\noindent 
{\bf Step 3.}~We finish the proof by deriving a contradiction. 
\par 
We see from \eqref{18/01/20/15:26} and \eqref{14/01/20/15:35} that for any sufficiently small $\omega>0$, there exists $L(\omega)\gg 1$ such that if $M\ge L(\omega)/R$, then   
\begin{equation}\label{14/01/19/11:06}
\begin{split}
&
\Big[
\Im \int_{\mathbb{R}^{d}} \frac{M|x|}{M+|x|}
\frac{x}{|x|} \cdot \nabla \psi(x,t)\overline{\psi(x,t)}\,dx 
\Big]_{t_{1}}^{t_{3}}
\\[6pt]
&=
\int_{t_{1}}^{t_{3}}\frac{d}{dt}\Im \int_{\mathbb{R}^{d}}
\frac{M|x|}{M+|x|}
\frac{x}{|x|} \cdot \nabla \psi(x,t)\overline{\psi(x,t)}\,dx\,dt 
\\[6pt]
&\ge C_{0}^{*}(\omega)
\sum_{j=1}^{n}\int_{I_{j}} \big\{ e^{\mu |t-\tau_{j}|}
-
C_{1}^{*}(\omega)\big\} 
\widetilde{d}_{\omega}(\psi(\tau_{j}))\,dt 
\\[6pt]
&= C_{0}^{*}(\omega)
 \sum_{j=1}^{n} \frac{1}{\mu}
\Big\{ 
e^{\mu |\sup{I_{j}}-\tau_{j}|}
+
e^{\mu |\inf{I_{j}}-\tau_{j}|}
-2 
\Big\} 
\widetilde{d}_{\omega}(\psi(\tau_{j}))
\\[6pt]
&\quad 
-C_{0}^{*}(\omega)C_{1}^{*}(\omega)\sum_{j=1}^{n}|I_{j}| \widetilde{d}_{\omega}(\psi(\tau_{j}))
.
\end{split}
\end{equation}
Here, it follows from \eqref{14/01/18/12:28} and \eqref{14/01/18/11:35} that 
for any $1\le j\le n$, 
\begin{align}
\label{14/01/21/13:33}
\delta_{X}&=\widetilde{d}_{\omega}(\psi(\sup{I_{j}}))
\sim 
e^{\mu |\sup{I_{j}}-\tau_{j}|}\widetilde{d}_{\omega}(\psi(\tau_{j}))
\quad 
\mbox{except for $j=n$},
\\[6pt]
\label{14/01/21/13:34}
\delta_{X}&=\widetilde{d}_{\omega}(\psi(\inf{I_{j}}))
\sim 
e^{\mu |\inf{I_{j}}-\tau_{j}|}\widetilde{d}_{\omega}(\psi(\tau_{j}))
\quad 
\mbox{except for $j=1$}.
\end{align}
Moreover, we see from \eqref{13/12/18/05:36} and $\widetilde{d}_{\omega}(\psi(\tau_{j}))\ge R$ (see \eqref{15/04/02/16:47}) that if $M_{0}$ is sufficiently large dependently only on $d$, $p$ and $\omega$, 
\begin{align}
\label{14/01/21/14:03}
\mu 
C_{1}^{*}(\omega)
|I_{j}|\widetilde{d}_{\omega}(\psi(\tau_{j}))
\lesssim 
C_{1}^{*}(\omega) \frac{1}{M_{0}}\delta_{X}
\ll 
\delta_{X}
.
\end{align}
Thus, \eqref{14/01/19/11:06} together with  \eqref{14/01/21/13:33},  \eqref{14/01/21/13:34}, $\widetilde{d}_{\omega}(\psi(\tau_{j}))\le \delta_{*} \ll \delta_{X}$ and \eqref{14/01/21/14:03} shows that for any sufficiently large $M>0$ depending only on $d$, $p$, $\omega$ and $R$,  
\begin{equation}\label{14/01/20/17:11}
\begin{split}
&
\Big[
\Im \int_{\mathbb{R}^{d}}
\frac{M|x|}{M+|x|}
\frac{x}{|x|} \cdot \nabla \psi(x,t)\overline{\psi(x,t)}
\Big]_{t_{1}}^{t_{3}}
\\[6pt]
&\ge   
2 C^{**}(\omega)
\sum_{j=1}^{n} 
\big\{ \delta_{X}-\delta_{*} \big\} 
-C^{**}(\omega)\sum_{j=1}^{n} \delta_{X}
\ge 
\frac{1}{2} C^{**}(\omega) n\delta_{X}\ge C^{**}(\omega)\delta_{X}
\end{split}
\end{equation}
for some positive constant $C^{**}(\omega)$ depending only on $d$, $p$ and $\omega$. Note here that  
\begin{equation}\label{13/12/18/2:36}
\bigg[
\Im \int_{\mathbb{R}^{d}}\frac{M|x|}{M+|x|}
\frac{x}{|x|} \cdot 
\nabla \Phi_{\omega}(x)\Phi_{\omega}(x)
\,dx
\bigg]_{t_{1}}^{t_{3}}=0.
\end{equation}
Moreover, we have 
\begin{equation}\label{13/12/18/2:49}
\bigg|
\Im \int_{\mathbb{R}^{d}}
\frac{M|x|}{M+|x|}
\frac{x}{|x|} \cdot 
\Phi_{\omega}(x)\overline{\nabla \eta(x,t)}
\,dx
\bigg|
\le 
\||x| \Phi_{\omega}\|_{L^{2}}\|\nabla \eta(t)\|_{L^{2}}
.
\end{equation}
Here, since $\Phi_{\omega}$ decays exponentially, $\||x| \Phi_{\omega}\|_{L^{2}}$ is finite.

We see from the decomposition \eqref{13/04/27/9:38}, \eqref{13/12/18/2:36}, \eqref{13/12/18/2:49}, Lemma \ref{13/05/11/11:52}, Lemma \ref{13/05/05/15:47},  $\|\eta(t_{1})\|_{H^{1}}\sim \widetilde{d}_{\omega}(t_{1})=R$ and $\|\eta(t_{3})\|_{H^{1}}\sim \widetilde{d}_{\omega}(t_{3})\lesssim R$ that if $R_{*}$ is sufficiently small dependently on $\omega$, and $M= L(\omega)/R$ for some 
 sufficiently large constant $L(\omega)$ depending only on $d$, $p$ and $\omega$, then  
\begin{equation}\label{14/01/21/14:14}
\begin{split}
&
\biggm|
\bigg[ 
\Im \int_{\mathbb{R}^{d}}
\frac{M|x|}{M+|x|}
\frac{x}{|x|} \cdot \nabla \psi(x,t)\overline{\psi(x,t)}\,dx
\bigg]_{t_{1}}^{t_{3}}
\biggm|
\\[6pt]
&\le 
\biggm|
\bigg[ 
\Im \int_{\mathbb{R}^{d}}\frac{M|x|}{M+|x|}
\frac{x}{|x|} \cdot 
\bigm( 
\nabla \eta(x,t)\Phi_{\omega}(x) 
+
\nabla \Phi_{\omega}(x,t)\overline{\eta(x,t)}
\bigm)
\,dx
\bigg]_{t_{1}}^{t_{3}}
\biggm|
\\[6pt]
&\quad 
+
\biggm|
\bigg[ 
\Im \int_{\mathbb{R}^{d}}\frac{M|x|}{M+|x|}
\frac{x}{|x|} \cdot 
\nabla \eta(x,t) \overline{\eta(x,t)}\,dx
\bigg]_{t_{1}}^{t_{3}}
\biggm|
\\[6pt]
&\le 
\biggm|
\bigg[ 
\Im \int_{\mathbb{R}^{d}}\frac{M|x|}{M+|x|}
\frac{x}{|x|} \cdot 
\nabla \eta(x,t)\Phi_{\omega}(x) 
\,dx \bigg]_{t_{1}}^{t_{3}}
\biggm|
\\[6pt]
&\quad 
+
\biggm|
\bigg[ 
\Im \int_{\mathbb{R}^{d}}\frac{M|x|}{M+|x|}
\frac{x}{|x|} \cdot 
\Phi_{\omega}(x,t)\overline{\nabla \eta(x,t)}
\,dx
\bigg]_{t_{1}}^{t_{3}}
\biggm|
\\[6pt]
&\quad +
\biggm|
\bigg[ 
\Im \int_{\mathbb{R}^{d}}
\bigg\{ \frac{d M^{2}}{(M+|x|)^{2}}
+
\frac{(d-1)M|x|}{(M+|x|)^{2}}
\bigg\} 
\Phi_{\omega}(x,t)\overline{\eta(x,t)}
\,dx
\bigg]_{t_{1}}^{t_{3}}
\biggm|
\\[6pt]
&\quad 
+
\biggm|
\bigg[ 
\Im \int_{\mathbb{R}^{d}}\frac{M|x|}{M+|x|}
\frac{x}{|x|} \cdot 
\nabla \eta(x,t) \overline{\eta(x,t)}\,dx
\bigg]_{t_{1}}^{t_{3}}
\biggm|
\\[6pt]
&\lesssim 
\|\eta(t_{1})\|_{H^{1}}+\|\eta(t_{3})\|_{H^{1}}
+
M\|\eta(t_{1})\|_{H^{1}}^{2}+M\|\eta(t_{3})\|_{H^{1}}^{2}
\ll C^{**}(\omega) \delta_{X}
.
\end{split}
\end{equation}
This contradicts \eqref{14/01/20/17:11}. Hence, the case $\mathfrak{s}=1$ never happens. 
\\
\par 
We can deal with the case where $\mathfrak{s}=-1$ in a way similar to Section 4.1 of \cite{Nakanishi-Schlag2}. Hence, we have completed the proof.
\end{proof}

%%%%%%%%%%%%%%%%%%%%%%%%%%%%%%%%%%%%%%%%%%%%%%%%%%%%%%%%%%%%%%%%%%%%%%%%%%%%%%%%
\section{Proof of Theorem \ref{13/03/16/12:01}}\label{14/01/21/14:18}
Our aim in this section is to prove Theorem \ref{13/03/16/12:01}. By the time reversality, it suffices to show that there are only three possibilities (scattering, blowup, trapping) forward in time. Throughout this section, we fix $\omega\in (0,\omega_{*})$, where $\omega_{*}>0$ is the frequency given by the one-pass theorem (Theorem \ref{13/06/15/20:32}). 
\par 
We introduce several notation used in this section: 
\\
\noindent 
\textbullet~We use $I_{\max}(\psi)$ to denote the maximal existence-interval of a solution $\psi$ to \eqref{12/03/23/17:57}, and set $T_{\max}(\psi):=\sup{I_{\max}(\psi)}$ and $T_{\min}(\psi):=\inf{I_{\min}(\psi)}$.
\\
\textbullet~Since the set $A_{\omega}^{\varepsilon}$ is invariant under the flow defined by \eqref{12/03/23/17:57} for all $\varepsilon>0$ (see \eqref{13/03/02/20:21}), we use the convention $\psi \in A_{\omega}^{\varepsilon}$ to indicate that $\psi(t)\in A_{\omega}^{\varepsilon}$ for all $t \in I_{\max}(\psi)$.
\\
\noindent 
\textbullet~For given $R>0$ and $\varepsilon>0$, we define $S_{\omega,R}^{\varepsilon}$ by 
\begin{equation}
\label{14/01/23/11:24}
S_{\omega,R}^{\varepsilon}
:=
\bigg\{\psi \colon  
\begin{array}{l}
\mbox{$\psi$ is a radial solution to (NLS) such that $0\in I_{\max}(\psi)$, }
\\[3pt]
\mbox{ $\psi \in A_{\omega}^{\varepsilon}$, and $\widetilde{d}_{\omega}(\psi(t))\ge R$ for all $t \in [0,T_{\max}(\psi))$} 
\end{array} 
\bigg\}.
\end{equation}

Next, we give a fundamental fact in $A_{\omega}^{\varepsilon}$: 
\begin{lemma}\label{17/12/02/20:17}
Assume $d\ge 3$ and $1+\frac{4}{d}<p <2^{*}-1$. Let $\psi$ be a solution to \eqref{12/03/23/17:57} and $\varepsilon>0$. Furthermore, assume that $\psi \in A_{\omega}^{\varepsilon}$ and   
\begin{equation}\label{17/12/02/20:18}
\inf_{t\in I_{\max}(\psi)}\|\nabla \psi(t)\|_{L^{2}}^{2} 
\le \mathcal{H}(\Phi_{\omega}). 
\end{equation}
Then, we have either 
\begin{equation}\label{17/12/02/20:19}
\sup_{t\in I_{\max}(\psi)}\mathcal{K}(\psi(t))<0, 
\end{equation}
or 
\begin{equation}\label{17/12/02/20:20}
\inf_{t\in I_{\max}(\psi)}\mathcal{K}(\psi(t))>0. 
\end{equation}
\end{lemma}
\begin{proof}[Proof of Lemma \ref{17/12/02/20:17}]
Note first that $\mathcal{M}(\psi)=\mathcal{M}(\Phi_{\omega})$ and therefore $\psi$ is non-trivial. We also see from the assumption \eqref{17/12/02/20:18} that there exists $t_{0}\in I_{\max}(\psi)$ such that 
\begin{equation}\label{17/12/30/09:41}
\|\nabla \psi(t_{0})\|_{L^{2}}^{2} 
\le \frac{3}{2}\mathcal{H}(\Phi_{\omega}). 
\end{equation}
Furthermore, we see from $\mathcal{M}(\psi)=\mathcal{M}(\Phi_{\omega})$ and \eqref{17/12/30/09:41} that 
\begin{equation}\label{17/12/02/20:24}
\begin{split}
\mathcal{S}_{\omega}(\psi)
&=
\omega \mathcal{M}(\psi)
+
\frac{1}{2}\|\nabla \psi(t_{0})\|_{L^{2}}^{2}
-
\frac{1}{p+1}\|\psi(t_{0})\|_{L^{p+1}}^{p+1}
-
\frac{1}{2^{*}}\|\psi(t_{0})\|_{L^{2^{*}}}^{2^{*}}
\\[6pt]
&\le 
\omega \mathcal{M}(\Phi_{\omega})  
+
\frac{3}{4}\mathcal{H}(\Phi_{\omega})
=
\mathcal{S}_{\omega}(\Phi_{\omega})
-
\frac{1}{4}\mathcal{H}(\Phi_{\omega})
< m_{\omega}. 
\end{split} 
\end{equation}
Hence, the claim follows from \eqref{13/02/15/15:34} and Theorem \ref{15/03/24/16:05}.
\end{proof}

Now, let us recall that $\delta_{X}$ and $\delta_{S}$ denote the constants given by the ejection lemma (Lemma \ref{13/05/05/22:31}) and \eqref{15/07/11/11:18}, respectively (see also Remark \ref{18/01/27/14:03}). Furthermore, the one-pass theorem shows that there exist constants $0<\varepsilon_{*} \ll R_{*} \ll \delta_{*} (\ll \delta_{S} <\delta_{X})$ such that: $\varepsilon_{*}\le \varepsilon_{0}(\delta_{*})$ (see Lemma \ref{13/05/15/9:32} for the definition of $\varepsilon_{0}(\delta_{*})$); and for any $\varepsilon \in (0,\varepsilon_{*})$, any $R\in (\sqrt{2\varepsilon}, R_{*})$ and any solution $\psi$ to \eqref{12/03/23/17:57} with $\psi \in A_{\omega}^{\varepsilon}$, the following alternative holds: 
\\
{\bf Case 1}. There exists $t_{0}\in I_{\max}(\psi)$ such that $\widetilde{d}_{\omega}(\psi(t_{0}))<R$ and $\widetilde{d}_{\omega}(\psi(t))<R+R^{\frac{\min\{3,p+1\}}{2}}$ for all $t \in [t_{0}, T_{\max}(\psi))$; or  
\\
{\bf Case 2}. There exists $t_{1}\in I_{\max}(\psi)$ such that for all $t_{1}\le t <T_{\max}(\psi)$,  
\begin{equation}\label{14/01/21/16:18}
\widetilde{d}_{\omega}(\psi(t)) \ge R.
\end{equation}

In Case 1, it follows from $R+R^{\frac{\min\{3,p+1\}}{2}}<\delta_{S} <\delta_{X} $, \eqref{13/05/14/21:37} in Proposition \ref{18/01/26/21:25}, Remark \ref{18/01/27/14:03} and Lemma \ref{15/11/08/11:45} that $T_{\max}(\psi)=\infty$. Hence, $\psi$ is trapped by $\mathscr{O}(\Phi_{\omega})$ forward in time. 
\par 
Next, we consider Case 2. In this case, we have that $\psi(\cdot-t_{1}) \in S_{\omega,R}^{\varepsilon}$ and  
\begin{equation}\label{14/09/04/12:11}
\varepsilon 
\le  
\min\Big\{ \frac{\widetilde{d}_{\omega}(\psi(t))}{2},\, \varepsilon_{0}(\delta_{*}) \Big\}
\end{equation} 
for any $t\in [t_{1},T_{\max}(\psi))$. In Section \ref{14/03/19/11:16} and Section \ref{14/01/25/11:47}, we will prove that the solution $\psi$ blows up in a finite time or scatters. To this end, we need some preparations. Let us begin with the following: 
\begin{lemma}\label{14/01/26/08:13}
Let $\varepsilon \in (0, \varepsilon_{*})$, $R\in (\sqrt{2\varepsilon}, R_{*})$, and let $\psi \in S_{\omega,R}^{\varepsilon}$. Then, for any $T \in (0,T_{\max}(\psi))$, there exists $\tau \in [T,T_{\max}(\psi))$ such that 
\begin{equation}\label{14/01/26/08:16}
\widetilde{d}_{\omega}(\psi(\tau))\ge \delta_{X}.
\end{equation}
\end{lemma}
\begin{proof}[Proof of Lemma \ref{14/01/26/08:13}]
Suppose for contradiction that there exists $T_{0}\in (0,T_{\max}(\psi))$ such that for any $t \in  [T_{0},T_{\max}(\psi))$, 
\begin{equation}\label{14/01/26/08:33}
d_{\omega}(\psi(t))=\widetilde{d}_{\omega}(\psi(t))< \delta_{X},
\end{equation} 
where the equality follows from \eqref{13/05/14/21:37} in Proposition \ref{18/01/26/21:25} and Remark \ref{18/01/27/14:03}. Then, we see from Lemma \ref{15/11/08/11:45} that $T_{\max}(\psi)=\infty$. Let $L\gg \frac{1}{\mu}$, and consider a time $t_{L}\in [T_{0}, T_{0}+L\log{\frac{ \delta_{X}}{\varepsilon_{*}}}]$ such that $R_{L}:=d_{\omega}(\psi(t_{L}))$ is the minimum of $d_{\omega}(\psi(t))$ over the interval $[T_{0},T_{0}+L\log{\frac{ \delta_{X}}{\varepsilon_{*}}}]$. Note that $\delta_{X}>R_{L}\ge R>\varepsilon_{*}$. If $T_{0}\le t_{L}\le T_{0}+\frac{L}{2}\log{\frac{\delta_{X}}{\varepsilon_{*}}}$, then we see from the ejection lemma (Lemma \ref{13/05/05/22:31})  that 
\begin{equation}\label{14/01/26/08:41}
d_{\omega}
\big(\psi\Big(t_{L}+\frac{L}{2}\log{\frac{\delta_{X}}{\varepsilon_{*}}}\Big)\big)
\sim e^{\mu(\frac{L}{2}\log{\frac{\delta_{X}}{\varepsilon_{*}}})}
R_{L}
>
\left( \frac{\delta_{X}}{\varepsilon_{*}}\right)^{\frac{\mu L}{2}}\varepsilon_{*} 
\gg \delta_{X}.
\end{equation} 
However, this contradicts \eqref{14/01/26/08:33}. Similarly, when $t_{L}\ge T_{0}+\frac{L}{2}\log{\frac{\delta_{X}}{\varepsilon_{*}}}$, applying the ejection lemma backward in time, we reach a contradiction. Thus, we have proved the 
 claim.     
\end{proof}

Next, we show that the solution $\psi$ stays away from $\Phi_{\omega}$ after some time.  
\begin{lemma}\label{14/09/04/15:39}
Let $\varepsilon \in (0, \varepsilon_{*})$, $R \in (\sqrt{2\varepsilon}, R_{*})$, and let $\psi \in S_{\omega,R}^{\varepsilon}$. Then, there exists a time $T_{0} \in [0,T_{\max}(\psi))$ such that for any $t\in [T_{0},T_{\max}(\psi))$, 
\begin{equation}\label{14/09/01/20:14}
\widetilde{d}_{\omega}\big( \psi(t)\big)
\ge 
R_{*}.
\end{equation}
\end{lemma}
\begin{proof}[Proof of Lemma \ref{14/09/04/15:39}]
It suffices to consider the case where there exists a time $\tau_{0} \in [0,T_{\max}(\psi))$ such that  
\begin{equation}\label{14/09/02/16:56}
R\le \widetilde{R}_{*}:=\widetilde{d}_{\omega}\big( \psi(\tau_{0})\big) < R_{*}.\end{equation}
Here, it follows from Lemma \ref{14/01/26/08:13} that we may assume that 
\begin{equation}\label{15/02/21/15:39}
\frac{\sqrt{1+4R_{*}}-1}{2}
\le 
\widetilde{R}_{*}
< R_{*}.
\end{equation}
Then, we see from the one-pass theorem (Theorem \ref{13/06/15/20:32}) that either 
\begin{equation}\label{14/09/04/13:47}
\sup_{t\ge \tau_{0}}\widetilde{d}_{\omega}\big( \psi(t)\big)
\le 
\widetilde{R}_{*}+\widetilde{R}_{*}^{\frac{\min\{3,p+1\}}{2}}, 
\end{equation} 
or there exists $T_{0} \ge \tau_{0}$ such that 
\begin{equation}\label{14/09/04/13:48}
\inf_{t\ge T_{0}}\widetilde{d}_{\omega}\big( \psi(t)\big)
\ge \widetilde{R}_{*}+\widetilde{R}_{*}^{\frac{\min\{3,p+1\}}{2}}
\ge \widetilde{R}_{*}+\widetilde{R}_{*}^{2}
\ge 
R_{*}
. 
\end{equation} 
We find from Lemma \ref{14/01/26/08:13} and $R_{*}\ll \delta_{X}$ that the former case \eqref{14/09/04/13:47} never happens. Thus, the latter case \eqref{14/09/04/13:48} only happens, and the proof is completed. 
\end{proof}

Lastly, we introduce a sign function which will determine the scattering or blowup. 
\begin{lemma}\label{13/05/14/21:19} 
Let $\varepsilon \in (0, \varepsilon_{*})$ and $R\in (\sqrt{2\varepsilon}, R_{*})$. Then,  there exists a function $\mathfrak{S} \colon S_{\omega,R}^{\varepsilon} \to \{1, -1\}$ such that 
\begin{equation}\label{15/04/05/13:39}
\mathfrak{S}(\psi)
=\left\{ \begin{array}{ccl}
{\rm sign}[\lambda_{1}(t)]
&\mbox{if}& \mbox{$t\in [0,T_{\max}(\psi))$ with $\widetilde{d}_{\omega}(\psi(t))\le \delta_{E}(\omega)$}, 
\\[6pt]
{\rm sign}[\mathcal{K}(\psi(t))] 
&\mbox{if}& \mbox{$t\in [0,T_{\max}(\psi))$ with $\widetilde{d}_{\omega}(\psi(t))\ge \delta_{*}$}.
\end{array} \right.
\end{equation} 
In addition,  there exists a positive constant $C(\omega)$ which is independent of $\varepsilon$ and $R$ and satisfies that 
\begin{equation}\label{13/05/14/21:50}
\sup\big\{ \|\psi \|_{L_{t}^{\infty}H_{x}^{1}([0,T_{\max}(\psi)))} \colon  \psi \in S_{\omega,R}^{\varepsilon},\ \mathfrak{S}(\psi)=1 \big\} \le C(\omega). 
\end{equation} 
\end{lemma}
\begin{proof}[Proof of Lemma \ref{13/05/14/21:19}]
We introduced a sign function \eqref{15/04/02/20:40} in the proof of one-pass theorem (Theorem \ref{13/06/15/20:32}). Employing the same argument as \eqref{15/04/02/20:40} and Lemma \ref{14/01/26/08:13}, we can prove the existence of the desired sign function. Furthermore, we can prove \eqref{13/05/14/21:50} in the same way as \eqref{17/11/26/21:26}.  
\end{proof}

We divide $S_{\omega,R}^{\varepsilon}$ into two parts according to the sign  of $\mathfrak{S}$: 
\begin{equation}\label{14/01/23/11:58} 
S_{\omega,R,\pm}^{\varepsilon}
:=
\big\{ \psi \in S_{\omega,R}^{\varepsilon}\colon 
\mbox{$\mathfrak{S}(\psi)=\pm 1$} \big\}.
\end{equation}

%%%%%%%%%%%%%%%%%%%%%%%%%%%%%%%%%%%%%%%%%%%%%%%%%%%%%%%%%%%%%%%%%%%%%%%%%%%%%%%%%%%%%%%%%%%%%%%%%%%%%%%%%%%%%%%%%%%%%%%%%%%%%%%%%%%%%%%%%%%%%%%%%%%%%%%%%%%%%%%%%%%%%%%%%%%%%%%%%%%%%%%%%%%%%%%%%%%%%%%%%%%%%%%%%%%%%%%%%%%%%%%%%%%%%%%%%%%%%%%%

\subsection{Analysis on $\boldsymbol{S_{\omega,R,-}^{\varepsilon}}$}\label{14/03/19/11:16}

In this section, we shall prove that any radial solution $\psi \in S_{\omega,R,-}^{\varepsilon}$ blows up forward in time. 

\begin{proposition}\label{14/09/21/11:11}
Assume that $0< \varepsilon < \min\{ \varepsilon_{0}(R_{*}), \varepsilon_{*}\}$, where $\varepsilon_{0}(R_{*})$ is a constant given by Lemma \ref{13/05/15/9:32}. Let $R \in (\sqrt{2\varepsilon}, R_{*})$, and let $\psi \in S_{\omega,R,-}^{\varepsilon}$. Then, the maximal lifespan $T_{\max}(\psi)$ is finite: the finite time blowup forward in time.    
\end{proposition}
\begin{proof}[Proof of Proposition \ref{14/09/21/11:11}]
We see from Lemma \ref{14/09/04/15:39}, Lemma \ref{13/05/15/9:32} and Lemma \ref{17/12/02/20:17} that there exist a time $T_{0}\in [0,T_{\max}(\psi))$ and $\kappa_{1}(R_{*})>0$ such that 
\begin{equation}\label{15/06/01/11:57}
\inf_{t\in [T_{0},T_{\max}(\psi))}|\mathcal{K}(\psi(t)) | 
\ge 
\kappa_{1}(R_{*})
. 
\end{equation}
Moreover, it follows from Lemma \ref{14/01/26/08:13} that there exists a time $\tau \in [T_{0}, T_{\max}(\psi))$ such that  
\begin{equation}\label{14/09/20/20:37}
\widetilde{d}_{\omega}(\psi(\tau))\ge \delta_{X}\ge \delta_{*}. 
\end{equation}
Hence, Lemma \ref{13/05/14/21:19} together with $\mathfrak{S}(\psi)=-1$ shows that 
\begin{equation}\label{14/09/20/20:40}
\mathcal{K}(\psi(\tau))<0.
\end{equation}
Putting \eqref{15/06/01/11:57} and \eqref{14/09/20/20:40} together, we find that \begin{equation}\label{14/09/21/09:01}
\sup_{t\in [T_{0},T_{\max}(\psi))}\mathcal{K}(\psi(t))\le -\kappa_{1}(R_{*}).  
\end{equation}
Then, the same argument as the proof of Theorem 1.3 in \cite{AIKN} is available. Thus, we find that $T_{\max}(\psi)<\infty$. 
\end{proof}

%%%%%%%%%%%%%%%%%%%%%%%%%%%%%%%%%%%%%%%%%%%%%%%%%%%%%%%%%%%%%%%%%%%%%%%%%%%%%%%%
\subsection{Analysis on $\boldsymbol{S_{\omega,R,+}^{\varepsilon}}$}\label{14/01/25/11:47}

In this section, we shall prove that any solution $\psi \in S_{\omega,R,+}^{\varepsilon}$ scatters  forward in time. Lemma \ref{14/09/04/15:39} together with time-translation allows us to restrict ourselves  to the solutions in $S_{\omega,R_{*},+}^{\varepsilon}$. Then, Lemma \ref{13/05/15/9:32} together with Lemma \ref{17/12/02/20:17} determines constants $\varepsilon_{0}(R_{*})>0$ and $\kappa_{1}(R_{*})>0$ such that if $\varepsilon \in (0,\varepsilon_{0}(R_{*}))$, then any solution $\psi \in S_{\omega,R_{*},+}^{\varepsilon}$ obeys 
\begin{equation}\label{15/05/14/16:24}
\inf_{t\in [0,T_{\max}(\psi))} |\mathcal{K}(\psi(t))| 
\ge 
\kappa_{1}(R_{*})
. 
\end{equation}  
Furthermore, this together with Lemma \ref{14/01/26/08:13} and Lemma \ref{13/05/14/21:19} shows that for any $\varepsilon \in (0,\varepsilon_{0}(R_{*}))$ and any $\psi \in S_{\omega,R_{*},+}^{\varepsilon}$, 
\begin{equation}\label{15/05/14/16:33}
\inf_{[0,T_{\max}(\psi))}\mathcal{K}(\psi(t))\ge \kappa_{1}(R_{*}). 
\end{equation}  
\par 
It is well known (see, e.g, the claim (v) of Theorem 4.1 in \cite{AIKN2}) that the boundedness of a solution $\psi$ in the Strichartz-type space $W_{p+1}([0,T_{\max}(\psi))\cap W_{2^{*}}([0,T_{\max}(\psi))$ implies the scattering forward in time. Thus, our aim is to prove the following: 
\begin{proposition}\label{14/01/25/14:59}
For any $\omega \in (0,\omega_{*})$, there exists $\varepsilon \in (0,\varepsilon_{*})$ such that for any $\psi \in S_{\omega,R_{*},+}^{\varepsilon}$, 
\begin{equation}\label{14/01/25/15:01}
\|\psi\|_{W_{p+1}([0,T_{\max}(\psi))\cap W^{2^{*}}([0,T_{\max}(\psi)))}<\infty .
\end{equation} 
In particular, the solution $\psi$ scatters forward in time. 
\end{proposition}

\begin{proof}[Proof of Proposition \ref{14/01/25/14:59}] 
We divide the proof into several parts: In section \ref{15/07/11/11:26}, we suppose for contradiction that the claim was false, and extract some sequence of non-scattering solutions in $S_{\omega,R_{*},+}^{\varepsilon_{*}}$. In Section \ref{15/07/11/11:27}, we apply a profile decomposition to this sequence, and obtain ``linear profiles''. In Section \ref{15/07/11/11:28}, we introduce the ``nonlinear profiles'', and investigate their fundamental properties. Furthermore, in Section \ref{15/07/11/11:29}, we show the existence of a nonlinear profile whose Strichartz norm diverges. Finally, in Section \ref{15/07/11/11:30},  we derive a contradiction by showing the existence of the ``critical element'', and complete the proof.  

\subsubsection{Setup}\label{15/07/11/11:26}

 For any $E>0$, we define $\nu(E)$ by 
\begin{equation}\label{14/01/22/11:50}
\nu(E)
:=
\sup\Bigm\{ \|\psi\|_{W_{p+1}([0,T_{\max}(\psi)))\cap W_{2^{*}}([0,T_{\max}(\psi)))} \colon 
\psi \in S_{\omega,R_{*},+}^{\varepsilon_{*}},  
\ \mathcal{H}(\psi) 
\le E 
\Bigm\}.
\end{equation}
Furthermore, we put 
\begin{equation}\label{14/01/22/12:05}
E_{*}:=\sup\{E>0 \colon \nu(E)<\infty \}.
\end{equation}
If $\psi\in S_{\omega,R_{*},+}^{\varepsilon_{*}}$ and $\mathcal{H}(\psi)<\mathcal{H}(\Phi_{\omega})$, then we see from \eqref{15/05/14/16:33} that $\psi \in PW_{\omega,+}$. Hence, it follows from Theorem \ref{15/03/24/16:05} that  
\begin{equation}\label{14/09/04/15:23}
E_{*}\ge \mathcal{H}(\Phi_{\omega}). 
\end{equation}
Thus, what we want to prove is that $E_{*}>\mathcal{H}(\Phi_{\omega})$. We prove this by contradiction, and therefore suppose that 
\begin{equation}\label{14/01/25/15:35}
E_{*}=\mathcal{H}(\Phi_{\omega}). 
\end{equation}
Then, we can take a sequence $\{\varepsilon_{n}\}$ of constants in $(0,\varepsilon_{*})$ and a sequence $\{\psi_{n}\}$ of solutions such 
 that 
\begin{align}
\label{15/05/05/09:09}
& \lim_{n\to \infty}\varepsilon_{n}=0, 
\\[6pt]
\label{14/09/06/14:57}
&\varepsilon_{n} < \min\Big\{ \varepsilon_{0}(R_{*}), \ \frac{\kappa_{1}(R_{*})}{10 d(p-1)+10} \Big\},
\\[6pt]
\label{14/01/26/10:38}
&\psi_{n}\in S_{\omega,R_{*},+}^{\varepsilon_{n}},
\\[6pt]
\label{14/01/25/15:51}
&\|\psi_{n}\|_{W_{p+1}([0,T_{\max}(\psi_{n})))\cap W_{2^{*}}([0,T_{\max}(\psi_{n})))}=\infty.
\end{align}
In particular, we have that 
\begin{align}
\label{14/01/26/10:40}
&\inf_{n\ge 1} \inf_{t\in [0,T_{\max}(\psi_{n}))}\widetilde{d}_{\omega}(\psi_{n}(t))\ge R_{*},\\[6pt]
\label{14/01/25/15:50}
&\lim_{n\to \infty}\mathcal{H}(\psi_{n})= \mathcal{H}(\Phi_{\omega})=E_{*}. 
\end{align}
Furthermore, it follows from \eqref{15/05/14/16:33} and Lemma \ref{13/05/14/21:19} that 
\begin{align}
\label{14/01/26/11:03}
&\inf_{n\ge 1}\inf_{t\in [0,T_{\max}(\psi_{n}))}\mathcal{K}(\psi_{n}(t))
\ge 
\kappa_{1}(R_{*}),
\\[6pt]
\label{14/01/26/10:45}
&\sup_{n\ge 1}\sup_{t\in [0,T_{\max}(\psi_{n}))}\|\psi_{n}(t)\|_{H^{1}} <\infty
.
\end{align}

\subsubsection{Linear profiles}\label{15/07/11/11:27}
Let $\{\psi_{n}\}$ be the sequence of solutions obtained in Section \ref{15/07/11/11:26}. We employ the profile decomposition for a sequence in the homogeneous space $\dot{H}^{1}(\mathbb{R}^{d})$ (see Theorem 1.6 in \cite{Keraani} and Lemma 2.10 in \cite{Killip-Visan}). We apply it to the sequence $\{ |\nabla|^{-1} \langle \nabla \rangle \psi_{n}(0) \}$ in $\dot{H}^{1}(\mathbb{R}^{d})$, instead of $\{\psi_{n}\}$ itself. Then, we find\footnote{We do not use the assumption that every $\psi_{n}$ is radial in this subsection (Section \ref{15/07/11/11:27}) and therefore translation parameters appear in the linear profile decomposition. In Section \ref{15/07/11/11:29} and \ref{15/07/11/11:30}, we need the radial symmetry of the solutions so that we can take every translation $x_{n}^{j}=0$. } that there exists a subsequence of $\{\psi_{n}\}$ (denoted by the same symbol $\{\psi_{n} \}$) with the following properties: there exist a family $\{ \{(x_{n}^{1},t_{n}^{1},\lambda_{n}^{1})\}, \{(x_{n}^{2},t_{n}^{2},\lambda_{n}^{2})\}, \ldots \}$ of sequences in $\mathbb{R}^{d}\times \mathbb{R}\times (0,\infty)$, a family $\{\widetilde{u}^{1},\widetilde{u}^{2},\ldots \}$ of functions in $H^{1}(\mathbb{R}^{d})$, and a family $\{\{w_{n}^{1}\},\{w_{n}^{2}\}, \ldots \}$ of sequences in $H^{1}(\mathbb{R}^{d})$ such that:
\\
\noindent  
\textbullet~For any $j\ge 1$,  
\begin{align}
\label{14/01/26/11:46}
&\lim_{n\to \infty}t_{n}^{j}=t_{\infty}^{j} \in \mathbb{R}\cup \{\pm \infty \}, \\[6pt]
\label{14/01/26/11:47}
&\lim_{n\to \infty}\lambda_{n}^{j}=\lambda_{\infty}^{j} \in \{0, 1, \infty\},
\qquad 
\left\{ \begin{array}{ccc}
\lambda_{n}^{j}\equiv 1 &\mbox{if}& \lambda_{\infty}^{j}=1,
\\[6pt]
\lambda_{n}^{j}\le 1 &\mbox{if}& \lambda_{\infty}^{j}=0 .
\end{array}\right.   
\end{align}
\textbullet~For any $j'\neq j$, 
\begin{equation}\label{14/01/26/11:48}
\lim_{n\to \infty}
\Bigm\{ 
\frac{\lambda_{n}^{j'}}{\lambda_{n}^{j}}
+
\frac{\lambda_{n}^{j}}{\lambda_{n}^{j'}}
+
\frac{|x_{n}^{j}-x_{n}^{j'}|}{\lambda_{n}^{j}}
+
\frac{|t_{n}^{j}-t_{n}^{j'}|}{(\lambda_{n}^{j})^{2}}
\Bigm\}=\infty .
\end{equation}
\noindent 
\textbullet~For any $\dot{H}^{1}$-admissible pair\footnote{A pair $(q,r)$ is said to be $\dot{H}^{1}$-admissible if $\frac{1}{r}=\frac{d}{2}(\frac{1}{2}-\frac{1}{q}-\frac{1}{d})$ and $(q,r)\in [2,\infty]\times [2,\infty]$.} $(q,r)$,  
\begin{equation}\label{14/01/26/11:57}
\lim_{j\to \infty}\lim_{n\to \infty}\big\| |\nabla|^{-1}\langle \nabla \rangle e^{it\Delta}w_{n}^{j} \big\|_{L_{t}^{r}L_{x}^{q}(\mathbb{R})}=0
.
\end{equation}
\noindent 
\textbullet~For any $n\ge 1$ and $k \ge 1$, 
\begin{equation}\label{14/01/26/12:16}
\begin{split}
e^{it\Delta}\psi_{n}(0)
&=\sum_{j=1}^{k}\langle \nabla \rangle^{-1}|\nabla |
G_{n}^{j}\big( e^{it\Delta} |\nabla|^{-1}\langle \nabla \rangle \widetilde{u}^{j}\big)
+e^{it\Delta}w_{n}^{k}
\\[6pt]
&=
\sum_{j=1}^{k} \frac{\langle \nabla \rangle^{-1} \langle \lambda_{n}^{j} \nabla \rangle }{\lambda_{n}^{j}}e^{i(t-t_{n}^{j})\Delta}g_{n}^{j} \widetilde{u}^{j}
+e^{it\Delta}w_{n}^{k}
,
\end{split}
\end{equation}
where $G_{n}^{j}$ and $g_{n}^{j}$ are the operators defined by 
\begin{align}
\label{15/07/11/13:00}
&G_{n}^{j}v(x,t):=\frac{1}{(\lambda_{n}^{j})^{\frac{d-2}{2}}}v\Big( \frac{x-x_{n}^{j}}{\lambda_{n}^{j}},\frac{t-t_{n}^{j}}{(\lambda_{n}^{j})^{2}}\Big),
\\[6pt]
\label{14/01/26/14:02}
&g_{n}^{j}u(x):=\frac{1}{(\lambda_{n}^{j})^{\frac{d-2}{2}}}u\Big( \frac{x-x_{n}^{j}}{\lambda_{n}^{j}}\Big)
.
\end{align}
Note that 
\begin{equation}\label{17/11/02/10:37}
\frac{\langle \nabla \rangle^{-1} \langle \lambda_{n}^{j} \nabla \rangle }{\lambda_{n}^{j}}e^{i(t-t_{n}^{j})\Delta}g_{n}^{j} \widetilde{u}^{j}
=
g_{n}^{j} \sigma_{n}^{j} e^{i\frac{t-t_{n}^{j}}{(\lambda_{n}^{j})^{2}}\Delta }\widetilde{u}^{j}
=
G_{n}^{j}\sigma_{n}^{j} e^{it\Delta}\widetilde{u}^{j}
,
\end{equation}
where 
\begin{equation}\label{14/01/26/14:14}
\sigma_{n}^{j}:=\frac{\langle (\lambda_{n}^{j})^{-1}\nabla \rangle^{-1}\langle \nabla \rangle}{\lambda_{n}^{j}}.
\end{equation}
\noindent 
\textbullet~For any  $k\ge 1$ and any $1\le j \le k$, 
\begin{equation}\label{17/10/14/16:03}
\lim_{n \to \infty}
\langle \nabla \rangle^{-1}|\nabla| e^{-i\frac{t_{n}^{j}}{(\lambda_{n}^{j})}\Delta}(g_{n}^{j})^{-1} |\nabla|^{-1} \langle \nabla \rangle w_{n}^{k} =0
\quad \mbox{weakly in $H^{1}(\mathbb{R}^{d})$}
.
\end{equation}
\noindent 
\textbullet~For any $k\ge 1$,  
\begin{equation}\label{14/01/26/14:07}
\lim_{n\to \infty}\Big\{ 
\|\langle \nabla \rangle \psi_{n}(0)\|_{L^{2}}^{2}
-\sum_{j=1}^{k}
\| \langle \nabla \rangle g_{n}^{j}\sigma_{n}^{j}e^{-i\frac{t_{n}^{j}}{(\lambda_{n}^{j})^{2}}\Delta}\widetilde{u}^{j} \|_{L^{2}}^{2}
-\|\langle \nabla \rangle w_{n}^{k} \|_{L^{2}}^{2}
\Big\}=0
.
\end{equation}
Furthermore, the following hold for all $k\ge 1$ (see Lemma 2.2. in \cite{AIKN2}): 
\begin{align}
\label{17/10/14/16:19}
&\lim_{n\to \infty}\Big\{ 
\mathcal{M}(\psi_{n}) 
-\sum_{j=1}^{k}\mathcal{M}\big( g_{n}^{j}\sigma_{n}^{j}e^{-i\frac{t_{n}^{j}}{(\lambda_{n}^{j})^{2}}\Delta}\widetilde{u}^{j} \big) 
-\mathcal{M}(w_{n}^{k}) 
\Big\}=0,
\\[6pt]
\label{14/01/26/14:08}
&\lim_{n\to \infty}\Big\{ 
\mathcal{H}(\psi_{n}) 
-\sum_{j=1}^{k}\mathcal{H}\big( g_{n}^{j}\sigma_{n}^{j}e^{-i\frac{t_{n}^{j}}{(\lambda_{n}^{j})^{2}}\Delta}\widetilde{u}^{j} \big) 
-\mathcal{H}(w_{n}^{k}) 
\Big\}=0,
\\[6pt]
\label{14/01/26/14:09}
&\lim_{n\to \infty}\Big\{ 
\mathcal{S}_{\omega}(\psi_{n}) 
-\sum_{j=1}^{k}\mathcal{S}_{\omega}\big( g_{n}^{j}\sigma_{n}^{j}e^{-i\frac{t_{n}^{j}}{(\lambda_{n}^{j})^{2}}\Delta}\widetilde{u}^{j} \big) 
-\mathcal{S}_{\omega}(w_{n}^{k}) 
\Big\}=0,
\\[6pt]
\label{14/01/26/14:11}
&\lim_{n\to \infty}\Big\{ 
\mathcal{I}_{\omega}(\psi_{n}(0)) 
-
\sum_{j=1}^{k}\mathcal{I}_{\omega}\big( g_{n}^{j}\sigma_{n}^{j}e^{-i\frac{t_{n}^{j}}{(\lambda_{n}^{j})^{2}}\Delta}\widetilde{u}^{j} \big) 
-
\mathcal{I}_{\omega}(w_{n}^{k}) 
\Big\}=0.
\end{align}

Note that it follows from Strichartz' estimate, \eqref{14/01/26/10:45} and \eqref{14/01/26/14:07} that for any $k\ge 1$, there exists a number $N(k)$ such that for any $n\ge N(k)$, 
\begin{equation}\label{14/02/03/11:19}
\big\| \langle \nabla \rangle e^{it\Delta}  w_{n}^{k} \big\|_{St(\mathbb{R})}
\lesssim 
\|w_{n}^{k}\|_{H^{1}} \lesssim 1.
\end{equation}
Moreover, it follows from \eqref{14/01/26/11:57} and \eqref{14/02/03/11:19} that for any $1+\frac{4}{d} < q \le 2^{*}-1$,
\begin{equation}\label{15/02/28/16:41}
\begin{split}
\lim_{k\to \infty}\lim_{n\to \infty}\|e^{it\Delta}w_{n}^{k} \|_{W_{q+1}(\mathbb{R})}
&\le 
\lim_{k\to \infty}\lim_{n\to \infty}
\|e^{it\Delta}w_{n}^{k} \|_{W_{2+\frac{4}{d}}(\mathbb{R})}^{1-s_{q}}
\|e^{it\Delta}w_{n}^{k} \|_{W_{2^{*}}(\mathbb{R})}^{s_{q}}
\\[6pt]
&\lesssim
\lim_{k\to \infty}\lim_{n\to \infty}
\||\nabla|^{-1} \langle \nabla \rangle e^{it\Delta}w_{n}^{k} \|_{W_{2^{*}}(\mathbb{R})}^{s_{q}}
=0. 
\end{split} 
\end{equation}

We shall discuss fundamental properties of operators $g_{n}^{j}$ (see \eqref{14/01/26/14:02}) and $\sigma_{n}^{j}$ (see \eqref{14/01/26/14:14}). To this end, for each $j\ge 1$, we introduce an operator $\sigma_{\infty}^{j}$ as  
\begin{equation}\label{14/01/26/14:40}
\sigma_{\infty}^{j}
:=\left\{ \begin{array}{ccc} 
|\nabla |^{-1}\langle \nabla \rangle &\mbox{if}& \lambda_{\infty}^{j}=0,
\\[6pt]
1 &\mbox{if}& \lambda_{\infty}^{j}=1,
\\[6pt]
0  &\mbox{if}& \lambda_{\infty}^{j}=\infty.
\end{array} \right. 
\end{equation}

\begin{lemma}\label{14/09/01/11:30}
The following hold for all function $f\in H^{1}(\mathbb{R}^{d})$ and all $j\ge 1$:
\\
{\rm (i)}
\begin{equation}\label{14/09/01/11:32}
\lim_{n\to \infty} 
\big\| \nabla \big\{ \sigma_{n}^{j}f -\sigma_{\infty}^{j}f \big\} \big\|_{L^{2}}=0;
\end{equation}
\noindent 
{\rm (ii)} 
\begin{equation}\label{14/09/06/17:32}
\lim_{n\to \infty} 
\big\| g_{n}^{j} \sigma_{n}^{j}f \big\|_{L^{2}}
=
\lim_{n\to \infty} 
\big\| \lambda_{n}^{j} \sigma_{n}^{j}f \big\|_{L^{2}}
=\left\{ \begin{array}{ccc} 
0 &\mbox{if}& \lambda_{\infty}^{j}=0,
\\[6pt]
\|f\|_{L^{2}} &\mbox{if}& \lambda_{\infty}^{j}=1,
\\[6pt]
\|\langle \nabla \rangle f\|_{L^{2}}  &\mbox{if}& \lambda_{\infty}^{j}=\infty.
\end{array} \right.
\end{equation}
\end{lemma}
\begin{proof}[Proof of Lemma \ref{14/09/01/11:30}] 
We shall prove \eqref{14/09/01/11:32}. If $\lambda_{\infty}^{j}=1$, then $\sigma_{n}^{j} \equiv 1$ and therefore the claim is trivial. Assume that $\lambda_{\infty}^{j}=0$. Then, we see from Parseval's identity that 
\begin{equation}\label{14/09/01/11:43}
\begin{split}
\lim_{n\to \infty}\big\||\nabla| \big\{ \sigma_{n}^{j}f -\sigma_{\infty}^{j}f 
\big\} \big\|_{L^{2}}
&=
\lim_{n\to \infty}\bigg\| |\nabla| \bigg\{ 
\frac{ \langle (\lambda_{n}^{j})^{-1}\nabla \rangle^{-1} }{\lambda_{n}^{j}} - |\nabla|^{-1} \bigg\} \langle \nabla \rangle f 
\bigg\|_{L^{2}}
\\[6pt]
&=
\lim_{n\to \infty}\bigg\| \bigg\{ 
\frac{|\xi|}{\sqrt{(\lambda_{n}^{j})^{2}+|\xi|^{2}}} 
-1  \bigg\} \langle \xi \rangle \mathcal{F}[f] 
\bigg\|_{L_{\xi}^{2}}. 
\end{split}
\end{equation}
Moreover, since $\lim_{n\to \infty}\lambda_{n}^{j}=\lambda_{\infty}^{j}=0$, we 
see that 
\begin{equation}\label{17/10/29/14:13}
\lim_{n\to \infty}
\bigg\{ \frac{|\xi|}{\sqrt{(\lambda_{n}^{j})^{2}+|\xi|^{2}}} 
-1  \bigg\} \langle \xi \rangle \mathcal{F}[f](\xi)
=0 \qquad \mbox{for almost all $\xi \in \mathbb{R}^{d}$}
. 
\end{equation}
Hence, we find from Lebesgue's convergence theorem that \eqref{14/09/01/11:32} holds. Similarly, we can prove the case where $\lambda_{\infty}^{j}=\infty$.
\par 
Next, we give a proof of \eqref{14/09/06/17:32}. The case where $\lambda_{\infty}^{j}=1$ is trivial. Assume that $\lambda_{\infty}^{j}=0$. Then, we see from the substitution of variables and Parseval's identity that 
\begin{equation}\label{17/10/29/14:20}
\begin{split}
\lim_{n\to \infty}
\| g_{n}^{j} \sigma_{n}^{j} f \|_{L^{2}}
&=
\lim_{n\to \infty} 
\| \lambda_{n}^{j} \sigma_{n}^{j} f  \|_{L^{2}}
=
\lim_{n\to \infty}
\biggm\| 
\frac{1}{\sqrt{ 1+\frac{|\xi|^{2}}{(\lambda_{n}^{j})^{2}}}} 
 \langle \xi \rangle\mathcal{F}[f]
\biggm\|_{L_{\xi}^{2}}. 
\end{split}
\end{equation}
Since $\lim_{n\to \infty}\lambda_{n}^{j}=\lambda_{\infty}^{j}=0$,
we find from \eqref{17/10/29/14:20} and Lebesgue's convergence theorem that \eqref{14/09/06/17:32} holds. Similarly, we can prove the case where $\lambda_{\infty}^{j}=\infty$.  
\end{proof}

\begin{lemma}\label{14/06/21/15:10}
Let $j$ be a number such that $\lambda_{\infty}^{j} \neq \infty$. Then, for any  $n\ge 1$, any $q \in [2,\infty)$ and any function $f$ with $\langle \nabla \rangle f  \in L^{q}(\mathbb{R}^{d})$, 
\begin{equation} 
\label{14/06/21/15:12}
\big\| \lambda_{n}^{j} \sigma_{n}^{j}f \big\|_{L^{q}}
+
\big\| \nabla \sigma_{n}^{j}f \big\|_{L^{q}}
\lesssim 
\big\|\langle \nabla \rangle f \big\|_{L^{q}},
\end{equation}
where the implicit constant depends only on $d$ and $q$. 
\end{lemma}
\begin{proof}[Proof of Lemma \ref{14/06/21/15:10}]
If $\lambda_{\infty}^{j}=1$, then the claim is trivial. When $\lambda_{\infty}^{j}=0$, we find from the Mihlin multiplier theorem that 
\begin{equation}\label{14/06/21/15:20}
\begin{split}
\big\|\lambda_{n}^{j} \sigma_{n}^{j} f \big\|_{L^{q}}
+
\big\| \nabla \sigma_{n}^{j}f \big\|_{L^{q}}
&\lesssim  
\big\|\langle \nabla \rangle f \big\|_{L^{q}}
+
\big\| \sigma_{n}^{j}(\sigma_{\infty}^{j})^{-1} |\nabla| \sigma_{\infty}^{j}f \big\|_{L^{q}}
\\[6pt]
&\lesssim   
\big\|\langle \nabla \rangle f \big\|_{L^{q}}
+
\big\||\nabla| \sigma_{\infty}^{j} f \big\|_{L^{q}}
\lesssim 
\big\|\langle \nabla \rangle f \big\|_{L^{q}}. 
\end{split}
\end{equation}
Thus, we have proved the lemma.  
\end{proof}

\begin{lemma}\label{14/09/07/15:55}
Let $1+\frac{4}{d}<q<2^{*}-1$, and let $j\ge 1$ be a number for which $\lambda_{\infty}^{j}\in \{0, \infty\}$. Then, the linear profile $\tilde{u}^{j}$ satisfies   
\begin{equation}\label{14/09/20/11:58}
\lim_{n\to \infty}
\| G_{n}^{j}\sigma_{n}^{j}e^{it \Delta}\widetilde{u}^{j} \big|_{t=0} \|_{L^{q+1}}
=
\lim_{n\to \infty}
\| g_{n}^{j}\sigma_{n}^{j}e^{-i\frac{t_{n}^{j}}{(\lambda_{n}^{j})^{2}}\Delta}\widetilde{u}^{j}
\|_{L^{q+1}}
=0. 
\end{equation}
\end{lemma}
\begin{proof}[Proof of Lemma \ref{14/09/07/15:55}]
Assume that $\lambda_{\infty}^{j}\in \{0,\infty\}$. Then, we see from the Gagliardo-Nirenberg inequality that 
\begin{equation}\label{14/09/20/12:02}
\big\| g_{n}^{j}\sigma_{n}^{j}e^{-i\frac{t_{n}^{j}}{(\lambda_{n}^{j})^{2}}\Delta}\widetilde{u}^{j}\big\|_{L^{q+1}}^{q+1}
\lesssim 
\big\| \lambda_{n}^{j}\sigma_{n}^{j} \widetilde{u}^{j} \big\|_{L^{2}}^{q+1-\frac{d(q-1)}{2}} 
\big\|\nabla \sigma_{n}^{j} \widetilde{u}^{j} \big\|_{L^{2}}^{\frac{d(q-1)}{2}}
. 
\end{equation}
If $\lambda_{\infty}^{j}=0$, then it follows from \eqref{14/09/20/12:02} and Lemma \ref{14/09/01/11:30} that 
\begin{equation}\label{17/10/29/14:47}
\lim_{n\to \infty}
\big\| g_{n}^{j}\sigma_{n}^{j}e^{-i\frac{t_{n}^{j}}{(\lambda_{n}^{j})^{2}}\Delta}\widetilde{u}^{j}\big\|_{L^{q+1}}^{q+1}
\lesssim 
\lim_{n\to \infty} 
\big\| \lambda_{n}^{j}\sigma_{n}^{j} \widetilde{u}^{j} \big\|_{L^{2}}^{q+1-\frac{d(q-1)}{2}} 
\| \langle \nabla \rangle \widetilde{u}^{j} \|_{L^{2}}^{\frac{d(q-1)}{2}}
=0. 
\end{equation}
On the other hand, if $\lambda_{\infty}^{j}=\infty$, then it follows from \eqref{14/09/20/12:02} and Lemma \ref{14/09/01/11:30} that 
\begin{equation}\label{17/10/29/14:50}
\lim_{n\to \infty}
\big\| g_{n}^{j}\sigma_{n}^{j}e^{-i\frac{t_{n}^{j}}{(\lambda_{n}^{j})^{2}}\Delta}\widetilde{u}^{j}\big\|_{L^{q+1}}^{q+1}
\lesssim 
\big\| \langle \nabla \rangle \widetilde{u}^{j} \big\|_{L^{2}}^{q+1-\frac{d(q-1)}{2}} 
\lim_{n\to \infty}
\big\|\nabla \sigma_{n}^{j} \widetilde{u}^{j} \big\|_{L^{2}}^{\frac{d(q-1)}{2}}
=0. 
\end{equation}
Thus, we have completed the proof. 
\end{proof}

\subsubsection{Nonlinear profiles}\label{15/07/11/11:28}
In this subsection, we introduce the ``nonlinear profile'' associated with 
 $(\tilde{u}^{j}, \{(x_{n}^{j}, t_{n}^{j}, \lambda_{n}^{j})\})$ for each $j\ge 1$. From the point of view of \eqref{14/01/26/12:16}, we first introduce a function $U_{n}^{j}$ as the solution to \eqref{12/03/23/17:57} satisfying 
\begin{equation}\label{17/11/01/14:26}
U_{n}^{j}(0)=\frac{\langle \nabla \rangle^{-1} \langle \lambda_{n}^{j} \nabla \rangle }{\lambda_{n}^{j}}e^{-it_{n}^{j} \Delta}g_{n}^{j} \widetilde{u}^{j}
.
\end{equation}
Then, we find from \eqref{17/11/02/10:37} that 
\begin{equation}\label{17/11/01/13:09}
e^{it\Delta } U_{n}^{j}(0)= G_{n}^{j}\sigma_{n}^{j} e^{it \Delta} \widetilde{u}^{j}
.
\end{equation}
Furthermore, Duhamel's formula for $U_{n}^{j}$ together with \eqref{17/11/01/13:09} shows that $U_{n}^{j}$ satisfies 
\begin{equation}\label{17/11/02/11:01}
U_{n}^{j}(t)
-
G_{n}^{j}\sigma_{n}^{j} e^{it\Delta} \widetilde{u}^{j}
=
i\int_{0}^{t} 
e^{i(t-s)\Delta}
\bigm\{ 
F^{\dagger}[U_{n}^{j}](s)
+
F^{\ddagger}[U_{n}^{j}](s)
\bigm\} ds,
\end{equation}
where see \eqref{14/06/05/17:01} for the definitions of $F^{\dagger}$ and $F^{\ddagger}$. Applying $(\sigma_{n}^{j})^{-1}(G_{n}^{j})^{-1}$ to the both sides in \eqref{17/11/02/11:01},  we find that
\begin{equation}\label{17/11/02/11:02}
\begin{split}
& (\sigma_{n}^{j})^{-1}(G_{n}^{j})^{-1}
U_{n}^{j}(t)
-
e^{it\Delta} \widetilde{u}^{j}
\\[6pt]
&=
i (\sigma_{n}^{j})^{-1} (g_{n}^{j})^{-1}
\int_{-\frac{t_{n}^{j}}{(\lambda_{n}^{j})^{2}}}^{t} 
e^{i(\lambda_{n}^{j})^{2}(t-s)\Delta}
\bigm\{ 
(\lambda_{n}^{j})^{-\frac{(d-2)(p-1)}{2}} g_{n}^{j} 
F^{\dagger}[(G_{n}^{j})^{-1} U_{n}^{j}](s)
\\[6pt]
& \quad \hspace{30pt}
+
(\lambda_{n}^{j})^{-2} g_{n}^{j}F^{\ddagger}[ (G_{n}^{j})^{-1} U_{n}^{j}](s)
\bigm\} (\lambda_{n}^{j})^{2} \, ds
\\[6pt]
&=
i (\sigma_{n}^{j})^{-1} 
\int_{-\frac{t_{n}^{j}}{(\lambda_{n}^{j})^{2}}}^{t} 
e^{i(t-s)\Delta}
\bigm\{ 
(\lambda_{n}^{j})^{\frac{4-(d-2)(p-1)}{2}} 
F^{\dagger}[\sigma_{n}^{j}(\sigma_{n}^{j})^{-1}(G_{n}^{j})^{-1} U_{n}^{j}](s)
\\[6pt]
& \quad \hspace{17pt}
+
F^{\ddagger}[ \sigma_{n}^{j} (\sigma_{n}^{j})^{-1}(G_{n}^{j})^{-1} U_{n}^{j}](s)
\bigm\} ds
.
\end{split}
\end{equation}
From the point of view of \eqref{17/11/02/11:02},  for each number $j\ge 1$ with $\lambda_{\infty}^{j}\in \{0,1\}$, we define the nonlinear profile $\widetilde{\psi}^{j}$ associated with $(\widetilde{u}^{j},\{x_{n}^{j},t_{n}^{j}, \lambda_{n}^{j}\})$ as the solution to  
\begin{equation}\label{14/01/26/14:26}
\begin{split}
\widetilde{\psi}^{j}(t)
&=
e^{it\Delta}\widetilde{u}^{j}
+
i\int_{-\frac{t_{\infty}^{j}}{(\lambda_{\infty}^{j})^{2}}}^{t} 
e^{i(t-s)\Delta}(\sigma_{\infty}^{j})^{-1} 
\Big\{ 
(\lambda_{\infty}^{j})^{\frac{4-(d-2)(p-1)}{2}}
F^{\dagger}[\sigma_{\infty}^{j}\widetilde{\psi}^{j}](s)
\\[6pt]
&\qquad \hspace{30pt} +
F^{\ddagger}[ \sigma_{\infty}^{j}\widetilde{\psi}^{j}](s)
\Big\}\,ds.
\end{split}
\end{equation}
When $-\frac{t_{\infty}^{j}}{(\lambda_{\infty}^{j})^{2}}\in \{\pm \infty\}$, the equation \eqref{14/01/26/14:26} is interpreted  as the final value problem at $\pm \infty$. Moreover, if $j\ge 1$ is a number for which $\lambda_{\infty}^{j}=\infty$, then we define $\widetilde{\psi}^{j}(t):=e^{it\Delta}\widetilde{u}^{j}$
 (we define $\widetilde{\psi}^{j}$ as a free solution to \eqref{12/03/23/17:57}  so that \eqref{14/01/27/14:51} holds; in particular, if $\widetilde{u}^{j}\equiv 0$, then $\widetilde{\psi}^{j}\equiv 0$). 
\par 
We simply write $I_{\max}^{j}:= I_{\max}(\widetilde{\psi}^{j})$,   $T_{\max}^{j}:=T_{\max}(\widetilde{\psi}^{j})$ and $T_{\min}^{j}:=T_{\min}(\widetilde{\psi}^{j})$. Moreover, for an interval $I$ and a number $j\ge 1$, we define 
\begin{equation}\label{14/02/08/16:35}
\mathcal{W}^{j}(I)
:=\left\{ \begin{array}{ccl} 
W_{p+1}(I)\cap W_{2^{*}}(I) &\mbox{if}& \lambda_{\infty}^{j}\in \{1,\infty\},
\\
W_{2^{*}}(I) &\mbox{if}& \lambda_{\infty}^{j}=0. 
\end{array} \right.
\end{equation}

We find from the construction of the nonlinear profiles (see \eqref{14/01/26/14:26}) that the following hold for all numbers $j\ge 1$: 
\\[6pt]
\textbullet~$\widetilde{\psi}^{j} \in C(I_{\max}^{j},H^{1}(\mathbb{R}^{d}))$.
 \\[6pt]
\textbullet~$\widetilde{\psi}^{j}$ satisfies 
\begin{equation}\label{14/01/29/11:41}
\left\{ \begin{array}{rcl}
\displaystyle{i\frac{\partial \sigma_{\infty}^{j}\widetilde{\psi}^{j}}{\partial t} 
+\Delta \sigma_{\infty}^{j}\widetilde{\psi}^{j}
+F[\sigma_{\infty}^{j}\widetilde{\psi}^{j} ] =0}   
&\mbox{if}& \lambda_{\infty}^{j}=1, 
\\[12pt]
\displaystyle{i\frac{\partial \sigma_{\infty}^{j}\widetilde{\psi}^{j}}{\partial t} 
+\Delta \sigma_{\infty}^{j}\widetilde{\psi}^{j}
+F^{\ddagger}[\sigma_{\infty}^{j}\widetilde{\psi}^{j}]
=0} 
&\mbox{if}& \lambda_{\infty}^{j}=0,
\\[12pt]
\displaystyle{i\frac{\partial \widetilde{\psi}^{j}}{\partial t} 
+\Delta \widetilde{\psi}^{j}=0
}
&\mbox{if}& \lambda_{\infty}^{j}=\infty  
.   
\end{array} \right.
\end{equation}
\textbullet~There exists a number $N(j)$ such that for any $n\ge N(j)$, 
\begin{equation}\label{14/08/30/16:23}
-\frac{t_{n}^{j}}{(\lambda_{n}^{j})^{2}} \in I_{\max}^{j}
\end{equation} 
and 
\begin{equation}\label{14/01/27/14:51}
\lim_{n\to \infty}\big\| \widetilde{\psi}^{j}\Big( -\frac{t_{n}^{j}}{(\lambda_{n}^{j})^{2}} \Bigm)
-
e^{-i\frac{t_{n}^{j}}{(\lambda_{n}^{j})^{2}} \Delta} 
\widetilde{u}^{j} \big\|_{H^{1}}=0
.
\end{equation}
\noindent 
\textbullet~If $-\frac{t_{\infty}^{j}}{(\lambda_{\infty}^{j})^{2}}=\infty$, then $T_{\max}^{j}=\infty$ and for any $T>T_{\min}^{j}$, 
\begin{equation}\label{14/02/08/16:32}
\big\| \sigma_{\infty}^{j}\widetilde{\psi}^{j}\big\|_{\mathcal{W}^{j}([T,\infty))}<\infty
;
\end{equation}
\textbullet~If $-\frac{t_{\infty}^{j}}{(\lambda_{\infty}^{j})^{2}}=-\infty$,   then $T_{\min}^{j}=-\infty$ and for any $T<T_{\max}^{j}$, 
\begin{equation}\label{14/02/08/16:45}
\big\| \sigma_{\infty}^{j}\widetilde{\psi}^{j}\big\|_{\mathcal{W}^{j}((-\infty,T])}<\infty.
\end{equation}
When $\lambda_{\infty}^{j}\in\{ 0,\infty\}$, it is convenient rewriting \eqref{14/01/27/14:51} in the form 
\begin{equation}\label{14/01/29/20:27}
\left\{ \begin{array}{ccc}
\displaystyle{\lim_{n\to \infty}\big\| \widetilde{\psi}^{j}\bigg(\cdot-x_{n}^{j},\, -\frac{t_{n}^{j}}{(\lambda_{n}^{j})^{2}} \bigg)
-
g_{n}^{j}\sigma_{n}^{j}e^{-i\frac{t_{n}^{j}}{(\lambda_{n}^{j})^{2}} \Delta}\widetilde{u}^{j} 
\big\|_{H^{1}}=0}
&\mbox{if}& \lambda_{\infty}^{j}=1,
\\[12pt]
\displaystyle{\lim_{n\to \infty}\big\| \nabla g_{n}^{j}\Bigg\{ \sigma_{\infty}^{j}\widetilde{\psi}^{j}\bigg( -\frac{t_{n}^{j}}{(\lambda_{n}^{j})^{2}} \bigg)
-
\sigma_{n}^{j}e^{-i\frac{t_{n}^{j}}{(\lambda_{n}^{j})^{2}} \Delta}\widetilde{u}^{j} \big\} \Big\|_{L^{2}}=0}
&\mbox{if}& \lambda_{\infty}^{j}=0.
\end{array} \right.
\end{equation}
 
Now, for each number $j\ge 1$, we define  
\begin{equation}\label{14/02/01/10:50}
\psi_{n}^{j}:=G_{n}^{j}\sigma_{n}^{j}\widetilde{\psi}^{j},
\end{equation}
and let $I_{\max,n}^{j}$ be the maximal existence-interval of $\psi_{n}^{j}$. Then, we see from \eqref{14/01/29/11:41} that  
\begin{equation}\label{11/12/04/16:09}
\left\{ 
\begin{array}{ccc}
\displaystyle{i\frac{\partial \psi_{n}^{j}}{\partial t}
+
\Delta \psi_{n}^{j}+F[\psi_{n}^{j}]=0 }
&\mbox{if}& \lambda_{\infty}^{j}=1,
\\[6pt]
\displaystyle{i\frac{\partial \psi_{n}^{j}}{\partial t}
+
\Delta \psi_{n}^{j}
+(\sigma_{\infty}^{j})^{-1}
F^{\ddagger}[G_{n}^{j}\sigma_{\infty}^{j}\widetilde{\psi}^{j}]
=0
}
&\mbox{if}& \lambda_{\infty}^{j}=0,
\\[6pt]
\displaystyle{i\frac{\partial \psi_{n}^{j}}{\partial t}
+
\Delta \psi_{n}^{j}=0 }
&\mbox{if}& \lambda_{\infty}^{j}=\infty,  
\end{array}
\right.
\end{equation}
and 
\begin{equation}\label{14/02/01/10:57}
I_{\max,n}^{j}
=
\left\{ \begin{array}{ccc}
( (\lambda_{n}^{j})^{2}T_{\min}^{j}+t_{n}^{j}, \ (\lambda_{n}^{j})^{2}T_{\max}^{j}+t_{n}^{j})
&\mbox{if}& \lambda_{\infty}^{j}\in \{0,1\},
\\[6pt]
(-\infty, \infty)
&\mbox{if}& \lambda_{\infty}^{j}=\infty.
\end{array} \right.  
\end{equation} 
Furthermore, \eqref{14/08/30/16:23} shows that $0\in I_{\max,n}^{j}$ for any $n\ge N(j)$. 
\par 
In what follows, we discuss fundamental properties of nonlinear profiles in a series of lemmas. In particular, we show that there is at most one ``bad nonlinear profile''; the other profiles are ``good'' in the sense of \eqref{14/02/05/11:47} and \eqref{14/02/01/17:58}. 
\begin{lemma}\label{17/11/11/15:22}
Let $j$ be a number for which $\lambda_{\infty}^{j}=0$, and let $I$ be an interval on which  
\begin{equation}\label{17/11/11/15:23}
\big\| \langle \nabla \rangle \widetilde{\psi}^{j} \big\|_{St(I)}<\infty
.
\end{equation}
Then, the following hold:  
\begin{align}
\label{17/11/11/15:24}
&\lim_{n\to \infty}
\|
\sigma_{\infty}^{j}
\widetilde{\psi}^{j}
-
\sigma_{n}^{j} \widetilde{\psi}^{j}
\big\|_{W_{2^{*}}(I)}
\lesssim 
\lim_{n\to \infty}
\big\|
 \nabla 
\big\{
\sigma_{\infty}^{j}
\widetilde{\psi}^{j}
-
\sigma_{n}^{j} \widetilde{\psi}^{j}
\big\} 
\big\|_{V_{2^{*}}(I)}
=0,
\\[6pt]
\label{17/11/11/16:05}
&\lim_{n\to \infty}
\big\|
 \nabla 
\big\{
\sigma_{\infty}^{j}
\widetilde{\psi}^{j}
-
\sigma_{n}^{j} \widetilde{\psi}^{j}
\big\} 
\big\|_{V_{2+\frac{4}{d}}(I)}
=0,
\\[6pt]
\label{17/11/11/16:41}
&
\lim_{n\to \infty}
\big\|
 \lambda_{n}^{j} 
 \sigma_{n}^{j} \widetilde{\psi}^{j}
\big\|_{V_{2+\frac{4}{d}}(I)}
=0
.
\end{align}
\end{lemma}
\begin{proof}[Proof of Lemma \ref{17/11/11/15:22}]
The inequality in \eqref{17/11/11/15:24} follows from Sobolev's embedding. We see from H\"older's inequality, Lemma \ref{14/09/01/11:30}, Lemma \ref{14/06/21/15:10} and the assumption \eqref{17/11/11/15:23} that for almost all $t\in I$, 
\begin{equation}\label{17/11/11/15:00}
\begin{split}
&
\lim_{n\to \infty}\|
\nabla \{
\sigma_{\infty}^{j} \widetilde{\psi}^{j}(t)
-
\sigma_{n}^{j} \widetilde{\psi}^{j}(t)
\big\}
\|_{L^{\frac{2d(d+2)}{d^{2}+4}}}^{\frac{2(d+2)}{d-2}}
\\[6pt]
&\le 
\lim_{n\to \infty}
\|
\nabla \{
\sigma_{\infty}^{j} \widetilde{\psi}^{j}(t)
-
\sigma_{n}^{j} \widetilde{\psi}^{j}(t)
\big\}
\|_{L^{2}(I)}^{\frac{8}{d-2}}
\|
\langle \nabla \rangle  \widetilde{\psi}^{j}(t)
\|_{L^{2^{*}}(I)}^{2}=0
.
\end{split}
\end{equation}
Furthermore, we see from  Lemma \ref{14/06/21/15:10} and the assumption 
\eqref{17/11/11/15:23} that 
\begin{equation}\label{17/11/11/15:01}
\|
\nabla \{
\sigma_{\infty}^{j} \widetilde{\psi}^{j}(t)
-
\sigma_{n}^{j} \widetilde{\psi}^{j}(t)
\big\}
\|_{L^{\frac{2d(d+2)}{d^{2}+4}}}^{\frac{2(d+2)}{d-2}}
\lesssim
\| \langle \nabla \rangle 
\widetilde{\psi}^{j}(t)
\|_{L^{\frac{2d(d+2)}{d^{2}+4}}}^{\frac{2(d+2)}{d-2}} \in L^{1}(I)
.
\end{equation}
Thus, Lebesgue's convergence theorem shows that the equality in \eqref{17/11/11/15:24} holds. Similarly, we can prove \eqref{17/11/11/16:05}. It remains to prove \eqref{17/11/11/16:41}. We see from H\"older's inequality, Lemma \ref{14/09/01/11:30}, Lemma \ref{14/06/21/15:10} and the assumption \eqref{17/11/11/15:23} that for almost all $t\in I$, 
\begin{equation}\label{17/11/11/16:47}
\lim_{n\to \infty}
\big\|
 \lambda_{n}^{j} 
 \sigma_{n}^{j} \widetilde{\psi}^{j}(t)
\big\|_{L_{\frac{2(d+2)}{d}}}^{\frac{2(d+2)}{d}}
\le 
\lim_{n\to \infty}
\|
\lambda_{n}^{j} 
 \sigma_{n}^{j} \widetilde{\psi}^{j}(t)
\|_{L^{2}(I)}^{\frac{4}{d}}
\|
\langle \nabla \rangle  \widetilde{\psi}^{j}(t)
\|_{L^{2^{*}}(I)}^{2}=0
.
\end{equation}
Furthermore, we see from  Lemma \ref{14/06/21/15:10} and the assumption 
\eqref{17/11/11/15:23} that 
\begin{equation}\label{17/11/11/16:57}
\|
\lambda_{n}^{j}
\sigma_{n}^{j} \widetilde{\psi}^{j}(t)
\|_{L^{\frac{2(d+2)}{d}}}^{\frac{2(d+2)}{d}}
\lesssim
\| \langle \nabla \rangle 
\widetilde{\psi}^{j}(t)
\|_{L^{\frac{2(d+2)}{d}}}^{\frac{2(d+2)}{d}} \in L^{1}(I)
.
\end{equation}
Thus, Lebesgue's convergence theorem shows that \eqref{17/11/11/16:41} holds. 
\end{proof}

\begin{lemma}\label{14/07/02/10:30}
Let $j$ be a number for which $\lambda_{\infty}^{j}=0$, and let $I$ be an interval on which  
\begin{equation}\label{14/07/03/14:03}
\big\| \langle \nabla \rangle \widetilde{\psi}^{j} \big\|_{St(I)}<\infty
.
\end{equation}
Then, the nonlinear term in \eqref{11/12/04/16:09} satisfies 
\begin{equation}\label{14/07/02/10:31}
\lim_{n\to \infty}
\big\|
\langle \nabla \rangle
\big\{
(\sigma_{\infty}^{j})^{-1}
F^{\ddagger}[G_{n}^{j}\sigma_{\infty}^{j}\widetilde{\psi}^{j}]
-
 F[\psi_{n}^{j}]
\big\} 
\big\|_{L_{t,x}^{\frac{2(d+2)}{d+4}}(I_{n}^{j})}
=0,
\end{equation}
where $I_{n}^{j}:=( (\lambda_{n}^{j})^{2}\inf{I}+t_{n}^{j}, \ (\lambda_{n}^{j})^{2}\sup{I}+t_{n}^{j})$. 
\end{lemma}
\begin{proof}[Proof of Lemma \ref{14/07/02/10:30}]
Note first that 
\begin{equation}\label{18/01/24/18:17}
F[\psi_{n}^{j}]
=
F^{\ddagger}[G_{n}^{j}\sigma_{n}^{j}\widetilde{\psi}^{j}]
+
F^{\dagger}[G_{n}^{j}\sigma_{n}^{j}\widetilde{\psi}^{j}]
.
\end{equation}
Next, observe that 
\begin{equation}\label{17/11/09/09:01}
\begin{split}
& 
\lim_{n\to \infty}\big\|
\langle \nabla \rangle \big\{ (\sigma_{\infty}^{j})^{-1} 
F^{\ddagger}[G_{n}^{j}\sigma_{\infty}^{j}\widetilde{\psi}^{j}]
-
F^{\ddagger}[G_{n}^{j}\sigma_{n}^{j} \widetilde{\psi}^{j}]
\big\}
\big\|_{L_{t,x}^{\frac{2(d+2)}{d+4}}(I_{n}^{j})}
\\[6pt]
&=
\lim_{n\to \infty}
(\lambda_{n}^{j})^{-2} 
\big\|
G_{n}^{j} (\lambda_{n}^{j})^{-1} | \nabla | 
F^{\ddagger}[\sigma_{\infty}^{j}\widetilde{\psi}^{j}]
-
G_{n}^{j}
\langle (\lambda_{n}^{j})^{-1} \nabla \rangle 
F^{\ddagger}[\sigma_{n}^{j} \widetilde{\psi}^{j}]
\big\|_{L_{t,x}^{\frac{2(d+2)}{d+4}}(I_{n}^{j})}
\\[6pt]
&=
\lim_{n\to \infty}\big\|
|\nabla| 
F^{\ddagger}[ \sigma_{\infty}^{j}\widetilde{\psi}^{j}]
-
\lambda_{n}^{j}\langle (\lambda_{n}^{j})^{-1}\nabla \rangle 
F^{\ddagger}[\sigma_{n}^{j} \widetilde{\psi}^{j}]
\big\|_{L_{t,x}^{\frac{2(d+2)}{d+4}}(I)}
\\[6pt]
&\lesssim  
\lim_{n\to \infty}\big\|
\nabla 
\big\{ 
F^{\ddagger}[\sigma_{\infty}^{j} \widetilde{\psi}^{j}]
-
F^{\ddagger}[\sigma_{n}^{j} \widetilde{\psi}^{j}]
\big\}
\big\|_{L_{t,x}^{\frac{2(d+2)}{d+4}}(I)}
\\[6pt]
&\quad +
\lim_{n\to \infty}
\lambda_{n}^{j}
\big\|
\big\{ 
|(\lambda_{n}^{j})^{-1}\nabla|-\langle (\lambda_{n}^{j})^{-1}\nabla \rangle 
\big\}
F^{\ddagger}[ \sigma_{n}^{j}\widetilde{\psi}^{j}]
\big\|_{L_{t,x}^{\frac{2(d+2)}{d+4}}(I)}
.
\end{split}
\end{equation}

We consider the first term on the right-hand side of \eqref{17/11/09/09:01}. We see from elementary computations and H\"older's inequality  that 
\begin{equation}\label{14/07/02/10:39}
\begin{split}
&
\lim_{n\to \infty}\big\|
\nabla 
\big\{ 
F^{\ddagger}[\sigma_{\infty}^{j} \widetilde{\psi}^{j}]
-
F^{\ddagger}[\sigma_{n}^{j} \widetilde{\psi}^{j}]
\big\}
\big\|_{L_{t,x}^{\frac{2(d+2)}{d+4}}(I)}
\\[6pt]
&\le 
\lim_{n\to \infty}\big\|
|\sigma_{\infty}^{j} \widetilde{\psi}^{j}|^{\frac{4}{d-2}}
\nabla \sigma_{\infty}^{j} \widetilde{\psi}^{j}
-
|\sigma_{n}^{j} \widetilde{\psi}^{j}|^{\frac{4}{d-2}}
\nabla \sigma_{n}^{j} \widetilde{\psi}^{j}
\big\}
\big\|_{L_{t,x}^{\frac{2(d+2)}{d+4}}(I)}
\\[6pt]
&\quad +
\lim_{n\to \infty}\big\|
|\sigma_{\infty}^{j} \widetilde{\psi}^{j}|^{\frac{8-2d}{d-2}}(\sigma_{\infty}^{j} \widetilde{\psi}^{j})^{2}\overline{ \nabla \sigma_{\infty}^{j} \widetilde{\psi}^{j}}
-
|\sigma_{n}^{j} \widetilde{\psi}^{j}|^{\frac{8-2d}{d-2}}
(\sigma_{n}^{j} \widetilde{\psi}^{j})^{2}
\overline{\nabla \sigma_{n}^{j} \widetilde{\psi}^{j}}
\big\}
\big\|_{L_{t,x}^{\frac{2(d+2)}{d+4}}(I)}
.
\end{split}
\end{equation}
Furthermore, we see from the triangle inequality, H\"older's inequality, Sobolev's embedding, Lemma \ref{14/06/21/15:10}, Lemma \ref{17/11/11/15:22} and the assumption \eqref{14/07/03/14:03} that if $d\ge 6$ (hence $\frac{4}{d-2}\le 1$), then 
\begin{equation}\label{17/11/11/14:45}
\begin{split}
&
\lim_{n\to \infty}\big\|
|\sigma_{\infty}^{j} \widetilde{\psi}^{j}|^{\frac{4}{d-2}}
\nabla \sigma_{\infty}^{j} \widetilde{\psi}^{j}
-
|\sigma_{n}^{j} \widetilde{\psi}^{j}|^{\frac{4}{d-2}}
\nabla \sigma_{n}^{j} \widetilde{\psi}^{j}
\big\}
\big\|_{L_{t,x}^{\frac{2(d+2)}{d+4}}(I)}
\\[6pt]
&\quad +
\lim_{n\to \infty}\big\|
|\sigma_{\infty}^{j} \widetilde{\psi}^{j}|^{\frac{8-2d}{d-2}}(\sigma_{\infty}^{j} \widetilde{\psi}^{j})^{2}\overline{ \nabla \sigma_{\infty}^{j} \widetilde{\psi}^{j}}
-
|\sigma_{n}^{j} \widetilde{\psi}^{j}|^{\frac{8-2d}{d-2}}
(\sigma_{n}^{j} \widetilde{\psi}^{j})^{2}
\overline{\nabla \sigma_{n}^{j} \widetilde{\psi}^{j}}
\big\}
\big\|_{L_{t,x}^{\frac{2(d+2)}{d+4}}(I)}
\\[6pt]
&\lesssim 
\lim_{n\to \infty}\big\|
|\sigma_{\infty}^{j} \widetilde{\psi}^{j}|^{\frac{4}{d-2}}
\nabla \sigma_{\infty}^{j} \widetilde{\psi}^{j}
-
|\sigma_{n}^{j} \widetilde{\psi}^{j}|^{\frac{4}{d-2}}
\nabla \sigma_{\infty}^{j} \widetilde{\psi}^{j}
\big\}
\big\|_{L_{t,x}^{\frac{2(d+2)}{d+4}}(I)}
\\[6pt]
&\quad + \lim_{n\to \infty}\big\|
|\sigma_{\infty}^{j} \widetilde{\psi}^{j}|^{\frac{8-2d}{d-2}}(\sigma_{\infty}^{j} \widetilde{\psi}^{j})^{2}\overline{ \nabla \sigma_{\infty}^{j} \widetilde{\psi}^{j}}
-
|\sigma_{n}^{j} \widetilde{\psi}^{j}|^{\frac{8-2d}{d-2}}
(\sigma_{n}^{j} \widetilde{\psi}^{j})^{2}
\overline{\nabla \sigma_{\infty}^{j} \widetilde{\psi}^{j}}
\big\}
\big\|_{L_{t,x}^{\frac{2(d+2)}{d+4}}(I)}
\\[6pt]
&\quad +
\lim_{n\to \infty}\big\|
|\sigma_{n}^{j} \widetilde{\psi}^{j}|^{\frac{4}{d-2}}
\nabla \sigma_{\infty}^{j} \widetilde{\psi}^{j}
-
|\sigma_{n}^{j} \widetilde{\psi}^{j}|^{\frac{4}{d-2}}
\nabla \sigma_{n}^{j} \widetilde{\psi}^{j}
\big\}
\big\|_{L_{t,x}^{\frac{2(d+2)}{d+4}}(I)}
\\[6pt]
&\quad + 
\lim_{n\to \infty}\big\|
|\sigma_{n}^{j} \widetilde{\psi}^{j}|^{\frac{8-2d}{d-2}}(\sigma_{n}^{j} \widetilde{\psi}^{j})^{2}\overline{ \nabla \sigma_{\infty}^{j} \widetilde{\psi}^{j}}
-
|\sigma_{n}^{j} \widetilde{\psi}^{j}|^{\frac{8-2d}{d-2}}
(\sigma_{n}^{j} \widetilde{\psi}^{j})^{2}
\overline{\nabla \sigma_{n}^{j} \widetilde{\psi}^{j}}
\big\}
\big\|_{L_{t,x}^{\frac{2(d+2)}{d+4}}(I)}
\\[6pt]
&\lesssim 
\lim_{n\to \infty}\|
\sigma_{\infty}^{j} \widetilde{\psi}^{j}
-
\sigma_{n}^{j} \widetilde{\psi}^{j}
\|_{W_{2^{*}}(I)}^{\frac{4}{d-2}}
\| 
\nabla \sigma_{\infty}^{j} \widetilde{\psi}^{j}
\|_{V_{2+\frac{4}{d}}(I)}
\\[6pt]
&\quad +
\lim_{n\to \infty}\|
\sigma_{n}^{j} \widetilde{\psi}^{j}
\|_{W_{2^{*}}(I)}^{\frac{4}{d-2}}
\| 
\nabla \big\{
\sigma_{\infty}^{j} \widetilde{\psi}^{j}
-
\sigma_{\infty}^{j} \widetilde{\psi}^{j}
\big\}
\|_{V_{2+\frac{4}{d}}(I)}
\\[6pt]
&\lesssim 
\lim_{n\to \infty}\|
\sigma_{\infty}^{j} \widetilde{\psi}^{j}
-
\sigma_{n}^{j} \widetilde{\psi}^{j}
\|_{W_{2^{*}}(I)}^{\frac{4}{d-2}}
\| 
\langle \nabla \rangle \widetilde{\psi}^{j}
\|_{V_{2+\frac{4}{d}}(I)}
\\[6pt]
&\quad +
\lim_{n\to \infty}\| \langle \nabla \rangle
\widetilde{\psi}^{j}
\|_{V_{2^{*}}(I)}^{\frac{4}{d-2}}
\| 
\nabla \big\{
\sigma_{\infty}^{j} \widetilde{\psi}^{j}
-
\sigma_{\infty}^{j} \widetilde{\psi}^{j}
\big\}
\|_{V_{2+\frac{4}{d}}(I)}
\\[6pt]
&=0.
\end{split}
\end{equation}
Similarly, we can verify that if $3\le d \le 5$ (hence $1<\frac{4}{d-2}=1+\frac{6-d}{d-2}$), then 
\begin{equation}\label{17/11/11/16:25}
\begin{split}
&
\lim_{n\to \infty}\big\|
|\sigma_{\infty}^{j} \widetilde{\psi}^{j}|^{\frac{4}{d-2}}
\nabla \sigma_{\infty}^{j} \widetilde{\psi}^{j}
-
|\sigma_{n}^{j} \widetilde{\psi}^{j}|^{\frac{4}{d-2}}
\nabla \sigma_{n}^{j} \widetilde{\psi}^{j}
\big\}
\big\|_{L_{t,x}^{\frac{2(d+2)}{d+4}}(I)}
\\[6pt]
&\quad +
\lim_{n\to \infty}\big\|
|\sigma_{\infty}^{j} \widetilde{\psi}^{j}|^{\frac{8-2d}{d-2}}(\sigma_{\infty}^{j} \widetilde{\psi}^{j})^{2}\overline{ \nabla \sigma_{\infty}^{j} \widetilde{\psi}^{j}}
-
|\sigma_{n}^{j} \widetilde{\psi}^{j}|^{\frac{8-2d}{d-2}}
(\sigma_{n}^{j} \widetilde{\psi}^{j})^{2}
\overline{\nabla \sigma_{n}^{j} \widetilde{\psi}^{j}}
\big\}
\big\|_{L_{t,x}^{\frac{2(d+2)}{d+4}}(I)}
\\[6pt]
&\lesssim 
\lim_{n\to \infty}
\big\{ 
\| \sigma_{\infty}^{j} \widetilde{\psi}^{j}\|_{W_{2^{*}}}^{\frac{6-d}{d-2}}
+
\|\sigma_{n}^{j} \widetilde{\psi}^{j}
\|_{W_{2^{*}}(I)}^{\frac{6-d}{d-2}}
\big\}
\|
\sigma_{\infty}^{j} \widetilde{\psi}^{j}
-
\sigma_{n}^{j} \widetilde{\psi}^{j}
\|_{W_{2^{*}}(I)}
\| 
\langle \nabla \rangle \widetilde{\psi}^{j}
\|_{V_{2+\frac{4}{d}}(I)}
\\[6pt]
&\quad +
\lim_{n\to \infty}
\| \sigma_{n}^{j}\widetilde{\psi}^{j}
\|_{W_{2^{*}}(I)}^{\frac{4}{d-2}}
\| 
\nabla \big\{
\sigma_{\infty}^{j} \widetilde{\psi}^{j}
-
\sigma_{\infty}^{j} \widetilde{\psi}^{j}
\big\}
\|_{V_{2+\frac{4}{d}}(I)}
\\[6pt]
&=0
.
\end{split}
\end{equation}

Next, we consider the second term on the right-hand side of \eqref{17/11/09/09:01}. We see from the Mihlin multiplier theorem, the assumption \eqref{14/07/03/14:03} and Lemma \ref{17/11/11/15:22} that 
\begin{equation}\label{17/11/10/13:25}
\begin{split}
&
\lim_{n\to \infty}\lambda_{n}^{j}
\big\| 
\big\{ 
|(\lambda_{n}^{j})^{-1}\nabla|- \langle (\lambda_{n}^{j})^{-1}\nabla \rangle 
\big\}
F^{\ddagger}[ \sigma_{n}^{j}\widetilde{\psi}^{j}]
\big\|_{L_{t,x}^{\frac{2(d+2)}{d+4}}(I)}
\\[6pt]
&\lesssim 
 \lim_{n\to \infty} \lambda_{n}^{j}\big\|
F^{\ddagger}[ \sigma_{n}^{j}\widetilde{\psi}^{j}]
\big\|_{L_{t,x}^{\frac{2(d+2)}{d+4}}(I)}
\lesssim  
\lim_{n\to \infty}\big\|
\sigma_{n}^{j}\widetilde{\psi}^{j}
\big\|_{W_{2^{*}}(I)}^{\frac{4}{d-2}}
\| \lambda_{n}^{j} \sigma_{n}^{j}\widetilde{\psi}^{j} \|_{V_{2+\frac{4}{d}}(I)}
=0. 
\end{split}
\end{equation}
Putting \eqref{17/11/09/09:01} through \eqref{17/11/10/13:25} together, we find that 
\begin{equation}\label{17/11/11/17:59}
\lim_{n\to \infty}\big\|
\langle \nabla \rangle (\sigma_{\infty}^{j})^{-1} 
F^{\ddagger}[G_{n}^{j}\sigma_{\infty}^{j}\widetilde{\psi}^{j}]
-
\langle \nabla \rangle 
F^{\ddagger}[G_{n}^{j}\sigma_{n}^{j} \widetilde{\psi}^{j}]
\big\|_{L_{t,x}^{\frac{2(d+2)}{d+4}}(I_{n}^{j})}
=0
.
\end{equation}
We can also verify in a way similar to the estimates \eqref{17/11/09/09:01} and \eqref{17/11/10/13:25} that 
\begin{equation}\label{14/07/02/11:50}
\begin{split}
&\lim_{n\to \infty}
\big\|\langle \nabla \rangle
F^{\dagger}[ G_{n}^{j}\sigma_{n}^{j} \widetilde{\psi}^{j} ]
\big\|_{L_{t,x}^{\frac{2(d+2)}{d+4}}(I_{n}^{j})}
\\[6pt]
&=
\lim_{n\to \infty}
(\lambda_{n}^{j})^{\frac{4-(d-2)(p-1)}{2}}\big\|\lambda_{n}^{j} \langle (\lambda_{n}^{j})^{-1} \nabla \rangle
F^{\dagger}[ \sigma_{n}^{j} \widetilde{\psi}^{j} ]
\big\|_{L_{t,x}^{\frac{2(d+2)}{d+4}}(I_{n}^{j})}
\\[6pt]
&\lesssim 
\lim_{n\to \infty}
(\lambda_{n}^{j})^{\frac{4-(d-2)(p-1)}{2}}
\lambda_{n}^{j}\big\| \big\{ |(\lambda_{n}^{j})^{-1} \nabla | -  \langle (\lambda_{n}^{j})^{-1} \nabla \rangle \big\}
F^{\dagger}[ \sigma_{n}^{j} \widetilde{\psi}^{j} ]
\big\|_{L_{t,x}^{\frac{2(d+2)}{d+4}}(I_{n}^{j})}
\\[6pt]
&\quad + 
\lim_{n\to \infty}
(\lambda_{n}^{j})^{\frac{4-(d-2)(p-1)}{2}}\big\| \nabla 
F^{\dagger}[ \sigma_{n}^{j} \widetilde{\psi}^{j} ]
\big\|_{L_{t,x}^{\frac{2(d+2)}{d+4}}(I_{n}^{j})}
\\[6pt]
&\lesssim 
\lim_{n\to \infty}
(\lambda_{n}^{j})^{\frac{4-(d-2)(p-1)}{2}}
\lambda_{n}^{j}\big\| F^{\dagger}[ \sigma_{n}^{j} \widetilde{\psi}^{j} ]
\big\|_{L_{t,x}^{\frac{2(d+2)}{d+4}}(I_{n}^{j})}
\\[6pt]
&\quad + 
\lim_{n\to \infty}
(\lambda_{n}^{j})^{\frac{4-(d-2)(p-1)}{2}}\big\| \nabla 
F^{\dagger}[ \sigma_{n}^{j} \widetilde{\psi}^{j} ]
\big\|_{L_{t,x}^{\frac{2(d+2)}{d+4}}(I_{n}^{j})}
\\[6pt]
&\lesssim
\lim_{n\to \infty}
(\lambda_{n}^{j})^{\frac{4-(d-2)(p-1)}{2}}
\big\|
\sigma_{n}^{j}\widetilde{\psi}^{j}
\big\|_{W_{p+1}(I)}^{p-1}
\| \lambda_{n}^{j} \sigma_{n}^{j}\widetilde{\psi}^{j} \|_{V_{2+\frac{4}{d}}(I)}
\\[6pt]
&\quad + 
\lim_{n\to \infty}
(\lambda_{n}^{j})^{\frac{4-(d-2)(p-1)}{2}}
\big\|
\sigma_{n}^{j}\widetilde{\psi}^{j}
\big\|_{W_{p+1}(I)}^{p-1}
\| \nabla \sigma_{n}^{j}\widetilde{\psi}^{j} \|_{V_{2+\frac{4}{d}}(I)}
\\[6pt]
&\le
\lim_{n\to \infty}
\| \lambda_{n}^{j} \sigma_{n}^{j} \widetilde{\psi}^{j} 
\|_{V_{2+\frac{4}{d}}(I)}^{(p-1)(1-s_{p})}
\| \sigma_{n}^{j} \widetilde{\psi}^{j} 
\|_{W_{2^{*}}(I)}^{(p-1)s_{p}}
\| \lambda_{n}^{j} \sigma_{n}^{j} \widetilde{\psi}^{j} 
\|_{V_{2+\frac{4}{d}}(I)}
\\[6pt]
&\quad + 
\lim_{n\to \infty}
\| \lambda_{n}^{j} \sigma_{n}^{j}\widetilde{\psi}^{j}
\big\|_{V_{2+\frac{4}{d}}(I)}^{(p-1)(1-s_{p})}
\| \lambda_{n}^{j} \sigma_{n}^{j}\widetilde{\psi}^{j}
\big\|_{W_{2^{*}}(I)}^{(p-1)s_{p}}
\| \nabla \sigma_{n}^{j}\widetilde{\psi}^{j} \|_{V_{2+\frac{4}{d}}(I)}
\\[6pt]
&=0.
\end{split} 
\end{equation}
Then, the desired result \eqref{14/07/02/10:31} follows from \eqref{17/11/11/17:59} and \eqref{14/07/02/11:50}.
\end{proof}

\begin{lemma}\label{17/11/25/12:18}
For any $\delta >0$ and any $k\ge 1$, there exists a number $N(\delta,k)$ such that for any $n\ge N(\delta,k)$, 
\begin{equation}\label{17/11/25/12:19}
\sum_{{1\le j \le k}\atop{\lambda_{\infty}^{j}= \infty}}
\| \langle \nabla \rangle  
F\big[ \psi_{n}^{j}\big]
\|_{L_{t,x}^{\frac{2(d+2)}{d+4}}(\mathbb{R})}
\le \delta, 
\end{equation}
where the sum is taken over all integers $j$ satisfying $1\le j \le k$ and $\lambda_{\infty}^{j}=\infty$.  
\end{lemma}
\begin{proof}[Proof of Lemma \ref{17/11/25/12:18}]
Let $q$ denote $p$ or $2^{*}-1$, and let $j\ge 1$ be a number for which $\lambda_{\infty}^{j}=\infty$, so that $\psi_{n}^{j}=G_{n}^{j}\sigma_{n}^{j} e^{it\Delta}\widetilde{u}^{j}$. Then, we see from H\"older's inequality, Sobolev's embedding, Strichartz' estimate, Lemma \ref{14/09/01/11:30} with $\sigma_{\infty}^{j}=0$ (see \eqref{14/01/26/14:40}) that 
\begin{equation}\label{17/11/25/14:08}
\begin{split}
&
\lim_{n\to \infty}
\|   
| \psi_{n}^{j}|^{q-1} \psi_{n}^{j}
\|_{L_{t,x}^{\frac{2(d+2)}{d+4}}(\mathbb{R})}
\\[6pt]
&\le 
\lim_{n\to \infty} \|   
\psi_{n}^{j}
\|_{W_{q+1}(I_{n})}^{q-1}
\| \psi_{n}^{j}
\|_{V_{2+\frac{4}{d}}(\mathbb{R})}
\\[6pt]
&\lesssim 
\lim_{n\to \infty}(\lambda_{n}^{j})^{(q-1)(1-s_{q})}
\|   |\nabla|^{s_{q}}
\sigma_{n}^{j} e^{it\Delta}\widetilde{u}^{j}
\|_{V_{q+1}(I_{n})}^{q-1}
\lambda_{n}^{j} \| \sigma_{n}^{j} e^{it\Delta}\widetilde{u}^{j}
\|_{V_{2+\frac{4}{d}}(\mathbb{R})}
\\[6pt]
&\lesssim 
\lim_{n\to \infty}(\lambda_{n}^{j})^{(q-1)(1-s_{q})}
\|   |\nabla|^{s_{q}}
\sigma_{n}^{j}  \widetilde{u}^{j} \|_{L^{2}}^{q-1}
\|   \lambda_{n}^{j} \sigma_{n}^{j}  \widetilde{u}^{j} \|_{L^{2}}
\\[6pt]
&\lesssim 
\lim_{n\to \infty}\| \lambda_{n}^{j} \sigma_{n}^{j}  \widetilde{u}^{j} \|_{L^{2}}^{(q-1)(1-s_{q})}
\| \nabla \sigma_{n}^{j}  \widetilde{u}^{j} \|_{L^{2}}^{(q-1)s_{q}}
\| \lambda_{n}^{j} \sigma_{n}^{j}  \widetilde{u}^{j} \|_{L^{2}}
=0. 
\end{split}
\end{equation}
Similarly, we can verify that 
\begin{equation}\label{17/11/25/14:37}
\begin{split}
\lim_{n\to \infty}
\|  \nabla \{  
| \psi_{n}^{j}|^{q-1} \psi_{n}^{j} \}
\|_{L_{t,x}^{\frac{2(d+2)}{d+4}}(\mathbb{R})}
&\lesssim 
\lim_{n\to \infty}
\| | \psi_{n}^{j}|^{q-1} | \nabla \psi_{n}^{j}|
\|_{L_{t,x}^{\frac{2(d+2)}{d+4}}(\mathbb{R})}
\\[6pt]
&\le 
\lim_{n\to \infty} \|   
\psi_{n}^{j}
\|_{W_{q+1}(I_{n})}^{q-1}
\| \nabla \psi_{n}^{j}
\|_{V_{2+\frac{4}{d}}(\mathbb{R})}
=0. 
\end{split}
\end{equation}
Thus, we find from \eqref{17/11/25/14:08} and \eqref{17/11/25/14:37} that for any $\delta>0$ and any $k\ge 1$, there exists $N(\delta,k)\ge 1$ such that for any $n\ge N(\delta,k)$, 
\begin{equation}\label{17/11/25/14:33}
\sum_{{1\le j \le k}\atop{\lambda_{\infty}^{j}= \infty}}
\|\langle \nabla \rangle \big\{ | \psi_{n}^{j}|^{q-2} \psi_{n}^{j} 
\big\}
\|_{L_{t,x}^{\frac{2(d+2)}{d+4}}(\mathbb{R})}\le \delta,
\end{equation}
which gives us the desired result. 
\end{proof}

\begin{lemma}\label{14/06/18/10:16}
Let $j \ge 1$, and let $I$ be an interval on which $\widetilde{\psi}^{j}$ exists. Then, we have 
\begin{equation}
\label{14/06/18/11:27}
\sup_{n\ge 1}\| \langle \nabla \rangle \psi_{n}^{j}\|_{St(I_{n}^{j})}
\lesssim 
\|\langle \nabla \rangle \widetilde{\psi}^{j}\|_{St(I)},
\end{equation}
where $I_{n}^{j}:=( (\lambda_{n}^{j})^{2}\inf{I}+t_{n}^{j}, \ (\lambda_{n}^{j})^{2}\sup{I}+t_{n}^{j})$. 
\end{lemma}
\begin{proof}[Proof of Lemma \ref{14/06/18/10:16}]
We can easily verify that 
\begin{align}
\label{14/06/18/10:17}
&\|\langle \nabla \rangle \psi_{n}^{j}\|_{L_{t}^{\infty}L_{x}^{2}(I_{n}^{j})}
\le 
\|\lambda_{n}^{j} \sigma_{n}^{j} \widetilde{\psi}^{j}\|_{L_{t}^{\infty}L_{x}^{2}(I)}
+
\| \nabla \sigma_{n}^{j}\widetilde{\psi}^{j}\|_{L_{t}^{\infty}L_{x}^{2}(I)},
\\[6pt]
\label{14/06/18/11:00}
&\|\langle \nabla \rangle \psi_{n}^{j}\|_{L_{t}^{2}L_{x}^{2^{*}}(I_{n}^{j})}
\le 
\|\lambda_{n}^{j} \sigma_{n}^{j} \widetilde{\psi}^{j}\|_{L_{t}^{2}L_{x}^{2^{*}}(I)}
+
\| \nabla \sigma_{n}^{j}\widetilde{\psi}^{j}\|_{L_{t}^{2}L_{x}^{2^{*}}(I)}. 
\end{align}
When $\lambda_{\infty}^{j}\in \{0,1\}$, these estimates \eqref{14/06/18/10:17} and \eqref{14/06/18/11:00} together with Lemma \ref{14/06/21/15:10} prove the lemma. Assume $\lambda_{\infty}^{j}= \infty$, so that $\widetilde{\psi}^{j}(t)=e^{it\Delta}\widetilde{u}^{j}$. Then, we see from  \eqref{14/06/18/10:17}, \eqref{14/06/18/11:00} and Strichartz' estimate that 
\begin{equation}\label{17/11/24/15:33}
\begin{split}
\|\langle \nabla \rangle \psi_{n}^{j}\|_{L_{t}^{\infty}L_{x}^{2}(I_{n}^{j})}
+
\|\langle \nabla \rangle \psi_{n}^{j}\|_{L_{t}^{2}L_{x}^{2^{*}}(I_{n}^{j})}
& \lesssim 
\|\lambda_{n}^{j} \sigma_{n}^{j} \widetilde{u}^{j}\|_{L^{2}}
+
\| \nabla  \sigma_{n}^{j} \widetilde{u}^{j}\|_{L^{2}}
\\[6pt]
&\lesssim 
\| \langle \nabla \rangle \widetilde{u}^{j} \|_{L^{2}}
=
\| \langle \nabla \rangle \widetilde{\psi}^{j} \|_{L_{t}^{\infty}L_{x}^{2}(I)}
.
\end{split}
\end{equation}
Thus, we have completed the proof. 
\end{proof}

\begin{lemma}[cf. Lemma 6.8 in \cite{AIKN2}]\label{14/01/26/17:46}
There exists a number $J_{0}$ such that $I_{\max}^{j}=\mathbb{R}$ for any $j>J_{0}$; and for any $r\ge 2$, 
\begin{equation}
\label{14/02/01/16:22}
\sum_{j> J_{0}} 
\|\langle \nabla \rangle \widetilde{\psi}^{j}\|_{St(\mathbb{R})}^{r}
\lesssim 
\sum_{j> J_{0}}\| \widetilde{u}^{j}\|_{H^{1}}^{r}
<\infty.
\end{equation} 
\end{lemma}
\begin{proof}[Proof of Lemma \ref{14/01/26/17:46}]
It follows from \eqref{14/01/26/14:07} and the uniform boundedness \eqref{14/01/26/10:45} that 
\begin{equation}\label{14/01/26/21:41}
\lim_{n\to \infty}\sum_{j=1}^{k}
\big\| g_{n}^{j}\sigma_{n}^{j} e^{-i\frac{t_{n}^{j}}{(\lambda_{n}^{j})^{2}}\Delta }
\widetilde{u}^{j} \big\|_{H^{1}}^{2}\lesssim 1
\end{equation}
for all $k\ge 1$. Furthermore, we see from Lemma \ref{14/09/01/11:30} that: if $\lambda_{\infty}^{j}=0$, then  
\begin{align}
\label{17/11/09/10:25}
&
\lim_{n\to \infty}
\big\| g_{n}^{j}\sigma_{n}^{j} e^{-i\frac{t_{n}^{j}}{(\lambda_{n}^{j})^{2}}\Delta }\widetilde{u}^{j} 
\big\|_{L^{2}}
=
\lim_{n\to \infty}
\big\| \lambda_{n}^{j} \sigma_{n}^{j} \widetilde{u}^{j} \big\|_{L^{2}}
=0,
\\[6pt]
\label{14/09/06/17:23}
&
\lim_{n\to \infty}
\big\| \nabla g_{n}^{j}\sigma_{n}^{j} e^{-i\frac{t_{n}^{j}}{(\lambda_{n}^{j})^{2}}\Delta }\widetilde{u}^{j} 
\big\|_{L^{2}}
=
\lim_{n\to \infty}\big\| \nabla  \sigma_{n}^{j} \widetilde{u}^{j} \big\|_{L^{2}}=
\| \widetilde{u}^{j} \|_{H^{1}};
\end{align}
and if $\lambda_{\infty}^{j}\in \{1,\infty\}$, then  
\begin{equation}\label{14/09/06/17:50}
\lim_{n\to \infty}
\big\| g_{n}^{j}\sigma_{n}^{j} e^{-i\frac{t_{n}^{j}}{(\lambda_{n}^{j})^{2}}\Delta }\widetilde{u}^{j} 
\big\|_{H^{1}}
=
\| \widetilde{u}^{j} \|_{H^{1}}.
\end{equation}
Putting \eqref{14/01/26/21:41} through \eqref{14/09/06/17:50} together, we obtain that 
\begin{equation}\label{14/09/06/17:53}
\sum_{j=1}^{\infty}\| \widetilde{u}^{j} \|_{H^{1}}^{2}
=
\lim_{k\to \infty} \sum_{j=1}^{k}\| \widetilde{u}^{j} \|_{H^{1}}^{2} 
\lesssim 
\lim_{k\to \infty}\sum_{j=1}^{k}\lim_{n\to \infty}
\big\| g_{n}^{j}\sigma_{n}^{j} e^{-i\frac{t_{n}^{j}}{(\lambda_{n}^{j})^{2}}\Delta }\widetilde{u}^{j} \big\|_{H^{1}}^{2}\lesssim 1
. 
\end{equation}
In particular, for any $\delta>0$, we can take $J(\delta)\ge 1$ such that for any $j\ge J(\delta)$, 
\begin{equation}\label{15/02/23/10:19}
\| \widetilde{u}^{j} \|_{H^{1}} < \delta.
\end{equation}
Note here that if $\lambda_{\infty}^{j}=0$, then \eqref{15/02/23/10:19} is rewritten as $\| \nabla \sigma_{\infty}^{j} \widetilde{u}^{j} \|_{L^{2}}^{2}< \delta$. Since $\widetilde{\psi}^{j}$ satisfies \eqref{14/01/29/11:41} and \eqref{14/01/27/14:51}, the small-data theory (Lemma \ref{14/01/30/10:19}) together with \eqref{15/02/23/10:19} implies that there exists a number $J_{0}$ such that  for any $j\ge J_{0}$, 
\begin{equation}\label{14/01/26/21:47}
\|\langle \nabla \rangle \widetilde{\psi}^{j}\|_{St(\mathbb{R})}
\lesssim  
\| \widetilde{u}^{j}\|_{H^{1}}.
\end{equation}
Then, the claim \eqref{14/02/01/16:22} follows from \eqref{14/09/06/17:53}, \eqref{14/01/26/21:47} and Jensen's inequality.
\end{proof}

\begin{lemma}\label{14/01/26/17:10}
For any number $k\ge 1$, there exists a number $N(k)$ such that for any $n\ge N(k)$, \begin{equation}\label{14/01/29/21:05}
\sum_{j=1}^{k}\mathcal{I}_{\omega}\big( g_{n}^{j}\sigma_{n}^{j}e^{-i\frac{t_{n}^{j}}{(\lambda_{n}^{j})^{2}}\Delta}\widetilde{u}^{j} \big) 
+
\mathcal{I}_{\omega}(w_{n}^{k}) 
\le 
m_{\omega}-\frac{\kappa_{1}(R_{*})}{10 d(p-1)}.
\end{equation}
\end{lemma}
\begin{proof}[Proof of Lemma \ref{14/01/26/17:10}]
It follows from \eqref{14/09/06/14:57}, \eqref{14/01/26/11:03} and $\mathcal{S}_{\omega}(\psi_{n})\le m_{\omega}+\varepsilon_{n}$ that for any numbers $n$ and $k$, 
\begin{equation}\label{14/01/26/16:48}
\begin{split}
m_{\omega}-\frac{\kappa_{1}(R_{*})}{2d(p-1)}
&\ge 
m_{\omega}+\varepsilon_{n}
-\frac{2}{d(p-1)}\mathcal{K}(\psi_{n}(0))
\\[6pt]
&\ge 
\mathcal{S}_{\omega}(\psi_{n})
-
\frac{2}{d(p-1)}\mathcal{K}(\psi_{n}(0))
=
\mathcal{I}_{\omega}(\psi_{n}(0))
.
\end{split}
\end{equation}
This together with \eqref{14/01/26/14:11} gives us the desired result \eqref{14/01/29/21:05}. 
\end{proof}

\begin{lemma}\label{14/01/29/12:01}
For any number $k\ge 1$, there exists a number $N(k)$ such that for any $n\ge N(k)$ and any $j \le k$ for which the linear profile $\widetilde{u}^{j}$ is non-trivial,   
\begin{align}
\label{14/01/28/09:07}
&\mathcal{H}\big( 
g_{n}^{j}\sigma_{n}^{j}e^{-i\frac{t_{n}^{j}}{(\lambda_{n}^{j})^{2}}\Delta}\widetilde{u}^{j}\big)
>
\frac{1}{2}\mathcal{K}\big( 
g_{n}^{j}\sigma_{n}^{j}e^{-i\frac{t_{n}^{j}}{(\lambda_{n}^{j})^{2}}\Delta}\widetilde{u}^{j}
\big) > 0,
\\[6pt]
\label{14/01/29/12:09}
&\mathcal{H}(w_{n}^{k})\ge \frac{1}{2}\mathcal{K}(w_{n}^{k})\ge 0.
\end{align}
\end{lemma}
\begin{proof}[Proof of Lemma \ref{14/01/29/12:01}]
Since $\mathcal{I}_{\omega}$ is non-negative, Lemma \ref{14/01/26/17:10} 
 together with \eqref{12/03/23/18:17} shows that for any $k\ge 1$, there exists a number $N(k)$ such that for any $n\ge N(k)$ and any $j\le k$ for which $\widetilde{u}^{j}$ is non-trivial, 
\begin{equation}\label{15/02/23/10:53}
\mathcal{K}\big( 
g_{n}^{j}\sigma_{n}^{j}e^{-i\frac{t_{n}^{j}}{(\lambda_{n}^{j})^{2}}\Delta}\widetilde{u}^{j}
\big) >0, 
\quad 
\mathcal{K}(w_{n}^{k})\ge 0. 
\end{equation}
This together with \eqref{14/09/19/16:12} gives us the desired result. 
\end{proof}

\begin{lemma}\label{14/01/29/11:33}
There is at most one number $j_{0}$ such that $\lambda_{\infty}^{j_{0}} \in \{0,1\}$ and 
\begin{equation}\label{14/01/29/11:31}
\left\{ 
\begin{array}{ccc}
\mathcal{S}_{\omega}(\sigma_{\infty}^{j_{0}}\widetilde{\psi}^{j_{0}}) \ge m_{\omega}
&\mbox{if}& \lambda_{\infty}^{j_{0}}=1,
\\[6pt]
\mathcal{H}^{\ddagger}(\sigma_{\infty}^{j_{0}}\widetilde{\psi}^{j_{0}}) 
\ge 
\mathcal{H}^{\ddagger}(W)
&\mbox{if}& \lambda_{\infty}^{j_{0}}=0.
\end{array}
\right.
\end{equation}
\end{lemma}
\begin{proof}[Proof of Lemma \ref{14/01/29/11:33}] 
First, we consider a number $k$ such that     
\begin{equation}\label{14/01/29/20:11}
\mathcal{S}_{\omega} \big( g_{n}^{k}\sigma_{n}^{k}e^{-i\frac{t_{n}^{k}}{(\lambda_{n}^{k})^{2}}\Delta}\widetilde{u}^{k} \big)
\le \frac{2}{3}m_{\omega}.
\end{equation}
Then, it follows from Lemma \ref{14/09/01/11:30}, Lemma \ref{14/09/07/15:55}, \eqref{14/01/29/11:41}, \eqref{14/01/29/20:27}, \eqref{15/02/24/20:55} and \eqref{12/03/23/17:48} that for  any sufficiently large number $n$,
\begin{equation}\label{14/01/29/20:13}
\left\{ \begin{array}{lcc}
\mathcal{S}_{\omega}( \sigma_{\infty}^{k}\widetilde{\psi}^{k})
=
\mathcal{S}_{\omega}\big( \sigma_{\infty}^{k}\widetilde{\psi}^{k}\Big( -\frac{t_{n}^{k}}{(\lambda_{n}^{k})^{2}} \Big)\big)<m_{\omega}
&\mbox{if}& \lambda_{\infty}^{k}=1,
\\[9pt]
\mathcal{H}^{\ddagger}(\sigma_{\infty}^{k}\widetilde{\psi}^{k})
=
\mathcal{H}^{\ddagger}\big( \sigma_{\infty}^{k}\widetilde{\psi}^{k}\Big( -\frac{t_{n}^{k}}{(\lambda_{n}^{k})^{2}} \Big)
\big)
<
\mathcal{H}^{\ddagger}(W)
&\mbox{if}& \lambda_{\infty}^{k}=0.
\end{array} \right.
\end{equation}
Thus, it suffices for the desired result to show that there is at most one number $j_{0}$ such that for any sufficiently large $n$, 
\begin{equation}\label{14/01/29/19:56}
\mathcal{S}_{\omega}\big( g_{n}^{j_{0}}\sigma_{n}^{j_{0}}e^{-i\frac{t_{n}^{j_{0}}}{(\lambda_{n}^{j_{0}})^{2}}\Delta}\widetilde{u}^{j_{0}} \big)\ge \frac{2}{3}m_{\omega}.
\end{equation}
We see from Lemma \ref{14/01/29/12:01} that for any $j\ge 1$ and any sufficiently large $n$,  
\begin{equation}\label{14/01/29/11:56}
\mathcal{S}_{\omega} \big( g_{n}^{j}\sigma_{n}^{j}e^{-i\frac{t_{n}^{j}}{(\lambda_{n}^{j})^{2}}\Delta}\widetilde{u}^{j} \big) \ge 0.
\end{equation}
Hence, it follows from \eqref{14/01/26/14:09} and $\mathcal{S}_{\omega}(\psi_{n}(0))<m_{\omega}+\varepsilon_{n}$ that there are at most one linear profile satisfying \eqref{14/01/29/19:56}. 
\end{proof}

Now, using Lemma \ref{14/01/29/11:33} and reordering the indices, we may assume that for any $j\ge 2$, 
\begin{equation}\label{14/02/01/17:53}
\left\{ \begin{array}{ccc}
\mathcal{S}_{\omega}(\sigma_{\infty}^{j}\widetilde{\psi}^{j})<m_{\omega}
&\mbox{if}& \lambda_{\infty}^{j}=1,
\\[6pt]
\mathcal{H}^{\ddagger}(\sigma_{\infty}^{j}\widetilde{\psi}^{j})
<
\mathcal{H}^{\ddagger}(W)
&\mbox{if}& \lambda_{\infty}^{j}=0.
\end{array}\right.
\end{equation}

\begin{lemma}\label{14/01/29/20:50}
Assume \eqref{14/02/01/17:53}. Then, for any $j\ge 1$ for which the nonlinear profile $\widetilde{\psi}^{j}$ is non-trivial and $\lambda_{\infty}^{j}\in \{0,1\}$, there exists a number $N(j)$ such that for any $n\ge N(j)$,  
\begin{equation}\label{14/01/28/11:45}
\left\{ \begin{array}{lcl}
\mathcal{K}\big( 
\widetilde{\psi}^{j}
\Big( -\frac{t_{n}^{j}}{(\lambda_{n}^{j})^{2}}\Big)
\big) > 0 \qquad 
&\mbox{if}& \lambda_{\infty}^{j}=1,
\\[9pt]
\mathcal{K}^{\ddagger}\big( 
\sigma_{\infty}^{j}\widetilde{\psi}^{j}\Big(-\frac{t_{n}^{j}}{(\lambda_{n}^{j})^{2}} \Big) >0
&\mbox{if}& \lambda_{\infty}^{j}=0.
\end{array} \right.
\end{equation} 
Furthermore, if $\lambda_{\infty}^{j}=0$ and $\mathcal{H}^{\ddagger}(\sigma_{\infty}^{j}\widetilde{\psi}^{j})
<
\mathcal{H}^{\ddagger}(W)$, then for any $n\ge N(j)$, 
\begin{equation}\label{15/02/24/21:28}
\big\|
\nabla \sigma_{\infty}^{j}\widetilde{\psi}^{j}
\Big(-\frac{t_{n}^{j}}{(\lambda_{n}^{j})^{2}} \Big)
\big\|_{L^{2}}^{2}
<
\|\nabla W\|_{L^{2}}^{2}.
\end{equation}
\end{lemma}
\begin{proof}[Proof of Lemma \ref{14/01/29/20:50}] 
We see from Lemma \ref{14/01/26/17:10} that there exists a number $N(j)$ such that for any $n \ge N(j)$, 
\begin{equation}\label{14/01/26/22:17}
\mathcal{I}_{\omega}\big(g_{n}^{j}\sigma_{n}^{j}  e^{-i\frac{t_{n}^{j}}{(\lambda_{n}^{j})^{2}}\Delta} \widetilde{u}^{j} \big) 
\le m_{\omega}-\frac{\kappa_{1}(R_{*})}{100 d(p-1)} 
.
\end{equation} 
When $\lambda_{\infty}^{j}=1$, the desired result \eqref{14/01/28/11:45} follows from \eqref{14/01/29/20:27}, \eqref{14/01/26/22:17} and \eqref{12/03/23/18:17}. In order to prove \eqref{14/01/28/11:45} in the case $\lambda_{\infty}^{j}=0$, note that it follows from \eqref{14/09/20/13:55} and \eqref{14/01/26/22:17} that for any $n \ge N(j)$, 
\begin{equation}\label{18/01/28/18:06}
\begin{split}
&
\mathcal{I}_{\omega}^{\ddagger}\big(g_{n}^{j}\sigma_{n}^{j}  e^{-i\frac{t_{n}^{j}}{(\lambda_{n}^{j})^{2}}\Delta} \widetilde{u}^{j} \big) 
=
\mathcal{H}^{\ddagger}\big(g_{n}^{j}\sigma_{n}^{j}  e^{-i\frac{t_{n}^{j}}{(\lambda_{n}^{j})^{2}}\Delta} \widetilde{u}^{j} \big) 
-
\frac{2}{d(p-1)}
\mathcal{K}^{\ddagger}\big(g_{n}^{j}\sigma_{n}^{j}  e^{-i\frac{t_{n}^{j}}{(\lambda_{n}^{j})^{2}}\Delta} \widetilde{u}^{j} \big) 
\\[6pt]
&\le 
\mathcal{S}_{\omega}\big(g_{n}^{j}\sigma_{n}^{j}  e^{-i\frac{t_{n}^{j}}{(\lambda_{n}^{j})^{2}}\Delta} \widetilde{u}^{j} \big) 
-
\frac{2}{d(p-1)}
\mathcal{K}\big(g_{n}^{j}\sigma_{n}^{j}  e^{-i\frac{t_{n}^{j}}{(\lambda_{n}^{j})^{2}}\Delta} \widetilde{u}^{j} \big) 
\\[6pt]
&=
\mathcal{I}_{\omega}\big(g_{n}^{j}\sigma_{n}^{j}  e^{-i\frac{t_{n}^{j}}{(\lambda_{n}^{j})^{2}}\Delta} \widetilde{u}^{j} \big) 
\le 
m_{\omega}-\frac{\kappa_{1}(R_{*})}{100 d(p-1)} 
.
\end{split} 
\end{equation} 
Furthermore, we see from \eqref{18/01/28/18:06} and \eqref{14/01/29/20:27} that for any $n \ge N(j)$, 
\begin{equation}\label{15/02/24/20:52}
\mathcal{I}_{\omega}^{\ddagger} \big( \sigma_{\infty}^{j}\widetilde{\psi}^{j}\Big(-\frac{t_{n}^{j}}{(\lambda_{n}^{j})^{2}} \Big) \big) 
=
\lim_{n\to \infty}\mathcal{I}^{\ddagger}\big( 
g_{n}^{j}\sigma_{n}^{j} e^{-i\frac{t_{n}^{j}}{(\lambda_{n}^{j})^{2}}\Delta}\widetilde{u}^{j} \big)
\le 
m_{\omega}-\frac{\kappa_{1}(R_{*})}{100 d(p-1)}
.
\end{equation}
Thus, when $\lambda_{\infty}^{j}=0$, \eqref{15/08/10/17:31} together with \eqref{15/02/24/20:55}, \eqref{12/03/23/17:48} and \eqref{15/02/24/20:52} shows that for any $n\ge N(j)$, 
\begin{equation}\label{15/02/24/11:13}
\mathcal{K}^{\ddagger}\big( 
\sigma_{\infty}^{j}\widetilde{\psi}^{j}\Big(-\frac{t_{n}^{j}}{(\lambda_{n}^{j})^{2}} \Big) 
\big)
=
\big\|\nabla \sigma_{\infty}^{j}\widetilde{\psi}^{j}\Big(-\frac{t_{n}^{j}}{(\lambda_{n}^{j})^{2}} \Big) \big\|_{L^{2}}^{2}
-
\big\|\sigma_{\infty}^{j}\widetilde{\psi}^{j}\Big(-\frac{t_{n}^{j}}{(\lambda_{n}^{j})^{2}} \Big) \big\|_{L^{2^{*}}}^{2^{*}}
>0,
\end{equation}
which completes the proof of \eqref{14/01/28/11:45}.
\par 
Next, we shall show \eqref{15/02/24/21:28}. Assume $\lambda_{\infty}^{j}=0$ and $\mathcal{H}^{\ddagger}(\sigma_{\infty}^{j}\widetilde{\psi}^{j})
<
\mathcal{H}^{\ddagger}(W)$. Then, we see from \eqref{15/02/24/11:13}, \eqref{15/02/24/20:55} and \eqref{12/08/22/13:32} that for any $n\ge N(j)$,
\begin{equation}\label{14/01/28/21:10}
\begin{split}
&\frac{1}{d}\big\|\nabla \sigma_{\infty}^{j}\widetilde{\psi}^{j}\Big(-\frac{t_{n}^{j}}{(\lambda_{n}^{j})^{2}} \Big) \big\|_{L^{2}}^{2}
\\[6pt]
&=
\frac{1}{2}\big\|\nabla \sigma_{\infty}^{j}\widetilde{\psi}^{j}\Big(-\frac{t_{n}^{j}}{(\lambda_{n}^{j})^{2}} \Big) \big\|_{L^{2}}^{2}
-
\frac{1}{2^{*}}\big\|\nabla \sigma_{\infty}^{j}\widetilde{\psi}^{j}\Big(-\frac{t_{n}^{j}}{(\lambda_{n}^{j})^{2}} \Big) \big\|_{L^{2}}^{2}
\\[6pt]
&\le 
\frac{1}{2}\big\|\nabla \sigma_{\infty}^{j}\widetilde{\psi}^{j}\Big(-\frac{t_{n}^{j}}{(\lambda_{n}^{j})^{2}} \Big) \big\|_{L^{2}}^{2}
-
\frac{1}{2^{*}}\big\|\sigma_{\infty}^{j}\widetilde{\psi}^{j}\Big(-\frac{t_{n}^{j}}{(\lambda_{n}^{j})^{2}} \Big) \big\|_{L^{2^{*}}}^{2^{*}}
\\[6pt]
&=
\mathcal{H}^{\ddagger}\big( 
\sigma_{\infty}^{j}\widetilde{\psi}^{j}
\big) 
<
\mathcal{H}^{\ddagger}(W)
=
\frac{1}{d}\|\nabla W\|_{L^{2}}^{2}
.
\end{split} 
\end{equation}
Thus, we have completed the proof. 
\end{proof}

We see from Lemma \ref{14/01/29/20:50} and \eqref{14/02/01/17:53} that for any $j\ge 2$ for which $\widetilde{\psi}^{j}$ is non-trivial and $\lambda_{\infty}^{j}\neq \infty$, 
\begin{equation}\label{14/02/05/11:47}
\left\{ \begin{array}{ccc}
\sigma_{\infty}^{j}\widetilde{\psi}^{j} \in PW_{\omega,+} &\mbox{if}& \lambda_{\infty}^{j}=1,
\\[6pt]
\sigma_{\infty}^{j}\widetilde{\psi}^{j} \in PW_{+}^{\ddagger} &\mbox{if}& \lambda_{\infty}^{j}=0. 
\end{array}\right.
\end{equation}
Furthermore, since $\widetilde{\psi}^{j}(t)=e^{it\Delta}\widetilde{u}^{j}$ for $j\ge 1$ with $\lambda_{\infty}^{j}=\infty$, it follows from Theorem \ref{15/04/05/15:27}, Theorem \ref{15/03/24/16:05} and Lemma \ref{14/01/26/17:46} that $I_{\max}^{j}=\mathbb{R}$ for any $j\ge 2$, and 
\begin{equation}
\label{14/02/01/17:58}
\sup_{j\ge 2} \|\langle \nabla \rangle \widetilde{\psi}^{j}\|_{St(\mathbb{R})}<\infty. 
\end{equation}

\subsubsection{Existence of bad profile}\label{15/07/11/11:29}

In the previous section, we showed that there are at most one bad profile (see Lemma \ref{14/01/29/11:33}), and the candidate is $\widetilde{\psi}^{1}$ (see \eqref{14/02/01/17:58}). Our aim in  this section is to show that for any $T \in (T_{\min}^{1},T_{\max}^{1})$, 
\begin{equation}\label{15/07/11/14:25}
\|\sigma_{\infty}^{1} \widetilde{\psi}^{1} \|_{\mathcal{W}^{1}([T,T_{\max}^{1}))}
=\infty,
\end{equation}
where $\mathcal{W}^{1}$ is the function space defined by \eqref{14/02/08/16:35}  (see Proposition \ref{14/02/02/14:35} below). To this end, we observe properties of $\widetilde{\psi}^{1}$ on an interval where  
\begin{equation}\label{14/06/14/14:38}
\|\sigma_{\infty}^{1}\widetilde{\psi}^{1} \|_{\mathcal{W}^{1}(I)}
<\infty. 
\end{equation}

\begin{lemma}\label{14/02/02/18:02}
For any interval $I$ satisfying \eqref{14/06/14/14:38}, we can take a constant $A(I)>0$ such that 
\begin{equation}\label{14/02/02/18:04}
\|\langle \nabla \rangle \widetilde{\psi}^{1} \|_{St(I)}\le A(I).
\end{equation}
\end{lemma}
\begin{proof}[Proof of Lemma \ref{14/02/02/18:02}]
When $\lambda_{\infty}^{1}=\infty$, $\widetilde{\psi}^{1}(t)=e^{it \Delta}\widetilde{u}^{1}$ and therefore the claim follows from Strichartz' estimate. Moreover, when $I$ is a compact interval in $I_{\max}^{1}$, the claim follows from the well-posedness theory. Thus, we may assume that $\lambda_{\infty}^{j}\in \{0,1\}$, and $\sup{I}=T_{\max}^{1}$ or $\inf{I}=T_{\min}^{1}$. We only consider the case where $\sup{I}=T_{\max}^{1}$ and $T_{\min}^{1}< \inf{I}$. The same proof is applicable for the other cases. Since $\sup{I}=T_{\max}^{1}$, the blowup criterion (cf. Theorem 4.1 in \cite{AIKN2} and Lemma 2.11 in \cite{Kenig-Merle}) together with $\eqref{14/06/14/14:38}$ shows that $T_{\max}^{1}=\infty$ and   
\begin{equation}\label{14/02/02/18:03}
\|\sigma_{\infty}^{1}\widetilde{\psi}^{1} \|_{\mathcal{W}^{1}([\inf{I},\infty))}
<\infty.
\end{equation}

We shall show that  
\begin{equation}\label{14/02/02/18:05}
\sup_{t\in [\inf{I},\infty)} \|\widetilde{\psi}^{1} (t) \|_{H^{1}}<\infty.
\end{equation}
Suppose for contradiction that \eqref{14/02/02/18:05} was false. Then, we could take a sequence $\{t_{n}\}$ in $[\inf{I},\infty)$ such that $\lim_{n\to \infty}t_{n}=\infty$ and 
\begin{equation}\label{14/02/02/20:54}
\lim_{n\to \infty}\big\|\widetilde{\psi}^{1}(t_{n})\big\|_{H^{1}}=\infty.
\end{equation}
Assume $\lambda_{\infty}^{1}=1$. Then, $\widetilde{\psi}^{1}$ is a solution to \eqref{12/03/23/17:57}, and we see from Strichartz' estimate and H\"older's inequality that for any $t_{m}<t_{n}$, 
\begin{equation}\label{14/02/02/20:58}
\begin{split}
&\big\|e^{-it_{m}\Delta}\widetilde{\psi}^{1}(t_{m})-e^{-it_{n}\Delta}\widetilde{\psi}^{1}(t_{n}) \big\|_{H^{1}}
\le 
\sup_{T\in [t_{m},t_{n}]}
\Big\|\int_{t_{m}}^{T}e^{-it\Delta} 
F[\widetilde{\psi}^{1}(t)]
\,dt \Big\|_{H^{1}}
\\[6pt]
&\lesssim 
\| \widetilde{\psi}^{1} \|_{W_{p+1}([t_{m},\infty))}^{p-1}
\|\langle \nabla \rangle \widetilde{\psi}^{1} \|_{V_{p+1}([t_{m},\infty))}+
\| \widetilde{\psi}^{1} \|_{W_{2^{*}}([t_{m},\infty))}^{\frac{4}{d-2}}
\|\langle \nabla \rangle \widetilde{\psi}^{1} \|_{V_{2^{*}}([t_{m}, \infty ))}.
\end{split}
\end{equation}
Moreover, we see from \eqref{14/02/02/18:03} and $\sigma_{\infty}^{1}=1$ that 
\begin{equation}\label{14/02/02/21:39}
\lim_{T \to \infty} \|\widetilde{\psi}^{1} \|_{W_{p+1}([T,\infty))\cap W_{2^{*}}([T,\infty))}=0.
\end{equation}
In particular, for any $\delta >0$, there exists $T(\delta)>0$ such that 
\begin{equation}\label{14/02/02/18:07}
\|\widetilde{\psi}^{1} \|_{W_{p+1}([T(\delta),\infty))\cap W_{2^{*}}([T(\delta),\infty))}\le \delta.
\end{equation}
An estimate similar to \eqref{14/02/02/20:58} together with \eqref{14/02/02/18:07} also yields that  
\begin{equation}\label{14/02/02/21:15}
\begin{split}
&
\| \langle \nabla \rangle \widetilde{\psi}^{1} 
\|_{V_{p+1}([T(\delta),\infty))\cap V_{2^{*}}([T(\delta),\infty))}
\\[6pt]
&\lesssim 
\|\widetilde{\psi}^{1}\big(T(\delta)\big)\|_{H^{1}}
+(\delta^{p-1}+\delta^{\frac{4}{d-2}})
\| \langle \nabla \rangle \widetilde{\psi}^{1} 
\|_{V_{p+1}([T(\delta),\infty))\cap V_{2^{*}}([T(\delta),\infty))}.
\end{split}
\end{equation}
Thus, we can take $\delta_{0}>0$ such that  
\begin{equation}\label{14/02/02/21:31}
\| \langle \nabla \rangle \widetilde{\psi}^{1} \|_{V_{p+1}([T(\delta_{0}),\infty))\cap V_{2^{*}}([T(\delta_{0}) ,\infty))} 
\lesssim 
\|\widetilde{\psi}^{1}\big(T(\delta_{0})\big)\|_{H^{1}}.
\end{equation}
Combining \eqref{14/02/02/20:58} with \eqref{14/02/02/21:39} and \eqref{14/02/02/21:31}, we find that $\{e^{-it_{n}\Delta}\widetilde{\psi}^{1}(t_{n})\}$ is a Cauchy sequence in $H^{1}(\mathbb{R}^{d})$. Similarly, when $\lambda_{\infty}^{1}=0$, we can verify that  $\{e^{-it_{n}\Delta}\sigma_{\infty}^{1} \widetilde{\psi}^{1}(t_{n})\}$ is a Cauchy sequence in $\dot{H}^{1}(\mathbb{R}^{d})$. Thus, in both cases $\lambda_{\infty}^{1}=1$ and $\lambda_{\infty}^{1}=0$, we can take $\phi_{+}\in H^{1}(\mathbb{R}^{d})$ such that 
\begin{equation}\label{14/02/02/21:41}
\lim_{n\to \infty}\| \widetilde{\psi}^{1}(t_{n})\|_{H^{1}}
=
\|\phi_{+}\|_{H^{1}}<\infty.
\end{equation}
However, this contradicts \eqref{14/02/02/20:54}. Hence, we have proved \eqref{14/02/02/18:05}. 
\par 
Now, we are able to show \eqref{14/02/02/18:04}. If $\lambda_{\infty}^{1}=1$, then $e[\widetilde{\psi}^{1}]\equiv 0$ (see \eqref{14/05/05/17:02} for the definition of $e[\widetilde{\psi}^{1}]$), and  Lemma \ref{14/06/11/11:16} together with \eqref{14/02/02/18:03} and \eqref{14/02/02/18:05} gives us the desired estimate \eqref{14/02/02/18:04}. If $\lambda_{\infty}^{1}=0$, then $e^{\ddagger}[\sigma_{\infty}^{1}\widetilde{\psi}^{1}]\equiv 0$. Moreover, it follows from \eqref{14/02/08/16:35}, \eqref{14/02/02/18:03} and \eqref{14/02/02/18:05} that 
\begin{equation}
\label{14/06/14/16:13}
\| \sigma_{\infty}^{1}\widetilde{\psi}^{1} \|_{W_{2^{*}}([\inf{I},\infty))}<\infty, 
\quad 
\sup_{t\in [\inf{I},\infty)}
\| \nabla \sigma_{\infty}^{1}\widetilde{\psi}^{1} \|_{L^{2}}
<\infty.
\end{equation} 
Hence, we see from Lemma \ref{14/06/11/11:16} that 
\begin{equation}\label{14/06/14/16:29}
\|\langle \nabla \rangle \widetilde{\psi}^{1} \|_{St([\inf{I},\infty))}
=
\| \nabla \sigma_{\infty}^{1}\widetilde{\psi}^{1} \|_{St([\inf{I},\infty))}
<\infty,
\end{equation}
which completes the proof. 
\end{proof}

Since the estimate \eqref{14/01/26/11:57} is insufficient to control the remainder $e^{it\Delta}w_{n}^{j}$ in the Strichartz space $\langle \nabla \rangle^{-1}St(\mathbb{R})$, we need the following estimate: 
\begin{lemma}\label{14/08/26/11:23}
Assume \eqref{14/02/01/17:53}, and let $q$ denote $p$ or $2^{*}-1$. Then, for any $\delta>0$, any interval $I$ satisfying \eqref{14/06/14/14:38}, and any number $j\ge 1$, there exists a number $K(\delta,I,j)$ with the following property: for any $k\ge K(\delta,I,j)$, there exists a number $N(\delta,I,j,k)$ such that for any $n\ge N(\delta,I,j,k)$ and $s\in \{0,1\}$,  
\begin{equation}
\label{14/02/26/11:25}
\big\| \psi_{n}^{j}| \nabla|^{s} e^{it\Delta}w_{n}^{k} 
\big\|_{L_{t,x}^{\frac{2(d+2)(q-1)}{d(q-1)+4}}(I_{n})}
\le \delta,
\end{equation}
where $I_{n}:=( (\lambda_{n}^{1})^{2}\inf{I}+t_{n}^{1}, \ (\lambda_{n}^{1})^{2}\sup{I}+t_{n}^{1})$. 
\end{lemma}
\begin{proof}[Proof of Lemma \ref{14/08/26/11:23}] 
Let $I$ be an interval on which \eqref{14/06/14/14:38} holds. Put $I^{1}:=I$ and  $I^{j}:=\mathbb{R}$ for  $j\ge 2$. Then, it follows from Lemma \ref{14/02/02/18:02} and \eqref{14/02/01/17:58} that there exists a constant $A(I)>0$ such that  
\begin{equation}\label{14/08/27/11:22}
\sup_{j\ge 1}\|\langle \nabla \rangle \widetilde{\psi}^{j}\|_{St(I^{j})}\le A(I).  
\end{equation} 
Moreover, we see from \eqref{14/02/01/10:50} and elementary computations that for any numbers $j, k, n \ge 1$  and $s\in \{0,1\}$,  
\begin{equation}\label{17/11/15/10:19}
\begin{split}
&\| \psi_{n}^{j}  |\nabla |^{s} e^{it\Delta} w_{n}^{k} \|_{L_{t,x}^{\frac{2(d+2)(q-1)}{d(q-1)+4}}(I_{n})}
\\[6pt]
&\le 
(\lambda_{n}^{j})^{2-s_{q}}
\| \sigma_{n}^{j}\widetilde{\psi}^{j}  
(G_{n}^{j})^{-1} |\nabla |^{s} e^{it\Delta}  w_{n}^{k}
\|_{L_{t,x}^{\frac{2(d+2)(q-1)}{d(q-1)+4}}(I^{j})}
.
\end{split}
\end{equation}
We consider the right-hand side of \eqref{17/11/15/10:19} according to $s$ and $\lambda_{\infty}^{j}$:
\\
\noindent 
{\bf Case 1.1.}~Assume $s=0$ and $\lambda_{\infty}^{j}\in \{0,1 \}$ in \eqref{17/11/15/10:19}. Then, we see from H\"older's inequality, Lemma \ref{14/06/21/15:10}, \eqref{14/08/27/11:22} and \eqref{15/02/28/16:41} that there exists  a number $K_{1}(\delta,I)$ with the following property: for any $k\ge K_{1}(\delta,I)$, there exists a number $N_{1}(\delta,I,k)$ such that  for any $j\ge 1$ and any $n \ge N_{1}(\delta,I, k)$, 
\begin{equation}\label{14/08/08/16:12}
\begin{split}
&
(\lambda_{n}^{j})^{2-s_{q}}
\| \sigma_{n}^{j}\widetilde{\psi}^{j}  
(G_{n}^{j})^{-1} e^{it\Delta}  w_{n}^{k}
\|_{L_{t,x}^{\frac{2(d+2)(q-1)}{d(q-1)+4}}(I^{j})}
\\[6pt]
&\le 
\lambda_{n}^{j}
\| 
\sigma_{n}^{j}\widetilde{\psi}^{j}
\|_{V_{2+\frac{4}{d}}(I^{j})} 
\| e^{it\Delta}  w_{n}^{k} \|_{W_{q+1}(\mathbb{R})}
\lesssim 
\| 
\langle \nabla \rangle \widetilde{\psi}^{j}
\|_{V_{2+\frac{4}{d}}(I^{j})} 
\| e^{it\Delta}  w_{n}^{k} \|_{W_{q+1}(\mathbb{R})}
\le 
\delta.
\end{split}
\end{equation}

\noindent 
{\bf Case 1.2.}~Assume $s=0$ and $\lambda_{\infty}^{j}=\infty$. Then, $\widetilde{\psi}^{j}(t)=e^{it\Delta}\widetilde{u}^{j}$ and we see from an estimate similar to \eqref{14/08/08/16:12}, Strichartz' estimate, Lemma \ref{14/09/01/11:30} and Lemma \ref{14/01/26/17:46} that for any $k\ge K_{1}(\delta,I)$, there exists a number $N_{1}(\delta,I,k)$ such that  for any $j\ge 1$ and any $n \ge N_{1}(\delta,I, k)$, 
\begin{equation}\label{18/01/28/10:07}
\begin{split}
&
(\lambda_{n}^{j})^{2-s_{q}}
\| \sigma_{n}^{j}\widetilde{\psi}^{j}  
(G_{n}^{j})^{-1} e^{it\Delta}  w_{n}^{k}
\|_{L_{t,x}^{\frac{2(d+2)(q-1)}{d(q-1)+4}}(I^{j})}
\\[6pt]
&\le 
\lambda_{n}^{j}
\| 
e^{it\Delta}\sigma_{n}^{j}\widetilde{u}^{j}
\|_{V_{2+\frac{4}{d}}(I^{j})} 
\| e^{it\Delta}  w_{n}^{k} \|_{W_{q+1}(\mathbb{R})}
\lesssim 
\| 
\langle \nabla \rangle \widetilde{u}^{j}
\|_{L^{2}} 
\| e^{it\Delta}  w_{n}^{k} \|_{W_{q+1}(\mathbb{R})}
\le 
\delta.
\end{split}
\end{equation}

\noindent 
{\bf Case 2.1.}~Assume $s=1$ and $\lambda_{\infty}^{j}\in \{0,1\}$ in \eqref{17/11/15/10:19}. Take a sequence $\{v_{m}^{j} \}_{m\ge 1}$ of smooth functions on $\mathbb{R}^{d}\times \mathbb{R}$ with the following properties: 
\begin{equation}\label{14/08/26/11:53}
\lim_{m\to \infty}
\big\| \langle \nabla \rangle \big( 
\widetilde{\psi}^{j}-v_{m}^{j} 
\big)
\big\|_{V_{2+\frac{4}{d}}(I^{j})\cap  V_{2^{*}}(I^{j})}=0, 
\end{equation}
and for each $m\ge 1$, there exist $R_{m}^{j}>0$,  $T_{m}^{j}>0$ and $a_{m}^{j} \in \mathbb{R}$ such that 
\begin{equation}\label{17/11/15/10:42}
E_{m}^{j}:=\big\{ (x,t)\colon |x|\le R_{m}^{j}, \ |t-a_{m}^{j}|\le T_{m}^{j} \big\}\subset \mathbb{R}^{d}\times I^{j}
\end{equation} 
and the support of $v_{m}^{j}$ is contained in $E_{m}^{j}$.

We see from \eqref{14/08/27/11:22} and \eqref{14/08/26/11:53} that for each $j\ge 1$, there exists a number $M_{1}(I,j)$ such that for any $m\ge M_{1}(I,j)$, 
\begin{equation}\label{14/08/08/16:09}
\| \langle \nabla \rangle v_{m}^{j} 
\|_{V_{2+\frac{4}{d}}(I^{j})\cap V_{2^{*}}(I^{j})} \le A(I)+1.
\end{equation}
Moreover, it follows from \eqref{14/08/26/11:53} that for any $\delta>0$ and any $j\ge 1$, there exists $M_{2}(\delta,I,j)\ge 1$ such that for any $m\ge M_{2}(\delta,I,j)$, 
\begin{equation}\label{14/08/08/16:08}
\| \langle \nabla \rangle \big( 
\widetilde{\psi}^{j}-v_{m}^{j} 
\big) \|_{V_{2+\frac{4}{d}}(I^{j})\cap V_{2^{*}}(I^{j})}
\le 
\delta .
\end{equation}
Put $m_{0}:=\max\{M_{1}(I,j),M_{2}(\delta,I,j)\}$. We see from the triangle inequality, H\"older's inequality and Sobolev's embedding that the right-hand side of \eqref{17/11/15/10:19} with $s=1$ is estimated as follows: 
\begin{equation}\label{14/08/08/16:10}
\begin{split}
&(\lambda_{n}^{j})^{2-s_{q}}
\| \sigma_{n}^{j}\widetilde{\psi}^{j}  
(G_{n}^{j})^{-1} |\nabla | e^{it\Delta}  w_{n}^{k}
\|_{L_{t,x}^{\frac{2(d+2)(q-1)}{d(q-1)+4}}(I^{j})}
\\[6pt]
&\le 
(\lambda_{n}^{j})^{1-s_{q}}
\|   \sigma_{n}^{j} \widetilde{\psi}^{j} 
- 
\sigma_{n}^{j} v_{m_{0}}^{j} 
\|_{W_{q+1}(I^{j})} 
\|
|\nabla | e^{it\Delta}  w_{n}^{k}
\|_{V_{2+\frac{4}{d}}(I^{j})}
\\[6pt]
&\quad 
+
(\lambda_{n}^{j})^{2-s_{q}} \| \sigma_{n}^{j}v_{m_{0}}^{j}
(G_{n}^{j})^{-1} |\nabla | e^{it\Delta}  w_{n}^{k} \|_{L_{t,x}^{\frac{2(d+2)(q-1)}{d(q-1)+4}}(I^{j})}
.
\end{split}
\end{equation}

We consider the first term on the right-hand side of \eqref{14/08/08/16:10}. 
 Using  \eqref{14/02/03/11:19}, H\"older's inequality, Lemma \ref{14/06/21/15:10}  and \eqref{14/08/08/16:08} that if $n$ is sufficiently large dependently on $k$, then    
\begin{equation}\label{14/08/08/16:11}
\begin{split}
&
(\lambda_{n}^{j})^{1-s_{q}}
\|   \sigma_{n}^{j} \widetilde{\psi}^{j} 
- 
\sigma_{n}^{j} v_{m_{0}}^{j} 
\|_{W_{q+1}(I^{j})} 
\|
|\nabla | e^{it\Delta}  w_{n}^{k}
\|_{V_{2+\frac{4}{d}}(I^{j})}
\\[6pt]
&\lesssim  
(\lambda_{n}^{j})^{1-s_{q}}
\|  \sigma_{n}^{j} \widetilde{\psi}^{j}
- \sigma_{n}^{j} v_{m_{0}}^{j} 
\|_{V_{2+\frac{4}{d}}(I^{j})}^{1-s_{q}}
\|  \sigma_{n}^{j} \widetilde{\psi}^{j}
- \sigma_{n}^{j} v_{m_{0}}^{j} \|_{W_{2^{*}}(I^{j})}^{s_{q}}
\\[6pt]
&\lesssim 
\|  \langle \nabla \rangle  
\{ \widetilde{\psi}^{j}- v_{m_{0}}^{j} \} \|_{V_{2+\frac{4}{d}}(I)}
^{1-s_{q}}
\|  \langle \nabla \rangle  
\{ \widetilde{\psi}^{j}- v_{m_{0}}^{j} \} \|_{V_{2^{*}}(I)}^{1-s_{q}}
\le  
\delta.
\end{split}
\end{equation}

Next, we consider the second term on the right-hand side of \eqref{14/08/08/16:10}. Assume first that $\lambda_{\infty}^{j}=1$, so that $\lambda_{n}^{j}\equiv 1$ and $\sigma_{n}^{j}\equiv 1$. Then, we see from H\"older's inequality, Sobolev's embedding,  \eqref{14/08/08/16:09}, \eqref{14/02/03/11:19},  Lemma 2.5 of \cite{Killip-Visan} and \eqref{15/02/28/16:41} that there exists a number $K_{2}(\delta,I,j)$ with the following property: for any $k\ge K_{2}(\delta,I,j)$, there exists $N_{2}(\delta,I,j,k)\ge 1$ such that for any $n\ge N_{2}(\delta,I,j,k)$, 
\begin{equation}\label{17/11/13/15:37}
\begin{split}
&
(\lambda_{n}^{j})^{2-s_{q}} \| \sigma_{n}^{j}v_{m_{0}}^{j}
(G_{n}^{j})^{-1} |\nabla | e^{it\Delta}  w_{n}^{k} \|_{L_{t,x}^{\frac{2(d+2)(q-1)}{d(q-1)+4}}(I^{j})}
\\[6pt]
&=
\| v_{m_{0}}^{j}
|\nabla | e^{it\Delta}  w_{n}^{k} 
\|_{L^{\frac{2(d+2)(q-1)}{d(q-1)+4}}(E_{m_{0}}^{j})}
\\[6pt]
&\le 
\| v_{m_{0}}^{j} \|_{W_{2^{*}}(I^{j})}
\| |\nabla| e^{it\Delta} w_{n}^{k} 
\|_{L^{\frac{(d+2)(q-1)}{q+1}}(E_{m_{0}}^{j})}
\\[6pt]
&\lesssim 
\| \nabla v_{m_{0}}^{j}  \|_{V_{2^{*}}(I^{j})}
\| |\nabla| e^{it\Delta} w_{n}^{k} 
\|_{L^{\frac{(d+2)(q-1)}{q+1}}(E_{m_{0}}^{j})}
\\[6pt]
&\lesssim 
\{ A(I) +1\}
\| |\nabla| e^{it\Delta} w_{n}^{k} \|_{L^{2}(E_{m_{0}}^{j})}^{1-s_{q}}\| |\nabla| e^{it\Delta} w_{n}^{k} \|_{V_{2+\frac{4}{d}}(\mathbb{R})}^{s_{q}}
\\[6pt]
&\lesssim 
\{A(I) +1\}
(T_{m_{0}}^{j})^{\frac{2(1-s_{q})}{3(d+2)}}
(R_{m_{0}}^{j})^{\frac{(3d+2)(1-s_{q})}{6(d+2)}}
\| e^{it\Delta}w_{n}^{k} \|_{W_{2^{*}}(\mathbb{R})}^{\frac{1-s_{q}}{3}}
\| \nabla w_{n}^{k} \|_{L^{2}}^{\frac{2(1-s_{q})}{3}}
\\[6pt]
&\lesssim 
\{ A(I) +1\}
(T_{m_{0}}^{j})^{\frac{2(1-s_{q})}{3(d+2)}}
(R_{m_{0}}^{j})^{\frac{(3d+2)(1-s_{q})}{6(d+2)}}
\| e^{it\Delta}w_{n}^{k} \|_{W_{2^{*}}(\mathbb{R})}^{\frac{1-s_{q}}{3}}
\le
\delta
.
\end{split}
\end{equation}
Next, we assume that $\lambda_{\infty}^{j}=0$. We see from Lemma \ref{14/06/21/15:10} and the support property of $v_{m_{0}}^{j}$ that 
\begin{equation}\label{17/11/15/18:22}
\|  \lambda_{n}^{j}\sigma_{n}^{j} v_{m_{0}}^{j}(t)  \|_{L^{2}}^{2}
\lesssim  
\| \langle \nabla \rangle v_{m_{0}}^{j}(t)  \|_{L^{2}}^{2}
\in 
L^{1}(\mathbb{R}).
\end{equation}
We also see from Lemma \ref{14/09/01/11:30} that for any $t\in \mathbb{R}$, 
\begin{equation}\label{17/11/15/18:27}
\lim_{n\to \infty}\|  \lambda_{n}^{j}\sigma_{n}^{j} v_{m_{0}}^{j}(t)  \|_{L^{2}}=0
.
\end{equation}
Hence, Lebesgue's convergence theorem shows that 
\begin{equation}\label{17/11/15/18:03}
\lim_{n\to \infty}
\|  \lambda_{n}^{j}\sigma_{n}^{j} v_{m_{0}}^{j}  \|_{L_{t,x}^{2}(I^{j})}^{2}
=
\lim_{n\to \infty}\int_{\mathbb{R}} \|  \lambda_{n}^{j}\sigma_{n}^{j} v_{m_{0}}^{j}(t)  \|_{L^{2}}^{2}\,dt 
=0.
\end{equation}
Furthermore, it follows from H\"older's inequality, \eqref{14/02/03/11:19}, Sobolev's embedding,  Lemma \ref{14/06/21/15:10}, \eqref{14/08/08/16:09}, $\lambda_{n}^{j}\le 1$ and \eqref{17/11/15/18:03} that for any $k\ge 1$, there exists $N_{3}(\delta,I,j,k)\ge 1$ such that for any $n\ge N_{3}(\delta,I,j,k)$, 
\begin{equation}\label{17/11/15/15:25}
\begin{split}
&
(\lambda_{n}^{j})^{2-s_{q}} \| \sigma_{n}^{j}v_{m_{0}}^{j}
(G_{n}^{j})^{-1} |\nabla | e^{it\Delta}  w_{n}^{k} \|_{L_{t,x}^{\frac{2(d+2)(q-1)}{d(q-1)+4}}(I^{j})}
\\[6pt]
&\le 
(\lambda_{n}^{j})^{1-s_{q}} 
\|  \sigma_{n}^{j} v_{m_{0}}^{j}  \|_{V_{2+\frac{4}{d}}(I^{j})}^{1-s_{q}}
\|  \sigma_{n}^{j} v_{m_{0}}^{j}  \|_{W_{2^{*}}(I^{j})}^{s_{q}}
\| |\nabla| e^{it\Delta} w_{n}^{k} 
\|_{V_{2+\frac{4}{d}}(\mathbb{R})}
\\[6pt]
&\lesssim   
\|  \lambda_{n}^{j}\sigma_{n}^{j} v_{m_{0}}^{j}  \|_{V_{2+\frac{4}{d}}(I^{j})}^{1-s_{q}}
\{ A(I)+ 1 \}^{s_{q}}
\\[6pt]
&\lesssim   
\|  \lambda_{n}^{j}\sigma_{n}^{j} v_{m_{0}}^{j}  \|_{L_{t}^{2}L_{x}^{2}(I^{j})}^{\frac{1-s_{q}}{2}}
(\lambda_{n}^{j})^{\frac{1-s_{q}}{2}}
\|  \sigma_{n}^{j} v_{m_{0}}^{j}  \|_{W_{2^{*}}(I^{j})}^{\frac{1-s_{q}}{2}}
\{ A(I)+ 1 \}^{s_{q}}
\\[6pt]
&\lesssim
\|  \lambda_{n}^{j}\sigma_{n}^{j} v_{m_{0}}^{j}  \|_{L_{t}^{2}L_{x}^{2}(I^{j})}^{\frac{1-s_{q}}{2}}
\|  \nabla \sigma_{n}^{j} v_{m_{0}}^{j}  \|_{V_{2^{*}}(I^{j})}^{\frac{1-s_{q}}{2}}
\{ A(I)+ 1 \}^{s_{q}}
\\[6pt]
&\lesssim 
\| \lambda_{n}^{j}\sigma_{n}^{j} v_{m_{0}}^{j}  \|_{L_{t}^{2}L_{x}^{2}(I_{j})}^{\frac{1-s_{q}}{2}}
\{A(I)+1 \}^{\frac{1+s_{q}}{2}}
\le \delta
.
\end{split}
\end{equation}

\noindent 
{\bf Case 2.2.}~Assume $s=1$ and $\lambda_{\infty}^{j}=\infty$ in \eqref{17/11/15/10:19}. We see from H\"older's inequality, \eqref{14/02/03/11:19}, Sobolev's embedding, Strichartz' estimate and Lemma \ref{14/09/01/11:30} that for any $k\ge 1$, there exists $N_{4}(\delta,j,k)\ge 1$ such that for any $n\ge N_{4}(\delta,j,k)$, 
\begin{equation}\label{17/11/25/17:30}
\begin{split}
&
(\lambda_{n}^{j})^{2-s_{q}} \| \sigma_{n}^{j}\widetilde{\psi}^{j}
(G_{n}^{j})^{-1} |\nabla | e^{it\Delta}  w_{n}^{k} \|_{L_{t,x}^{\frac{2(d+2)(q-1)}{d(q-1)+4}}(I^{j})}
\\[6pt]
&\le 
(\lambda_{n}^{j})^{1-s_{q}} 
\|  \sigma_{n}^{j} \widetilde{\psi}^{j} \|_{V_{2+\frac{4}{d}}(I^{j})}^{1-s_{q}}
\|  \sigma_{n}^{j} \widetilde{\psi}^{j}  \|_{W_{2^{*}}(I^{j})}^{s_{q}}
\| |\nabla| e^{it\Delta} w_{n}^{k} 
\|_{V_{2+\frac{4}{d}}(\mathbb{R})}
\\[6pt]
&\lesssim  
(\lambda_{n}^{j})^{1-s_{q}} 
\|  e^{it\Delta} \sigma_{n}^{j} \widetilde{u}^{j} \|_{V_{2+\frac{4}{d}}(I^{j})}^{1-s_{q}}
\|  e^{it\Delta}\nabla \sigma_{n}^{j} \widetilde{u}^{j}  \|_{V_{2^{*}}(I^{j})}^{s_{q}}
\\[6pt]
&\lesssim   
\| \lambda_{n}^{j}\sigma_{n}^{j} \widetilde{u}^{j} \|_{L^{2}}^{1-s_{q}}
\| \nabla \sigma_{n}^{j} \widetilde{u}^{j} \|_{L^{2}}^{s_{q}}
\le \delta
.
\end{split}
\end{equation}

Thus, putting the estimates \eqref{17/11/15/10:19}, \eqref{14/08/08/16:12}, \eqref{14/08/08/16:10}, \eqref{14/08/08/16:11}, \eqref{17/11/13/15:37}, \eqref{17/11/15/15:25} and \eqref{17/11/25/17:30} together, we obtain the desired result \eqref{14/02/26/11:25}. 
\end{proof}

Now, for a number $k\ge 1$, we define an approximate solution $\psi_{n}^{k-app}$ of $\psi_{n}$ by 
\begin{equation}\label{14/02/03/11:10}
\psi_{n}^{k-app}(t):=\sum_{j=1}^{k}\psi_{n}^{j}(t)+e^{it\Delta}w_{n}^{k}. 
\end{equation}
Assume \eqref{14/02/01/17:53}. Then, we have $I_{\max}^{j}=\mathbb{R}$ for all $j\ge 2$, which together with \eqref{14/02/01/10:57} shows that $I_{\max,n}^{j}=\mathbb{R}$ for all $j\ge 2$. Hence, for any numbers $k$ and $n$, the maximal existence-interval of $\psi_{n}^{k-app}$ is $I_{\max,n}^{1}:=\big((\lambda_{n}^{1})^{2} T_{\min}^{1}+t_{n}^{1}, (\lambda_{n}^{1})^{2}T_{\max}^{1}+t_{n}^{1}\big)$.  Furthermore, we see from \eqref{14/08/30/16:23} that $0\in I_{\max,n}^{1}$ for any sufficiently large  number $n$.  

\begin{lemma}\label{14/08/14/14:30}
Assume \eqref{14/02/01/17:53}, and let $q$ denote one of the numbers $1+\frac{4}{d}$, $p$ and $2^{*}-1$. Then, for any $j_{0}\ge 2$,  any $k>j_{0}$ and any $\delta>0$, there exists a number $N(j_{0}, k,\delta)$ such that for any $n\ge N(j_{0}, k,\delta)$, 
\begin{align}
\label{14/08/14/14:34}
&\big\| \sum_{j=j_{0}}^{k}\psi_{n}^{j}
\big\|_{W_{q+1}(\mathbb{R})}^{\frac{(d+2)(q-1)}{2}}
\le
\sum_{j=j_{0}}^{k}
\| \psi_{n}^{j} \|_{W_{q+1}(\mathbb{R})}^{\frac{(d+2)(q-1)}{2}}
+
\delta,
\\[6pt]
\label{17/11/25/10:07}
&\big\| \nabla \sum_{j=j_{0}}^{k}\psi_{n}^{j}
\big\|_{V_{2+\frac{4}{d}}(\mathbb{R})}^{\frac{2(d+2)}{d}}
\le
\sum_{j=j_{0}}^{k}
\| \nabla \psi_{n}^{j} \|_{V_{2+\frac{4}{d}}(\mathbb{R})}^{\frac{2(d+2)}{d}}
+
\delta.
\end{align}
\end{lemma}
\begin{proof}[Proof of Lemma \ref{14/08/14/14:30}] 
Let $j_{0}$ and $k$ be numbers with $2\le j_{0}<k$. Then, we can verify 
 that there exists a constant $C(j_{0}, k)$ such that for any number $n$, 
\begin{equation}\label{14/08/14/14:45}
\begin{split}
&
\big\| \sum_{j=j_{0}}^{k}\psi_{n}^{j}
\big\|_{W_{q+1}(\mathbb{R})}^{\frac{(d+2)(q-1)}{2}}
\\[6pt]
&\le  
\sum_{j=j_{0}}^{k}
\|\psi_{n}^{j} \|_{W_{q+1}(\mathbb{R})}^{\frac{(d+2)(q-1)}{2}}
+
C(j_{0}, k)\sum_{j=2}^{k}
\sum_{{2\le j'\le k;}\atop {j'\neq j}}
\int_{\mathbb{R}}\int_{\mathbb{R}^{d}}
| \psi_{n}^{j} |^{\frac{(d+2)(q-1)}{2}-1}
| \psi_{n}^{j'} |\,dxdt
.
\end{split}
\end{equation}
Furthermore, the orthogonality \eqref{14/01/26/11:48} shows that for any $\delta>0$, we can take a number $N(j_{0},k,\delta)$ such that for any distinct numbers $j,j' \in \{2,\ldots, k\}$ and any $n\ge N(j_{0}, k,\delta)$, 
\begin{equation}\label{14/08/23/10:45}
\int_{\mathbb{R}}
\int_{\mathbb{R}^{d}}
\big| \psi_{n}^{j}(x,t) \big|^{\frac{(d+2)(q-1)}{2}-1}
\big| \psi_{n}^{j'}(x,t) \big|\,dxdt
\le 
\frac{\delta}{C(j_{0},k)k^{2}}.
\end{equation}
Putting \eqref{14/08/14/14:45} and \eqref{14/08/23/10:45} together, we obtain the desired result \eqref{14/08/14/14:34}. Similarly, we can verify that \eqref{17/11/25/10:07} holds. 
\end{proof}

\begin{lemma}\label{14/06/17/11:32}
Assume \eqref{14/02/01/17:53}. Then, for any interval $I$ satisfying \eqref{14/06/14/14:38}, we can take a constant $B(I)>0$ with the following property: for any number $k$, there exists a number $N(k)$ such that for any $n\ge N(k)$, 
\begin{equation}\label{14/06/17/11:36}
\|
\psi_{n}^{k\mbox{-}app}
\|_{W_{p+1}(I_{n})\cap W_{2^{*}}(I_{n})}
+
\| \langle \nabla \rangle 
\psi_{n}^{k\mbox{-}app} \|_{V_{2+\frac{4}{d}}(I_{n})}
\le B(I),
\end{equation}
where $I_{n}:=((\lambda_{n}^{1})^{2}\inf{I}+t_{n}^{1}, 
(\lambda_{n}^{1})^{2}\sup{I}+t_{n}^{1})$.
\end{lemma}
\begin{proof}[Proof of Lemma \ref{14/06/17/11:32}]
We consider the first term on the right-hand side of \eqref{14/06/17/11:36}. Let $q$ denote $p$ or $2^{*}-1$, and let $J_{0}$ be the number found in Lemma \ref{14/01/26/17:46}.  Then, it follows from Lemma \ref{14/08/14/14:30} with $\delta=1$, Lemma \ref{14/06/18/10:16}, \eqref{14/02/03/11:19},  Lemma \ref{14/02/02/18:02}, \eqref{14/02/01/17:58} and Lemma \ref{14/01/26/17:46}  that for any  number $k$, there exists a number $N(k)$ such that for any $n\ge N(k)$, 
\begin{equation}\label{11/07/02/18:18}
\begin{split}
&\| \psi_{n}^{k\mbox{-}app} \|_{W_{q+1}(I_{n})}^{\frac{(d+2)(q-1)}{2}}
\\[6pt]
&\lesssim  
\| \psi_{n}^{1} \|_{W_{q+1}(I_{n})}^{\frac{(d+2)(q-1)}{2}}
+
\sum_{j=2}^{k}
\| \psi_{n}^{j} \|_{W_{q+1}(\mathbb{R})}^{\frac{(d+2)(q-1)}{2}}
+1+
\|e^{it\Delta}w_{n}^{k}\|_{W_{q+1}(\mathbb{R})}^{\frac{(d+2)(q-1)}{2}}
\\[6pt]
&\lesssim  
\| \langle \nabla \rangle \widetilde{\psi}^{1} \|_{St(I)}^{\frac{(d+2)(q-1)}{2}}+
\sum_{j=2}^{J_{0}}
\| \langle \nabla \rangle \widetilde{\psi}^{j} 
\|_{St(\mathbb{R})}^{\frac{(d+2)(q-1)}{2}}
+
\sum_{j=J_{0}}^{\infty}
\| \langle \nabla \rangle \widetilde{\psi}^{j} 
\|_{St(\mathbb{R})}^{\frac{(d+2)(q-1)}{2}}
+
1
\\[6pt]
&\lesssim  
A(I)^{\frac{(d+2)(q-1)}{2}}
+
J_{0} 
+
1.
\end{split}
\end{equation}
Thus, we have obtained the desired estimate for the first term. We can deal with  the second term in a similar way.
\end{proof}

\begin{lemma}\label{14/08/08/16:01}
Assume \eqref{14/02/01/17:53}, and  let $q$ denote $p$ or $2^{*}-1$. Then, for any $\delta>0$ and any interval $I$ satisfying \eqref{14/06/14/14:38}, 
 we can take a number $K(\delta,I)$ with the following property: for any $k\ge K(\delta,I)$, there exists a number $N(\delta,I,k)$ such that for any $n\ge N(\delta, I,k)$ and $s\in \{0,1\}$, 
\begin{equation}
\label{14/08/08/16:03}
\big\| |\psi_{n}^{k\mbox{-}app}|^{q-1}| \nabla|^{s}e^{it\Delta}w_{n}^{k} 
\big\|_{L_{t,x}^{\frac{2(d+2)}{d+4}}(I_{n})}\le \delta,
\end{equation}
where $I_{n}:=( (\lambda_{n}^{1})^{2}\inf{I}+t_{n}^{1}, \ (\lambda_{n}^{1})^{2}\sup{I}+t_{n}^{1})$. 
\end{lemma}
\begin{proof}[Proof of Lemma \ref{14/08/08/16:01}] Let us begin with some preparation. We see from Sobolev's embedding, Lemma \ref{14/06/18/10:16} and Lemma \ref{14/01/26/17:46} that for any $\delta>0$, there exists a number $K_{1}(\delta)$ such that 
\begin{equation}\label{14/08/27/15:27}
\sum_{j=K_{1}(\delta)}^{\infty} \big\| \psi_{n}^{j} \big\|_{W_{q+1}(\mathbb{R})}^{\frac{(d+2)(q-1)}{2}} 
\lesssim 
\sum_{j=K_{1}(\delta)}^{\infty} 
\big\|\langle \nabla \rangle \widetilde{\psi}^{j} \big\|_{St(\mathbb{R})}^{\frac{(d+2)(q-1)}{2}}
\le \delta. 
\end{equation}
This together with Lemma \ref{14/08/14/14:30} shows that for any $\delta>0$ and any $k>K_{1}(\delta)$, there exists a number $N_{1}(\delta,k)$ such that for any $n\ge N_{1}(\delta,k)$, 
\begin{equation}\label{14/08/27/17:40}
\big\| \sum_{j=K_{1}(\delta)+1}^{k} \psi_{n}^{j} 
\big\|_{W_{q+1}(\mathbb{R})}^{\frac{(d+2)(q-1)}{2}} 
\lesssim \delta. 
\end{equation}
Furthermore, we see from H\"older's inequality, \eqref{14/02/03/11:19} and \eqref{15/02/28/16:41} that for any $\delta>0$, there exists  a number $K_{2}(\delta)$ with the following property: for any $k\ge K_{2}(\delta)$, there exists a number $N_{2}(\delta,k)$ such that for any $n\ge N_{2}(\delta,k)$ and $s\in \{0,1\}$,
\begin{equation}\label{14/08/28/9:00}
\begin{split}
\big\| e^{it\Delta} w_{n}^{k} | \nabla|^{s}e^{it\Delta}w_{n}^{k} 
\big\|_{L_{t,x}^{\frac{2(d+2)(q-1)}{d(q-1)+4}}(\mathbb{R})}^{\frac{(d+2)(q-1)}{2}}
&\le 
\| e^{it\Delta} w_{n}^{k} \|_{W_{q+1}(\mathbb{R})}^{\frac{(d+2)(q-1)}{2}}
\| |\nabla |^{s} e^{it\Delta} w_{n}^{k} \|_{V_{2+\frac{4}{d}}(\mathbb{R})}^{\frac{(d+2)(q-1)}{2}}
\\[6pt]
&\lesssim 
\| e^{it\Delta} w_{n}^{k} \|_{W_{q+1}(\mathbb{R})}^{\frac{(d+2)(q-1)}{2}}
\le \delta.
\end{split} 
\end{equation}

Now, we shall prove \eqref{14/08/08/16:03}. First, we consider the case where $q\le 2$. Then, we see from H\"older's inequality, \eqref{14/02/03/11:19}, \eqref{14/08/27/17:40} and \eqref{14/08/28/9:00} that for any $\delta>0$, any $k> \max\{K_{1}(\delta),K_{2}(\delta)\}$, any $n\ge \max\{N_{1}(\delta, k),N_{2}(\delta,k)\}$ and $s\in \{0,1\}$,
\begin{equation}\label{14/08/14/14:10}
\begin{split}
&
\| |\psi_{n}^{k\mbox{-}app}|^{q-1}| \nabla|^{s}e^{it\Delta}w_{n}^{k} 
\|_{L_{t,x}^{\frac{2(d+2)}{d+4}}(I_{n})}^{\frac{d+2}{2}}
\\[6pt]
&\le 
\| \psi_{n}^{k\mbox{-}app} | \nabla|^{s}e^{it\Delta}w_{n}^{k} 
\|_{L_{t,x}^{\frac{2(d+2)(q-1)}{d(q-1)+4}}(I_{n})}^{\frac{(d+2)(q-1)}{2}}
\| | \nabla|^{s}e^{it\Delta}w_{n}^{k} \|_{V_{2+\frac{4}{d}}(I_{n})}^{\frac{(d+2)(2-q)}{2}}
\\[6pt]
&\lesssim 
\| \psi_{n}^{k\mbox{-}app} | \nabla|^{s}e^{it\Delta}w_{n}^{k} 
\|_{L_{t,x}^{\frac{2(d+2)(q-1)}{d(q-1)+4}}(I_{n})}^{\frac{(d+2)(q-1)}{2}}
\\[6pt]
&\lesssim
C(\delta) \sum_{j=1}^{K_{1}(\delta)} 
\| \psi_{n}^{j} | \nabla|^{s}e^{it\Delta}w_{n}^{k} 
\|_{L_{t,x}^{\frac{2(d+2)(q-1)}{d(q-1)+4}}(I_{n})}^{\frac{(d+2)(q-1)}{2}}
\\[6pt]
&\quad +
\| \sum_{j=K_{1}(\delta)+1}^{k}  \psi_{n}^{j} 
\|_{W_{q+1}(\mathbb{R})}^{\frac{(d+2)(q-1)}{2}} 
\| | \nabla|^{s}e^{it\Delta}w_{n}^{k} \|_{V_{2+\frac{4}{d}}(\mathbb{R})}^{\frac{(d+2)(q-1)}{2}}
\\[6pt]
&\quad +
\big\| e^{it\Delta} w_{n}^{k} | \nabla|^{s}e^{it\Delta}w_{n}^{k} 
\big\|_{L_{t,x}^{\frac{2(d+2)(q-1)}{d(q-1)+4}}(\mathbb{R})}^{\frac{(d+2)(q-1)}{2}}
\\[6pt]
&\lesssim 
C(\delta)K_{1}(\delta)\sup_{1\le j \le K_{1}(\delta)} 
\| \psi_{n}^{j} | \nabla|^{s}e^{it\Delta}w_{n}^{k} 
\|_{L_{t,x}^{\frac{2(d+2)(q-1)}{d(q-1)+4}}(I_{n})}^{\frac{(d+2)(q-1)}{2}}
+
\delta
,
\end{split}
\end{equation}
where $C(\delta)$ is some constant depending only on $d$, $q$ and $K_{1}(\delta)$. Moreover,  Lemma \ref{14/08/26/11:23} shows that we can take a number $K_{3}(\delta,I)$ with the following property: for any $k\ge K_{3}(\delta,I)$, there exists a number $N_{3}(\delta,I,k)$ such that for any $1\le j \le K_{1}(\delta)$, any $n\ge N_{3}(\delta,I,k)$ and $s \in \{0,1\}$,  
\begin{equation}\label{14/08/28/09:48}
\big\| \psi_{n}^{j} | \nabla|^{s}e^{it\Delta}w_{n}^{k} 
\big\|_{L_{t,x}^{\frac{2(d+2)(q-1)}{d(q-1)+4}}(I_{n})}^{\frac{(d+2)(q-1)}{2}}
\le \frac{\delta}{C(\delta)K_{1}(\delta)}.
\end{equation}
Putting \eqref{14/08/14/14:10} and \eqref{14/08/28/09:48} together, we obtain the desired result \eqref{14/08/08/16:03} in the case $q\le 2$.
\par 
Next, we consider the case where $q>2$. In this case, we see from H\"older's inequality and Lemma \ref{14/06/17/11:32} that there exists $B(I)>0$ with the following property: for any number $k$, there exists a number $N_{4}(k)$ such that for any $n\ge N_{4}(k)$ and $s \in \{0,1\}$,   
\begin{equation}\label{14/08/25/11:50}
\begin{split}
&\big\| |\psi_{n}^{k\mbox{-}app}|^{q-1}| \nabla|^{s}e^{it\Delta}w_{n}^{k} 
\big\|_{L_{t,x}^{\frac{2(d+2)}{d+4}}(I_{n})}
\\[6pt]
&\le 
\| \psi_{n}^{k\mbox{-}app} | \nabla|^{s}e^{it\Delta}w_{n}^{k} 
\|_{L_{t,x}^{\frac{2(d+2)(q-1)}{d(q-1)+4}}(I_{n})}
\| \psi_{n}^{k\mbox{-}app} \|_{W_{q+1}}^{q-2}
\\[6pt]
&\le 
\| \psi_{n}^{k\mbox{-}app} | \nabla|^{s}e^{it\Delta}w_{n}^{k} 
\|_{L_{t,x}^{\frac{2(d+2)(q-1)}{d(q-1)+4}}(I_{n})}
B(I)^{q-2}.
\end{split}
\end{equation}
Then, we can obtain the desired estimate \eqref{14/08/08/16:03} in a way 
 similar to the case $q \le 2$.   
\end{proof}

The following lemma enables us to control the error term of the approximate solution: 
\begin{lemma}\label{14/06/14/10:52}
Assume \eqref{14/02/01/17:53}. Then, for any $\delta\in (0,1)$ and any interval $I$ satisfying \eqref{14/06/14/14:38}, there exist numbers $k_{0}$ (depending on $\delta$ and $I$) and $N(\delta,I)$ such that for any $n\ge N(\delta,I)$, 
\begin{equation}\label{11/07/02/17:21}
\|\langle \nabla \rangle  
e[\psi_{n}^{k_{0}\mbox{-}app}]
\|_{L_{t,x}^{\frac{2(d+2)}{d+4}}(I_{n})}
\le \delta,
\end{equation}
where $I_{n}:=((\lambda_{n}^{1})^{2}\inf{I}+t_{n}^{1}, 
(\lambda_{n}^{1})^{2}\sup{I}+t_{n}^{1})$.
\end{lemma}
\begin{proof}[Proof of Lemma \ref{14/06/14/10:52}] 
Let $q$ denote $p$ or $2^{*}-1$. Then, we see from \eqref{15/02/28/16:41}, Lemma \ref{14/06/17/11:32} and Lemma \ref{14/08/08/16:01} that for any $\delta\in (0,1)$ and any interval $I$ satisfying \eqref{14/06/14/14:38}, there exist a number $k_{0}\ge 1$ (depending on $\delta$ and $I$), a constant $B(I)>0$ (independent of $\delta$) and a number $N_{1}(\delta,I)\ge 1$ such that if $n\ge N_{1}(\delta,I)$, then 
\begin{align}
\label{14/08/05/12:00}
&
\| e^{it\Delta}w_{n}^{k_{0}} \|_{W_{q+1}(\mathbb{R})}
\le 
\dfrac{\delta}{(1+B(I))^{2^{*}}}, 
\\[6pt]
\label{14/06/14/14:19}
&
\|\psi_{n}^{k_{0}\mbox{-}app} 
\|_{W_{q+1}(I_{n})}
+
\|\langle \nabla \rangle \psi_{n}^{k_{0}\mbox{-}app} 
\|_{V_{2+\frac{4}{d}}(I_{n})}
\le B(I),
\\[6pt]
\label{14/08/13/16:53}
&
\| |\psi_{n}^{k_{0}\mbox{-}app} |^{q-1}
e^{it\Delta}w_{n}^{k_{0}}
\|_{L_{t,x}^{\frac{2(d+2)}{d+4}}(I_{n})}
+
\| |\psi_{n}^{k_{0}\mbox{-}app} |^{q-1}
|\nabla |  e^{it\Delta}w_{n}^{k_{0}}
\|_{L_{t,x}^{\frac{2(d+2)}{d+4}}(I_{n})}
\le 
\delta
.
\end{align}

We rewrite the ``error term''  $e[\psi_{n}^{k_{0}\mbox{-}app}]$ as follows:   
\begin{equation}\label{11/07/15/9:59}
\begin{split} 
e[\psi_{n}^{k_{0}\mbox{-}app}]
&=
\sum_{j=1}^{k_{0}}
\Bigm( i\frac{\partial \psi_{n}^{j}}{\partial t}
+
\Delta \psi_{n}^{j}
\Bigm)
+
F\big[ \sum_{j=1}^{k_{0}}\psi_{n}^{j} \big]
\\[6pt]
&\quad +
F\big[\psi_{n}^{k_{0}\mbox{-}app}\big]
-
F\big[\psi_{n}^{k_{0}\mbox{-}app}-e^{it\Delta}w_{n}^{k_{0}}\big]
.
\end{split}
\end{equation}
Then, we have 
\begin{equation}\label{14/07/01/11:45}
\begin{split}
&
\|\langle \nabla \rangle  
e[\psi_{n}^{k_{0}\mbox{-}app}]
\|_{L_{t,x}^{\frac{2(d+2)}{d+4}}(I_{n})}
\\[6pt]
&\le 
\| \langle \nabla \rangle 
\Bigm\{ 
\sum_{j=1}^{k_{0}} \Big( 
i\frac{\partial \psi_{n}^{j}}{\partial t}
+
\Delta \psi_{n}^{j}
\Big)
+
F\big[ 
\sum_{j=1}^{k_{0}} \psi_{n}^{j}
\big]
\Bigm\} 
\|_{L_{t,x}^{\frac{2(d+2)}{d+4}}(I_{n})}
\\[6pt]
&\quad +
\big\| \langle \nabla \rangle 
\bigm\{
F\big[\psi_{n}^{k_{0}\mbox{-}app}\big]
-
F\big[\psi_{n}^{k_{0}\mbox{-}app}-e^{it\Delta}w_{n}^{k_{0}}\big] 
\bigm\}
\big\|_{L_{t,x}^{\frac{2(d+2)}{d+4}}(I_{n})}
.
\end{split}
\end{equation}
We consider the first term on the right-hand side of \eqref{14/07/01/11:45}. It follows from \eqref{11/12/04/16:09}, Lemma \ref{14/07/02/10:30} and Lemma \ref{17/11/25/12:18} that there exists a number $N_{2}(\delta,I)$ such that for any $n\ge N_{2}(\delta,I)$, 
\begin{equation}\label{11/07/17/13:49}
\begin{split}
&
\| \langle \nabla \rangle \Bigm\{ 
\sum_{j=1}^{k_{0}} 
\Big( 
i\frac{\partial \psi_{n}^{j}}{\partial t}
+
\Delta \psi_{n}^{j}
\Big)
+
F\big[ 
\sum_{j=1}^{k_{0}} \psi_{n}^{j}
\big]
\Bigm\}
\|_{L_{t,x}^{\frac{2(d+2)}{d+4}}(I_{n})}
\\[6pt]
&\le 
\| \langle \nabla \rangle  
\Big\{
\sum_{j=1}^{k_{0}}  
F\big[ \psi_{n}^{j}\big]
-
F\big[ 
\sum_{j=1}^{k_{0}} \psi_{n}^{j}
\big]
\Big\}
\|_{L_{t,x}^{\frac{2(d+2)}{d+4}}(I_{n})}
+\delta.
\end{split}
\end{equation}
Then, we see from elementary computations that for $q=p$ or $q=2^{*}-1$, 
\begin{equation}\label{14/08/04/11:29}
\begin{split}
&
\| 
\sum_{j=1}^{k_{0}}  
\big|\psi_{n}^{j}|^{q-1}\psi_{n}^{j}
-
| \sum_{j=1}^{k_{0}} \psi_{n}^{j} \big|^{q-1}\sum_{j=1}^{k_{0}} \psi_{n}^{j}
\|_{L_{t,x}^{\frac{2(d+2)}{d+4}}(I_{n})}
\\[6pt]
&\lesssim 
\sum_{j=1}^{k_{0}} \sum_{{1\le k \le k_{0}}\atop {k \neq j}}
\| |  \psi_{n}^{j} |^{q-1} \psi_{n}^{k} \|_{L_{t,x}^{\frac{2(d+2)}{d+4}}(I_{n})}\end{split}
\end{equation}
and 
\begin{equation}\label{15/03/01/11:45}
\begin{split}
& \| \nabla   
\Big\{
\sum_{j=1}^{k_{0}}  |\psi_{n}^{j} |^{q-1}\psi_{n}^{j}
-
\big| \sum_{j=1}^{k_{0}} \psi_{n}^{j} \big|^{q-1}\sum_{j=1}^{k_{0}}
 \psi_{n}^{j}
\Big\} \|_{L_{t,x}^{\frac{2(d+2)}{d+4}}(I_{n})}
\\[6pt]
&\lesssim
\sum_{j=1}^{k_{0}} \sum_{{1\le k\le k_{0}}\atop {k\neq j}}
\| |  \psi_{n}^{j} |^{q-1}  \nabla \psi_{n}^{k} \|_{L_{t,x}^{\frac{2(d+2)}{d+4}}(I_{n})}.
\end{split} 
\end{equation}
Here, when $q>2$, we must add the following term to the right-hand side of \eqref{15/03/01/11:45}: 
\begin{equation}\label{15/04/24/17:17}
\sum_{i=1}^{k_{0}} 
\sum_{j=1}^{k_{0}} 
\sum_{{1\le k \le k_{0}}\atop {k \neq j}}
\| |  \psi_{n}^{i} |^{q-2} 
 |\nabla \psi_{n}^{j}|  |\psi_{n}^{k}| \|_{L_{t,x}^{\frac{2(d+2)}{d+4}}(I_{n})}.
\end{equation}
We see from the orthogonality \eqref{14/01/26/11:48} that there exists a number $N_{3}(\delta, I)$ such that for any distinct numbers $j,k\in \{1,\ldots, k_{0}\}$ and any $n\ge N_{3}(\delta,I)$,
\begin{equation}\label{15/04/24/16:15}
\| | \psi_{n}^{j} |^{q-1}  \psi_{n}^{k} 
\|_{L_{t,x}^{\frac{2(d+2)}{d+4}}(I_{n})}
+
\| | \psi_{n}^{j} |^{q-1}  \nabla \psi_{n}^{k} 
\|_{L_{t,x}^{\frac{2(d+2)}{d+4}}(I_{n})}
\le 
\frac{\delta}{k_{0}^{2}}
.
\end{equation}
When $q>2$, we also see that for any number $i \in \{1,\ldots, k_{0}\}$, any distinct numbers $j,k\in \{1,\ldots, k_{0}\}$ and any $n\ge N_{3}(\delta,I)$, 
\begin{equation}\label{15/03/01/11:55}
\| | \psi_{n}^{i} |^{q-2}  \nabla \psi_{n}^{j}  \, \psi_{n}^{k}
\|_{L_{t,x}^{\frac{2(d+2)}{d+4}}(I_{n})}
\le \frac{\delta}{k_{0}^{3}}. 
\end{equation}
Putting the estimates \eqref{11/07/17/13:49} through \eqref{15/03/01/11:55} together, we find that for any $n\ge \max\{N_{2}(\delta,I),N_{3}(\delta,I)\}$, 
\begin{equation}\label{14/08/03/12:10}
\| \langle \nabla \rangle 
\Bigm\{
\sum_{j=1}^{k_{0}} 
\Big(
i\frac{\partial \psi_{n}^{j}}{\partial t}
+
\Delta \psi_{n}^{j}
\Big)
+
F\big[ 
\sum_{j=1}^{k_{0}} \psi_{n}^{j}
\big]
\Bigm\}
\|_{L_{t,x}^{\frac{2(d+2)}{d+4}}(I_{n})}
\lesssim \delta. 
\end{equation}
We move on to the second term on the right-hand side of \eqref{14/07/01/11:45}.  We can verify that for $q=p$ or $q=2^{*}-1$, and any number $n\ge 1$, 
\begin{equation}\label{11/07/16/10:08}
\begin{split}
&
\| \langle \nabla \rangle 
\Big\{
|\psi_{n}^{k_{0}\mbox{-}app}|^{q-1}\psi_{n}^{k_{0}\mbox{-}app}
-
| \psi_{n}^{k_{0}\mbox{-}app}\!-e^{it\Delta}w_{n}^{k_{0}} |^{q-1} 
( \psi_{n}^{k_{0}\mbox{-}app}\!-e^{it\Delta}w_{n}^{k_{0}})
\Big\} 
\|_{L_{t,x}^{\frac{2(d+2)}{d+4}}(I_{n})}
\\[6pt]
&\lesssim  
\| | \psi_{n}^{k_{0}\mbox{-}app} |^{q-1}
e^{it\Delta}w_{n}^{k_{0}} \|_{L_{t,x}^{\frac{2(d+2)}{d+4}}(I_{n})}
+
\| | e^{it\Delta}w_{n}^{k_{0}}  |^{q-1}\psi_{n}^{k_{0}\mbox{-}app}
 \|_{L_{t,x}^{\frac{2(d+2)}{d+4}}(I_{n})}
\\[6pt]
&\quad 
+
\| | \psi_{n}^{k_{0}\mbox{-}app} |^{q-1}
\nabla e^{it\Delta}w_{n}^{k_{0}}
\|_{L_{t,x}^{\frac{2(d+2)}{d+4}}(I_{n})}
+
\| | e^{it\Delta}w_{n}^{k_{0}}  |^{q-1}
\nabla \psi_{n}^{k_{0}\mbox{-}app} \|_{L_{t,x}^{\frac{2(d+2)}{d+4}}(I_{n})}
\\[6pt]
&\quad 
+
\| \langle \nabla \rangle 
\big\{ 
| e^{it\Delta}w_{n}^{k_{0}} |^{q-1}
e^{it\Delta}w_{n}^{k_{0}} \big\} \|_{L_{t,x}^{\frac{2(d+2)}{d+4}}(I_{n})}
.
\end{split}
\end{equation}
When $q>2$, we must add the following terms to the right-hand side of \eqref{11/07/16/10:08}:   
\begin{align}
\label{11/08/21/11:18}
&
\| | \psi_{n}^{k_{0}\mbox{-}app} |^{q-2}
 e^{it\Delta}w_{n}^{k_{0}} 
 \nabla  \psi_{n}^{k_{0}\mbox{-}app} \|_{L_{t,x}^{\frac{2(d+2)}{d+4}}(I_{n})},
\\[6pt] 
\label{15/02/25/22:12}
&
\| | e^{it\Delta}w_{n}^{k_{0}} |^{q-2}
 \psi_{n}^{k_{0}\mbox{-}app} 
\nabla e^{it\Delta}w_{n}^{k_{0}} \|_{L_{t,x}^{\frac{2(d+2)}{d+4}}(I_{n})}
.
\end{align}
Furthermore, using \eqref{14/08/13/16:53}, H\"older's inequality,  \eqref{14/08/05/12:00}, \eqref{14/06/14/14:19} and \eqref{14/02/03/11:19}, we can verify that  the right-hand side of \eqref{11/07/16/10:08} vanishes as $n$ tends to the infinity (cf. the proof of (3.10) in \cite{Killip-Visan}). Thus, we conclude that 
\begin{equation}\label{15/03/01/09:00}
\lim_{n\to \infty}
\big\| \langle \nabla \rangle 
\big\{
F\big[\psi_{n}^{k_{0}\mbox{-}app}\big]
-
F\big[\psi_{n}^{k_{0}\mbox{-}app}-e^{it\Delta}w_{n}^{k_{0}}\big] 
\big\}
\big\|_{L_{t,x}^{\frac{2(d+2)}{d+4}}(I_{n})}
=0. 
\end{equation}
Putting \eqref{14/07/01/11:45}, \eqref{14/08/03/12:10} and \eqref{15/03/01/09:00} together, we find that the desired estimate \eqref{11/07/02/17:21} holds. 
\end{proof}

\begin{proposition}\label{14/02/02/14:35}
Assume \eqref{14/02/01/17:53}. Then, $\lambda_{\infty}^{1}\in \{0,1\}$, and for any $T\in I_{\max}^{1}=(T_{\min}^{1}, T_{\max}^{1})$,  
\begin{equation}\label{14/02/02/14:36}
\|\sigma_{\infty}^{1} \widetilde{\psi}^{1} \|_{\mathcal{W}^{1}([T,T_{\max}^{1}))}=\infty.
\end{equation}
Furthermore, the following holds: 
\begin{equation}\label{14/02/05/11:10}
\left\{ 
\begin{array}{ccc}
\mathcal{S}_{\omega}(\sigma_{\infty}^{1}\widetilde{\psi}^{1}) \ge m_{\omega}
&\mbox{if}& \lambda_{\infty}^{1}=1,
\\[6pt]
\mathcal{H}^{\ddagger}(\sigma_{\infty}^{1}\widetilde{\psi}^{1}) 
\ge \frac{1}{d}\sigma^{\frac{d}{2}}>m_{\omega}
&\mbox{if}& \lambda_{\infty}^{1}=0.
\end{array}
\right.
\end{equation}
\end{proposition}
\begin{proof}[Proof of Proposition \ref{14/02/02/14:35}]
First, we assume $\lambda_{\infty}^{1}\in \{0,1\}$ and prove \eqref{14/02/02/14:36}. Suppose for contradiction that \eqref{14/02/02/14:36} failed for some $T \in I_{\max}^{1}$. Then, we see from the well-posedness theory that $T_{\max}^{1}=\infty$ and for any $\tau \in (T_{\min}^{1},\infty)$,  
\begin{equation}\label{14/02/03/11:02}
\|\sigma_{\infty}^{1} \widetilde{\psi}^{1} \|_{\mathcal{W}^{1}([\tau, \infty))}
<\infty.
\end{equation}
In order to derive a contradiction, we consider the approximate solution $\psi_{n}^{k\mbox{-}app}$ defined by \eqref{14/02/03/11:10}. Note that for any numbers $k$ and $n$, the maximal existence interval of $\psi_{n}^{k\mbox{-}app}$ is identical to $I_{\max,n}^{1}:=((\lambda_{n}^{1})^{2}T_{\min}^{1}+t_{n}^{1},\infty)$.\par 
We shall show that there exists a constant $B>0$ with the following property: for any number $k$, there exists a number $N(k)$ such that for any $n\ge N(k)$, 
\begin{equation}\label{14/08/30/17:04}
\| \psi_{n}^{k\mbox{-}app} \|_{V_{2+\frac{4}{d}}([0,\infty))\cap W_{2^{*}}([0,\infty))}
\le B.
\end{equation}
If $-\frac{t_{\infty}^{1}}{(\lambda_{\infty}^{1})^{2}}=-\infty$, then we see from \eqref{14/02/03/11:02} and \eqref{14/02/08/16:45} that Lemma \ref{14/06/17/11:32} is available on the whole interval $\mathbb{R}$. Thus, we can take a constant $B_{1}>0$ with the following property: for any number $k$, there exists a number $N_{1}(k)$ such that 
 for any $n\ge N_{1}(k)$, 
\begin{equation}\label{14/08/30/16:40}
\| \psi_{n}^{k\mbox{-}app} \|_{V_{2+\frac{4}{d}}(\mathbb{R})\cap W_{2^{*}}(\mathbb{R})}
\le B_{1}.
\end{equation}
On the other hand, if $-\frac{t_{\infty}^{1}}{(\lambda_{\infty}^{1})^{2}}\neq -\infty$, then it follows from the construction of $\widetilde{\psi}^{1}$ (see \eqref{14/01/26/14:26}) that $-\frac{t_{\infty}^{1}}{(\lambda_{\infty}^{1})^{2}} \in (T_{\min}^{1}, \infty)\cup \{\infty\}$. This implies that there exist $\tau_{0} \in (T_{\min}^{1}, \infty)$ and a number $N_{0}$ such that $\tau_{0}<-\frac{t_{n}^{1}}{(\lambda_{n}^{1})^{2}}$ for any $n\ge N_{0}$, so that $[0, \infty) \subset ((\lambda_{n}^{1})^{2}\tau_{0}+t_{n}^{1}, \infty) \subset I_{\max,n}^{1}$. Furthermore, we see from Lemma \ref{14/06/17/11:32} and \eqref{14/02/03/11:02} with $\tau=\tau_{0}$ that there exists a constant $B_{2}>0$ with the following property: for any number $k$, there exists a number $N_{2}(k)$ such that for any $n\ge \max\{N_{0}, N_{2}(k)\}$, 
\begin{equation}\label{14/02/03/15:00}
\| \psi_{n}^{k\mbox{-}app} \|_{V_{2+\frac{4}{d}}([0,\infty))\cap W_{2^{*}}([0,\infty))}
\le B_{2}.
\end{equation}
Thus, \eqref{14/08/30/16:40} and \eqref{14/02/03/15:00} give the desired result \eqref{14/08/30/17:04}. 
\par 
Now, we shall finish the proof of \eqref{14/02/02/14:36} for $\lambda_{\infty}^{1}\in \{0,1\}$. Note that it follows from \eqref{14/01/26/12:16}, \eqref{17/11/02/10:37} and Lemma \ref{14/06/21/15:10} that 
\begin{equation}\label{14/02/03/15:06}
\begin{split}
\| \psi_{n}(0)-\psi_{n}^{k\mbox{-}app}(0) \|_{H^{1}}
&=
\big\| \sum_{j=1}^{k}g_{n}^{j}\sigma_{n}^{j}e^{-i\frac{t_{n}^{j}}{(\lambda_{n}^{j})^{2}}\Delta}\widetilde{u}^{j}
-
\sum_{j=1}^{k}g_{n}^{j}\sigma_{n}^{j}\widetilde{\psi}^{j}\Big( -\frac{t_{n}^{j}}{(\lambda_{n}^{j})^{2}}\Big) 
\big\|_{H^{1}}
\\[6pt]
&\lesssim 
\sum_{j=1}^{k}
\big\| e^{-\frac{t_{n}^{j}}{(\lambda_{n}^{j})^{2}} 
\Delta}
\widetilde{u}^{j}
-
\widetilde{\psi}^{j}\Big( -\frac{t_{n}^{j}}{(\lambda_{n}^{j})^{2}}\Big) 
\big\|_{H^{1}}.
\end{split}
\end{equation}
This estimate together with \eqref{14/01/27/14:51} shows that for any number $k$, there exists a number  $N_{3}(k)$ such that
\begin{equation}\label{14/02/03/15:39}
\sup_{n\ge N_{3}(k)} \| \psi_{n}(0)-\psi_{n}^{k\mbox{-}app}(0) \|_{H^{1}}\le 1. \end{equation}
We also see from \eqref{14/01/26/10:45} that   
\begin{equation}\label{14/02/03/20:19} 
\sup_{n\ge 1}\|\psi_{n}\|_{L_{t}^{\infty}H_{x}^{1}([0,\infty))} \lesssim 1. 
\end{equation}
Let  $\delta_{0}>0$ be a constant determined by the long-time perturbation theory (Lemma \ref{14/06/05/10:08}) together with \eqref{14/08/30/17:04}, \eqref{14/02/03/15:39} and \eqref{14/02/03/20:19}. Then, we see from Sobolev's embedding, Strichartz' estimate, \eqref{14/02/03/15:06} and \eqref{14/01/27/14:51} that for any number $k$, there exists a number $N_{4}(k)$ such that for any $n\ge N_{4}(k)$, 
\begin{equation}\label{14/02/03/15:50}
\begin{split}
&\| e^{it\Delta}\big\{ \psi_{n}(0)-\psi_{n}^{k\mbox{-}app}(0) \big\} 
\|_{V_{2+\frac{4}{d}}([0,\infty)\cap W_{2^{*}}([0,\infty))} 
\\[6pt]
&\lesssim 
\| \langle \nabla \rangle e^{it\Delta}\big\{ \psi_{n}(0)-\psi_{n}^{k\mbox{-}app}(0) \big\} \|_{St([0,\infty))}
\lesssim 
\| \psi_{n}(0)-\psi_{n}^{k\mbox{-}app}(0) \|_{H^{1}}
\ll   
\delta_{0}.
\end{split}
\end{equation}
Moreover, it follows from Lemma \ref{14/06/14/10:52} that we can take numbers $k_{0}$ and $N_{5}$ such that for any $n \ge N_{5}$, 
\begin{equation}\label{14/02/03/16:05} 
\| \langle \nabla \rangle  e[\psi_{n}^{k_{0}\mbox{-}app}] 
\|_{L_{t,x}^{\frac{2(d+2)}{d+4}}([0,\infty))}
\le \delta_{0}.
\end{equation}
Thus, applying the long-time perturbation theory (Lemma \ref{14/06/05/10:08}) to $\psi_{n}$ and $\psi_{n}^{k_{0}\mbox{-}app}$ with $n\ge \max\{N(k_{0}), N_{3}(k_{0}), N_{4}(k_{0}), N_{5} \}$, we find that 
\begin{equation}\label{14/08/29/16:17}
\| \langle \nabla \rangle \psi_{n} \|_{St([0,\infty))}<\infty. 
\end{equation}
However, this contradicts \eqref{14/01/25/15:51}. Thus, we have shown that \eqref{14/02/02/14:36} holds when $\lambda_{\infty}^{1}\in \{0,1\}$. 
\par 
Next, we assume that $\lambda_{\infty}^{1}=\infty$. Then, \eqref{14/08/30/17:04} holds without the hypothesis \eqref{14/02/03/11:02}. Furthermore, the same argument   as \eqref{14/02/03/15:06} through \eqref{14/08/29/16:17} shows that the case $\lambda_{\infty}^{1}=\infty$ never happens. 
\par 
It remains to prove  \eqref{14/02/05/11:10}. Suppose to the contrary that the claim \eqref{14/02/05/11:10} was false. Then, it follows from \eqref{14/01/29/11:41}, Theorem \ref{15/04/05/15:27}, Theorem \ref{15/03/24/16:05}, Lemma \ref{14/01/29/20:50} and \eqref{14/02/02/14:36} that $\| \sigma_{\infty}^{j}\tilde{\psi}^{j}\|_{St([0,\infty))}<\infty$. However, this contradicts \eqref{14/02/02/14:36}. Thus, we have proved that the  claim \eqref{14/02/05/11:10}  is true.  
\end{proof}

\subsubsection{Critical element and completion of the proof}\label{15/07/11/11:30}
We shall show the existence of ``critical element'' (see Lemma \ref{14/02/04/10:06} below). Furthermore, we will derive a contradiction by using it under the hypothesis \eqref{14/01/25/15:35}, which completes the proof of Proposition \ref{14/01/25/14:59}. 
\par 
Let us begin by defining the functionals $\mathcal{S}_{\omega}^{j}$ and $\mathcal{I}_{\omega}^{j}$. For each $j\ge 1$, we define   
\begin{equation}\label{14/02/06/09:41}
\mathcal{S}_{\omega}^{j}
:=\left\{ \begin{array}{ccc}
\mathcal{S}_{\omega} &\mbox{if}& \lambda_{\infty}^{j}\in \{1,\infty\},
\\[6pt]
\mathcal{H}^{\ddagger} &\mbox{if}& \lambda_{\infty}^{j}=0,
\end{array} \right.
\qquad 
\mathcal{I}_{\omega}^{j}:=
\left\{ \begin{array}{lcc}
\mathcal{I}_{\omega}
&\mbox{if}& \lambda_{\infty}^{j}\in \{1,\infty\},
\\[6pt]
\mathcal{I}^{\ddagger}
&\mbox{if}& \lambda_{\infty}^{j}=0. 
\end{array} \right.
\end{equation}

\begin{lemma}\label{14/02/04/12:07}
Assume \eqref{14/02/01/17:53}. Then, for any $\delta>0$ and any number $k\ge 2$, there exists a number $N(\delta,k)$ such that for any $t\in \mathbb{R}$ and any $n\ge N(\delta,k)$, 
\begin{equation}\label{14/02/05/09:16}
\sum_{j=2}^{k} \mathcal{I}_{\omega}^{j}\big( \sigma_{\infty}^{j}\widetilde{\psi}^{j}(t) \big)
+
\mathcal{I}_{\omega}(w_{n}^{k}) \le \delta.
\end{equation}
\end{lemma}
\begin{proof}[Proof of Lemma \ref{14/02/04/12:07}]
We see from  \eqref{15/05/05/09:09}, \eqref{14/01/26/10:38} and \eqref{14/01/26/14:09} that for any $\delta>0$ and any number $k$, there exists a number $N_{1}(\delta, k)$ such that for any $n\ge N_{1}(\delta,k)$,
\begin{equation}\label{14/02/05/09:44}
\mathcal{S}_{\omega}\big( g_{n}^{1}\sigma_{n}^{1}e^{-i\frac{t_{n}^{1}}{(\lambda_{n}^{1})^{2}}\Delta}\widetilde{u}^{1} \big)
+
\sum_{j=2}^{k} \mathcal{S}_{\omega}\big( g_{n}^{j}\sigma_{n}^{j}e^{-i\frac{t_{n}^{j}}{(\lambda_{n}^{j})^{2}}\Delta}\widetilde{u}^{j} \big)
+
\mathcal{S}_{\omega}(w_{n}^{k})
\le  m_{\omega}+ \delta.
\end{equation}
If $\lambda_{\infty}^{1}=\lim_{n\to \infty}\lambda_{n}^{1}=0$, 
 then Lemma \ref{14/09/01/11:30} shows that for each $t \in I_{\max}^{1}$, 
\begin{equation}\label{14/02/04/20:17}
\lim_{n\to \infty} \| g_{n}^{1}\sigma_{n}^{1} \widetilde{\psi}^{1}(t) \|_{L^{2}}=0.
\end{equation} 
We see from Lemma \ref{14/01/29/20:50}, \eqref{14/01/29/12:09} in Lemma \ref{14/01/29/12:01}, \eqref{14/01/29/20:27}, \eqref{14/02/05/09:44}, \eqref{14/02/04/20:17} and \eqref{14/02/05/11:10} in Proposition \ref{14/02/02/14:35}  that for any $\delta>0$ and any number $k$, there exists $N_{2}(\delta,k)$ such that for any $t\in \mathbb{R}$ and any $n\ge N_{2}(\delta,k)$, 
\begin{equation}\label{14/02/05/11:30}
\begin{split} 
\sum_{j=2}^{k} \mathcal{I}_{\omega}^{j}\big( \sigma_{\infty}^{j}\widetilde{\psi}^{j}(t) \big)
+
\mathcal{I}_{\omega}(w_{n}^{k})
&\le 
\sum_{j=2}^{k} \mathcal{S}_{\omega}^{j}\big( \sigma_{\infty}^{j}\widetilde{\psi}^{j} \big)
+
\mathcal{S}_{\omega}(w_{n}^{k})
\\[6pt]
&\le 
\sum_{j=2}^{k} \mathcal{S}_{\omega}^{j}\big( g_{n}^{j}\sigma_{n}^{j}e^{-i\frac{t_{n}^{j}}{(\lambda_{n}^{j})^{2}}\Delta}\widetilde{u}^{j} \big)
+
\delta
+
\mathcal{S}_{\omega}(w_{n}^{k})
\\[6pt]
&\le 
m_{\omega}-\mathcal{S}_{\omega}^{1}( \sigma_{\infty}^{1}\widetilde{\psi}^{1})
+2\delta 
\le 
2\delta 
.
\end{split}
\end{equation}
Thus, we have completed the proof. 
\end{proof}

Put $\tau_{n}^{1}:=-\frac{t_{n}^{j}}{(\lambda_{n}^{1})^{2}}$. Then, passing to a subsequence, we may assume that 
\begin{equation}\label{15/04/27/10:15}
\tau_{\infty}^{1}:=\lim_{n\to \infty}\tau_{n}^{1} \in \mathbb{R}\cup \{\pm \infty\}.
\end{equation}

\begin{lemma}\label{14/09/05/09:52}
Assume \eqref{14/02/01/17:53}. Then, we have that $\lambda_{\infty}^{1}=1$. Furthermore, there exists a time $T \in I_{\max}^{1}$ such that 
\begin{equation}\label{14/05/29/11:08}
\inf_{t\in [T,T_{\max}^{1})}\widetilde{d}_{\omega}\big( \widetilde{\psi}^{1}(t) \big) \ge \frac{R_{*}}{2}
\end{equation}
and 
\begin{equation}\label{14/09/06/14:30}
\inf_{t\in [T,T_{\max}^{1})}\mathcal{K}\big( \widetilde{\psi}^{1}(t) \big)\ge 
\frac{\kappa_{1}(R_{*})}{2},
\end{equation}
where $\kappa_{1}(R_{*})$ is the constant appearing in \eqref{14/01/26/11:03}. 
\end{lemma}
\begin{proof}[Proof of Lemma \ref{14/09/05/09:52}]
Proposition \ref{14/02/02/14:35} together with \eqref{14/02/08/16:32} shows that $\lambda_{\infty}^{1}\neq \infty$ and $\tau_{\infty}^{1}\neq \infty$. If $\tau_{\infty}^{1}=-\infty$, then it follows from \eqref{14/02/08/16:45} that $T_{\min}^{1}=-\infty$ and $\|\sigma_{\infty}^{1}\widetilde{\psi}^{1}\big\|_{\mathcal{W}^{1}((-\infty,T])}<\infty$ for any $T<T_{\max}^{1}$.
 If $\tau_{\infty}^{1}\in \mathbb{R}$, then we see from the construction of $\widetilde{\psi}^{1}$  (see \eqref{14/01/26/14:26}) that $\tau_{\infty}^{1} \in I_{\max}^{1}$.  Put 
\begin{equation}\label{14/09/05/16:23}
\tau_{\min}^{1}:=
\left\{ \begin{array}{ccl}
\displaystyle{\frac{ T_{\min}^{1}+ \tau_{\infty}^{1}}{2}} &\mbox{if}& T_{\min}^{1}>-\infty, 
\\[6pt]
\tau_{\infty}^{1}-1 &\mbox{if}& \mbox{$\tau_{\infty}^{1}\in \mathbb{R}$ and $T_{\min}^{1}=-\infty$},
\\[6pt]
-\infty &\mbox{if}& \tau_{\infty}^{1}=-\infty.
\end{array} \right.
\end{equation}
Then, $T_{\min}^{1}\le \tau_{\min}^{1}\le \tau_{\infty}^{1}$, and we can take a number $N_{1}$ such that for any $n\ge N_{1}$, 
\begin{equation}\label{14/09/05/10:50}
\tau_{\min}^{1} < -\frac{t_{n}^{1}}{(\lambda_{n}^{1})^{2}} <  T_{\max}^{1}.
\end{equation}
Note that $0\in I_{\max,n}^{1}$ for all $n\ge N_{1}$. Furthermore, 
 for each $T \in (\tau_{\min}^{1}, T_{\max}^{1})$, 
\begin{equation}\label{14/09/05/11:09}
\| \sigma_{\infty}^{1} \widetilde{\psi}^{1}\|_{\mathcal{W}^{1}((\tau_{\min}^{1},T])} <\infty,
\end{equation}
which together with Lemma \ref{14/02/02/18:02} shows that for any $T \in (\tau_{\min}^{1}, T_{\max}^{1})$, there exists a constant $A(T)>0$ such that  
\begin{equation}\label{14/09/05/10:45}
\|\langle \nabla \rangle \widetilde{\psi}^{1}\|_{St((\tau_{\min}^{1},T])}
\le A(T).
\end{equation}
Sobolev's embedding also gives us that 
\begin{equation}\label{14/09/05/17:10}
\|\sigma_{\infty}^{1} \widetilde{\psi}^{1}\|_{\mathcal{W}^{1}((\tau_{\min}^{1},T])}
\lesssim  
\|\langle \nabla \rangle \widetilde{\psi}^{1}\|_{St((\tau_{\min}^{1},T])}
\le A(T).
\end{equation}

Now, we consider the approximate solution $\psi_{n}^{k\mbox{-}app}$ defined by \eqref{14/02/03/11:10}. Lemma \ref{14/06/17/11:32} together with \eqref{14/09/05/17:10} shows that for any  $T \in (\tau_{\min}^{1}, T_{\max}^{1})$, there exists $B(T)>0$ with the following property: for any number $k$, there exists a number $N_{2}(k)$ such that for any $n\ge N_{2}(k)$, 
\begin{equation}\label{14/09/05/11:52}
\|  \psi_{n}^{k\mbox{-}app} 
\|_{V_{2+\frac{4}{d}}(I_{n}(T))\cap W_{2^{*}}(I_{n}(T))}
+
\|\langle \nabla \rangle \psi_{n}^{k\mbox{-}app} 
\|_{L_{t}^{\infty}L_{x}^{2}(I_{n}(T))}
\le B(T),
\end{equation}
where 
\begin{equation}\label{14/09/05/12:30}
I_{n}(T):=\big((\lambda_{n}^{1})^{2}\tau_{\min}^{1}+t_{n}^{1},\,  (\lambda_{n}^{1})^{2}T+t_{n}^{1} \big].
\end{equation}
Lemma \ref{14/06/14/10:52} together with \eqref{14/09/05/17:10} shows that for any $\delta \in (0,1)$ and any $T\in (\tau_{\min}^{1},T_{\max}^{1})$, there exists numbers $k_{0}$ (depending on $\delta$ and $T$) and $N_{3}(\delta,T)$ such that for any $n \ge N_{3}(\delta,T)$, 
\begin{equation}\label{14/09/05/14:17} 
\| \langle \nabla \rangle e[\psi_{n}^{k_{0}\mbox{-}app}] 
\|_{L_{t,x}^{\frac{2(d+2)}{d+4}}(I_{n}(T))}
\le \delta.
\end{equation}
We also see from \eqref{14/01/26/12:16}, Lemma \ref{14/06/21/15:10} and \eqref{14/01/27/14:51} that for any number $k$ and any $\gamma>0$, there exists a number $N_{4}(k,\gamma)$ such that for any $n\ge N_{4}(k,\gamma)$,
\begin{equation}\label{14/09/05/14:30}
\begin{split}
\big\| \psi_{n}(0)-\psi_{n}^{k\mbox{-}app}(0)\big\|_{H^{1}}
&= 
\big\| \sum_{j=1}^{k}g_{n}^{j}\sigma_{n}^{j}e^{-i\frac{t_{n}^{j}}{(\lambda_{n}^{j})^{2}}\Delta}\widetilde{u}^{j}
-
\sum_{j=1}^{k}g_{n}^{j}\sigma_{n}^{j}\widetilde{\psi}^{j}\Big( -\frac{t_{n}^{j}}{(\lambda_{n}^{j})^{2}}\Big) 
\big\|_{H^{1}}
\\[6pt]
&\lesssim 
\sum_{j=1}^{k}
\big\| 
e^{-i\frac{t_{n}^{j}}{(\lambda_{n}^{j})^{2}} 
\Delta}
\widetilde{u}^{j}
-
\widetilde{\psi}^{j}\Big( -\frac{t_{n}^{j}}{(\lambda_{n}^{j})^{2}}\Big) 
\big\|_{H^{1}}
\le \gamma.
\end{split}
\end{equation}
Furthermore, it follows from Strichartz' estimate and \eqref{14/09/05/14:30} that for any number $k$ and any $\delta>0$, there exists a number $N_{5}(k,\delta)$ such that for any $n\ge N_{5}(k,\delta)$,
\begin{equation}\label{14/09/05/14:35}
\big\| \langle \nabla \rangle e^{it\Delta}\big\{ \psi_{n}(0)-\psi_{n}^{k\mbox{-}app}(0) \big\} \big\|_{V_{p+1}(\mathbb{R})}
\lesssim
\big\| \psi_{n}(0)-\psi_{n}^{k\mbox{-}app}(0) \big\|_{H^{1}}
\le \delta.
\end{equation}

The long-time perturbation theory (Lemma \ref{14/06/05/10:08}) together with 
 \eqref{14/09/05/11:52}, \eqref{14/09/05/14:17}, \eqref{14/09/05/14:30} and \eqref{14/09/05/14:35} shows that for any $T \in (\tau_{\min}^{1},T_{\max}^{1})$, there exist constants $\delta_{0}(T,\gamma)>0$ and $C(T,\gamma)>0$ with the following properties: for any $\delta \in (0, \delta_{0}(T,\gamma))$, we can take numbers $k_{0}$ (depending on $\delta$ and $T$) and $N_{6}(\delta,T)$ such that for any $n\ge N_{6}(\delta,T)$,
\begin{align}
\label{14/09/05/14:40} 
\big\|\langle \nabla \rangle \psi_{n} \big\|_{St(I_{n}(T))}
&\le C(T,\gamma),
\\[6pt]
\label{14/09/05/14:41}
\sup_{t\in I_{n}(T)} \big\|\langle \nabla \rangle \big\{ \psi_{n}(t)-\psi_{n}^{k_{0}\mbox{-}app}(t) \big\} \|_{L^{2}} 
&\le C(T,\gamma) \gamma 
.
\end{align}
Here, $C(T,\gamma)$ is non-decreasing with respect to $\gamma$. Furthermore, it follows from \eqref{14/09/05/14:41}, Lemma \ref{14/09/01/11:30} and Lemma \ref{14/02/04/12:07}  that for any $T\in (\tau_{\min}^{1},T_{\max}^{1})$, any $\delta  \in (0, \delta_{0}(T,\gamma))$ and the number $k_{0}$ determined by $T$ and $\delta$, we can take a number $N_{7}(\delta, T, \omega)$ such that for any $n\ge N_{7}(\delta, T, \omega)$ and $s\in \{0,1\}$,
\begin{equation}\label{14/09/05/14:47}
\begin{split}
&\| |\nabla|^{s}  \big\{
\psi_{n}\big( (\lambda_{n}^{1})^{2}T+t_{n}^{1}\big)-\psi_{n}^{1}\big( (\lambda_{n}^{1})^{2}T+t_{n}^{1}\big) \big\}
\|_{L^{2}}^{2}
\\[6pt]
&\lesssim  
\| | \nabla |^{s} \big\{
\psi_{n}\big( (\lambda_{n}^{1})^{2}T+t_{n}^{1}\big)
- 
\psi_{n}^{k_{0}\mbox{-}app}\big( (\lambda_{n}^{1})^{2}T+t_{n}^{1}\big)
\big\} \|_{L^{2}}^{2}
\\[6pt]
&\quad +
\big\| | \nabla |^{s}  
\sum_{j=2}^{k_{0}} 
\psi_{n}^{j}\big( (\lambda_{n}^{1})^{2}T+t_{n}^{1}\big)
\big\|_{L^{2}}^{2}
+
\||\nabla |^{s} w_{n}^{k_{0}}\|_{L^{2}}^{2}
\\[6pt]
&\lesssim  
C(T,\gamma)^{2} \gamma^{2}
+
k_{0}
\sum_{j=2}^{k_{0}} 
\big\| | \nabla |^{s}  \psi_{n}^{j} \big( (\lambda_{n}^{1})^{2}T+t_{n}^{1}\big)
\big\|_{L^{2}}^{2}
+
\|| \nabla |^{s} w_{n}^{k_{0}}\|_{L^{2}}^{2}
+
\delta
\\[6pt]
&=  
C(T,\gamma)^{2}\gamma^{2}
+
k_{0}
\sum_{j=2}^{k_{0}} 
\big\| | \nabla |^{s} g_{n}^{j} \sigma_{n}^{j} 
\widetilde{\psi}^{j}(T)
\big\|_{L^{2}}^{2}
+
\||\nabla |^{s} w_{n}^{k_{0}}\|_{L^{2}}^{2}
+
\delta
\\[6pt]
&\lesssim 
C(T,\gamma)^{2}\gamma^{2}+\delta.
\end{split}
\end{equation}

We shall show that $\lambda_{\infty}^{1}=1$. Suppose for contradiction that $\lambda_{\infty}^{1}=0$. Fix a time $T \in (\tau_{\min}^{1},T_{\max}^{1})$. Then, it follows from Lemma \ref{14/09/01/11:30} that 
\begin{equation}\label{14/09/05/14:48}
\lim_{n\to \infty}\big\| \psi_{n}^{1}\big( (\lambda_{n}^{1})^{2}T+t_{n}^{1}\big) \big\|_{L^{2}}
=
\lim_{n\to \infty}\big\| g_{n}^{1}\sigma_{n}^{1} \widetilde{\psi}^{1}(T) \big\|_{L^{2}}
=0.
\end{equation}
Furthermore, we see from $\mathcal{M}(\psi_{n})\equiv \mathcal{M}(\Phi_{\omega})$, \eqref{14/09/05/14:47} and \eqref{14/09/05/14:48} that for any sufficiently small $\omega>0$ and any $\delta \in (0,\delta_{0}(T,\gamma))$,  
\begin{equation}\label{14/09/01/15:31}
\|\Phi_{\omega}\|_{L^{2}}^{2}
=
\lim_{n\to \infty}\| \psi_{n} \big( (\lambda_{n}^{1})^{2}T+t_{n}^{1}\big) 
\|_{L^{2}}^{2}
\lesssim  
C(T,\gamma)^{2}\gamma^{2}+\delta.
\end{equation}
Taking $\gamma \to 0$ and $\delta\to 0$, we deduce that $\Phi_{\omega} \equiv 0$. However, this is a contradiction. Thus, we find that $\lambda_{\infty}^{1}=1$, so that $\lambda_{n}^{1}\equiv 1$, $\tau_{\infty}^{1}=-t_{\infty}^{1}$ and $\sigma_{n}^{1}\equiv 1$. 
\par 
Now, we are in a position to prove \eqref{14/05/29/11:08} and \eqref{14/09/06/14:30}. We first consider the case $\tau_{\infty}^{1}=-\infty$. In this case, we have $t_{\infty}^{1}=\infty$. Suppose for contradiction that \eqref{14/05/29/11:08} failed. Then, for any $T \in (\tau_{\min}^{1},T_{\max}^{1})$, we can take $T' \in [T, T_{\max}^{1})$ such that 
\begin{equation}\label{14/09/06/11:16}
\widetilde{d}_{\omega}\big( \widetilde{\psi}^{1}(T') \big)\le \frac{2}{3}R_{*}.
\end{equation}
Since $\lim_{n\to \infty}t_{n}^{1}=\infty$, we can take a number $N_{8}(T')$ such that $T'+t_{n}^{1} \ge 0$ for any $n\ge N_{8}(T')$. Thus, we find from \eqref{14/01/26/10:40}, \eqref{14/09/05/14:47} with $\lambda_{n}^{1}\equiv 1$ and \eqref{14/09/06/11:16} that 
\begin{equation}\label{14/09/05/17:47}
\begin{split}
R_{*}
&\le \widetilde{d}_{\omega}\big( \psi_{n}\big( T'+t_{n}^{1}\big) \big)
\le 
\widetilde{d}_{\omega}\big( \psi_{n}^{1}\big( T'+t_{n}^{1} \big) \big)
+o_{\gamma}(1)
+o_{\delta}(1)
\\[6pt]
&=
\widetilde{d}_{\omega}\big( \widetilde{\psi}^{1}(T') \big)
+o_{\gamma}(1)+o_{\delta}(1)
\le \frac{2}{3}R_{*}+o_{\gamma}(1)+o_{\delta}(1).
\end{split}
\end{equation}
However, this is impossible for sufficiently small $\gamma$ and $\delta$. Hence, the claim \eqref{14/05/29/11:08} is true. Suppose next that \eqref{14/09/06/14:30} failed. Then, for any $T \in (\tau_{\min}^{1}, T_{\max}^{1})$, there exists $T' \in [T, T_{\max}^{1})$ such that 
\begin{equation}\label{14/09/07/10:53}
\mathcal{K}\big( \widetilde{\psi}^{1}(T') \big)\le \frac{2}{3}\kappa_{1}(R_{*}).\end{equation}
Let $N_{8}(T')$ be the number obtained above. Then, we find from \eqref{14/01/26/11:03}, \eqref{14/09/05/14:47} with $\lambda_{n}^{1}\equiv 1$ and \eqref{14/09/07/10:53} that 
\begin{equation}\label{14/09/07/11:07}
\begin{split}
\kappa_{1}(R_{*}) 
&\le 
\mathcal{K}\big( \psi_{n}\big( T'+t_{n}^{1}\big) \big)
\\[6pt]
&\le 
\mathcal{K}\big( \psi_{n}^{1}\big( T'+t_{n}^{1} \big) \big)
+o_{\gamma}(1)+ o_{\delta}(1)  
\le \frac{2}{3} \kappa_{1}(R_{*})+o_{\gamma}(1)+o_{\delta}(1).
\end{split}
\end{equation}
However, for a sufficiently small $\gamma$ and $\delta$, this is a contradiction. Hence, the claim \eqref{14/09/06/14:30} is true.
\par 
It remains to consider the case $\tau_{\infty}^{1}\in \mathbb{R}$. In this case, we have that $\tau_{\min}^{1}< \tau_{\infty}^{1}=-t_{\infty}^{1}<T_{\max}^{1}$ (see \eqref{14/09/05/10:50}). Let $T' \in [-t_{\infty}^{1},T_{\max}^{1})$ be a time for which \eqref{14/09/06/11:16} or \eqref{14/09/07/10:53} holds. Then, we can derive a contradiction as well as the case $\tau_{\infty}^{1}=-\infty$. 
\end{proof}
\begin{lemma}\label{14/09/18/07:00}
Assume \eqref{14/02/01/17:53}. Then, we have $\mathcal{M}( \widetilde{\psi}^{1})=\mathcal{M}(\Phi_{\omega})$.
\end{lemma} 
\begin{proof}[Proof of Lemma \ref{14/09/18/07:00}]
We see from  $\lambda_{n}^{1}\equiv 1$ (see Lemma \ref{14/09/05/09:52}), 
 \eqref{17/10/14/16:19} with $k=2$, \eqref{14/01/27/14:51}, Lemma \ref{14/02/04/12:07} and  Lemma \ref{14/09/01/11:30} that for any $\delta>0$, there exists a number $N(\delta)$ such that for any $n\ge N(\delta)$, 
\begin{equation}\label{14/09/18/15:06}
\begin{split}
\omega \delta
&\ge 
 \omega \big| \mathcal{M}(\psi_{n}(0)) 
-
\mathcal{M}( e^{-i t_{n}^{1}\Delta}\widetilde{u}^{1}) 
\big|
-
\omega \mathcal{M}( g_{n}^{2}e^{-i t_{n}^{2}\Delta}\widetilde{u}^{2}) 
-
\omega \mathcal{M}(w_{n}^{2})
\\[6pt]
&\ge 
\omega \big|
\mathcal{M}(\psi_{n}(0)) 
-
\mathcal{M}( e^{-i t_{n}^{1}\Delta}\widetilde{u}^{1} ) 
\big|
-\delta 
\\[6pt]
&\ge 
\omega \big|
\mathcal{M}(\Phi_{\omega}) 
-
\mathcal{M}( \widetilde{\psi}^{1}) 
\big|
-\omega \delta -\delta.
\end{split}
\end{equation}
Since $\delta$ is an arbitrary constant, \eqref{14/09/18/15:06} implies the desired result.   
\end{proof}

\begin{proposition}\label{14/02/04/10:06}
There exists a solution $\Psi$ to \eqref{12/03/23/17:57} such that 
\begin{align}
\label{14/02/04/08:01}
&\Psi \in S_{\omega,R_{*},+}^{\varepsilon_{*}}, 
\\[6pt]
\label{14/09/30/09:20}
&\inf_{t\in [0,T_{\max})}\mathcal{K}(\Psi(t))\ge \frac{\kappa_{1}(R_{*})}{2},
\\[6pt]
\label{14/02/04/08:02}
&\mathcal{H}(\Psi)=E_{*}, 
\\[6pt]
\label{14/02/04/08:03}
&\|\Psi\|_{W_{p+1}([0,T_{\max}))\cap W_{2^{*}}([0,T_{\max}))}=\infty,
\\[6pt]
\label{14/09/19/17:12}
&\big\{\Psi(t) \colon t\in [0,T_{\max}) \big\} \ \mbox{is precompact in $H^{1}(\mathbb{R}^{d})$},
\end{align}
where $T_{\max}$ denotes the maximal lifespan of $\Psi$. 
\end{proposition}
\begin{proof}[Proof of Proposition \ref{14/02/04/10:06}]
Without loss of generality, we may assume \eqref{14/02/01/17:53}. Then, we see from Lemma \ref{14/09/05/09:52} that $\lambda_{\infty}^{1}=1$ (hence $\sigma_{\infty}^{1}=1$ and $\widetilde{\psi}^{1}$ is a solution to \eqref{12/03/23/17:57}), and there exists $T \in I_{\max}^{1}$ such that 
\begin{equation}\label{14/09/07/11:18}
\inf_{t\in [T,T_{\max}^{1})} \widetilde{d}_{\omega}\big(\widetilde{\psi}^{1}(t)\big) \ge \frac{R_{*}}{2},
\quad 
\inf_{t\in [T,T_{\max}^{1})} \mathcal{K}\big(\widetilde{\psi}^{1}(t)\big)\ge \frac{\kappa_{1}(R_{*})}{2}. 
\end{equation}
We also see from \eqref{15/05/05/09:09}, \eqref{14/01/26/10:38}, \eqref{14/01/25/15:50}, \eqref{14/01/26/14:09}, \eqref{14/01/27/14:51} and Lemma \ref{14/01/29/12:01} that 
\begin{equation}\label{15/06/27/14:51}
\mathcal{S}_{\omega}(\widetilde{\psi}^{1}) 
\le 
\lim_{n\to \infty}\mathcal{S}_{\omega}(\psi_{n}) 
\le
\omega \mathcal{M}(\Phi_{\omega})+E_{*}=m_{\omega}.
\end{equation}
This together with Lemma \ref{14/09/18/07:00} and \eqref{14/09/07/11:18} shows that $\widetilde{\psi}^{1}(\cdot +T) \in S_{\omega,\frac{R_{*}}{2}}^{\frac{\varepsilon_{*}}{2}}$. Furthermore, we find from Lemma \ref{14/01/26/08:13}, Lemma \ref{13/05/14/21:19}, \eqref{14/09/07/11:18} and Lemma \ref{14/09/04/15:39}
 that $\widetilde{\psi}^{1}(\cdot +T_{0}) \in S_{\omega, R_{*},+}^{\varepsilon_{*}}$ for some $T_{0}\ge T$. Put $\Psi(t):=\widetilde{\psi}^{1}(t+T_{0})$. Then, the maximal lifespan of $\Psi$ is $T_{\max}(\Psi):=T_{\max}^{1}-T_{0}$, and $\Psi \in S_{\omega,R_{*},+}^{\varepsilon_{*}}$. We also find from Proposition \ref{14/02/02/14:35}, \eqref{14/09/19/16:12} and \eqref{14/09/07/11:18} that 
\begin{align}
\label{14/02/04/10:48}
&\| \Psi \|_{W_{p+1}([0, T_{\max}(\Psi)))\cap W_{2^{*}}([0,T_{\max}(\Psi))} =\infty ,
\\[6pt]
\label{14/02/04/11:46}
&
\mathcal{H}(\Psi)\ge \frac{1}{2}
\inf_{t\in [0,T_{\max}(\Psi))}\mathcal{K}\big(\Psi(t) \big) \ge \frac{\kappa_{1}(R_{*})}{4}.
\end{align}
Since $\mathcal{M}(\Psi)=\mathcal{M}(\widetilde{\psi}^{1})=\mathcal{M}(\Phi_{\omega})$, the definition of $E_{*}$ (see \eqref{14/01/22/12:05}) together with \eqref{15/06/27/14:51} and \eqref{14/02/04/10:48} shows the property \eqref{14/02/04/08:02}. 

Now, suppose for contradiction that there existed a number $j\ge 2$ such that $\widetilde{\psi}^{j}$ is non-trivial. Without loss of generality, we may assume that $\widetilde{\psi}^{2}$ is non-trivial.  Then, it follows from \eqref{14/02/05/11:47} and Theorem \ref{15/04/05/15:27} that 
\begin{equation}\label{14/09/07/15:38}
\left\{ \begin{array}{ccc}
\mathcal{H}\big( \widetilde{\psi}^{2}\big)>0
&\mbox{if}&  \lambda_{\infty}^{2}=1,
\\[6pt]
\mathcal{H}^{\ddagger}\big( \sigma_{\infty}^{2}\widetilde{\psi}^{2}\big)>0
&\mbox{if}&  \lambda_{\infty}^{2}=0 .
\end{array} \right. 
\end{equation}
We see from \eqref{14/01/29/20:27}, Lemma \ref{14/09/07/15:55} and \eqref{14/09/07/15:38} that 
\begin{equation}\label{14/09/07/15:43}
0<h_{2}:=\left\{ \begin{array}{ccc}
\lim_{n\to \infty}\mathcal{H}\big( g_{n}^{2}\sigma_{n}^{2}e^{-i\frac{t_{n}^{2}}{(\lambda_{n}^{2})^{2}}\Delta}\widetilde{u}^{2} \big)
=
\mathcal{H}\big( \widetilde{\psi}^{2}\big) 
&\mbox{if}& \lambda_{\infty}^{2}=1,
\\[6pt]
\lim_{n\to \infty}\mathcal{H}\big( g_{n}^{2}\sigma_{n}^{2}e^{-i\frac{t_{n}^{2}}{(\lambda_{n}^{2})^{2}}\Delta}\widetilde{u}^{2} \big)
=
\mathcal{H}^{\ddagger}\big( \sigma_{\infty}^{2}\widetilde{\psi}^{2}\big)
&\mbox{if}& \lambda_{\infty}^{2}=0. 
\end{array} \right.
\end{equation}
Furthermore, we see from  \eqref{14/01/26/14:08}, \eqref{14/01/25/15:50}, \eqref{14/01/29/12:09} in Lemma \ref{14/01/29/12:01} and \eqref{14/09/07/15:43} that 
\begin{equation}\label{14/09/07/14:07}
\begin{split} 
0&=
\lim_{n\to \infty}\Big\{ 
\mathcal{H}(\psi_{n}) 
-\sum_{j=1}^{2}\mathcal{H}\big( g_{n}^{j}\sigma_{n}^{j}e^{-i\frac{t_{n}^{j}}{(\lambda_{n}^{j})^{2}}\Delta}\widetilde{u}^{j} \big) 
-\mathcal{H}(w_{n}^{2}) 
\Big\}
\\
&\le 
E_{*} 
-
\lim_{n\to \infty}\mathcal{H}\big( e^{-i t_{n}^{1}\Delta}\widetilde{u}^{1} \big)-
h_{2}.
\end{split}
\end{equation}
This together with \eqref{14/01/29/20:27} shows that 
\begin{equation}\label{14/09/08/11:17}
\mathcal{H}( \Psi )
=
\mathcal{H}\big( \widetilde{\psi}^{1} \big)
=
\lim_{n\to \infty}\mathcal{H}\big( e^{-i t_{n}^{1}\Delta}\widetilde{u}^{1} \big)\le E_{*}-h_{2}.
\end{equation}
However, this contradicts the proved property \eqref{14/02/04/08:02}. Thus, we have found that $\widetilde{\psi}^{j}\equiv 0$ for all $j\ge 2$. Then, we also 
 see that $\widetilde{u}^{j}\equiv 0$ for all $j\ge 2$.
\par 
We return to the decomposition \eqref{14/01/26/12:16}:  
 $\psi_{n}(0)=e^{-it_{n}^{1}\Delta}\widetilde{u}^{1}+w_{n}^{1}$. 
 It follows from \eqref{17/10/14/16:19}, \eqref{14/01/29/20:27} and $\mathcal{M}(\widetilde{\psi}^{1})=\mathcal{M}(\Psi)=\mathcal{M}(\Phi_{\omega})=\mathcal{M}(\psi_{n})$ that 
\begin{equation}\label{14/09/19/15:50}
\lim_{n\to \infty}\|w_{n}^{1} \|_{L^{2}}=0. 
\end{equation}
Moreover, it follows from \eqref{14/01/25/15:50}, \eqref{14/01/26/14:08}, \eqref{14/01/29/20:27} and \eqref{14/02/04/08:02} that 
\begin{equation}\label{14/09/19/15:59}
\lim_{n\to \infty}\mathcal{H}\big( w_{n}^{1} \big)=0, 
\end{equation}
which together with \eqref{14/01/29/12:09} in Lemma \ref{14/01/29/12:01} shows  \begin{equation}\label{14/09/19/17:25}
\begin{split}
\lim_{n\to \infty} \|\nabla w_{n}^{1} \|_{L^{2}}^{2}
&\le 
\frac{d}{s_{p}} \lim_{n\to \infty}
\Big\{\mathcal{H}\big( w_{n}^{1} \big)
-\frac{2}{d(p-1)}\mathcal{K}(w_{n}^{k})
\Big\}
\le 
\frac{d}{s_{p}} \lim_{n\to \infty} \mathcal{H}\big( w_{n}^{1} \big)=0.
\end{split}
\end{equation}
Thus, we find that 
\begin{equation}\label{14/09/19/17:33}
\lim_{n\to \infty}\|w_{n}^{1} \|_{H^{1}}=0. 
\end{equation}

Finally, we shall prove the precompactness of $\{\Psi(t)\}$ in $H^{1}(\mathbb{R}^{d})$. Take a sequence $\{\tau_{n}\}$ in $[0,T_{\max}(\Psi))$. By the continuity in time, it suffices to consider the case where $\lim_{n\to \infty}\tau_{n}=T_{\max}(\Psi)$. Applying the above argument to $\{\Psi(t+\tau_{n})\}$, we can take a subsequence of $\{\tau_{n}\}$ (still denoted by the same symbol), a sequence $\{t_{n}\}$ in $\mathbb{R}$, a sequence $\{w_{n}\}$ in $H^{1}(\mathbb{R}^{d})$ and a function $\widetilde{u} \in H^{1}(\mathbb{R}^{d})$ such that 
\begin{align}
\label{14/09/20/09:45}
&\Psi(\tau_{n})=e^{-it_{n}\Delta}\widetilde{u}+w_{n},
\\[6pt]
\label{14/09/20/09:47}
&\lim_{n\to \infty}\|w_{n}\|_{H^{1}}=0,
\\[6pt]
\label{14/09/20/10:08}
&t_{\infty}:=\lim_{n\to \infty}t_{n}\in \mathbb{R}\cup \{\infty \},
\end{align}
and if $t_{\infty}=\infty$, then $T_{\min}(\Psi)=-\infty$.  We see from  \eqref{14/09/20/09:45} and \eqref{14/09/20/09:47} that 
\begin{equation}\label{14/09/29/11:20}
\lim_{n\to \infty} \| \Psi(\tau_{n})-e^{-it_{n}\Delta}\widetilde{u} \|_{H^{1}}
=0.
\end{equation}
Hence, in order to prove the precompactness, it suffices to show that $t_{\infty}\in \mathbb{R}$. Suppose for contradiction that  $t_{\infty}=\infty$. Then, we find from Strichartz' estimate and \eqref{14/09/29/11:20} that 
\begin{equation}\label{14/09/29/11:05}
\begin{split}
&
\lim_{n\to \infty} \| e^{it\Delta} \Psi(\tau_{n}) \|_{ V_{2+\frac{4}{d}}((-\infty,0])\cap W_{2^{*}}((-\infty,
0])}
\\[6pt]
&=
\lim_{n\to \infty}
\| e^{i(t+t_{n})\Delta}\Psi(\tau_{n}) \|_{ V_{2+\frac{4}{d}}((-\infty,-t_{n}])\cap W_{2^{*}}((-\infty,-t_{n}])}
\\[6pt]
&\le
\lim_{n\to \infty}\| e^{i(t+t_{n})\Delta} e^{-it_{n}\Delta}\widetilde{u} \|_{ V_{2+\frac{4}{d}}((-\infty,-t_{n}])\cap W_{2^{*}}((-\infty,-t_{n}])}
\\[6pt]
&\quad +
\lim_{n\to \infty}\|  e^{i(t+t_{n})\Delta}\big\{ \Psi(\tau_{n})-e^{-it_{n}\Delta}\widetilde{u}\big\} \|_{ V_{2+\frac{4}{d}}((-\infty,-t_{n}])\cap W_{2^{*}}((-\infty,-t_{n}])}
\\[6pt]
&\lesssim
\lim_{n\to \infty} \| \langle \nabla \rangle e^{it\Delta}\widetilde{u} \|_{ V_{2+\frac{4}{d}}((-\infty,-t_{n}])\cap V_{2^{*}}((-\infty,-t_{n}])}
+
\lim_{n\to \infty} \| \Psi(\tau_{n})-e^{-it_{n}\Delta}\widetilde{u} \|_{H^{1}}
\\[6pt]
&=0.
\end{split}
\end{equation}
Thus, the small-data theory (see Lemma \ref{14/01/30/10:19} below) shows that there exists a number $N\ge 1$ such that for any $n\ge N$,  
\begin{equation}\label{14/09/20/09:54}
\| \Psi \|_{ W_{p+1}((-\infty,\tau_{n}])\cap W_{2^{*}}((-\infty,\tau_{n}])}
=
\| \Psi(\cdot+\tau_{n}) \|_{ W_{p+1}((-\infty,0])\cap W_{2^{*}}((-\infty,0])}
\lesssim 1,
\end{equation}
which together with $\lim_{n\to \infty}\tau_{n}=T_{\max}(\Psi)$ shows 
\begin{equation}\label{14/09/20/09:10}
\| \Psi \|_{ W_{p+1}((-\infty,T_{\max}(\Psi)))\cap W_{2^{*}}((-\infty,T_{\max}(\Psi)))}
\lesssim 1. 
\end{equation}
However, this contradicts \eqref{14/02/04/08:03}. Thus, we have $t_{\infty}\in \mathbb{R}$ and therefore the precompactness holds.  
\end{proof}

Now, we finish the proof of Proposition \ref{14/01/25/14:59}. Under the hypothesis \eqref{14/01/25/15:35}, we have shown the existence of solution $\Psi$ with the properties \eqref{14/02/04/08:01} through \eqref{14/09/19/17:12} in Proposition \ref{14/02/04/10:06}. However, the same argument as the proof of Proposition 5.3 in \cite{Kenig-Merle} shows that such a solution $\Psi$ never exists, and therefore we have arrived at a contradiction. Thus, the hypothesis \eqref{14/01/25/15:35} is false, and we have proved Proposition \ref{14/01/25/14:59}. 
\end{proof}

%%%%%%%%%%%%%%%%%%%%%%%%%%%%%%%%%%%%%%%%%%%%%%%%%%%%%%%%%%%%%%%%%%%%%%%%%%%%%%%%

\appendix 

%%%%%%%%%%%%%%%%%%%%%%%%%%%%%%%%%%%%%%%%%%%%%%%%%%%%%%%%%%%%%%%%%%%%%%%%%%%%%%%%
\section{Existence of ground state}\label{17/12/16/15:13}

In this section, we show the existence of ground state of \eqref{12/03/23/18:09} in three dimensions. We also give a proof of \eqref{15/02/24/20:55} at the end of this section.
\begin{proposition}\label{17/12/25/17:17}
Assume $d=3$ and $1<p<5$. Then, there exists $\omega_{3}>0$ such that for any $0< \omega <\omega_{3}$, a ground state of \eqref{12/03/23/18:09} exists. 
\end{proposition}

In order to prove Proposition \ref{17/12/25/17:17}, we consider the following equation in $\mathbb{R}^{d}$:
\begin{equation}\label{17/12/26/11:11} 
v -\Delta v
-\omega^{-\frac{5-p}{4}}|v|^{p-1}v
-|v|^{\frac{4}{d-2}}v
=0,
\qquad 
v \in H^{1}(\mathbb{R}^{d})\setminus\{0\} 
.
\end{equation}
The action associated with the equation \eqref{17/12/26/11:11} is 
\begin{equation}\label{17/12/26/11:52}
\widetilde{\mathcal{S}}_{\omega}(v)
=
\frac{1}{2}\|v \|_{L^{2}}^{2}
+
\frac{1}{2}\|\nabla v \|_{L^{2}}^{2}
-
\frac{\omega^{-\frac{5-p}{4}}}{p+1} \|v \|_{L^{p+1}}^{p+1}
-
\frac{1}{2^{*}}\|v \|_{L^{2^{*}}}^{2^{*}}
.
\end{equation}

It is easy to see that: if $d=3$ and $v$ is a solution to \eqref{17/12/26/11:11}, then the $H^{1}$-rescaled function $\omega^{\frac{1}{4}}v(\sqrt{\omega}\cdot )$ becomes a solution to \eqref{12/03/23/18:09} in $\mathbb{R}^{3}$; and if $u$ is a solution to \eqref{12/03/23/18:09} in $\mathbb{R}^{3}$, then $\omega^{-\frac{1}{4}}u(\cdot/\sqrt{\omega})$ becomes a solution to  \eqref{17/12/26/11:11}. Furthermore, if $d=3$ and $\widetilde{\Phi}_{\omega}$ is a ground state of \eqref{17/12/26/11:11} (namely, $\widetilde{\Phi}_{\omega}$ is a solution which minimizes the action $\widetilde{\mathcal{S}}_{\omega}$ among the solutions to \eqref{17/12/26/11:11}), then for any solution $u$ to \eqref{12/03/23/18:09} with $d=3$, 
\begin{equation}\label{17/12/26/12:17}
\mathcal{S}_{\omega}(\omega^{\frac{1}{4}}\widetilde{\Phi}_{\omega}(\sqrt{\omega}\cdot))
=
\widetilde{\mathcal{S}}_{\omega}(\widetilde{\Phi}_{\omega})
\le
\widetilde{\mathcal{S}}_{\omega}(\omega^{-\frac{1}{4}}u(\cdot/\sqrt{\omega}))
=
\mathcal{S}_{\omega}(u).
\end{equation}
Thus, we find that it is sufficient for Proposition \ref{17/12/25/17:17} to prove the existence of ground state of \eqref{17/12/26/11:11} in $\mathbb{R}^{3}$. An advantage to consider the equation \eqref{17/12/26/11:52} is that we can take into account the smallness of $\omega$ easier than the original problem \eqref{12/03/23/18:09} in our variational argument below.  
\par 
We define 
\begin{equation}\label{17/12/26/11:12}
\widetilde{m}_{\omega}
:=
\inf\big\{ \widetilde{\mathcal{T}}_{\omega}(v) \colon 
v \in H^{1}(\mathbb{R}^{3})\setminus \{0\}, \ \widetilde{\mathcal{N}}_{\omega}(v)\le 0
\big\},
\end{equation}
where 
\begin{align}
\label{17/12/26/16:33}
\widetilde{\mathcal{N}}_{\omega}(v)
&:= \|v\|_{L^{2}}^{2}+\|\nabla v \|_{L^{2}}^{2}
-\omega^{-\frac{5-p}{4}} \|v\|_{L^{p+1}}^{p+1}
-\|v \|_{L^{2^{*}}}^{2^{*}},
\\[6pt]
\label{17/12/26/17:25}
\widetilde{\mathcal{T}}_{\omega}(v)
&:= 
\frac{p-1}{2(p+1)}\big\{ \|v \|_{L^{2}}^{2}+ \|\nabla v \|_{L^{2}}^{2}\big\}
+
\frac{5-p}{6(p+1)}\|v \|_{L^{2^{*}}}^{2^{*}}
.
\end{align}
It is easy to verify that 
\begin{equation}\label{17/12/30/11:04}
\widetilde{\mathcal{T}}_{\omega}=\widetilde{\mathcal{S}}_{\omega}-\frac{1}{p+1}\widetilde{\mathcal{N}}_{\omega}.
\end{equation}  
Moreover, we have the following:
\begin{lemma}\label{17/12/30/11:01}
Assume $d=3$ and $1 < p < 5$. Then, we have the following: 
\\
{\rm (i)}
\begin{equation}\label{17/12/26/16:45} 
\widetilde{m}_{\omega}
=
\inf\big\{ \widetilde{\mathcal{S}}_{\omega}(v) \colon 
v \in H^{1}(\mathbb{R}^{3})\setminus \{0\}, \ \widetilde{\mathcal{N}}_{\omega}(v)=0
\big\}.
\end{equation}
{\rm (ii)}~Any minimizer of the variational problem associated with $\widetilde{m}_{\omega}$ becomes a ground state of \eqref{17/12/26/11:11}. 
\end{lemma}
\begin{proof}[Proof of Lemma \ref{17/12/30/11:01}]
First, we shall prove {\rm (i)}. Since $\widetilde{\mathcal{T}}_{\omega}(v) = \widetilde{\mathcal{S}}_{\omega}(v)$ for every $v \in H^1(\mathbb{R}^{3})$ with $\widetilde{\mathcal{N}}_{\omega}(v) = 0$, we see that $\widetilde{m}_\omega \le \inf\big\{ \widetilde{\mathcal{S}}_{\omega}(v) \colon 
v \in H^{1}(\mathbb{R}^{3})\setminus \{0\}, \ \widetilde{\mathcal{N}}_{\omega}(v)=0 \big\}$. In order to prove the opposite inequality, we note that for any $v_{0} \in H^{1}(\mathbb{R}^{3}) \setminus \{0\}$ with $\widetilde{\mathcal{N}}_{\omega}(v_{0}) \le 0$, there exists $0<\lambda_{0} \le 1$ such that $\widetilde{\mathcal{N}}_{\omega} (\lambda_{0} v_{0}) = 0$. Furthermore, this together with \eqref{17/12/30/11:04} shows that   
\begin{equation}\label{17/12/30/11:16}
\inf\big\{ \widetilde{\mathcal{S}}_{\omega}(v) \colon 
v \in H^{1}(\mathbb{R}^{3})\setminus \{0\}, \ \widetilde{\mathcal{N}}_{\omega}(v)=0
\big\}
\le 
\widetilde{\mathcal{S}}_{\omega} (\lambda_{0} v_{0}) 
= 
\widetilde{\mathcal{T}}_{\omega}(\lambda_{0} v_{0}) 
\le 
\widetilde{\mathcal{T}}_{\omega}(v_{0}).
\end{equation}
Since $v_{0}$ is an arbitrary function in $H^{1}(\mathbb{R}^{3}) \setminus \{0\}$ with $\widetilde{\mathcal{N}}_{\omega}(v_{0}) \le 0$, we find that \eqref{17/12/26/16:45} holds.  
\par 
Next, we shall prove the claim {\rm (ii)}. It is well known that if $\widetilde{Q}_{\omega}$ is a minimizer of $\inf\big\{ \widetilde{\mathcal{S}}_{\omega}(v) \colon 
v \in H^{1}(\mathbb{R}^{3})\setminus \{0\}, \ \widetilde{\mathcal{N}}_{\omega}(v)=0 \big\}$, then $\widetilde{Q}_{\omega}$ is also a ground state of \eqref{17/12/26/11:11}. Hence, we see from \eqref{17/12/30/11:04} and \eqref{17/12/26/16:45} that it suffices to prove that $\widetilde{\mathcal{N}}_{\omega}(\widetilde{Q}_{\omega})=0$ for all minimizer $\widetilde{Q}_{\omega}$ for $\widetilde{m}_{\omega}$. Suppose to the contrary that there exists a minimizer $\widetilde{Q}_{\omega}$ for $\widetilde{m}_{\omega}$ such that $\widetilde{\mathcal{N}}_{\omega}(\widetilde{Q}_{\omega})<0$. Then, we could take $0< \lambda_{0}<1$ such that $\widetilde{\mathcal{N}}_{\omega}(\lambda_{0}\widetilde{Q}_{\omega})=0$. Furthermore, we see from the definition of $\widetilde{m}_{\omega}$ (see \eqref{17/12/26/11:12}) and $0< \lambda_{0}<1$ that 
\begin{equation}\label{17/12/30/11:55}
\widetilde{m}_{\omega}
\le 
\widetilde{\mathcal{T}}_{\omega}(\lambda_{0}\widetilde{Q}_{\omega})
< \widetilde{\mathcal{T}}_{\omega}(\widetilde{Q}_{\omega})
=
\widetilde{m}_{\omega},
\end{equation}
which is a contradiction. Thus, $\widetilde{\mathcal{N}}_{\omega}(\widetilde{Q}_{\omega})=0$. 
\end{proof}

A key lemma to prove the existence of the minimizer is the following: 
\begin{lemma}\label{17/07/08/17:13}
Assume $d=3$ and $1<p<5$. Then, there exists $\omega_{3}>0$ such that for any $0< \omega <\omega_{3}$,  
\begin{equation}\label{17/07/08/17:15}
0< \widetilde{m}_{\omega}< \frac{1}{3}\sigma^{\frac{3}{2}}
.
\end{equation}
\end{lemma}
\begin{proof}[Proof of Lemma \ref{17/07/08/17:13}]
First, we shall show that $\widetilde{m}_{\omega}>0$. Let $v$ be a nontrivial function in $H^{1}(\mathbb{R}^{3})$ such that $\widetilde{\mathcal{N}}_{\omega}(v)\le 0$. Then, it follows from Sobolev's embedding that  
\begin{equation}\label{17/12/26/18:47}
\|v\|_{H^{1}}^{2}
=
\|v\|_{L^{2}}^{2}+\|\nabla v \|_{L^{2}}^{2}
\le \omega^{-\frac{5-p}{4}} \|v\|_{L^{p+1}}^{p+1}
+\|v \|_{L^{2^{*}}}^{2^{*}}
\lesssim 
\omega^{-\frac{5-p}{4}} \|v\|_{H^{1}}^{p+1} 
+\|v \|_{H^{1}}^{6}
.
\end{equation}
In particular, \eqref{17/12/26/18:47} implies that 
\begin{equation}\label{17/12/26/18:59}
\omega^{\frac{5-p}{4}} \lesssim 
\|v\|_{H^{1}}^{p-1} + \omega^{\frac{5-p}{4}} \|v \|_{H^{1}}^{4}
.
\end{equation}
Since the implicit constant of the inequality \eqref{17/12/26/18:59} depends only on $p$, we find that for any $1<p<5$ and any $\omega >0$, there exists $C(\omega)>0$ such that $C(\omega) \le \widetilde{\mathcal{T}}_{\omega}(v)$ for all $v \in H^{1}(\mathbb{R}^{d})\setminus \{0\}$ with $\widetilde{\mathcal{N}}_{\omega}(v)\le 0$. Thus, we have proved that  $\widetilde{m}_{\omega}>0$.

Next, we shall show that $\widetilde{m}_{\omega}<\frac{1}{3}\sigma^{\frac{3}{2}}$. We define 
\begin{equation}\label{17/07/08/13:33}
W_{\omega}(x)
:=
\omega^{-\frac{1}{2}}W(x / \omega)
=
\omega^{\frac{1}{2}}\Big(\omega^{2}+\frac{|x|^{2}}{3} \Big)^{-\frac{1}{2}}.
\end{equation}
Note that $W_{\omega}$ does not belong to $L^{2}(\mathbb{R}^{3})$. Let $\chi$ be an even smooth function on $\mathbb{R}$ such that $\chi(r)=1$ for $0\le r\le 1$, $\chi(r)=0$ for $r\ge 2$, and $\chi$ is non-increasing on $[0,\infty)$. Then, we define 
\begin{equation}\label{17/07/08/11:48}
\widetilde{W}_{\omega}(x)
:=
\chi(|x|)W_{\omega}(x)
.
\end{equation}
We see from \eqref{12/08/22/13:32} that 
\begin{align}
\label{17/07/08/13:27}
\|\nabla \widetilde{W}_{\omega}\|_{L^{2}}^{2}
&=
\sigma^{\frac{3}{2}}+O(\omega)
,
\\[6pt]
\label{17/07/08/14:29}
\| \widetilde{W}_{\omega} \|_{L^{2^{*}}}^{2^{*}}
&=
\sigma^{\frac{3}{2}}+O(\omega^{3})
.
\end{align}
Moreover, we find that for any $2\le q <6$, 
\begin{equation}\label{17/07/08/20:47}
\|\widetilde{W}_{\omega}\|_{L^{q}}^{q} 
=
\left\{ \begin{array}{ccc} 
O(\omega^{\frac{6-q}{2}})
&\mbox{if}& 3<q <6, 
\\[6pt]
O(\omega^{\frac{3}{2}}|\log{\omega}|)
&\mbox{if}& q=3,  
\\[6pt]
O(\omega^{\frac{q}{2}})
&\mbox{if}& 2\le  q <3 .
\end{array}
\right. 
\end{equation}
Note here that there exists a unique $\tau_{\omega,0}>0$ such that 
\begin{equation}\label{17/12/17/13:32}
\widetilde{\mathcal{N}}_{\omega}(\tau_{\omega, 0} \widetilde{W}_{\omega})=0.
\end{equation} 
Furthermore, it follows from \eqref{17/07/08/13:27} through \eqref{17/12/17/13:32} that: if $2<p<5$, then 
\begin{equation}\label{17/12/17/11:20}
\begin{split}
0&=
\tau_{\omega,0}^{2}\|\widetilde{W}_{\omega} \|_{L^{2}}^{2}
+
\tau_{\omega,0}^{2}
\|\nabla \widetilde{W}_{\omega} \|_{L^{2}}^{2}
-
\omega^{-\frac{5-p}{4}} 
\tau_{\omega,0}^{p+1} 
\|\widetilde{W}_{\omega} \|_{L^{p+1}}^{p+1}
-
\tau_{\omega,0}^{2^{*}} \|\widetilde{W}_{\omega} \|_{L^{2^{*}}}^{2^{*}}
\\[6pt]
&=
\tau_{\omega,0}^{2}O(\omega)
+
\tau_{\omega,0}^{2}
\{\sigma^{\frac{3}{2}}+O(\omega)\}
-
\tau_{\omega,\varepsilon}^{p+1} 
O(\omega^{\frac{5-p}{4}})
-
\tau_{\omega,0}^{6} 
\{\sigma^{\frac{3}{2}}+O(\omega^{3})\}
;
\end{split}
\end{equation}
if $p=2$, then 
\begin{equation}\label{17/12/26/17:35}
0=
\tau_{\omega,0}^{2}O(\omega)
+
\tau_{\omega,0}^{2}
\{\sigma^{\frac{3}{2}}+O(\omega)\}
-
\tau_{\omega,0}^{3} 
O(\omega^{\frac{3}{4}}|\log{\omega}|)
-
\tau_{\omega,0}^{6} 
\{\sigma^{\frac{3}{2}}+O(\omega^{3})\}
;
\end{equation}
and if $1< p<2$, then 
\begin{equation}\label{17/12/26/17:37}
0=
\tau_{\omega,0}^{2}O(\omega)
+
\tau_{\omega,0}^{2}
\{\sigma^{\frac{3}{2}}+O(\omega)\}
-
\tau_{\omega,0}^{p+1} 
O(\omega^{\frac{p+1}{2}-\frac{5-p}{4}})
-
\tau_{\omega,0}^{6} 
\{\sigma^{\frac{3}{2}}+O(\omega^{3})\}
.
\end{equation}
Thus, we find from \eqref{17/12/17/11:20} through \eqref{17/12/26/17:37} that   \begin{equation}\label{17/12/17/13:24}
\tau_{\omega,0}^{4}
=
1+O(\omega)
-
\tau_{\omega,0}^{p-1} o_{\omega}(1),
\end{equation}
where $o_{\omega}(1)$ means a certain function such that $\lim_{\omega \to 0}o_{\omega}(1)=0$. Furthermore, since $0< p-1 <4$, this implies that
\begin{equation}\label{17/12/17/13:25}
\lim_{\omega \to + 0}\tau_{\omega,0}\ge \frac{1}{2}.
\end{equation}

We introduce a function $y_{\omega} \colon (0,\infty) \to \mathbb{R}$ as 
\begin{equation}\label{17/07/08/15:55}
y_{\omega}(t):=\frac{1}{2}t^{2} \big\{ 
 \|\widetilde{W}_{\omega}\|_{L^{2}}^{2}
+
\|\nabla \widetilde{W}_{\omega}\|_{L^{2}}^{2}
\big\} 
-\frac{t^{2^{*}}}{2^{*}} 
\|\widetilde{W}_{\omega}\|_{L^{2^{*}}}^{2^{*}}
.
\end{equation}
It is easy to verify that the function $y_{\omega}$ attains its maximum only at the point
\begin{equation}\label{17/07/08/16:02}
\tau_{\omega, \max}:=\frac{\big\{ 
\|\widetilde{W}_{\omega}\|_{L^{2}}^{2}
+
\|\nabla \widetilde{W}_{\omega}\|_{L^{2}}^{2}
\big\}^{\frac{1}{4}} 
}{ \|\widetilde{W}_{\omega}\|_{L^{2^{*}}}^{\frac{3}{2}}
}.
\end{equation}
Note here that it follows from the definition of $\sigma$ (see \eqref{12/03/24/17:52}), \eqref{17/07/08/13:27} and \eqref{17/07/08/14:29} that
\begin{equation}\label{17/07/08/15:19}
\frac{\|\nabla \widetilde{W}_{\omega}\|_{L^{2}}^{2}}{\|\widetilde{W}_{\omega}\|_{L^{2^{*}}}^{2}}
=
\frac{\sigma^{\frac{3}{2}} 
+ 
O( \omega )}{\sigma^{\frac{1}{2}}
+
O(\omega^{3})}
=
\sigma +  O(\omega).
\end{equation} 
Moreover, we see from \eqref{17/07/08/13:27} and \eqref{17/07/08/20:47} that 
\begin{equation}\label{17/07/08/20:56}
\frac{\|\widetilde{W}_{\omega}\|_{L^{2}}^{2}}{\|\nabla \widetilde{W}_{\omega}\|_{L^{2}}^{2}}
=
O(\omega)
.
\end{equation} 
Hence, if $\omega$ is sufficiently small, then we have 
\begin{equation}\label{17/07/08/16:06}
\begin{split}
y(\tau_{\omega,\max})
&=
\frac{1}{3}
\Big(
\frac{ \|\widetilde{W}_{\omega}\|_{L^{2}}^{2}
+
\|\nabla\widetilde{W}_{\omega}\|_{L^{2}}^{2}
}{ \|\widetilde{W}_{\omega}\|_{L^{2^{*}}}^{2}}
\Big)^{\frac{3}{2}} 
=
\frac{1}{3}
\frac{\|\nabla \widetilde{W}_{\omega}\|_{L^{2}}^{3}
}{ \|\widetilde{W}_{\omega}\|_{L^{2^{*}}}^{3}
}
\Big(
1+ 
\frac{ \|\widetilde{W}_{\omega}\|_{L^{2}}^{2}
}{ \|\nabla \widetilde{W}_{\omega}\|_{L^{2}}^{2}
}
\Big)^{\frac{3}{2}} 
\\[6pt]
&=
\frac{1}{3}
\frac{\|\nabla \widetilde{W}_{\omega}\|_{L^{2}}^{3}
}{ \|\widetilde{W}_{\omega}\|_{L^{2^{*}}}^{3}
}
\Big\{
1+ \frac{3}{2}
\frac{ \|\widetilde{W}_{\omega}\|_{L^{2}}^{2}
}{ \|\nabla \widetilde{W}_{\omega}\|_{L^{2}}^{2}
}
+
O\big( 
\frac{ \|\widetilde{W}_{\omega}\|_{L^{2}}^{4}
}{ \|\nabla \widetilde{W}_{\omega}\|_{L^{2}}^{4}} 
\big)
\Big\} 
=
\frac{1}{3}\sigma^{\frac{3}{2}} +O(\omega)
.
\end{split}
\end{equation}
Furthermore, it follows from  \eqref{17/12/26/16:45}, \eqref{17/07/08/20:47}, \eqref{17/12/17/13:32}, the same computations concerning $\|\widetilde{W}_{\omega}\|_{L^{p+1}}^{p+1}$ as \eqref{17/12/17/11:20} through \eqref{17/12/26/17:37}, and \eqref{17/12/17/13:25} that 
\begin{equation}\label{17/07/08/15:33}
\begin{split}
\widetilde{m}_{\omega} 
&\le \widetilde{\mathcal{S}}_{\omega}(\tau_{\omega,0} \widetilde{W}_{\omega})
=
y_{\omega}(\tau_{\omega,0})
-
\frac{ \omega^{-\frac{5-p}{4}}}{p+1} 
\tau_{\omega,0}^{p+1}\|\widetilde{W}_{\omega} \|_{L^{p+1}}^{p+1}
\\[6pt]
&\le 
y_{\omega}(\tau_{\omega,\max})
-
c \omega^{\theta}
=
\frac{1}{3}\sigma^{\frac{3}{2}} 
+ 
O(\omega)
-
c \omega^{\theta}
\end{split} 
\end{equation}
for some $0< \theta <1$ and some $c>0$ depending only on $p$. Thus, we find that if $\omega$ is sufficiently small depending only on $p$, then \eqref{17/07/08/17:15} holds. 
\end{proof}

Now, we give a brief proof of Proposition \ref{17/12/25/17:17}:  
\begin{proof}[Proof of  Proposition \ref{17/12/25/17:17}] 
As mentioned above, if $d=3$ and $\widetilde{\Phi}_{\omega}$ is a ground state of \eqref{17/12/26/11:11}, then $\omega^{\frac{1}{4}}\widetilde{\Phi}_{\omega}(\sqrt{\omega} \cdot)$ is one of \eqref{12/03/23/18:09}. We also see from Lemma \ref{17/12/30/11:01} that any minimizer of the variational problem associated with $\widetilde{m}_{\omega}$ becomes a ground state of \eqref{17/12/26/11:11}. Thus, it suffices to prove that for the frequency $\omega_{3}>0$ given in Lemma \ref{17/07/08/17:13} and any $0<\omega <\omega_{3}$, there exists a minimizer for $\widetilde{m}_{\omega}$. The proof is standard: using Lemma \ref{17/07/08/17:13} and the Schwarz symmetrization,  we can prove the existence of minimizer for $\widetilde{m}_{\omega}$  (see, e.g, Propositoin 2.1 in \cite{AIKN}, and \cite{Brezis-Nirenberg}).
\end{proof}

At the end of this section, we discuss the variational value $m_{\omega}$ defined by \eqref{12/03/24/17:51} in three dimensions.
\begin{proposition}\label{17/12/30/10:09}
Assume $d=3$ and $1+\frac{4}{3}<p<5$. Let $\omega_{3}$ be the frequency given in Proposition \ref{17/12/25/17:17}. Then, for any $0< \omega <\omega_{3}$,
\begin{equation}\label{17/12/30/10:10}
0< m_{\omega}<\frac{1}{3}\sigma^{\frac{3}{2}}.
\end{equation}
Furthermore, any ground state $Q_{\omega}$ of \eqref{12/03/23/18:09} satisfies 
\begin{equation}\label{17/12/27/12:12}
\mathcal{S}_{\omega}(Q_{\omega})=m_{\omega}. 
\end{equation} 
\end{proposition}
\begin{proof}[Proof of Proposition \ref{17/12/30/10:09}]
First, we shall show that $m_{\omega}>0$. To this end, we recall that \eqref{12/03/23/18:17} holds for $d=3$ and $1+\frac{4}{3}<p<5$. Let $u$ be a nontrivial function in $H^{1}(\mathbb{R}^{3})$ such that $\mathcal{K}(u)\le 0$. Then, it follows from Sobolev's embedding that  
\begin{equation}\label{17/12/30/13:16}
\|\nabla u \|_{L^{2}}^{2}
\lesssim 
\|u\|_{L^{2}}^{\frac{5-p}{2}} 
\|\nabla u\|_{L^{2}}^{\frac{3(p-1)}{2}}
+
\|\nabla u \|_{L^{2}}^{6}
.
\end{equation}
In particular, if $\|u\|_{L^{2}}\ll 1$, then \eqref{17/12/30/13:16} implies that $1 \lesssim \|\nabla u\|_{L^{2}}$. Hence, if $\|u\|_{L^{2}}\ll 1$, then $\mathcal{I}_{\omega}(u)\ge \frac{p-1-\frac{4}{3}}{p-1}\|\nabla u\|_{L^{2}}^{2}\gtrsim 1$. On the other hand, if $\|u\|_{L^{2}}\gtrsim 1$, then $\mathcal{I}_{\omega}(u)\ge \frac{\omega}{2} \|u\|_{L^{2}}^{2} \gtrsim \omega$. Thus, we find that $\mathcal{I}_{\omega}(u)\gtrsim \min\{1,  \omega\}$ for all $u \in H^{1}(\mathbb{R}^{3}) \setminus \{0\}$ with $\mathcal{K}(u)\le 0$. This together with \eqref{12/03/23/18:17} shows $m_{\omega}>0$. 
\par 
Next, we shall prove \eqref{17/12/30/10:10}. Let $\widetilde{Q}_{\omega}$ be a minimizer of the variational problem associated with $\widetilde{m}_{\omega}$ (see the proof of Proposition \ref{17/12/25/17:17} for the existence of the minimizer). Then, it follows from Lemma \ref{17/12/30/11:01} that $\widetilde{Q}_{\omega}$ is a ground state of \eqref{17/12/26/11:11}. Furthermore, it is easy to verify that $\widetilde{\mathcal{N}}_{\omega}(\widetilde{Q})=0$. Put $Q_{\omega}:=\omega^{\frac{1}{4}}\widetilde{Q}_{\omega}(\sqrt{\omega}\cdot)$. Then, $Q_{\omega}$ becomes a ground state of \eqref{12/03/23/18:09} and satisfies the following 
 identity: 
\begin{equation}\label{17/12/27/12:37}
\mathcal{K}(Q_{\omega})=0.
\end{equation}
Hence, we find from the definition of $m_{\omega}$, \eqref{17/12/26/16:45} and Lemma \ref{17/07/08/17:13} that 
\begin{equation}\label{17/12/27/12:37}
m_{\omega} 
\le 
\mathcal{\mathcal{S}}_{\omega}(Q_{\omega})
=
\widetilde{\mathcal{S}}_{\omega}(\widetilde{Q}_{\omega})
=
\widetilde{m}_{\omega}<\frac{1}{3}\sigma^{\frac{3}{2}}. 
\end{equation}
Then, the standard argument (see, e.g., Proposition 2.1 in \cite{AIKN}) proves the existence of minimizer of the variational problem associated with $m_{\omega}$. Furthermore, it is well-known that any minimizer for $m_{\omega}$ becomes a ground state of \eqref{12/03/23/18:09} (see, e.g., Proposition 1.1 in \cite{AIKN}). Thus, \eqref{17/12/27/12:12} holds. 
\end{proof} 

Next, we give a proof of \eqref{15/02/24/20:55}. 
\begin{proof}[Proof of \eqref{15/02/24/20:55}]
We find from \eqref{12/08/22/13:32} that the function $W$ (see \eqref{13/03/16/16:43}) satisfies $\mathcal{K}^{\ddagger}(W)=0$ and $\mathcal{I}^{\ddagger}(W)=\mathcal{H}^{\ddagger}(W)=\frac{1}{d}\sigma^{\frac{d}{2}}$. Hence, $m^{\ddagger} \le \frac{1}{d}\sigma^{\frac{d}{2}}$ and therefore it remains to prove $\frac{1}{d}\sigma^{\frac{d}{2}} \le m^{\ddagger}$. Let $u \in H^{1}(\mathbb{R}^{d}) \setminus \{0\}$ satisfy $\mathcal{K}^{\ddagger}(u)\le 0$. Then, we can take $0< \lambda_{0} \le 1$ such that $\mathcal{K}^{\ddagger}(\lambda_{0} u)=0$. Furthermore, it follows from the definition of $\sigma$ (see \eqref{12/03/24/17:52}) and $\mathcal{K}^{\ddagger}(\lambda_{0} u)=\lambda_{0}^{2}\|\nabla u \|_{L^{2}}^{2}-\lambda_{0}^{2^{*}}\|u\|_{L^{2^{*}}}^{2^{*}}=0$ that 
\begin{equation}\label{18/01/31/17:37}
\sigma 
\le 
\frac{\|\nabla u \|_{L^{2}}^{2}}{\|  u \|_{L^{2^{*}}}^{2}}
\le 
\lambda_{0}^{\frac{4}{d-2}}\|  u \|_{L^{2^{*}}}^{\frac{4}{d-2}}
.
\end{equation}
Raising the both sides of \eqref{18/01/31/17:37} to the power of $\frac{d}{2}$, and using $\mathcal{K}^{\ddagger}(\lambda_{0} u)=0$ and $\lambda_{0}\le 1$, we see that  
\begin{equation}\label{18/01/31/17:51}
\begin{split}
\frac{1}{d}\sigma^{\frac{d}{2}}
&\le 
\frac{1}{d}\| \lambda_{0} u \|_{L^{2^{*}}}^{2^{*}}
=
\mathcal{H}^{\ddagger}(\lambda_{0} u )
\\[6pt]
&=
\mathcal{H}^{\ddagger}(\lambda_{0} u )
-
\frac{2}{d(p-1)}\mathcal{K}^{\ddagger}(\lambda_{0} u )
=
\mathcal{I}^{\ddagger}(\lambda_{0} u )
\le 
\mathcal{I}^{\ddagger}(u ).
\end{split}
\end{equation}
Since $u$ is arbitrary, we find that $\frac{1}{d}\sigma^{\frac{d}{2}} \le m^{\ddagger}$. Thus, we have completed the proof. 
\end{proof}

%%%%%%%%%%%%%%%%%%%%%%%%%%%%%%%%%%%%%%%%%%%%%%%%%%%%%%%%%%%%%%%%%%%%%%%%%%%%%%%%

\section{Fundamental properties of the linearized operators}
\label{15/01/29/14:22}

In this section, we discuss fundamental properties of the operators $L_{\omega,+}$ and $L_{\omega,-}$ defined by \eqref{13/02/02/17:56} and \eqref{13/02/02/17:57}, respectively. Throughout this section, we assume that $d\ge 3$, $1+\frac{4}{d}<p <2^{*}-1$ and $0<\omega <\omega_{1}$, where $\omega_{1}$ is the frequency given by Proposition \ref{13/03/29/15:30}. 
\par 
First, we can easily verify that for any $f \in H^{1}(\mathbb{R}^{d})$, 
\begin{equation}\label{13/02/19/10:42}
\Big| 
\langle 
L_{\omega,+} f, f
\rangle_{H^{-1},H^{1}}
\Big|
+
\Big| 
\langle   
L_{\omega,-} f, f
\rangle_{H^{-1},H^{1}}
\Big| 
\lesssim  
\big( \omega + 1+ 
\|\Phi_{\omega}\|_{L^{p+1}}^{p-1}
+
\|\Phi_{\omega}\|_{L^{2^{*}}}^{\frac{4}{d-2}} 
\big)
\|f\|_{H^{1}}^{2},
\end{equation}
where the implicit constant is independent of $\omega$. Moreover, we see from \eqref{13/01/04/14:46} and \eqref{13/02/20/9:45} that for any $u,v \in H_{real}^{1}(\mathbb{R}^{d})$,  
\begin{equation}\label{12/12/31/21:56}
\big[ \mathcal{S}_{\omega}''(\Phi_{\omega})u \big]v
=
\langle L_{\omega,+}\Re[u] 
+ 
i L_{\omega,-}\Im[u], v \rangle_{H^{-1},H^{1}}
. 
\end{equation}

\begin{lemma}\label{13/03/18/15:50}
The operator $L_{\omega,-}$ is non-negative and ${\rm Ker}\, L_{\omega,-}={\rm span}\{ \Phi_{\omega}\}$; in particular, there exists $\delta_{\omega}>0$ depending on $\omega$ such that for any $u\in H^{1}(\mathbb{R}^{d})$ with $(u,\Phi_{\omega})_{L_{real}^{2}}=0$, 
\begin{equation}\label{13/04/12/12:19}
\langle L_{\omega,-}u,u \rangle_{H^{-1},H^{1}}\ge \delta_{\omega} \| u \|_{H^{1}}^{2}.
\end{equation}
\end{lemma}
\begin{proof}[Proof of Lemma \ref{13/03/18/15:50}]
We can prove the claim in a way similar to \cite{Weinstein1}.
\end{proof}

Differentiating the equation \eqref{12/03/23/18:09} for $\Phi_{\omega}$ with respect to $x_{j}$, $1\le j \le d$, we have  
\begin{equation}\label{13/02/03/10:31}
\begin{split}
0
=\partial_{j}(\omega \Phi_{\omega}-\Delta\Phi_{\omega} - \Phi_{\omega}^{p}- \Phi_{\omega}^{2^{*}-1} )
=
L_{\omega,+}\partial_{j}\Phi_{\omega}.
\end{split}
\end{equation}

\begin{lemma}\label{12/12/31/22:31}
There exists $\omega_{3} >0$ such that for any $\omega \in (0, \omega_{3})$, 
 we have 
\begin{equation}\label{15/01/29/11:54}
{\rm Ker}\, L_{\omega,+}
=
{\rm span}\{ \partial_{1}\Phi_{\omega},\ldots, \partial_{d}\Phi_{\omega}\}
.
\end{equation}
\end{lemma}
\begin{proof}[Proof of Lemma \ref{12/12/31/22:31}]
We can prove this lemma by regarding the operator $L_{\omega,+}$ as a perturbation of $L_{\omega}^{\dagger}$ (see \cite{Grillakis} and Lemma \ref{13/03/36/15:35}). 
\end{proof}

Now, we introduce the functional $\mathcal{N}_{\omega}$ as  
\begin{equation}\label{13/04/14/17:06}
\mathcal{N}_{\omega}(u):=\omega \|u\|_{L^{2}}^{2}+\|\nabla u \|_{L^{2}}^{2}
-\|u\|_{L^{p+1}}^{p+1}-\|u \|_{L^{2^{*}}}^{2^{*}}. 
\end{equation}
Then, we can verify that  
\begin{equation}\label{13/04/14/17:08}
m_{\omega}=\inf\left\{ \mathcal{S}_{\omega}(u) \colon 
u\in H^{1}(\mathbb{R}^{d})\setminus \{0\} , \ 
\mathcal{N}_{\omega}(u)=0
\right\}.
\end{equation}

\begin{lemma}\label{13/01/01/16:58}
For any $u \in H^{1}(\mathbb{R}^{d})$ satisfying  
\begin{equation}\label{13/01/01/21:13}
\bigm( (p-1)\Phi_{\omega}^{p}
+(2^{*}-2)\Phi_{\omega}^{2^{*}-1},
\,  u \bigm)_{L_{real}^{2}}
=0,
\end{equation}
we have  
\begin{equation}\label{13/04/14/13:34}
\big\langle L_{\omega,+}u, u \big\rangle_{H^{-1},H^{1}} \ge 0.
\end{equation}
\end{lemma}
\begin{proof}[Proof of Lemma \ref{13/01/01/16:58}]
The proof is similar to Lemma 2.3 in \cite{Nakanishi-Schlag1}.  
\end{proof}

\begin{lemma}\label{13/01/01/16:22}
As an operator in $L_{rad}^{2}(\mathbb{R}^{d})$, $L_{\omega,+}$ has only one negative eigenvalue which is non-degenerate and $0$ is not an eigenvalue. 
\end{lemma}
\begin{proof}[Proof of Lemma \ref{13/01/01/16:22}]
We can prove this lemma in a way similar to \cite{Weinstein1} and \cite{Grillakis}. 
\end{proof}

\begin{lemma}\label{13/01/02/12:08}
Let $\omega_{2}$ be the frequency given by Proposition \ref{13/02/20/9:41}. Then, there exists $\omega_{0} \in (0,\omega_{2})$ such that the following holds for  all $0<\omega <\omega_{0}$: Let $\mu>0$ be a positive eigenvalue of $i\mathscr{L}_{\omega}$ as an operator in $L_{real}^{2}(\mathbb{R}^{d})$, and let  $\mathscr{U}_{+}$ be a corresponding eigenfunction. Then, we have:
\\
{\rm (i)} For any non-trivial, real-valued radial function $g \in H^{1}(\mathbb{R}^{d})$ with $(g,f_{2})_{L_{real}^{2}}=0$, 
\begin{equation}
\label{13/01/02/12:11}
\langle L_{\omega,+}g,g \rangle_{H^{-1},H^{1}} \sim \|g\|_{H^{1}}^{2},
\end{equation}
where the implicit constant may depend on $\omega$.
\\[6pt]
{\rm (ii)} For any non-trivial, real-valued radial function $g \in H^{1}(\mathbb{R}^{d})$ with $(g,\partial_{\omega}\Phi_{\omega})_{L_{real}^{2}}=0$,  
\begin{equation}\label{13/01/02/12:12}
\langle L_{\omega,-}g,g \rangle_{H^{-1},H^{1}} \sim \|g\|_{H^{1}}^{2},
\end{equation}
where the implicit constant may depend on $\omega$. 
\end{lemma}
\begin{proof}[Proof of Lemma \ref{13/01/02/12:08}]
We prove the claim {\rm (i)}. 
\par 
First, we shall show that for any $0< \omega <\omega_{2}$ and any nontrivial, radial real-valued  function $g \in H^{1}(\mathbb{R}^{d})$ satisfying $(g,f_{2})_{L_{real}^{2}}=0$, 
\begin{equation}\label{14/09/03/10:32}
\langle L_{\omega,+}g,g \rangle_{H^{-1},H^{1}}>0.
\end{equation}
Suppose for contradiction that there exists  $\omega>0$ and a nontrivial radial real-valued  function $g_{-} \in H^{1}(\mathbb{R}^{d})$ such that 
\begin{align}
\label{14/09/03/10:37}
&(g_{-},f_{2})_{L_{real}^{2}}=0, 
\\[6pt]
\label{14/09/03/10:38}
&\langle L_{\omega,+}g_{-}, g_{-} \rangle_{H^{-1},H^{1}}
\le 0.
\end{align}
Since $L_{\omega,+}$ is self-adjoint in  $L_{real}^{2}(\mathbb{R}^{d})$, we see from \eqref{13/01/02/16:25} and \eqref{14/09/03/10:37} that  
\begin{equation}\label{13/01/02/16:51}
\langle L_{\omega,+}g_{-},f_{1} \rangle_{H^{-1},H^{1}}
=
\big( g_{-},L_{\omega, +}f_{1} \big)_{L_{real}^{2}}
=
-\mu \big( g_{-}, f_{2} \big)_{L_{real}^{2}}
=0. 
\end{equation}
Moreover, it follows from \eqref{13/01/02/16:25} and \eqref{13/03/02/9:56} that \begin{equation}\label{13/03/02/10:12}
\langle L_{\omega,+}f_{1}, f_{1} \rangle_{H^{-1},H^{1}}
= 
-\big(\mu f_{2}, f_{1}\big)_{L_{real}^{2}}
=
-\mu \big( f_{1}, f_{2} \big)_{L_{real}^{2}} 
<0.
\end{equation}
Note here that \eqref{13/01/02/16:51} and \eqref{13/03/02/10:12} show that $g_{-}$ and $f_{1}$ are linearly independent in $L_{real}^{2}(\mathbb{R}^{d})$. 
We see from the hypothesis \eqref{14/09/03/10:38} and \eqref{13/01/02/16:51} that for any $a,b \in \mathbb{R}$, 
\begin{equation}\label{13/03/02/10:43}
\begin{split}
&\langle 
L_{\omega,+}\Bigm(a f_{1}+b \frac{g_{-}}{\|g_{-}\|_{L^{2}}} \Bigm),
\, 
\Bigm(af_{1}+b \frac{g_{-}}{\|g_{-}\|_{L^{2}}}\Bigm)
\rangle_{H^{-1},H^{1}}
\\[6pt]
&=
a^{2}
\big(L_{\omega,+}f_{1}, f_{1} \big)_{L_{real}^{2}}
+
\frac{b^{2}}{\|g_{-}\|_{L^{2}}^{2}}\langle L_{\omega,+}g_{-}, g_{-} \rangle_{H^{-1},H^{1}}
+
\frac{2ab}{\|g_{-}\|_{L^{2}}} \langle L_{\omega,+}g_{-}, f_{1} \rangle_{H^{-1},H^{1}}
\\[6pt]
&\le 
a^{2}\big(L_{\omega,+}f_{1}, f_{1} \big)_{L_{real}^{2}} \le 0.
\end{split}
\end{equation}
Put 
\begin{equation}
\label{14/09/03/17:45}
e_{1}:=\frac{g_{-}}{\|g_{-}\|_{L^{2}}},
\quad 
e_{2}:=\frac{ f_{1}-(f_{1},e_{1})_{L_{real}^{2}}e_{1}}{\big\| f_{1}-(f_{1},e_{1})_{L_{real}^{2}}e_{1} \big\|_{L^{2}}}.
\end{equation}
Then, it is clear that $\|e_{1}\|_{L^{2}}=\|e_{2}\|_{L^{2}}=1$ and $(e_{1},e_{2})_{L_{real}^{2}}=0$. Moreover, we find from \eqref{13/03/02/10:43} that for any $\alpha,\beta \in \mathbb{R}$, 
\begin{equation}\label{14/09/03/17:53}
\langle L_{\omega,+}\big( \alpha e_{1}+\beta e_{2}\big),\,  \alpha e_{1} +\beta e_{2} \rangle_{H^{-1},H^{1}}
\le 0. 
\end{equation}
Since Weyl's essential spectrum theorem shows $\sigma_{ess}(L_{\omega,+})=\sigma_{ess}(-\Delta+\omega^{2})=[\omega^{2},\infty)$, the min-max theorem (see Theorem 12.1 in \cite{Lieb-Loss}) together with \eqref{14/09/03/17:53} shows  that the second eigenvalue of $L_{\omega,+}$ is non-positive. However, this contradicts the fact that $L_{\omega,+}$ has only one non-positive eigenvalue as an operator in $L_{rad}^{2}(\mathbb{R}^{d})$ (see Lemma \ref{13/01/01/16:22}). Thus, we have proved the claim \eqref{14/09/03/10:32}.
\par 
Now, we are in a position to prove the claim {\rm (i)}. It follows from \eqref{13/02/19/10:42} and \eqref{13/03/03/10:05} that there exists $\omega_{0}>0$ such that for any $\omega \in (0,\omega_{0})$ and any $g \in H^{1}(\mathbb{R}^{d})$,
\begin{equation}\label{13/12/29/15:55}
\langle L_{\omega,+}g,g \rangle_{H^{-1},H^{1}} 
\lesssim 
\big(
\omega
+1
+\omega^{\frac{p-1}{p+1}(1-s_{p})}\|U\|_{L_{p+1}}^{p-1}
+
\omega^{\frac{4}{(d-2)(p-1)}-1}\|U\|_{L_{2^{*}}}^{2^{*}-2}
 \big) \|g\|_{H^{1}}^{2}.
\end{equation}  
Thus, it suffices to show that  
\begin{equation}\label{13/03/02/10:37}
\inf_{{g \in H_{rad}^{1}(\mathbb{R}^{d},\mathbb{R})}
\atop 
{ {\|g\|_{H^{1}}=1} \atop  {(g,f_{2})_{L_{real}^{2}}=0}}
} \langle L_{\omega,+} g, \, g \rangle_{H^{-1},H^{1}} >0,
\end{equation} 
where $H_{rad}^{1}(\mathbb{R}^{d},\mathbb{R})$ denotes the space of radial real-valued functions in $H^{1}(\mathbb{R}^{d})$. Suppose for a contradiction that \eqref{13/03/02/10:37} failed. Then, we could take a sequence $\{g_{n}\}$ of radial real-valued  functions in $H^{1}(\mathbb{R}^{d})$ such that 
\begin{align}
\label{14/09/03/12:02}
&\|g_{n}\|_{H^{1}}\equiv 1, 
\\[6pt]
\label{14/09/03/10:16}
&(g_{n},f_{2})_{L_{real}^{2}}\equiv 0, 
\\[6pt]
\label{13/03/02/10:15}
&
\lim_{n\to \infty}\langle L_{\omega,+}g_{n}, g_{n} \rangle_{H^{-1},H^{1}}
=0.
\end{align}
Here, passing to some subsequence, we may assume that there exists a radial real-valued function $g_{\infty}\in H^{1}(\mathbb{R}^{d})$ such that 
\begin{equation}\label{14/09/03/12:05}
\lim_{n\to \infty}g_{n}=g_{\infty}
\qquad 
\mbox{weakly in $H^{1}(\mathbb{R}^{d})$}.
\end{equation}  
If $g_{\infty}=0$, then we see from \eqref{13/03/02/10:15} and \eqref{14/09/03/12:05} that for any $\omega \in (0,1)$, 
\begin{equation}\label{14/09/03/12:13}
\begin{split}
0&=
\lim_{n\to \infty}\langle L_{\omega,+}g_{n}, g_{n} \rangle_{H^{-1},H^{1}}
=
\lim_{n\to \infty}\big( L_{\omega,+}g_{n}, g_{n} \big)_{L_{real}^{2}}
\\[6pt]
&=
\lim_{n\to \infty}
\bigg\{
\omega \int_{\mathbb{R}^{d}}\big| g_{n}\big|^{2}
+
\int_{\mathbb{R}^{d}}\big| \nabla g_{n}\big|^{2}
-
p \int_{\mathbb{R}^{d}}\Phi_{\omega}^{p-1}\big| g_{n}\big|^{2}
-
(2^{*}-1)\int_{\mathbb{R}^{d}}\Phi_{\omega}^{\frac{4}{d-2}}\big| g_{n}\big|^{2}\bigg\}
\\[6pt]
&\ge \omega
-\lim_{n\to \infty}
\bigg\{
p \int_{\mathbb{R}^{d}}\Phi_{\omega}^{p-1}\big| g_{n}\big|^{2}
+(2^{*}-1)\int_{\mathbb{R}^{d}}\Phi_{\omega}^{\frac{4}{d-2}}
\big| g_{n}\big|^{2}
\bigg\}
\\[6pt]
&=\omega,
\end{split}
\end{equation}
which is a contradiction. Thus, we have found that $g_{\infty}$ is nontrivial. Then, it follows from \eqref{13/03/02/10:15}, the lower semicontinuity and \eqref{14/09/03/10:32} that 
\begin{equation}\label{14/09/03/12:27}
0=
\lim_{n\to \infty}\langle L_{\omega,+}g_{n}, g_{n} \rangle_{H^{-1},H^{1}}
\ge 
\langle L_{\omega,+}g_{\infty}, g_{\infty} \rangle_{H^{-1},H^{1}}
>0.
\end{equation}
This absurd conclusion comes from the hypothesis that \eqref{13/03/02/10:37} failed.  Thus, we have prove the claim {\rm (i)}.
\par 
Next, we prove the claim {\rm (ii)}. Let $g \in H^{1}(\mathbb{R}^{d})$ be a non-trivial, real-valued, radial function with $(g,\partial_{\omega}\Phi_{\omega})_{L_{real}^{2}}=0$. We write $g$ in the form 
\begin{equation}\label{13/03/02/12:04}
g=a\Phi_{\omega}+h,
\end{equation}
where $h$ satisfies $(h,\Phi_{\omega})_{L_{real}^{2}}=0$. Note that the condition $(g,\partial_{\omega}\Phi_{\omega})_{L_{real}^{2}}=0$ implies that 
\begin{equation}\label{13/03/02/12:07}
a=-\frac{\big(h ,\partial_{\omega}\Phi_{\omega}\big)_{L_{real}^{2}}}{\big(\Phi_{\omega}, \partial_{\omega}\Phi_{\omega}\big)_{L_{real}^{2}}}.
\end{equation}
Hence, we find from $(h,\Phi_{\omega})_{L_{real}^{2}}=0$ that     
\begin{equation}\label{13/03/02/16:01}
\begin{split}
&\|g\|_{H^{1}}^{2}
= \big( g, g \big)_{L_{real}^{2}}
+
\big( \nabla g, \nabla  g \big)_{L_{real}^{2}}
\\[6pt]
&=a^{2}\|\Phi_{\omega}\|_{H^{1}}^{2}+\|h\|_{H^{1}}^{2}
- 
2a \big( \nabla \Phi_{\omega}, \nabla h \big)_{L_{real}^{2}}
\\[6pt]
&\le  
\frac{\|\partial_{\omega}\Phi_{\omega}\|_{L^{2}}^{2}\|h\|_{L^{2}}^{2}}{\big| (\Phi_{\omega},\partial_{\omega}\Phi_{\omega})_{L_{real}^{2}}\big|^{2}}\|\Phi_{\omega}\|_{H^{1}}^{2}
+
\|h\|_{H^{1}}^{2}
+
\frac{2\|\partial_{\omega}\Phi_{\omega}\|_{L^{2}}\|h\|_{L^{2}}}{\big| (\Phi_{\omega},\partial_{\omega}\Phi_{\omega})_{L_{real}^{2}}\big|} \| \nabla \Phi_{\omega}\|_{L^{2}}\|\nabla h\|_{L^{2}}
\\[6pt]
&\lesssim \|h\|_{H^{1}}^{2}.
\end{split}
\end{equation}
Moreover, we see from $L_{\omega,-}\Phi_{\omega}=0$ and  Lemma \ref{13/03/18/15:50}, that  
\begin{equation}\label{13/03/02/17:21}
\begin{split}
\langle L_{\omega,-}g, g \rangle_{H^{-1},H^{1}}
&=
\langle L_{\omega,-}( a\Phi_{\omega}+h ),\,   a\Phi_{\omega}+h  
\rangle_{H^{-1},H^{1}}
\\[6pt]
&=
\langle L_{\omega,-}h,  h  \rangle_{H^{-1},H^{1}}
\gtrsim \|h\|_{H^{1}}^{2}.
\end{split}
\end{equation}
Putting \eqref{13/03/02/16:01} and \eqref{13/03/02/17:21} together, we obtain
\begin{equation}\label{13/03/02/17:24}
\|g\|_{H^{1}}^{2}\lesssim 
\langle L_{\omega,-}g, g \rangle_{H^{-1},H^{1}}.
\end{equation}
The opposite relation follows from \eqref{13/02/19/10:42}. Hence, we have completed the proof. 
\end{proof}

%%%%%%%%%%%%%%%%%%%%%%%%%%%%%%%%%%%%%%%%%%%%%%%%%%%%%%%%%%%%%%%%%%%%%%%%%%%%%%%%
\section{Inequalities for the radial functions}\label{13/12/29/11:25}

We see from the fundamental theorem of calculus and Hardy's inequality that 
 for any $d\ge 3$ and any radial function $g \in H^{1}(\mathbb{R}^{d})$,
\begin{equation}\label{13/12/26/21:32}
\begin{split}
&\sup_{x\in \mathbb{R}^{d}}\{ |x|^{d-2}|g(x)|^{2}\}
=
\sup_{x\in \mathbb{R}^{d}}\int_{|x|}^{\infty} \frac{d}{dr}\bigg\{- r^{d-2}|g(r)|^{2} \bigg\} \,dr 
\\[6pt]
&\lesssim \int_{0}^{\infty} r^{\frac{d-3}{2}}|g(r)| r^{\frac{d-1}{2}}|g'(r)|
\,dr 
\\[6pt]
&\le 
\bigg( 
\int_{0}^{\infty} |g(r)|^{2} r^{d-3}\,dr 
\bigg)^{\frac{1}{2}} 
\bigg(\int_{0}^{\infty}
|g'(r)|^{2}r^{d-1}
\,dr 
\bigg)^{\frac{1}{2}}
\lesssim 
\|\nabla g \|_{L^{2}(\mathbb{R}^{d})}^{2}
.
\end{split}
\end{equation}
Similarly, we can verify that for any $d\ge 3$ and any radial function $g \in H^{1}(\mathbb{R}^{d})$,  
\begin{equation}\label{14/01/18/14:30}
\begin{split}
\sup_{x\in \mathbb{R}^{d}}\big\{ |x|^{d-1}|g(x)|^{2}\big\}
&\lesssim 
\int_{0}^{\infty}  r^{\frac{d-1}{2}}|g(r)| r^{\frac{d-1}{2}}|g'(r)|\,dr 
\lesssim 
\| g \|_{L^{2}(\mathbb{R}^{d})} \|\nabla g \|_{L^{2}(\mathbb{R}^{d})}
.
\end{split}
\end{equation}

%%%%%%%%%%%%%%%%%%%%%%%%%%%%%%%%%%%%%%%%%%%%%%%%%%%%%%%%%%%%%%%%%%%%%%%%%%%%%%%%
\section{Small-data theory}\label{14/01/30/10:18}

We record standard small-data theories for \eqref{12/03/23/17:57} and \eqref{14/01/31/11:03} (see Lemma \ref{14/06/05/10:08} below and Corollary 3.9 in \cite{Killip-Visan-book}):
\begin{lemma}\label{14/01/30/10:19} 
{\rm (i)}~Assume $d\ge 3$ and $1+\frac{4}{d}<p< 2^{*}-1$. Then, for any $A>0$, there exists $\delta(A)>0$ with the following property: for any $t_{0}\in \mathbb{R}$ and any $\psi_{0}\in H^{1}(\mathbb{R}^{d})$ satisfying       
\begin{align}
\label{17/11/25/16:33}
&\|\langle \nabla \rangle \psi_{0}\|_{L^{2}}\le A,
\\[6pt]
\label{17/11/25/17:25}
&\| e^{i(t-t_{0})\Delta} \psi_{0}\|_{W_{2+\frac{4}{d}}(\mathbb{R})\cap W_{2^{*}}(\mathbb{R})}\le \delta(A),
\end{align} 
there exists a unique strong $H^{1}$-solution $\psi$ to \eqref{12/03/23/17:57} with $\psi(t_{0})=\psi_{0}$. Furthermore, the maximal existence-interval of $\psi$ is the whole of $\mathbb{R}$, and $\psi$ satisfies  
\begin{equation}\label{14/01/30/12:12}
\| \langle \nabla \rangle \psi \|_{St(\mathbb{R})}\lesssim 
\|\langle \nabla \rangle \psi_{0}\|_{L^{2}}. 
\end{equation} 
\noindent 
{\rm (ii)}~Assume $d\ge 3$. Then, there exists $\delta_{0}>0$ with the following property: for any $t_{0}\in \mathbb{R}$ and any $\psi_{0} \in \dot{H}^{1}(\mathbb{R}^{d})$ satisfying     
\begin{equation}
\label{17/11/03/21:27}
\|\nabla \psi_{0}\|_{L^{2}}\le \delta_{0}, 
\end{equation} 
there exists a unique strong $\dot{H}^{1}$-solution $\psi$ to \eqref{14/01/31/11:03} with $\psi(t_{0})=\psi_{0}$. 
 Furthermore, the maximal existence-interval of $\psi$ is the whole of $\mathbb{R}$ and 
\begin{equation}\label{14/01/30/12:12}
\|  \nabla  \psi \|_{St(\mathbb{R})}
\lesssim 
\|\nabla \psi_{0}\|_{L^{2}}. 
\end{equation} 
\end{lemma}

We can control the norm of the full Strichartz space by a few particular norms:  
\begin{lemma}\label{14/06/11/11:16}
Assume $d\ge 3$, and let $A_{1},A_{2},A_{3}>0$. Then, there exists a constant $C(A_{1},A_{2},A_{3})>0$ with the following property: 
\\
{\rm (i)} for any interval $I$ and any space-time function $u$ satisfying  
\begin{align}
\label{11/05/14/14:46}
&\left\| u \right\|_{L^{\infty}(I,H^{1})}\le A_{1},
\\[6pt] 
\label{15/11/15/15:57}
&\left\| u \right\|_{W_{p+1}(I) \cap W_{2^{*}}(I)}\le A_{2},
\\[6pt]
\label{11/05/27/4:09}
&\big\|\langle \nabla \rangle e[u] \big\|_{L_{t,x}^{\frac{2(d+2)}{d+4}}(I)} 
\le A_{3},
\end{align} 
we have 
\begin{equation}\label{15/11/15/17:07}
\left\| \langle \nabla \rangle u \right\|_{St(I)}
\le C(A_{1},A_{2},A_{3}).
\end{equation}
{\rm (ii)} for any interval $I$ and any space-time function $u$ satisfying  
\begin{align}
\label{11/06/14/16:23}
&\left\| \nabla u \right\|_{L^{\infty}(I,L^{2})}\le A_{1},
\\[6pt] 
&\left\| u \right\|_{W_{2^{*}}(I)}\le A_{2},
\\[6pt]
\label{11/06/14/16:24}
&\big\|\nabla  e^{\ddagger}[u] \big\|_{L_{t,x}^{\frac{2(d+2)}{d+4}}(I)} 
\le A_{3},
\end{align} 
we have 
\begin{equation}\label{15/11/15/17:05}
\left\| \nabla  u \right\|_{St(I)}
\le C(A_{1},A_{2},A_{3}).
\end{equation}

\end{lemma}
\begin{proof}[Proof of Lemma \ref{14/06/11/11:16}]
See the proof of Lemma 4.1 in \cite{AIKN2}. 
\end{proof}

%%%%%%%%%%%%%%%%%%%%%%%%%%%%%%%%%%%%%%%%%%%%%%%%%%%%%%%%%%%%%%%%%%%%%%%%%%%%%%%%
\section{Long-time perturbation theory}\label{15/12/09/14:27}

In this section, we give the long-time perturbation theory for a general nonlinear Schr\"odinger equation including \eqref{12/03/23/17:57} with $1+\frac{4}{d}\le p <2^{*}-1$:
\begin{equation}\label{15/12/04/09:01}
i\frac{\partial \psi}{\partial t} +\Delta \psi +f(\psi)=0
,
\end{equation}
where $f\colon \mathbb{C}\to \mathbb{C}$ is a continuously differentiable function in the real-sense.
\par 
Fix an even, smooth cut-off function $\chi$ defined on $\mathbb{R}$ such that 
 $\chi(r)=1$ if $0\le r \le 1$ and $\chi(r)=0$ if $r\ge 2$, and define 
\begin{align}
\label{15/12/02/10:01}
f_{\le 1}(z)&:=\chi(|z|)f(z),
\\[6pt] 
\label{15/12/02/10:02}
f_{>1}(z)&:=f(z)-f_{\le 1}(z)
.
\end{align} 
Then, our assumptions about the nonlinearity $f$ are the followings: 
\begin{enumerate}
\item
\begin{equation}\label{15/12/02/10:51}
f(0)=\frac{\partial f}{\partial z}(0)=\frac{\partial f}{\partial \bar{z}}(0)=0.
\end{equation}
\item There exists a constant $C_{\le 1}>0$ such that for any $z_{1},z_{2} \in \mathbb{C}$,
\begin{equation}\label{15/12/02/11:06}
\begin{split}
&\Big| \frac{\partial f_{\le 1}}{\partial z} (z_{1}) - \frac{\partial f_{\le 1}}{\partial z} (z_{2}) \Big|
+
\Big| \frac{\partial f_{\le 1}}{\partial \bar{z}} (z_{1}) - \frac{\partial f_{\le 1}}{\partial \bar{z}} (z_{2}) \Big|
\\[6pt]
&\le C_{\le 1} 
\left\{ \begin{array}{ccc} 
|z_{1}-z_{2}| \min\big\{ 10, (|z_{1}|+|z_{2}|)^{\frac{4}{d}-1} \big\} &\mbox{if}& d=3,
\\[6pt]
|z_{1}-z_{2}|^{\frac{4}{d}} &\mbox{if} & d\ge 4.
\end{array} \right.
\end{split}
\end{equation}
\item There exists a constant $C_{>1}>0$ such that for any $z_{1},z_{2} \in \mathbb{C}$,
\begin{equation}\label{15/12/02/11:17}
\begin{split}
&\Big| \frac{\partial f_{> 1}}{\partial z} (z_{1}) - \frac{\partial f_{> 1}}{\partial z} (z_{2}) \Big|
+
\Big| \frac{\partial f_{> 1}}{\partial \bar{z}} (z_{1}) - \frac{\partial f_{> 1}}{\partial \bar{z}} (z_{2}) \Big|
\\[6pt]
&\le C_{>1} 
\left\{ \begin{array}{ccc} 
|z_{1}-z_{2}| (|z_{1}|+|z_{2}|)^{\frac{4}{d-2}-1} &\mbox{if}& 3\le d \le 5,
\\[6pt]
|z_{1}-z_{2}|^{\frac{4}{d-2}} &\mbox{if} & d\ge 6 . 
\end{array} \right.
\end{split}
\end{equation} 
\end{enumerate}

Throughout this section, we allow the implicit constants to depend on $C_{\le 1}$ and $C_{>1}$, so that the assumptions \eqref{15/12/02/10:51} through \eqref{15/12/02/11:17} imply that  
\begin{align}
\label{15/12/02/12:01}
&|f_{\le 1}(z)| 
\lesssim 
|z|^{1+\frac{4}{d}},
\\[6pt]  
\label{15/12/02/12:02}
&|f_{> 1}(z)|
\lesssim 
|z|^{1+\frac{4}{d-2}},
\\[6pt]
\label{15/12/04/09:11}
&
\Big|\frac{\partial f_{\le 1}}{\partial z}(z) \Big|
+
\Big|\frac{\partial f_{\le 1}}{\partial \bar{z}}(z) \Big|
\lesssim |z|^{\frac{4}{d}},
\\[6pt]
\label{15/12/04/09:12}
&
\Big|\frac{\partial f_{> 1}}{\partial z}(z) \Big|
+
\Big|\frac{\partial f_{> 1}}{\partial \bar{z}}(z) \Big|
\lesssim |z|^{\frac{4}{d-2}}
\end{align}
for any $z \in \mathbb{C}$.

Now, we state the long-time perturbation theory for the equation \eqref{15/12/04/09:01}: 
\begin{lemma}\label{14/06/05/10:08}
Let $d\ge 3$, $I$ an interval, $\psi \in C(I,H^{1}(\mathbb{R}^{d}))$ a solution to \eqref{15/12/04/09:01}, and let $u$ be a function in  $C(I,H^{1}(\mathbb{R}^{d}))$. Assume that  
\begin{align}
\label{14/06/05/07:01}
&\| u \|_{L_{t}^{\infty}H_{x}^{1}(I)}\le A, 
\\[6pt]
\label{14/06/05/07:02}
&\| u \|_{V_{2+\frac{4}{d}}(I)\cap W_{2^{*}}(I)} \le B,
\\[6pt]
\label{14/06/08/07:03}
&\|  \psi(t_{0})-u(t_{0}) \|_{H^{1}}
\le \gamma
\end{align} 
for some $t_{0}\in I$ and some constants $A>0$, $B>0$ and $\gamma >0$. Then, there exist constants $\delta_{0}(A,B,\gamma)>0$ and $C(A,B,\gamma)>0$ such that if\begin{align}
\label{15/12/07/18:26}
&\| e^{i(t-t_{0})\Delta} \big\{ \psi(t_{0})-u(t_{0})
\big\} \|_{V_{2+\frac{4}{d}}(I)\cap W_{2^{*}}(I)}
\le \delta,
\\[6pt]
\label{15/12/07/18:27}
&\| \langle \nabla \rangle e(u) \|_{L_{t,x}^{\frac{2(d+2)}{d+4}}(I)}
\le \delta 
\end{align}
for some  $0<\delta <\delta_{0}(A,B,\gamma)$, then we have 
\begin{align}
\label{14/06/05/07:07}
&\|  \psi-u \|_{V_{2+\frac{4}{d}}(I)\cap W_{2^{*}}(I)}
\le C(A,B,\gamma) \delta^{\frac{1}{d(d+2)}}  ,
\\[6pt]
\label{15/12/08/21:07}
&\|\langle \nabla \rangle \{ \psi-u \} \|_{St(I)}
\le C(A,B,\gamma) \gamma ,
\\[6pt]
\label{14/06/05/07:06}
&\|\langle \nabla \rangle \psi \|_{St(I)}\le C(A,B,\gamma)
.
\end{align}
Here, the constant $C(A,B,\gamma)$ is non-decreasing with respect to each of the parameters $A$, $B$ and $\gamma$.
\end{lemma}

We can prove Lemma \ref{14/06/05/10:08} in a way similar to the proof of 
  Theorem 3.8 in \cite{Killip-Visan-book}. The main difference of the proof is that we need another exotic Strichartz estimate. For example, the following 
 exotic Strichartz norms work well: 
\begin{align}
\label{11/05/21/12:00}
\| u \|_{ES(I)}
&:=
\| \langle \nabla \rangle^{\frac{4}{d+2}} u 
\|_{L_{t}^{\frac{d(d+2)}{2(d-2)}}L_{x}^{\frac{2d^{2}(d+2)}{d^{3}-4d+16}}(I)}
+
\| u\|_{L_{t}^{\frac{d+2}{2}}L_{x}^{\frac{2d(d+2)}{(d-2)(d+4)}}(I)}
, 
\\[6pt] 
\label{11/05/21/14:05}
\| u \|_{ES^{*}(I)}
&:=
\| \langle \nabla \rangle^{\frac{4}{d+2}}u 
\|_{L_{t}^{\frac{d}{2}}L_{x}^{\frac{2d^{2}(d+2)}{d^{3}+4d^{2}+4d-16}}(I)}
+
\| u\|_{L_{t}^{\frac{d(d+2)}{2(d+4)}}L_{x}^{\frac{2d^{2}(d+2)}{(d+4)^{2}(d-2)}}(I)}
,
\end{align}
where $I$ is an interval. In particular, we have the following exotic Strichartz estimate. 
\begin{lemma}[Exotic Strichartz estimate]
\label{10/10/16/22:32}
Let $d\ge 3$. Then, we have  
\begin{equation}\label{10/10/16/15:57}
\| u - e^{i(t-t_{0})\Delta}u(t_{0}) \|_{ES(I)} 
\lesssim 
\| i\frac{\partial u}{\partial t}+\Delta u -g\|_{ES^{*}(I)}
+
\|\langle \nabla \rangle g \|_{L_{t,x}^{\frac{2(d+2)}{d+4}}(I)}
\end{equation}
for any interval $I$, any $t_{0} \in I$, and any functions $u$ and $g$ defined on $\mathbb{R}^{d}\times I$.
\end{lemma}
Moreover, we can derive an analogue of Lemma 3.11 in \cite{Killip-Visan-book}: 
\begin{lemma}\label{15/10/23/11:30}
Let $d\ge 3$. Then, we have 
\begin{equation}
\label{15/12/06/16:29}
\|u\|_{ES(I)}
\lesssim
\| u \|_{L_{t,x}^{\frac{2(d+2)}{d-2}}(I)}^{\frac{1}{d+2}}
\| \nabla u \|_{St(I)}^{\frac{d+1}{d+2}}
+
\| u \|_{L_{t,x}^{\frac{2(d+2)}{d}}(I)}^{\frac{16(d-2)}{d^{2}(d+2)}}
\|\langle \nabla \rangle u\|_{St(I)}^{\frac{d^{3}+2d^{2}-16d+32}{d^{2}(d+2)}}
\end{equation}
for any interval $I$ and any function $u$ on $\mathbb{R}^{d}\times I$. In particular, we have 
\begin{equation}\label{11/06/02/22:24}
\| u \|_{ES(I)}
\lesssim 
\| \langle \nabla \rangle u 
\|_{St(I)}.
\end{equation}
Moreover, there exist constants $\theta_{0},\theta_{1} \in (0, 1]$ depending only on $d$ such that  
\begin{align}
\label{15/10/23/11:37}
\|u\|_{L_{t,x}^{\frac{2(d+2)}{d}}(I)}
&\le 
\| u\|_{St(I)}^{1-\theta_{0}}
\| u \|_{ES(I)}^{\theta_{0}},
\\[6pt] 
\label{15/10/23/11:51}
\|u\|_{L_{t,x}^{\frac{2(d+2)}{d-2}}(I)}
&\le 
\| \nabla u\|_{St(I)}^{1-\theta_{1}}
\| u \|_{ES(I)}^{\theta_{1}}
\end{align}
for any interval $I$ and any function $u$ on $\mathbb{R}^{d}\times I$.
\end{lemma}

%%%%%%%%%%%%%%%%%%%%%%%%%%%%%%%%%%%%%%%%%%%%%%%%%%%%%%%%%%%%%%%%%%%%%%%%%%%%%%%

\newpage 

\section{Table of notation}\label{17/12/09/16:47}

%%%%%%%%%%%%%%%%%%%%%%%%%%%%%%%%%%%%%%%%%%%%%%%%%%%%%%%%%%%%%%%%%%%%%%%%%%%%%%%%
\begin{center}
\begin{tabular}{|c|c|}
\hline
Symbols 
& Description or equation number 
\\ \hline
$\mathcal{M}$ 
& \eqref{13/03/30/2:01}
\\
$\mathcal{H}, \mathcal{H}^{\dagger}, \mathcal{H}^{\ddagger}, $ 
& \eqref{13/03/30/2:02},  \eqref{15/05/06/16:50}, \eqref{14/02/01/17:08}
\\
$\mathcal{S}_{\omega}, \mathcal{S}^{\dagger}_{\omega}, \mathcal{S}^{j}_{\omega}$ 
& \eqref{13/03/30/2:03}, \eqref{13/03/05/23:59}, \eqref{14/02/06/09:41} 
\\
$\mathcal{K}, \widetilde{K}_{\omega}, \mathcal{K}^{\dagger}, \mathcal{K}^{\ddagger}$ 
& \eqref{13/03/30/12:13}, \eqref{13/03/17/16:24}, \eqref{13/03/06/0:11}, \eqref{14/02/01/17:06} 
\\
$\mathcal{I}_{\omega}, \widetilde{\mathcal{I}}_{\omega}, \mathcal{I}^{\ddagger}, 
\mathcal{I}^{j}_{\omega}$   
& \eqref{13/03/30/12:14}, \eqref{13/03/17/16:25}, \eqref{14/09/20/13:55}, 
\eqref{14/02/06/09:41}  
\\
$\mathcal{J}_{\omega}$
& \eqref{13/03/08/10:38} 
\\
$\Phi_{\omega}, U$ 
& positive, radial ground states of \eqref{12/03/23/18:09} and \eqref{15/01/29/09:41} 
\\
$W$ 
& \eqref{13/03/16/16:43} 
\\
$m_{\omega}, \sigma, \nu_{\omega}$ 
& \eqref{12/03/24/17:51}, \eqref{12/03/24/17:52}, \eqref{14/06/27/16:27} 
\\
$\widetilde{PW}_{\omega}, PW_{\omega} , PW_{\omega,+}, PW_{\omega,-}$
& \eqref{15/03/22/11:23}, \eqref{15/03/21/11:40}, \eqref{13/02/15/15:32}, \eqref{15/01/24/00:01}
\\
$PW^{\dagger}, PW_{\omega}^{\dagger}, PW^{\dagger,\varepsilon}$ 
& \eqref{15/05/06/16:35}, \eqref{15/05/06/18:09}, \eqref{15/05/06/17:43} 
\\
$PW_{+}^{\ddagger}$
& \eqref{14/02/01/16:53}
\\
$T_{\omega}$
& \eqref{13/03/30/12:20}
\\
$s_{p}$ 
& \eqref{13/03/08/10:39}
\\
$\mathcal{L}_{\omega}, L_{\omega, +}, L_{\omega, -}$ 
& \eqref{15/02/02/10:25}, \eqref{13/02/02/17:56}, \eqref{13/02/02/17:57} 
\\
$L^{\dagger}_{+}, L^{\dagger}_{-}$ 
& \eqref{13/03/30/12:01}, \eqref{13/04/12/12:14}
\\
$\widetilde{L}_{\omega, +}, \widetilde{L}_{\omega, -}$ 
& \eqref{13/03/30/11:40}, \eqref{13/04/12/14:40} 
\\
$\mu$ 
& positive eigen-value of $- i \mathcal{L}_{\omega}$
\\
$\mathscr{U}_{\pm}$
& eigen-function of $- i \mathcal{L}_{\omega}$ corresponding to $\pm \mu$
\\
$\lambda_{\pm}(t)$ 
& \eqref{13/05/04/9:35} 
\\
$\Gamma(t)$ 
& \eqref{13/05/04/14:30}
\\
$\lambda_{1}(t), \lambda_{2}(t)$
& \eqref{13/05/07/16:17}
\\
$d_{\omega}(\psi(t)), \widetilde{d}_{\omega}(\psi(t))$ 
& \eqref{13/05/05/13:27}, \eqref{13/05/14/21:30} 
\\
$\delta_{E}, \delta_{X}, \delta_{S}, \delta_{*}$ 
& \eqref{13/05/05/13:33}, \eqref{15/02/08/22:02}, 
\eqref{15/07/11/11:18}, \eqref{15/02/17/10:01}
\\
$T_{X}$ 
& \eqref{13/06/29/10:51}
\\
$\kappa_{1}(\delta)$ 
& \eqref{13/05/14/17:21}
\\
$A_{\omega}^{\varepsilon}, S_{\omega,R}^{\varepsilon}, S_{\omega,R,\pm}^{\varepsilon}$ 
& \eqref{13/03/02/20:21}, \eqref{14/01/23/11:24}, \eqref{14/01/23/11:58} 
\\
$St(I), V_{q}(I), V(I), W_{q}(I), W(I), W^{j}(I)$ 
& \eqref{14/01/30/12:07}, \eqref{14/01/30/11:19}, \eqref{14/02/08/16:35}
\\
$G^{j}_{n}, g^{j}_{n}$ 
& \eqref{15/07/11/13:00}, \eqref{14/01/26/14:02} 
\\
$\sigma^{j}_{n}, \sigma^{j}_{\infty}$ 
& \eqref{14/01/26/14:14}, \eqref{14/01/26/14:40}
\\
$x^{j}_{n}, t^{j}_{n}, \lambda^{j}_{n}$ 
&
\eqref{14/01/26/11:46}, \eqref{14/01/26/11:47}, \eqref{14/01/26/11:48}
\\
$\widetilde{u}^{j}, w^{k}_{n}$ 
& 
\eqref{14/01/26/12:02}, \eqref{14/01/26/11:57}
\\ 
$\widetilde{\psi}^{j}, \psi^{j}_{n}$ 
& \eqref{14/01/26/14:26}, \eqref{14/02/01/10:50}
\\ 
$\psi_{n}^{k\mbox{-}app}$ 
& \eqref{14/02/03/11:10} 
\\
\hline
\end{tabular}
\end{center}

%%%%%%%%%%%%%%%%%%%%%%%%%%%%%%%%%%%%%%%%%%%%%%%%%%%%%%%%%%%%%%%%%%%%%%%%%%%%%%%%
\section*{Acknowledgements}
T.A. is partially supported by the Grant-in-Aid 
for Young Scientists (B) \# 25800077 of JSPS. 
\\
S. I. is partially supported by NSERC Discovery grant \# 371637-2014.
\\
H. K. is partially supported by the Grant-in-Aid for Young Scientists (B) \# 00612277
of JSPS.
\\
H. N. is partially supported by the Grant-in-Aid for Scientific Research (B) 23340030 of JSPS and the Grant-in-Aid for Challenging  Exploratory Research 23654052 of JSPS.

%%%%%%%%%%%%%%%%%%%%%%%%%%%%%%%%%%%%%%%%%%%%%%%%%%%%%%%%%%%%%%%%%%%%%%%%%%%%%%%%
\bibliographystyle{jplain}

\end{document}